\documentclass[a5paper,10pt,twoside,openany]{book}
\usepackage{euscript,amssymb,amsfonts,amsmath,amsthm,mathtext,mathrsfs,mathtools}
\usepackage{relsize}
\usepackage{fancyhdr}
\usepackage[toctitles]{titlesec}
\usepackage{graphics}
\usepackage{mathtext}
\usepackage[dvips]{graphicx}
\usepackage[cp1251]{inputenc}
\usepackage[T2A]{fontenc}
\usepackage[english,russian]{babel}

\usepackage[dvipsnames]{xcolor}
\usepackage{hyperref}
\hypersetup{breaklinks, colorlinks,
    linkcolor = {Blue},
    citecolor = {Blue},
    urlcolor  = {Blue}
}

\usepackage[%
   left=1.9cm,% 1.9
    top=1.5cm,% 1.5 
    right=1.9cm,% 1.9
    bottom=1.5cm,% 1.5
    headsep=0.2cm,%
    footskip=0.5cm,%
    includehead,%
    includefoot]{geometry}
\newtheorem{theorem}{\indent \sc Теорема}[chapter]
\newtheorem{lemma}{\indent \sc Лемма}[chapter]
\newtheorem{corollary}{\indent \sc Следствие}[chapter]
\newtheorem{definition}{\indent \sc Определение}[chapter]

\theoremstyle{remark}
\newtheorem{myremark}{\indent \sc Замечание}[chapter]
\newtheorem{example}{\indent \sc Пример}[chapter]
\newtheorem{excersize}{\indent \sc Упражнение}[chapter]

\makeatletter
\renewcommand{\@oddfoot}{}
\renewcommand{\@evenfoot}{}

\renewcommand{\@pnumwidth}{2em} % для оглавления
\renewcommand{\@tocrmarg}{2.5em} % для оглавления
\makeatother

\newcommand*{\hm}[1]{#1\nobreak \discretionary{}%
{\hbox {$\mathsurround=0pt #1$}}{}}
\newcommand{\abs}[1]{\left|#1\right|}
\newcommand{\norm}[1]{\left\|#1\right\|}
\newcommand{\tvd}{d_{\textsc{tv}}}
\newcommand{\tvnorm}[1]{\norm{#1}_{\textsc{tv}}}
\newcommand{\ud}{\rho}

\newcommand{\smallsum}{\mathsmaller{\sum}\limits}

\renewcommand{\phi}{\varphi}
\newcommand{\eps}{\varepsilon}

\renewcommand{\d}{\delta}
\newcommand{\la}{\lambda}

\newcommand{\R}{\mathbb R}

\newcommand{\N}{\mathbb N}
\newcommand{\Z}{\mathbb Z}
\renewcommand{\C}{\mathbb C}
\newcommand{\E}{{\sf E}}
\newcommand{\D}{{\sf D}}
\renewcommand{\Prob}{{\sf P}}
\newcommand{\I}{{\bf1}}
\newcommand{\sign}{\mathop{\mathrm{sign}}}
\newcommand{\vp}{\mathop{\mathrm{v.p.}}} %valeur principale
\renewcommand{\le}{\leqslant}
\renewcommand{\ge}{\geqslant}
\newcommand{\To}{\longrightarrow}
\newcommand{\eqd}{\stackrel{d}{=}}
\newcommand{\eqp}{\stackrel{\text{п.н.}}{=}}
\newcommand{\gep}{\stackrel{\text{п.н.}}{\ge}}
\newcommand{\lep}{\stackrel{\text{п.н.}}{\le}}
\newcommand{\gp}{\stackrel{\text{п.н.}}{>}}
\newcommand{\lp}{\stackrel{\text{п.н.}}{<}}
\newcommand{\pto}{\stackrel{P}{\To}}
\newcommand{\dto}{\Longrightarrow}
\newcommand{\dd}{\mathrm d}

\newcommand{\bet}{\beta_{2+\delta,\,j}}
\newcommand{\lyap}{L_{3,n}}
\newcommand{\lyapd}{L_{2+\delta,n}}

\newcommand{\Lpoisd}{L_{2+\d,\la}}
\newcommand{\poisbec}{M}

\newcommand{\MP}{\mathrm{MP}}
\newcommand{\NB}{\mathrm{{N\!B}}}
\newcommand{\pr}{\mathbf{p}}
\newcommand{\Bin}{{N_{n,\pr}}}
\newcommand{\Sbin}{{S_{n, \pr}}}
\newcommand{\Spois}{S_{\la}}

\newcommand{\comment}[1]{}%{{\color{blue}{#1}}}

\title{Оценки точности асимптотических вероятностных моделей\\[1cm]}
\author{И.\,Г. Шевцова\\[7cm]}

\date{}

\begin{document}

\abovedisplayshortskip=3pt plus 2pt minus 2pt
\belowdisplayshortskip=3pt plus 1pt minus 2pt

\addtocontents{toc}{\protect\thispagestyle{empty}}
\addtocontents{bbl}{\protect\thispagestyle{empty}}

\maketitle

%\newpage
%\thispagestyle{empty}
%
%\begin{footnotesize}  
%
%Конспект лекций по курсу ``Оценки точности асимптотических вероятностных моделей'' для студентов 5-го курса (первого года обучения в магистратуре) факультета ВМК МГУ имени М.\,В.\,Ломоносова.
%
%\end{footnotesize}

\setcounter{page}3

\chapter*{Предисловие}
%\addcontentsline{toc}{section}{Предисловие}
\thispagestyle{empty}
\markboth{Предисловие}{Предисловие}

Одним из ключевых моментов при построении математических моделей стохастических ситуаций является выбор модели для описания распределения имеющихся данных. При этом у выбранной модели желательно наличие не только хорошего согласия с \textit{эмпирическими} данными, но и некоторого \textit{теоретического} обоснования. Такие теоретические обоснования могут давать разнообразные предельные теоремы теории вероятностей, позволяющие в качестве модельного распределения выбирать его \textit{асимптотическую  аппроксимацию}, возникающую в той или иной предельной теореме.  Однако для выбора адекватной асимптотической аппроксимации необходимо уметь оценивать её точность, то есть скорость сходимости в соответствующей предельной теореме. 

Настоящее пособие как раз и посвящено построению оценок скорости сходимости в предельных теоремах для сумм независимых случайных величин, в том числе~--- в классической центральной предельной теореме, занимающей центральное место в теории вероятностей.  При этом некоторые результаты являются новыми и публикуются впервые, что придает данной книге и характер монографии.

Учебное пособие предназначено для студентов магистратуры, обучающихся по программе ``Статистический анализ и прогнозирование рисков'' факультета ВМК МГУ имени М.\,В.\,Ломоносова. В основу пособия легли лекции, читавшиеся автором на протяжении ряда лет студентам указанной программы.

Разделы 3.4 и 3.5 поддержаны грантом РНФ 18-11-00155, остальная часть книги поддержана грантами \mbox{МД-189.2019.1} и РФФИ 19-07-01220, 20-31-70054.

\begin{flushright}
%\it 
И.\,Г.\,Шевцова

%\smallskip

Москва -- Снегири

%28 декабря 2017\,г.
\today

\end{flushright}

\tableofcontents

%\chapter*{Введение}
%\addcontentsline{toc}{section}{Введение}
%\thispagestyle{empty}
%\markboth{Введение}{Введение}
%
%\input{introduction}
%
%
%
\chapter*{Список обозначений и сокращений}
\addcontentsline{toc}{section}{Список обозначений и сокращений}
\thispagestyle{empty}
\markboth{Список обозначений и сокращений}{Список обозначений и сокращений}
% Список сокращений и условных обозначений

\begin{itemize}

\item
ЦПТ~--- центральная предельная теорема;

\item
и.п.~--- измеримое пространство;

\item
в.п.~--- вероятностное пространство;

\item
с.в.~--- случайная(-ые) величина(-ы);

\item
о.р.(с.в.)~--- одинаково распределённые (случайные величины);

\item
р.р.(с.в.)~--- разнораспределённые (случайные величины), случай
одинакового распределения при этом не исключается;

\item
н.о.р.с.в.~--- независимые одинаково распределённые случайные
величины;

\item
ф.р.~--- функция распределения;

\item
ФОВ~--- функция ограниченной вариации;

\item
х.ф.~--- характеристическая функция;

\item
пр.ф.~--- производящая функция;

\item
п.н.~--- почти наверное (с вероятностью единица);

\item 
БДНЦ~--- безграничная делимость в классе неотрицательных целочисленных случайных величин;

\item $\Box$~--- конец доказательства;

\item
$\C,\R,\R_+,\Z,\N,\N_0$~--- соответственно множества всех комплексных,
действительных, неотрицательных, целых, натуральных и целых неотрицательных чисел;

\item
$\vee$~--- операция максимума, $a\vee b=\max\{a,b\}$;

\item
$\wedge$~--- операция минимума, $a\wedge b=\min\{a,b\}$;

\item
$x^+=x\vee0$, $x\in\R$;

\item
$x^-=(-x)\vee0$, $x\in\R$;

\item
$\sign x$~--- знак вещественного числа $x$;

\item
$\lfloor x\rfloor$~--- наибольшее целое, не превосходящее $x\in\R$;

\item
$\lceil x\rceil$~--- наименьшее целое, не меньшее $x\in\R$;

\item
$\Re\, z$, $\Im\, z$, $z\in\C$~--- соответственно действительная и
мнимая части $z$;

\item
$\overline z$~--- число, комплексно сопряжённое к $z\in\C;$

%\item
%$g(x)\asymp h(x)$, $x\to x_0$,~--- означает, что
%$\displaystyle0<\liminf_{x\to x_0} \frac{g(x)}{h(x)}\le
%\limsup_{x\to x_0}\frac{g(x)}{h(x)} <\infty$, где $g(x)$ и
%$h(x)$~--- положительные функции или последовательности, а $x_0$
%может быть бесконечностью;

\item
$g(x)\sim h(x)$, $x\to x_0$,~--- означает, что $\displaystyle
\lim_{x\to x_0}\frac{g(x)}{h(x)}=1$, где $g(x)$ и $h(x)$~---
положительные функции или последовательности, а $x_0$ может быть
бесконечностью;

\item
$i=\overline{m,n}$~--- $i$ пробегает множество всех целых чисел от
$m$ до $n$ включительно, $m,n\in\Z$, $m\le n$;

\item
$(\Omega,\mathcal{A},\Prob)$~--- вероятностное пространство (в.п.);

\item
$\mathcal B(\Omega)$, где $\Omega\subset\R^d$,~--- $\sigma$-алгебра, порождённая всеми открытыми подмножествами множества $\Omega$ (борелевская $\sigma$-алгебра);\\
$\mathcal B^d\coloneqq\mathcal B(\R^d)$;\\
$\mathcal B\coloneqq\mathcal B(\R)$;

%\item
%$\lambda$~--- мера Лебега, если не сказано иначе;

\item
$\E$~--- математическое ожидание (м.о.);

\item
$\D$~--- дисперсия;

\item
$\I(A)$~--- индикаторная функция события $A$;

\item
$\eqd$~--- совпадение распределений;

\item
$\eqp$, $\gep$, $\lep$, $\gp$, $\lp$ ~--- равенство, неравенства почти наверное  (с вероятностью $1$);

\item
$\dto$~--- слабая сходимость (распределений);

\item
$\displaystyle \Gamma(\alpha)=\int_0^\infty t^{\alpha-1}e^{-t}dt,\ \Gamma(\alpha,x)=\int_x^\infty t^{\alpha-1}e^{-t}dt,$
$\alpha>0$,~--- гамма-функция и неполная (верхняя) гамма-функция Эйлера;

\item
$n!=\Gamma(n+1)=n\cdot(n-1)\cdots2\cdot1$, $n\in\N$,~--- факториал;

\item
$\displaystyle C_n^k=\frac{\Gamma(n+1)}{\Gamma(k+1)\Gamma(n-k+1)}$~--- обобщённое число сочетаний из $n\in\R_+$ по $k\in[0,n]$;  $C_n^k=\frac{n!}{k!(n-k)!}$ для $n,k\in\N_0;$

\item
$\phi(x)=e^{-x^2/2}/\sqrt{2\pi}$,
$\Phi(x)=\int_{-\infty}^x\phi(t)\,dt$, $x\in\R$,~--- соответственно
плотность и ф.р. стандартного нормального закона;

\item 
$\xi\sim N(a,\sigma^2)$, $a\in\R,$ $\sigma^2>0,$~--- нормальное распределение с ф.р. $\Phi((x-a)/\sigma);$

\item 
$\xi\sim Ber(p)\coloneqq Bi(1,p)$~--- распределение Бернулли с параметром $p\in[0,1]$: $\Prob(\xi=1)=p=1-\Prob(\xi=0);$

\item 
$\xi\sim PB(n,\pr)$~--- пуассон-биномиальное распределение с параметрами $n\in\N$ и $\pr=(p_1,\ldots,p_n),$ $p_i\in(0,1]:$ $\xi\eqd\xi_1+\ldots+\xi_n$, $\xi_i\sim Ber(p_i)$,  $i=\overline{1,n},$ и $\{\xi_i\}_{i=1}^n$ независимы;

\item 
$\xi\sim Bi(n,p)\coloneqq PB(n,\pr)$ с $p_1=\ldots=p_n=p,$~--- биномиальное распределение с параметрами $n\in\N$ и $p\in(0,1]:$ $\Prob(\xi=k)\hm=C_n^kp^k(1-p)^{n-k}$, $k=\overline{0,n};$

\item 
$\xi\sim Pois(\la)$~--- пуассоновское распределение с параметром ${\la>0:}$ $\Prob(\xi=k)=e^{-\la}\la^k/k!$, $k=0,1,2,\ldots;$

\item 
$\xi\sim \NB(r,p)$~--- отрицательное биномиальное распределение с параметрами ${r>0}$ и $p\in(0,1):$ 
$$
\Prob(\xi=k)=C_{r+k-1}^kp^r(1-p)^k=\frac{\Gamma(r+k)}{k!\Gamma(r)}p^r(1-p)^k,\quad k\in\N_0;
$$

\item 
$\mathcal G(p)\coloneqq\NB(1,p)$~--- геометрическое распределение с параметром $p\in(0,1);$

\item 
$\xi\sim \Gamma(r,\la)$~--- гамма-распределение с параметром формы ${r>0}$ и масштаба ${\la>0}$, определяемое плотностью $p_\xi(x)\hm=\frac{\la^r}{\Gamma(r)}x^{r-1}e^{-\la x}\I(x>0)$;

\item 
$exp(\la)\coloneqq\Gamma(1,\la)$~--- показательное распределение с параметром ${\la>0},$ определяемое плотностью $p(x)=\la e^{-\la x}\I(x>0)$;

\item 
$\xi\sim St(r)$~--- распределение Стьюдента с ${r>0}$ степенями свободы, определяемое плотностью
$$
p_\xi(x)=\frac{\Gamma\big(\frac{r+1}2\big)}{\sqrt{\pi r}\,\Gamma\big(\frac{r}2\big)} \Big(1+\frac{x^2}r\Big)^{-(r+1)/2},\quad x\in\R;
$$

\item
$X,X_1,X_2\ldots$~---  независимые с.в., заданные на одном в.п.
$(\Omega,\mathcal{A},\Prob)$;

\item
$S_n=X_1+\ldots+X_n$, где $X_1,X_2\ldots$~---  независимые с.в.;

\item
$F(x)=\Prob(X<x)$, $F_j(x)=\Prob(X_j<x)$, $\overline
F_n(x)=\Prob(S_n-\E S_n<x\sqrt{\D S_n})$~--- ф.р. с.в. $X$, $X_j$,
$(S_n-\E S_n)/\sqrt{\D S_n}$ соответственно;

\item
$f(t)=\E e^{itX}$, $f_j(t)=\E e^{itX_j}$, $\overline f_n(t)=\E e^{it(S_n-\E S_n)/\sqrt{\D S_n}}$~--- х.ф. с.в. $X$, $X_j$, $(S_n-\E S_n)/\sqrt{\D S_n}$ соответственно;

\item
$\psi(z)=\E z^{X}$, $z\in\C,$ $|z|\le1,$~--- пр.ф. неотрицательной целочисленной с.в. $X$;

%\item
%$\nud =\nudf{F_1,\ldots,F_n}=|\overline F_n(x)-\Phi(x)|$, $x\in\R$;

%\item
%$\ud =\ud(F_1,\ldots,F_n)=\sup_{x\in\R}|\overline F_n(x)-\Phi(x)|$;
%$\ud(F)=\ud(F,\ldots,F)$;

\item
$\sigma_j^2=\D X_j$~--- дисперсия с.в. $X_j$, $j=\overline{1,n};$

\item
$\beta_{r,\,j}=\E|X_j|^r$~--- абсолютные моменты
порядка $r>0$ с.в. $X_j;$

\item
$B_n^2=\sum_{j=1}^n\sigma_j^2=\D S_n$;

\item
$\lyapd=B_n^{-(2+\d)}\sum_{j=1}^n\beta_{2+\d,j}$~--- дробь Ляпунова порядка $(2+\d)$, $\d\in(0,1]$;

\item
$L_n(\eps)=B_n^{-2}\sum_{j=1}^n\E X_j^2\I(|X_j|>\eps B_n)$, $\eps>0$~--- дробь Линдеберга;

%\item
%$M_n(\eps)=B_n^{-3}\sum_{j=1}^n\E|X_j|^3\I(|X_j|\le\eps B_n)$, $\eps>0$~--- второе слагаемое дроби Осипова;

\item
$F*G(x)\coloneqq\int_{-\infty}^\infty F(x-y)\dd G(y)=\int_{-\infty}^\infty G(x-y)\dd F(y)$ ---  свёртка ФОВ $F,G$ на $\R$; $F^{*1}\coloneqq F$, $F^{*n}\coloneqq F*F^{\ast(n-1)}$ --- $n$-кратная свертка ФОВ $F$ с собой;

\item
$\mathcal G$~--- класс всех неубывающих функций $g\colon\R_+\to\R_+$, таких что $g(x)>0$ для $x>0$ и $x/g(x)$ монотонно не убывает при $x>0$.

%\item
%$\F_{2+\d}$~--- множество всех ф.р. $F$ с.в. $X$ с ${\E X=0}$ (в
%главах\,1, 2) и $\E|X|^{2+\d}\hm<\infty$, $\d\ge0$; в главе 3
%требование центрированности ($\E X=0$) опускается;

%\item
%$\F_{2+\d,\,s}$~--- множество всех симметричных ф.р. из $\F_{2+\d}$;

%\item
%$\F_{2+\d}^h$~--- множество всех решётчатых ф.р. из $\F_{2+\delta}$
%с максимальным шагом $h>0$;

%\item
%$\F_{2+\d,\,s}^h$~--- множество всех решётчатых ф.р. из
%$\F_{2+\d,\,s}$ с максимальным шагом $h>0$;

\item
$\tvd(X,Y)=\sup\limits_{B\in\mathcal B(\R)}\abs{\Prob(X\in B)-\Prob(Y\in B)}$~---
расстояние между распределениями случайных величин $X$ и $Y$ по метрике полной вариации;

\item
$\rho(X,Y)=\sup\limits_{x\in\R}\abs{\Prob(X<x)-\Prob(Y<x)}$~--- 
расстояние по метрике Колмогорова между распределениями случайных величин $X$ и~$Y$ (равномерное расстояние между ф.р.);

\item
$L_\d(X,Y)=\inf\{\E|X'-Y'|^\d\colon X'\eqd X,\ Y'\eqd Y\}$~---
расстояние между распределениями случайных величин $X$ и $Y$ по минимальной интегральной метрике порядка $\d\in(0,1]$.

\end{itemize}

Всюду в пособии термины ``возрастающий'', ``убывающий'', ``положительный'', ``отрицательный'' понимаются в широком смысле (``неубывающий'' и т.п.), если не подчеркнуто иное.

\chapter{Предварительные сведения}
\thispagestyle{empty}
\section{Некоторые сведения из теории меры}

Материал данного раздела основан на книгах~\cite{Glivenko2007,Halmos1950, DyachenkoUlyanov2002,Bogachev2003,Bogachev2016}.

Пусть $\Omega$~--- произвольное непустое множество, $\mathcal A$~--- $\sigma$-алгебра подмножеств~$\Omega$, т.е. любой замкнутый относительно взятия дополнений и счётного числа объединений/пересечений класс подмножеств $\Omega,$ содержащий $\Omega$. Функция $\mu\colon \mathcal A\to\R$ называется \textit{мерой} (\textit{зарядом}) на $\mathcal A$, если она счётно-аддитивна в следующем смысле:
$$
\mu\bigg(\bigcup_{i=1}^\infty A_i\bigg)=\sum_{i=1}^\infty\mu( A_i),\quad A_1,A_2,\ldots\in\mathcal A,
\quad A_iA_j=\emptyset,\ i\neq j.
$$
Всюду далее такие счётно-аддитивные и, вообще говоря, знакопеременные  функции мы будем называть зарядами, а термин ``мера'' прибережём для неотрицательных зарядов $\mu$, т.е. таких, что $\mu(A)\ge0$ для любого $A\in\mathcal A.$

Множество $A\in\mathcal A$ называется \textit{положительным} (\textit{отрицательным}) относительно заряда $\mu$, если для любого $B\in\mathcal A$, $B\subseteq A$ выполнено неравенство $\mu(B)\ge0$ ($\mu(B)\le0$). Для каждого заряда $\mu$ найдется такое множество $A_+\in\mathcal A$, что оно положительно, а множество $A_-=\Omega\setminus A_+$ отрицательно относительно $\mu$. Представление $\Omega=A_+\cup A_-$ называется \textit{разложением Хана}, а представление 
$$
\mu=\mu^+-\mu^-,
$$
где $\mu^+(B)=\mu(B\cap A_+),\  \mu^-(B)=-\mu(B\cap A_-),\ B\in\mathcal A,$ --- \textit{разложением Жордана} заряда $\mu$, при этом $\mu^+,$ $\mu^-$ являются положительными мерами. Мера 
$$
|\mu|\coloneqq \mu^++\mu^-
$$
называется \textit{полной вариацией} заряда $\mu$, а функция
$$
\tvnorm{\mu}\coloneqq |\mu|(\Omega)
$$
является \textit{нормой} (полной вариации) и  называется \textit{вариацией} или \textit{вариационной нормой} заряда $\mu$. Заряд называется \textit{конечным} или \textit{ограниченным}, если $\tvnorm{\mu}<\infty.$

Если мера $\mu$ нормирована так, что $\mu(\Omega)=1$, то она называется \textit{вероятностной}. Для таких мер мы обычно будем использовать обозначение  $\Prob.$

Если $\Omega\subseteq \R^d$ и $\mathcal A\coloneqq\mathcal B(\Omega)$~--- борелевская $\sigma$-алгебра, т.е. $\sigma$-алгебра, порождённая всеми открытыми подмножествами $\Omega,$ то заряд $\mu$ называется \textit{борелевским}. Пример борелевской вероятностной меры даёт распределение случайной величины.

\begin{definition} Для всякого ограниченного борелевского заряда $\mu$ на прямой функция
$$
F_\mu(x)\coloneqq \mu\big((-\infty,x)\big),\quad x\in\R,
$$
называется функцией распределения заряда $\mu.$ 
\end{definition}

С функциями распределения зарядов тесно связано понятие функции ограниченной вариации на прямой. Напомним некоторые определения и факты из этой области.

Пусть $I\subseteq\R$~--- замкнутый, открытый или полуоткрытый интервал (любое выпуклое подмножество прямой).

\begin{definition} \textit{Вариацией} функции $F\colon I\to\R$ на интервале~$I$ называется точная верхняя грань
$$
V(F,I)\coloneqq \sup\Big\{\smallsum_{i=1}^n\abs{F(x_i)-F(x_{i-1})}\colon x_i\in I,\ x_0<x_1<\ldots<x_n\Big \}.
$$
Если $V(F,I)<\infty,$ то $F$ называется \textit{функцией ограниченной вариации} $($ФОВ$)$ на~$I$.
\end{definition}

Как известно, всякая монотонная функция $F$ на отрезке $[a,b]$ имеет ограниченную вариацию $V(F,[a,b])=|F(b)-F(a)|$. С другой стороны, всякая ФОВ на $[a,b]$ есть разность двух возрастающих функций. Отсюда, в частности, вытекает, что всякая ФОВ на отрезке имеет не более счётного множества точек разрыва. Отметим, что в разложении $F=F_1-F_2$ возрастающие функции $F_1,F_2$ можно выбрать так, чтобы
$$
V(F_1,[a,b])\le V(F,[a,b]),\quad V(F_2,[a,b])\le V(F,[a,b]),
$$
например,
$$
F_1(x)=V(F,[a,x]),\quad F_2(x)=V(F,[a,x])-F(x).
$$

Функция распределения $F_\mu$ всякого ограниченного заряда $\mu$ является ФОВ (в частности возрастающей функцией, если $\mu$~--- мера), непрерывной слева и имеющей нулевой предел на $-\infty$. И наоборот, любая непрерывная слева ФОВ $F$ с $\lim\limits_{x\to-\infty}F(x)=0$ является функцией распределения некоторого ограниченного борелевского заряда~$\mu$. 

Далее, по непрерывной слева ФОВ определяется интеграл Лебега--Стилтьеса, который для ф.р. $F_\mu$ ограниченных борелевских зарядов $\mu$ совпадает с соответствующим интегралом Лебега:
$$
\int_{-\infty}^\infty g(x)\dd F_\mu(x) \ =\ \int_\R g\dd \mu \ =\ \int g\dd \mu
$$
(если область интегрирования не указана, то~---  традиционно в теории меры~--- интегрирование ведётся по всему пространству) для каждой $\mu$-интегрируемой функции $g$, при этом интеграл по вариации $V\big(F_\mu,(-\infty,\cdot)\big)$ совпадает с вариационной нормой $\mu$:
$$
\tvnorm{\mu} =  V(F_\mu,\R) = \int_{-\infty}^\infty\dd V\big(F,(-\infty,x)\big)\eqqcolon \int_{-\infty}^\infty\abs{\dd F_\mu(x)}.
$$
%Djj,ot byntuhfk gj dfhbfwbb ytghthsdyjq cktdf AJD ;А; eckjdbvcz j,jpyfxfnm
%;;
%\ште_Х-\штаенЪ:чп(н)\фиыХ\вв А(н)Ъ\сщдщтуйй \ште_Х-\штаенЪ:чп(н)\вв М\ишп(Аб(-\штаенбн)\ишп)ю
%;;
%lkz k.,jq bpvthbvjq aeyrwbb ;п;ю D xfcnyjcnb? ;\ште_ф:и\фиыХ\вв А(ч)Ъ=М\ишп(Абхфби)\ишп)ю;

Подробнее со свойствами ФОВ можно ознакомиться, например, по книгам~\cite[Гл.\,5, \S\,18]{Glivenko2007}, \cite[Гл.\,5, \S\,5.2]{Bogachev2003}, или~\cite[Гл.\,5, \S\,24]{DyachenkoUlyanov2002}.

Множество всех ограниченных зарядов на заданной $\sigma$-алгебре~$\mathcal A$ с нормой $\tvnorm{\cdot}$ является \textit{банаховым}, т.е. полным относительно метрики $d(\mu,\nu)\coloneqq\tvnorm{\mu-\nu}$ нормированным пространством. Сходимость в этой метрике, т.е. $\tvnorm{\mu_n-\mu}\to0$, называется \textit{сходимостью по вариации}.

\begin{definition}
Последовательность ограниченных борелевских зарядов $\{\mu_n\}_{n\in\N}$ называется слабо сходящейся к борелевскому заряду $\mu,$ если для каждой непрерывной и ограниченной функции $g\colon\R\to\R$ имеет место сходимость
$$
\lim_{n\to\infty}\int_\R g\dd\mu_n=\int_\R g\dd\mu.
$$
Обозначение$:~ \mu_n\dto\mu.$ Аналогичное обозначение $F_{\mu_n}\dto F_\mu$ будем использовать и для соответствующих ф.р. 
\end{definition}

Если последовательность ограниченных борелевских зарядов $\mu_n$ сходится по вариации к заряду $\mu$, т.е. $\tvnorm{\mu_n-\mu}\to0,$ то $\mu_n\dto\mu.$ Обратное, вообще говоря, неверно.

%\begin{definition}
%Последовательность ф.р. $\{F_n\}_{n\in\N}$ сходится \textit{слабо} к ф.р. $F$, если для любой непрерывной и ограниченной функции $g\colon\R\to\R$
%$$
%\int g(x)\dd F_n(x)\to\int g(x)\dd F(x).
%$$ 
%Обозначение: $F_n\dto F.$
%\end{definition}

%Как известно, $F_n\dto F$  тогда и только тогда, когда $F_n(x)\to F(x)$ в каждой непрерывности предельной функции $F$.

%Для последовательности мер $\mu_1,\mu_2,\ldots$ слабая сходимость мер $\mu_n\dto\mu$ к мере $\mu$ равносильна сходимости их ф.р. $F_{\mu_n}(x)\to F_\mu(x)$ к ф.р. меры $\mu$ в каждой точке непрерывности последней. 

Для мер слабая сходимость равносильна сходимости соответствующих  ф.р.  в каждой точке непрерывности предельной. Однако для зарядов слабая сходимость не влечёт поточечную сходимость ф.р. на всюду плотном множестве.

Отметим также, что для непрерывной предельной ф.р. $F_\mu$ слабая сходимость мер $\mu_n\dto\mu$ равносильна сходимости в \textit{метрике Колмогорова}
$$
\rho(F_{\mu_n},F_\mu)\coloneqq \sup_{x\in\R}\abs{F_{\mu_n}(x)-F_\mu(x)}\to0,\quad n\to\infty.
$$

Слабая сходимость вероятностных мер фигурирует в определении \textit{сходимости по распределению} случайных величин. Случайная величина (с.в.) $\xi$~--- это 
%$\mathcal A/\mathcal B$-
измеримое отображение $\Omega\to\R$, т.е. такое что 
$$
\xi^{-1}(B)\coloneqq\{\omega\in\Omega\colon \xi(\omega)\in B\}\in\mathcal A \quad\text{для любого $B\in\mathcal B\coloneqq\mathcal B(\R)$.}
$$
Как и всякая измеримая в $(\R,\mathcal B)$ функция, с.в. $\xi$ индуцирует борелевскую вероятностную меру на прямой
$$
\Prob_\xi(B)\coloneqq \Prob\left (\xi^{-1}(B)\right ),
%=\Prob\big(\{\omega\in\Omega\colon \xi(\omega)\in B\}\big),
\quad B\in\mathcal B,
$$
называемую \textit{распределением} с.в.~$\xi.$ Функция распределения этой меры
$$
F(x)\coloneqq\Prob_\xi\big((-\infty,x)\big)=\Prob\big(\{\omega\in\Omega\colon \xi(\omega)<x\}\big),\quad x\in\R,
$$
также называется функцией распределения (ф.р.) с.в.~$\xi$. Иногда для удобства мы будем писать $\Prob(\xi\in B)$ вместо $\Prob_\xi(B)$.

\section{Метрика полной вариации}

Вводимые ниже \textit{простые вероятностные метрики} в качестве своих аргументов принимают маргинальные распределения (случайных величин), но иногда нам будет удобно в качестве аргументов указывать сами случайные величины или их ф.р.

\begin{definition}
Функция от пары маргинальных распределений с.в. $\xi$ и $\eta$, задаваемых функциями распределения $F_\xi$ и $F_\eta$
$$
\ud(\xi,\eta)\coloneqq\ud(F_\xi,F_\eta)\coloneqq\sup_{x\in\R}\abs{F_\xi(x)-F_\eta(x)}
$$
$($является метрикой и$)$ называется \textit{метрикой Колмогорова}.
\end{definition}

\begin{definition}
Функция от пары маргинальных распределений с.в. $\xi$ и $\eta$ на прямой 
\begin{equation}
\tvd(\xi,\eta)\coloneqq\tvd(F_\xi,F_\eta)\coloneqq\tvd(\Prob_\xi,\Prob_\eta)\coloneqq\sup_{B\in\mathcal B}\abs{\Prob_\xi(B)-\Prob_\eta(B)}
\label{TVDdef}
\end{equation}
$($является метрикой и$)$ называется \textit{метрикой полной вариации}.
\end{definition}

Ясно, что для всякой пары распределений с.в. $\xi,\eta$
$$
\rho(\xi,\eta) \le \tvd(\xi,\eta).
$$
Более того, метрика полной вариации топологически сильнее метрики Колмогорова: найдутся распределения с.в. $\xi,\xi_1,\xi_2,\ldots$ такие, что при $n\to\infty$
$$
\rho(\xi_n,\xi)\to0, \quad\text{но}\quad \tvd(\xi_n,\xi) \nrightarrow0.
$$

%Следующая теорема и её следствия устанавливает некоторые простые и полезные свойства метрики $\tvd.$ 
Нижеследующая теорема устанавливает, что точная верхняя грань в определении метрики полной вариации~\eqref{TVDdef} достигается и равняется половине нормы полной вариации борелевского заряда $\Prob_\xi-\Prob_\eta$, а также даёт конструктивный способ вычисления этой метрики.
%$$
%\tvd(\Prob_\xi,\Prob_\eta)\ =\ \tfrac12\tvnorm{\Prob_\xi-\Prob_\eta}=\tfrac12\int_{-\infty}^\infty\abs{\dd(F_\xi-F_\eta)(x)}.
%$$

\begin{theorem}\label{ThTVDv.s.DF}
Для любых с.в. $\xi$ и $\eta$ с ф.р. $F_\xi$ и $F_\eta$
\begin{multline*}
\tvd(\xi,\eta)=\max_{B\in\mathcal B}\abs{\Prob_\xi(B)-\Prob_\eta(B)} =
\\
=\frac12\int_{-\infty}^\infty\abs{\dd(F_\xi-F_\eta)(x)}=\frac12\tvnorm{\Prob_\xi-\Prob_\eta}.
\end{multline*}
\end{theorem}

\begin{corollary}\label{CorTVDforDiscrete+ACdistr}
Если с.в. $\xi$ и $\eta$ абсолютно непрерывны с плотностями $p_\xi$ и $p_\eta$, то 
$$
\tvd(\xi,\eta)=\frac12\int_{-\infty}^\infty\abs{p_\xi(x)-p_\eta(x)}\dd x.
$$
Если с.в. $\xi$ и $\eta$ дискретны и $B\coloneqq\{x_1,x_2,\ldots\}$~--- объединённое множество их возможных значений, т.е. $\Prob_\xi(B)=\Prob_\eta(B)=1$, то 
$$
\tvd(\xi,\eta)=\frac12\sum_{k}\abs{\Prob(\xi=x_k)-\Prob(\eta=x_k)}.
$$
В частности, если $\xi,\eta$~--- целочисленные с.в., то 
$\tvd(\xi,\eta)\hm=\frac12\sum_{k\in\Z}\abs{\Prob(\xi=k)-\Prob(\eta=k)}.$
\end{corollary}

\begin{proof}[Доказательство теоремы~\ref{ThTVDv.s.DF}] Пусть $A_+$ и $A_-=\overline{A_+}=\R\setminus A_+$~--- положительное и отрицательное  множества в разложении Хана для заряда $\mu\coloneqq\Prob_\xi-\Prob_\eta$ (см., например,~\cite[Теорема\,20.1]{DyachenkoUlyanov2002}). Тогда ф.р. $G=F_\xi-F_\eta$ заряда $\mu$ является возрастающей на $A_+$ и убывающей на $A_-$; кроме того, $\int\dd G=0$.

Покажем, что точная верхняя грань в определении $\tvd$ достигается на множествах $A_+$ и $A_-$. Действительно, для любого $B\in\mathcal B$ в силу возрастания $G$ на $A_+$ имеем 
$$
\Prob(\xi\in B)-\Prob(\eta\in B)=\int_B\dd G\le \int_{B\cap A_+}\dd G\le \int_{A_+}\dd G
$$
%$$
%\le \int_{A_+}\dd G= \Prob(\xi\in A_+)-\Prob(\eta\in A_+),
%$$
и аналогично, в силу убывания $G$ на $A_-$,
$$
\Prob(\eta\in B)-\Prob(\xi\in B)\le -\int_{B\cap A_-}\dd G\le -\int_{A_-}\dd G= \int_{A_+}\dd G,
$$
%$$
%=\Prob(\eta\in A_-)-\Prob(\xi\in A_-)=\Prob(\xi\in A_+)-\Prob(\eta\in A_+),
%$$
откуда вытекает, что
$$
\tvd(\xi,\eta)=\sup_{B\in\mathcal B}\abs{\Prob(\xi\in B)-\Prob(\eta\in B)}\le \biggl|\int_{A_-}\dd G\biggr|=\int_{A_+}\dd G=
$$
$$
= \Prob(\xi\in A_+)-\Prob(\eta\in A_+),
$$
причём точная верхняя грань и следующее равенство достигаются, т.к. $A_+,A_-\in\mathcal B$.

С другой стороны,
$$
\frac12\int\abs{\dd(F_\xi-F_\eta)} =\frac12\bigg(\int_{A_+}\dd G-\int_{A_-}\dd G\bigg) = \int_{A_+}\dd G=\tvd(\xi,\eta).\qedhere
$$
\end{proof}

\section{Дзета-метрики}

Для $s=m+\alpha>0$ с $m\in\N_0$, $\alpha\in(0,1]$ рассмотрим класс $\mathcal F_s$ функций $f\colon\R\to\R$, $m$ раз дифференцируемых и таких, что
$$
\abs{f^{(m)}(x)-f^{(m)}(y)}\le |x-y|^{\alpha},\quad x,\ y\in\R.
$$
Другими словами, $\mathcal F_s$ состоит из функций, $m$-я производная которых является гельдеровской (липшицевой при $\alpha=1$) с константой Гельдера (Липшица)~$1$. Ясно,  что каждая функция $f\in\mathcal F_s$ имеет непрерывную $m$-ю производную и ограничена полиномом от $|x|$ степени не выше $s$. Кроме того, при $\alpha=1$, т.\,е. $s=m+1\in\N$,  $f\in\mathcal F_s$ равносильно существованию $f^{(m+1)}(x)$ для $\lambda$-п.в. $x$ с 
$$
\norm{f^{(m+1)}}_{\infty}:=\inf\left\{M>0\colon |f^{(m+1)}(x)|\le M \text{ для } \lambda\text{-п.в. }x\in\R\right\} \le1.
$$

Класс $\mathcal F_s$ содержит, например, следующие функции: $f_1(x)=x^s/s!$ при $s\in\N$, 
$$
f_2(x)=\frac{\Gamma(1+\alpha)}{\Gamma(1+s)}\cdot|x|^s,\quad 
f_3(x)=\frac{\sin(tx)}{2^{1-\alpha}t^s},\quad f_4(x)=\frac{\cos(tx)}{2^{1-\alpha}t^s},\quad t\neq0.
$$

Доказательство того, что $f_2\in\mathcal F_s$, основано на следующем факте.

\begin{excersize}
Доказать, что если функция $f(x)$ вогнута при $x\ge0$ и $f(0)\ge0$, то $f$ субаддитивна на $[0,\infty)$, т.е.
$$
f(x+y)\le f(x)+f(y),\quad x,\ y\ge0.
$$
% 1) f(ax) \ge af(x), a\in[0,1], x\ge0,
% 2) f(x)+f(y) = f((x+y) * x/(x+y)) + f((x+y) * y/(x+y)) \ge f(x+y)*x/(x+y) + f(x+y)*y/(x+y) = f(x+y)
\end{excersize}

%Перед тем, как ввести определение дзета-метрики, скажем, что метрика в пространстве распределений (как функция от пары распределений) может зависеть от совместного распределения,  только маргинальны. 

Введем определение дзета-метрики Золотарева~\cite{Zolotarev1983} в пространстве вероятностных распределений.

\begin{definition}
Пусть $s>0$, $X$, $Y$~--- с.в. с ф.р. $F$, $G$ и $\E|X|^s<\infty,$ $\E|Y|^s<\infty$. Расстоянием в \textit{дзета-метрике} (\textit{расстоянием Золотарева}) порядка $s$ между распределениями с.в. $X$ и $Y$ называется
$$
\zeta_s(F,G):=\sup_{f\in\mathcal F_s}|\E f(X)-\E f(Y)| = \sup_{f\in\mathcal F_s}\abs{\int_{-\infty}^\infty f(x)\dd(F-G)(x)}.
$$
\end{definition}

Иногда  мы будем писать $\zeta_s(X,Y)$, подразумевая под этим расстояние в $\zeta$-метрике между соответствующими распределениями (функциями распределения).

Введенное расстояние обладает следующими свойствами:
\begin{enumerate}
\item
$\zeta_s(\,\cdot\,,\,\cdot\,)$ является метрикой в пространстве вероятностных распределений с конечными абсолютными моментами порядка $s$.

\item \textbf{Регулярность}: для любых с.в. $X,Y$ и независимой от них с.в.~$Z$
$$
\zeta_s(X+Z,Y+Z)\le \zeta_s(X,Y).
$$

\item \textbf{Однородность} порядка $s$:  для любых с.в. $X,Y$ и  вещественной константы $c\neq0$
\begin{equation}\label{zeta-regularity}
\zeta_s(cX,cY)=|c|^s\zeta_s(X,Y).
\end{equation}

\item \textbf{Полуаддитивность:} если $X_1,\ldots,X_n$~--- независимые с.в. и $Y_1,\ldots,Y_n$~--- независимые с.в., то 
$$
\zeta_s(X_1+\ldots+X_n,Y_1+\ldots+Y_n)\le \sum_{k=1}^n\zeta_s(X_k,Y_k).
$$

\item \textbf{Инвариантность} относительно сдвига:  для любых с.в. $X,Y$ и  $c\in\R$
$$
\zeta_s(X+c,Y+c)= \zeta_s(X,Y).
$$
\end{enumerate}

\begin{excersize}
Доказать, что свойство однородности дзета-метрики вытекает из неравенства, которое получается из~\eqref{zeta-regularity} заменой знака равенства на знак неравенства ($\le$ или $\ge$).
\end{excersize}

\begin{excersize}
Доказать, что свойства полуаддитивности и инвариантности вытекают из свойства регулярности дзета-метрики.
\end{excersize}

Заметим, что свойства  4) и 5) вытекают из 1)\,--\,3). Кроме того, в 3) достаточно потребовать выполнение неравенства ``$\le$''. Более подробно со свойствами $\zeta$-метрики  можно ознакомиться, например, в~\cite{Zolotarev1983}, \cite[\S\,1.4]{Zolotarev1986} и~\cite{MattnerShevtsova2019}. Простые метрики, обладающие свойствами 2)\,--\,3), впервые были рассмотрены В.\,М.\,Золотаревым в 1983\,г. в~\cite{Zolotarev1983} (также см.~\cite[\S\,1.4]{Zolotarev1986}) и названы \textit{идеальными} метриками порядка $s$.

C помощью $\zeta$-метрики можно оценивать средние значения гладких функций от рассматриваемых с.в., например:
$$
\abs{\E X^s-\E Y^s}\le s!\,\zeta_s(X,Y),\quad s\in\N,
$$
$$
\abs{\E|X|^s-\E|Y|^s}\le \frac{\Gamma(1+s)}{\Gamma(1+\alpha)}\zeta_s(X,Y),\quad s>0,
$$
$$
\abs{\E\cos(tX)-\E\cos(tY)}\le 2^{1-\alpha}|t|^s\zeta_s(X,Y),\quad s>0,
$$

Оказывается, дзета-метрика первого порядка, помимо определения, имеет еще три эквивалентных ему представления.

\begin{theorem}[см., например,~\cite{Zolotarev1986,Rachev1991}]
Для любых с.в. $X$, $Y$ с ф.р. $F$, $G$ и  конечными первыми моментами
\begin{multline*}
\zeta_1(F,G)=\inf\{\E|X'-Y'|\colon X'\eqd X,\ Y'\eqd Y\}
\\
=\int_0^1|F^{-1}(u)-G^{-1}(u)|du=\int_{-\infty}^\infty|F(x)-G(x)|\,\dd x=:\norm{F-G}_1.
\end{multline*}
\end{theorem} 

Первое представление для $\zeta_1(F,G)$ называется минимальной $L_1$-метрикой между с.в. $X,Y$. 

Второе представление означает, что точная нижняя грань в первом представлении достигается, причем на совместном распределении пары с.в. $(X,Y)$, определяемом следующим образом:
$$
X=F^{-1}(U),\quad Y=G^{-1}(U),\quad U\sim R(0,1),
$$
где $F^{-1}$~--- обобщенная обратная (квантильная) функция к ф.р. $F$. Напомним, что $F^{-1}$ можно определить как
$$
F^{-1}(u):=\sup\left\{x\in\R\colon F(x)\le u\right \},\quad u\in(0,1),
$$
или как
$$
F^{-1}(u):=\inf\left\{x\in\R\colon F(x)\ge u\right\},\quad u\in(0,1),
$$
(эти функции будут отличаться только на участках постоянства $F$, где первая будет давать самую правую точку интервала постоянства, а вторая~--- самую левую). В случае непрерывности и строгой монотонности~$F$ функция $F^{-1}$ есть обычная обратная.

Третье представление для $\zeta_1$ вытекает из геометрических соображений из второго и показывает, что $\zeta_1$ является также средней метрикой (метрикой Канторовича, Вассерштейна, $L_1$-метрикой) между ф.р. $F$ и $G$. 

Скажем еще несколько слов о средних метриках между ф.р. В общем виде они определяются следующим образом:
$$h(f
\norm{F-G}_p:= \bigg(\int_{-\infty}^\infty|F(x)-G(x)|^p\,\dd x\bigg)^{1\wedge1/p},\quad p>0.
$$
При $p=\infty$ расстояние $\norm{F-G}_p$ определяется как соответствующий предел при $p\to\infty$ и совпадает с равномерной метрикой (метрикой Колмогорова) между ф.р.:
$$
\norm{F-G}_\infty= \sup_{x\in\R}|F(x)-G(x)|.
$$
Средние метрики первого и бесконечного порядка представляют наибольший интерес, поскольку с их помощью и с помощью неравенства Гельдера можно оценить среднюю метрику любого промежуточного порядка:
$$
\norm{\,\cdot\,}_p\le \norm{\,\cdot\,}_1^{1/p}\times  \norm{\,\cdot\,}_\infty^{1-1/p},\quad p\ge1.
$$

%\setrightmark{Распределения и характеристические функции}
\section[Характеристические функции]{Характеристические функции: общие сведения}
\label{SecCh.F.General}

Пусть $(\Omega,\mathcal A,\Prob)$~--- в.п., $\xi$~--- с.в. на и.п. $(\Omega,\mathcal A)$, $\Prob_\xi$~--- распределение $\xi$, $F$--- ф.р.~$\xi$, $\la$~--- мера Лебега на $\R$. Интеграл Лебега от измеримой функции $g\colon\R\to\R$ по вероятностной мере $\Prob_\xi$ называется \textit{математическим ожиданием} и обозначается $\E g(\xi).$ Справедлива следующая формула замены переменных:
$$
\E g(\xi)\coloneqq \int_{\R} g\,\dd\Prob_\xi = \int_{-\infty}^\infty g(x)\dd F(x).
$$
Комплекснозначная функция вещественного переменного
$$
f(t)\coloneqq\int_{-\infty}^\infty e^{itx}\dd F(x)=\E e^{it\xi}\coloneqq\E\cos(t\xi)+i\E\sin(t\xi),\quad t\in\R,
$$
называется \textit{характеристической функцией} (х.ф.) с.в. $\xi$. Ясно, что $f(t)$ однозначно определяется распределением $\Prob_\xi$ или функцией распределения $F$. Верно и обратное: х.ф. однозначно определяет распределение и функцию распределения (теорема единственности), подробнее~--- см. формулу обращения~\eqref{FBVInversionFormulaExpcilit} ниже. Теорема единственности вкупе с теоремой непрерывности (см. п.\,\ref{point-continuity-theorem} ниже) обусловливают важную роль, которую играют характеристические функции для аналитических методов теории вероятностей.

\begin{definition}
С.в. $\xi$ $($её распределение, ф.р.$)$ называется дискретной $($дискретным$)$, если существует не более чем счётное множество $B$ такое, что $\Prob(\xi\in B)=1.$ В частности, если $B$ имеет арифметическую структуру, т.е. найдутся такие константы $a\in\R$ и $h>0$, что \begin{equation}\label{def_lattice_distr}
\sum_{k\in\Z}\Prob(\xi=a+kh)=1,
\end{equation}
то (распределение, ф.р.)  с.в. $\xi$ называется \textit{решётчатой}, при этом $h$ называется \textit{шагом} распределения.
\end{definition}

Ясно, что если $h$~--- шаг решётчатого распределения, то $h/2,h/3,\ldots$ также  будут шагами того же распределения. Поэтому, чтобы придать определению однозначность, говорят ещё о \textit{максимальном шаге} как максимальном $h>0$, для которого найдётся такое $a\in\R$, что справедливо~\eqref{def_lattice_distr}. 

\begin{definition}
С.в. $\xi$ $($её распределение, ф.р.$)$ называется непрерывной $($непрерывным$)$, если  $\Prob(\xi\in B)=0$ для любого не более чем счётного множества $B$.
\end{definition}
Ф.р. непрерывной с.в. $\xi$ является непрерывной.

\begin{definition}
С.в. $\xi$ $($её распределение, ф.р.$)$ называется абсолютно непрерывной $($абсолютно непрерывным$)$, если $\Prob(\xi\in B)=0$ для любого $B\in\mathcal B$ с $\la(B)=0$.  В этом случае распределение с.в. $\xi$ однозначно определяется \textit{плотностью} распределения, т.е. такой измеримой функцией $p\colon\R\to\R_+$, что $F(x)=\int_{-\infty}^xp(y)\dd y$ для любого $x\in\R$.
\end{definition}

\begin{definition}
С.в. $\xi$ $($её распределение, ф.р.$)$ называется сингулярной $($сингулярным$)$, если $\xi$~--- непрерывна и существует такое $B\in\mathcal B$, что $\la(B)=0$ и $\Prob(\xi\in B)=1.$
\end{definition}

\begin{theorem}[Теорема Лебега о разложении] Для любой ф.р. $F$ найдутся дискретная $F_1,$ абсолютно непрерывная $F_2$ и сингулярная~$F_3$ функции распределения и неотрицательные  числа $p_1,p_2,p_3$ с $p_1+p_2+p_3=1$ такие, что $F=p_1F_1+p_2F_2+p_3F_3.$
\end{theorem}

Перечислим некоторые важные свойства х.ф. Там, где это необходимо подчеркнуть, в качестве индекса будем указывать случайную величину с соответствующим распределением.

\begin{enumerate}

\item\textbf{Теорема единственности}: х.ф. однозначно определяет распределение. Более конкретно~--- см. формулу обращения~\eqref{FBVInversionFormulaExpcilit} ниже.

\item $f(0)=1,$ $|f(t)|\le1$ для любого $t\in\R$, причём $|f(t_0)|=1$ для некоторого $t_0\neq0$ тогда и только тогда, когда $|f(kt_0)|=1$ для любого $k\in\Z$, что, в свою очередь равносильно тому, что соответствующее распределение является решётчатым с шагом $h=\frac{2\pi}{t_0}$.

\item $f_\xi(-t)=\overline {f_\xi(t)}=f_{-\xi}(t)$, $t\in\R,$ где черта означает комплексное сопряжение.

\item $f(t)\in\R$ для любого $t\in\R$ $\Leftrightarrow$ $f$ чётна  $\Leftrightarrow$ $\xi\eqd-\xi$.

\item $\Re f(t)=\E\cos(t\xi)$ чётна, $\Im f(t)=\E\sin(t\xi)$ нечётна.

\item $f_{a\xi+b}(t)=e^{itb}f_\xi(at),$ $t\in\R,$ где $a,b$~--- вещественные константы.

\item Если $\xi_1,\ldots,\xi_n$~--- независимые с.в., то $f_{\xi_1+\ldots+\xi_n}=\prod_{k=1}^nf_{\xi_k}.$

\item В частности, если $\xi,\xi'$~--- н.о.р.с.в., то $f_{\xi-\xi'}(t)=\E\cos t(\xi-\xi')=|f_\xi(t)|^2$. Распределение $\xi-\xi'$ называется \textit{сверточной симметризацией} распределения $\xi$.

\item Если  $f_n$~--- х.ф., $p_n\ge0,$ $n\in\N,$ и  $\sum_{n=1}^\infty p_n=1$, то $\sum_{n=1}^\infty p_nf_n$~--- х.ф. 

\item В частности, $\Re f_\xi(t)=\frac12f_\xi(t)+\frac12f_{-\xi}(t)$ является х.ф. с.в. $\xi\cdot\eta,$ где  $\xi,\eta$~--- независимые с.в. и $\Prob(\eta=\pm1)=\frac12$.  Распределение $\xi\cdot\eta'$ называется \textit{рандомизированной симметризацией} распределения $\xi$.

\item \textbf{Теорема Римана--Лебега}: если распределение абсолютно непрерывно, то его х.ф. $f(t)\to0$ при $|t|\to\infty$.

\item \textbf{Равенство Парсеваля}: для любых ф.р. $F,G$ c х.ф. $f,g$
$$
\int f(t)\dd G(t) =\int g(t)\dd F(t).
$$

\item\label{point-continuity-theorem} \textbf{Теорема непрерывности}: последовательность ф.р. $F_1,F_2,\ldots$ сходится слабо к ф.р. $F$ тогда и только тогда, когда последовательность их х.ф. $f_1,f_2,\ldots$ сходится в каждой точке к функции $f(t)\coloneqq\lim\limits_{n\to\infty}f_n(t)$, непрерывной в нуле. При этом $f(t)=\int e^{itx}\dd F(x),$ $t\in\R,$ является  х.ф. и соответствует  ф.р.~$F$.

\end{enumerate}

\section{Характеристические функции и моменты: формула Тейлора}

Для с.в. $\xi$ с ф.р. $F$ введём алгебраические и абсолютные начальные моменты (распределения) с.в. $\xi$ 
$$
\alpha_k\coloneqq\E\xi^k=\int_{-\infty}^\infty x^k\dd F(x),\quad \beta_r\coloneqq\E|\xi|^r=\int_{-\infty}^\infty|x|^r\dd F(x),\quad %|\alpha_k|\le\beta_k,\quad 
%k\in\N,\ r>0,
$$
для всех $r>0$ и  $k\in\N$, для которых соответствующие интегралы существуют, принимая конечные или бесконечные значения, а также положим $\alpha_0\coloneqq1,$ $\beta_0\coloneqq1$.

Если $\beta_n<\infty$ для  ф.р. $F$ с х.ф.~$f$ при некотором $n\hm\in\N$, то существует равномерно непрерывная на $\R$ производная 
$$
f^{(n)}(t)\coloneqq\frac{d^n}{dt^n}\int e^{itx}\dd F(x) =  i^n\int x^ne^{itx}\dd F(x), \quad t\in\R,
$$
и справедлива Формула Тейлора:
$$
f(t)=1+\sum_{k=1}^n\frac{\alpha_k}{k!}\,(it)^k+o\left(t^n\right ),\quad t\to0.
$$
В частности, $f^{(n)}(0)=i^n\alpha_n$, $\abs{f^{(n)}(t)}\le\beta_n$.

Построим моментные оценки остаточного члена 
$$
R_n(t)\coloneqq f(t)-\sum_{k=0}^{n-1}\frac{\alpha_k}{k!}(it)^k,\quad t\in\R,
$$
в формуле Тейлора для х.ф., которые нам пригодятся в дальнейшем. Обозначим 
$$
r_n(x)\coloneqq e^{ix}-\sum_{k=0}^{n-1}\frac{(ix)^k}{k!},\quad x\in\R.
$$
Тогда $R_n(t)=\E r_n(tX)$ и, в силу монотонности интеграла Лебега,
$$
|R_n(t)|\le\E|r_n(tX)|,\quad t\in\R.
$$
Теперь для получения моментных оценок остаточного члена $R_n(t)$ достаточно 
построить полиномиальную оценку функции $r_n(x)$. Такая оценка устанавливается следующей леммой.

\begin{lemma}
Для любых $x\in\R,$ $n\in\N_0$ и $\d\in(0,1]$
\begin{equation}\label{e^(ix)TaylorEstim(n,d)Classic}
|r_{n+1}(x)|\le C_{n,\d}\cdot|x|^{n+\d},\quad\text{где}\quad C_{n,\d}\coloneqq 2^{1-\d}/\prod_{k=1}^{n}(k+\d)
\end{equation}
$($по определению полагаем $\prod_{k=1}^0(\cdot)\coloneqq1),$
причём при $\d=1$ константы $C_{n,\d}$ нельзя уменьшить в том смысле, что
$$
\lim_{x\to0}\frac{|r_{n+1}(x)|}{|x|^{n+1}}=\frac1{(n+1)!}=C_{n,1},\quad n\in\N_0.
$$
\end{lemma}

\begin{proof} Будем использовать индукцию по $n\in\N.$ Проверим базис индукции ($n=0$): для любого $x\in\R$ 
$$
|r_1(x)|=\abs{e^{ix}-1}=\sqrt{(\cos x-1)^2+\sin^2x}=\sqrt{2-2\cos x}=
$$
$$
=2\abs{\sin\tfrac{x}2}\le2\wedge|x|\le 2^{1-\d}|x|^\d=C_{0,\d}\cdot|x|^\d,\quad \d\in[0,1].
$$
Проверим индуктивный переход. Пусть $|r_{k}(x)|\le C_{k-1,\d}|x|^{k-1+\d}$ для $k=0,\ldots,n$. Тогда для $k=n+1$
$$
\abs{r_{n+1}(x)}=\abs{\int_0^xr_{n+1}'(y)\dd y}=\abs{i\int_0^xr_n(y)\dd y}\le \int_0^{|x|}\abs{r_n(y)}\dd y\le
$$
$$
\le C_{n-1,\d}\int_0^{|x|}y^{n-1+\d}\dd y = \frac{C_{n-1,\d}}{n+\d}|x|^{n+\d}=C_{n,\d}\cdot|x|^{n+\d}.\quad \qedhere 
$$
\end{proof}

Из доказанной леммы мгновенно вытекает

\begin{theorem}\label{ThCh.F.TaylorEstim(n,d)Classic}
Если $\xi$~--- с.в. c х.ф. $f$ и $\beta_{n+\d}\coloneqq\E|\xi|^{n+\d}<\infty$ при некоторых $n\in\N_0$ и $\d\in(0,1],$ то 
$$
|R_{n+1}(t)|= \bigg|f(t)-\sum_{k=0}^{n}\frac{\alpha_k}{k!}(it)^k\bigg|\le C_{n,\d}\beta_{n+\d}|t|^{n+\d},\quad t\in\R,
$$
при этом константы $C_{n,1}$ нельзя уменьшить в том смысле, что для вырожденной, например, с.в. $\xi\eqp a$
$$
\lim_{t\to0}\frac{|R_{n+1}(t)|}{|t|^{n+1}}= \frac{|a|^{n+1}}{(n+1)!} = C_{n,1}\cdot\beta_{n+1},\quad n\in\R.
$$
\end{theorem}

Таким образом, теорема~\ref{ThCh.F.TaylorEstim(n,d)Classic} позволяет не только конкретизировать порядок малости  остаточного члена в формуле Тейлора $R_n(t)\hm=o(t^{n})$ при  $t\to0,$  но и получить его конкретную оценку в каждой точке $t\in\R$.

Хотя для моментов целого порядка ($\d=1$) теорема~\ref{ThCh.F.TaylorEstim(n,d)Classic}
\begin{equation}\label{Ch.F.TaylorEstim(n,1)Classic}
|R_n(t)|= \bigg|f(t)-\sum_{k=0}^{n-1}\frac{\alpha_k}{k!}(it)^k\bigg| \le\frac{\beta_n|t|^n}{n!},\quad t\in\R,
\end{equation}
и неравенство~\eqref{e^(ix)TaylorEstim(n,d)Classic}
\begin{equation}\label{e^(ix)TaylorEstim(n,1)Classic}
|r_n(x)|=\bigg|e^{ix}-\sum_{k=0}^{n-1}\frac{(ix)^k}{k!}\bigg|\le \frac{|x|^n}{n!},\quad x\in\R,
\end{equation}
и устанавливают \textit{точные} %и в определённом смысле неулучшаемые 
оценки, всё же эти оценки далеки от совершенства.  В 1991\,г. была опубликована интересная статья шведского актуария Х.\,Правитца~\cite{Prawitz1991}, находившегося к тому моменту уже в весьма преклонном возрасте, в которой было показано, что коэффициент $1/n!$ в правой части~\eqref{e^(ix)TaylorEstim(n,1)Classic} можно уменьшить, если оставить в главной части (слева) некоторую  долю от следующего члена разложения $(ix)^n/n!$, так чтобы степенные функции $x^{n}$  фигурировали и слева, и справа, причём коэффициент  $1/n!$ в правой части уменьшится ровно настолько (до известного предела), сколько будет  ``перенесено'' в главную часть. А именно,  Правитц доказал  следующую оценку.

\begin{lemma}[см.~\cite{Prawitz1991}]
Для любых $x\in\R$ и $n\in\N$
$$
\abs{r_n(x)-\frac{n}{2(n+1)}\cdot\frac{(ix)^n}{n!}}
= \bigg|e^{ix}-\sum_{k=0}^{n-1}\frac{(ix)^k}{k!} -\frac{n}{2(n+1)}\cdot\frac{(ix)^n}{n!}\bigg|\le
$$
$$
\le\frac{n+2}{2(n+1)}\cdot\frac{|x|^n}{n!}.
$$
\end{lemma}

\begin{myremark}
Сумма коэффициентов $\frac{n}{2(n+1)}$ и $\frac{n+2}{2(n+1)}$ при $x^n/n!$ в левой и правой частях неравенства равна $1$.
\end{myremark}

\begin{proof}
Для $n\in\N$ обозначим
$$
P_n(x)\coloneqq \abs{r_n(x)-\frac{n}{2(n+1)}\cdot\frac{(ix)^n}{n!}}^2- \left(\frac{n+2}{2(n+1)}\cdot\frac{|x|^n}{n!}\right )^2,\quad x\in\R
$$
Тогда утверждение леммы равносильно тому, что $P_n(x)\le0$ для всех $x\in\R$, и в силу чётности $P_n$  это неравенство  достаточно доказать только для $x\ge0.$ 

Учитывая, что
$$
r_{n}(x)=r_{n+1}(x)+(ix)^{n}/n!,\quad r_{n+1}'(x)=ir_n(x),\quad x\in\R,
$$
$$
\abs{z_1+z_2}^2= |z_1|^2+z_1\cdot\overline{z_2\vphantom{d}}+\overline{z_1\vphantom{d}} \cdot z_2+|z_2|^2= |z_1|^2+2\Re z_1\cdot\overline{z_2\vphantom{d}}+|z_2|^2,
$$
для любых $z_1,z_2\in\C,$ получаем 
$$
P_n(x)= \abs{r_{n+1}(x)+\frac{n+2}{2(n+1)}\cdot\frac{(ix)^n}{n!}}^2- \left(\frac{n+2}{2(n+1)}\cdot\frac{|x|^n}{n!}\right )^2=
$$
$$
=|r_{n+1}(x)|^2 + \frac{n+2}{(n+1)!}\Re \left[r_{n+1}(x)(-ix)^n\right].
$$
Так  как $P_n(0)=0$, для неположительности функции $P_n(x)$ при $x\ge0$ достаточно доказать ее убывание в этой области, т.е. что $P_n'(x)\le0$  при $x>0$. С учётом того, что для любой комплекснозначной функции $z$ с $u=\Re z,$ $v=\Im z$
$$
(|z|^2 )'= (u^2+v^2)'=2uu'+2vv' =2\Re(u-iv)(u'+iv')=2\Re\big(\overline{z\vphantom{d}}\cdot z'\big),  
$$
имеем
\begin{multline*}
\frac{d}{dx}|r_{n+1}(x)|^2 = 2\Re\left[\overline{r_{n+1}(x)}\cdot r_{n+1}'(x)\right] =
\\
=2\Re\left[\left(\overline{r_n(x)}-\tfrac{(-ix)^n}{n!}\right)\cdot ir_n(x)\right]=
\\
=2\Re\left[i\abs{r_n(x)}^2+(-i)^{n+1}r_n(x)\frac{x^n}{n!}\right] =\frac{2x^n}{n!}\Re(-i)^{n+1}r_n(x),
\end{multline*}
\begin{multline*}
P_n'(x) =\frac{2x^n}{n!}\Re(-i)^{n+1}r_n(x) +\frac{n+2}{(n+1)!}\Re\big [ ir_n(x)(-ix)^n+ 
\\
+(-i)^nnx^{n-1}r_{n+1}(x) \big] \eqqcolon  \frac{nx^{n-1}}{(n+1)!}\Re Q_n(x),
\end{multline*}
где 
$$
 Q_n(x)=(-i)^{n+1}xr_n(x)\left (\tfrac{2(n+1)}n-\tfrac{n+2}n\right )+(n+2)(-i)^nr_{n+1}(x)=
$$
$$
=(-i)^{n+1}xr_n(x)+(n+2)(-i)^nr_{n+1}(x),\quad x\ge0,
$$
и $P_n'(x)\le0\ \Leftrightarrow\ \Re Q_n(x)\le0,$ для чего, в свою очередь, достаточно, чтобы $\Re Q_n'(x)\le0$ при $x\ge0$, т.к. $Q_n(0)=0.$ Имеем
$$
Q_n'(x)=(-i)^{n+1}r_n(x)+(-i)^{n+1}\cdot ixr_{n-1}(x) +(n+2)(-i)^n\cdot ir_n(x) =
$$
$$
=(-i)^{n-1}r_n(x)(n+1)+(-i)^{n}\cdot xr_{n-1}(x)\equiv Q_{n-1}(x).
$$
Продолжая в том же духе, приходим к заключению, что достаточно доказать, что $\Re Q_1(x)\le0$ при $x>0.$

Вычислим
$$
\Re Q_1(x)=\Re\left [-xr_1(x)-3ir_2(x)\right ] =\Re\left [-x(e^{ix}-1)-3i(e^{ix}-1-ix)\right ]=
$$
$$
=3\sin x-x(2+\cos x)\le 3-x\le 0\quad\text{при}\quad x\ge3.
$$
Для $x\in(0,3)$ проведём более аккуратный анализ. Несложно убедиться, что 
\begin{equation}\label{x*cos(x)-sin(x)<0,0<x<pi}
\frac{\sin x}{x} >\cos x\quad \text{при}\quad 0<x<\pi,
\end{equation}
поэтому производная интересующей нас функции
$$
\big(\Re Q_1(x)\big)'=2(\cos x-1)+x\sin x =4\sin\tfrac{x}2\left(\tfrac{x}2\cos\tfrac{x}2-\sin\tfrac{x}2\right )
$$
отрицательна при $x\in(0,3)\subset(0,2\pi),$ а следовательно, $\Re Q_1(x)\hm\le\Re Q_1(0)=0$ для всех $x\ge0.$
\end{proof}

Из доказанной леммы аналогично теореме~\ref{ThCh.F.TaylorEstim(n,d)Classic} вытекает 

\begin{theorem}[см.~\cite{Prawitz1991}]
\label{ThCh.F.TaylorEstimPrawitz}
Если $\xi$~--- с.в. c х.ф. $f$  и $\beta_{n}\coloneqq\E|\xi|^n<\infty$ при некотором $n\in\N,$ то для всех $ t\in\R$
\begin{multline}%\label{Ch.F.TaylorEstimPrawitz}
|R_n(t)-\frac{n}{2(n+1)}\cdot\frac{\alpha_n}{n!}\,(it)^n|\coloneqq
\\
\coloneqq \bigg|f(t)-\sum_{k=0}^{n-1}\frac{\alpha_k}{k!}(it)^k -\frac{n}{2(n+1)}\cdot\frac{\alpha_n}{n!}\,(it)^n\bigg|\le\frac{n+2}{2(n+1)} \frac{\beta_n|t|^n}{n!}.
\end{multline}
\end{theorem}

\begin{corollary}
В условиях теоремы~\ref{ThCh.F.TaylorEstimPrawitz} для любого $t\in\R$
\begin{equation}\label{Ch.F.TaylorEstimPrawitz}
|R_n(t)| \le\left[\frac{n}{2(n+1)}\,|\alpha_n|+\frac{n+2}{2(n+1)}\,\beta_n\right]\frac{|t|^n}{n!}.
\end{equation}
\end{corollary}

По сравнению с~\eqref{Ch.F.TaylorEstim(n,1)Classic}, в~\eqref{Ch.F.TaylorEstimPrawitz} удалось ``отщепить'' от абсолютного момента $\beta_n$ некоторую (всегда меньшую) часть $\frac{n}{2(n+1)}$ и заменить её на алгебраический момент $|\alpha_n|\le\beta_n$, который может равняться нулю при нечётных $n$, каким бы большим ни было значение $\beta_n$, например, в случае симметричного распределения~$\xi$. При этом  ``отщепляемая'' доля увеличивается с ростом $n$ и при больших $n$ приближается к $\frac12.$

Заметим также, что в~\eqref{Ch.F.TaylorEstimPrawitz}, как и выше в~\eqref{Ch.F.TaylorEstim(n,1)Classic}, асимптотическое равенство выполняется при $t\to0$ для любого вырожденного распределения с.в.~$\xi$.

\section{Формулы обращения}

\begin{enumerate}

\item \textbf{Формула обращения для приращений ф.р.}: если $x$ и ${x+h}$~--- точки непрерывности ф.р. $F$ с х.ф. $f$, то 
$$
F(x+h)-F(x)=\frac1{2\pi}\lim_{T\to\infty}\int_{-T}^T\left (1-e^{-ith}\right )\frac{e^{-itx}}{it} f(t)\dd t.
$$

\item\textbf{Формула обращения для ФОВ в явном виде} (для ф.р. доказана в 1948\,г. в работе Gurland~\cite{Gurland1948}): для любой ФОВ~$F$ с преобразованием Фурье--Стилтьеса 
$$
f(t)\coloneqq \int_{-\infty}^\infty e^{itx}\dd F(x),\quad t\in\R,
$$
в каждой точке $x\in\R$ справедливо соотношение
\begin{multline}\label{FBVInversionFormulaExpcilit}
\tfrac12(F(x+0)+F(x-0))= 
\\
= \tfrac12(F(+\infty)+F(-\infty)) +\frac1{2\pi}\vp\int_{-\infty}^\infty \frac{e^{-itx}}{-it} f(t)\dd t,
\end{multline}
где главное значение интеграла берётся и на бесконечности, и в нуле: %$\vp\int_{-\infty}^\infty\coloneqq\lim_{h\to0+,T\to\infty}\int_{h<|t|<T}$
$$
\vp\int_{-\infty}^\infty\coloneqq\lim_{T\to\infty}\lim_{h\to0+}\int_{h<|t|<T}.
$$
Отметим, что если $x$~--- точка непрерывности $F$, то левая часть~\eqref{FBVInversionFormulaExpcilit} равняется $F(x).$ Если $F$~--- ф.р., то первое слагаемое в правой части~\eqref{FBVInversionFormulaExpcilit} равняется $\frac12.$

\item \textbf{Формула обращения для плотностей}: если х.ф. $f\in\mathcal L_1(\R)$, то соответствующее распределение абсолютно непрерывно и его плотность $p$ можно найти по формуле
\begin{equation}\label{Inv.Form.4Density}
p(x)=\frac1{2\pi}\int_{-\infty}^\infty e^{-itx} f(t)\dd t
\end{equation}
для $\la$-почти всех $x\in\R.$ При этом интеграл в правой части является непрерывной и ограниченной функцией~$x$. Обобщения формулы обращения~\eqref{Inv.Form.4Density} для абсолютно непрерывных распределений с  неинтегрируемой х.ф. можно найти, например, в книге~\cite[\S~3.2]{Lukacs1970}. 

\end{enumerate}

\section{Примеры распределений и х.ф.}

\subsection{Нормальное распределение $N(a,\sigma^2)$} 

%Пусть $\xi\sim N(a,\sigma^2),$ где параметры $a\in\R$, $\sigma>0$.
Плотность $ p(x)=\frac1{\sigma}\phi\left (\frac{x-a}{\sigma}\right )$, $a\in\R$, $\sigma>0$, где $\phi(x)\coloneqq\frac1{\sqrt{2\pi}}e^{-x^2/2}$~--- плотность стандартного (с параметрами $a=0$, $\sigma^2=1$) нормального закона.
\\
Х.ф. $f(t)=e^{ita-\sigma^2t^2/2}$.
\\
Ф.р. $F(x)=\Phi\left (\frac{x-a}{\sigma}\right )$, где  $\Phi(x)\coloneqq\int_{-\infty}^{x}\phi(t)\dd t$~--- стандартная нормальная ф.р.

\subsection{Гамма-распределение  $\Gamma(\lambda,\alpha)$}

Плотность и х.ф. гамма-распределения с параметром формы $\alpha>0$ и масштаба $\lambda>0$ имеют вид:
$$
p(x)=\frac{\lambda^\alpha}{\Gamma(\alpha)}x^{\alpha-1}e^{-\lambda x}\I(x>0),\quad f(t)=\Big(\frac{\lambda}{\lambda-it}\Big)^{\alpha}.
$$
Заметим, что формула обращения~\eqref{Inv.Form.4Density} применима здесь только при $\alpha>1.$ При $\alpha=1$ (показательное распределение) х.ф. уже не интегрируема, а плотность разрывна (в нуле). При  $\alpha<1$ плотность не только разрывна, но и не ограничена (при $x\to0+$).

\subsection{Показательное распределение $exp(\lambda)$ и распределение Лапласа}

Частным случаем гамма-распределения с параметром формы $\alpha=1$ является показательное распределение: $exp(\lambda)=\Gamma(\lambda,1)$. Его плотность и х.ф. имеют вид:
$$
p(x)=\lambda e^{-\lambda x}\I(x>0),\quad f(t)=\frac{\lambda}{\lambda-it}= \frac{\lambda^2}{\lambda^2+t^2}+ i\,\frac{\lambda t}{\lambda^2+t^2} ,\quad 
\lambda>0.
$$
Пусть с.в. $\xi,\xi',\eta$ независимы, $\xi,\xi'\sim\exp(\lambda)$, $\Prob(\eta=\pm1)=\frac12$.
Несложно убедиться, что оба типа симметризации, рассмотренные в разделе~\ref{SecCh.F.General}, совпадают:
$$
f_{\xi-\xi'}(t)=|f(t)|^2=(\Re f(t))^2+(\Im f(t))^2=\frac{\la^2}{\la^2+t^2}=\Re f(t)=f_{\xi\eta}(t),
$$
т.е. с.в. $\xi-\xi'\eqd\xi\eta$ имеют так называемое распределение Лапласа с плотностью
$$
p_{\xi-\xi'}(x)=\tfrac\la2 e^{-\la|x|},
$$
которую можно найти, например, по формуле обращения~\eqref{Inv.Form.4Density} через х.ф. $f_{\xi-\xi'}$ или по формуле свёртки через плотность $p$.

\subsection{Распределение с треугольной х.ф.} \label{Ex:SmoothingDistrTriangChF}

Пусть $f(t)=(1-|t|)^+$, $t\in\R.$ Поскольку $f\in\mathcal L_1(\R)$, используя формулу обращения~\eqref{Inv.Form.4Density} и теорему единственности, несложно убедиться, что $f(t)$ является характеристической функцией абсолютно непрерывного распределения с плотностью
$$
p(x)=\frac{1-\cos x}{\pi x^2}=\frac1{2\pi}\Big(\frac{\sin\frac{x}2}{{x}/2}\Big)^2.
$$
Отметим, что у  данного распределения  настолько тяжёлые хвосты, что не существует моментов выше первого, однако довольно удобная финитная (т.е. отличная от нуля лишь на конечном интервале) х.ф., которая нам пригодится в дальнейшем.

\subsection{Равномерное распределение и его симметризации}

Плотность и х.ф. равномерного $U(a,b)$ на отрезке $[a,b]$ распределения имеют вид:
$$
p(x)=\frac1{b-a}\I(a\le x\le b),\quad f(t)=\frac{e^{itb}-e^{ita}}{it(b-a)},\quad a<b,
$$
в частности, на симметричном $[-a,a]$ отрезке
$$
p(x)=\frac1{2a}\I(-a\le x\le a),\quad f(t)=\frac{\sin at}{at},\quad a>0.
$$
Пусть с.в. $\xi,\xi'\sim U(0,1)$ и независимы, а с.в. $\zeta,\zeta'\sim U(-\frac12,\frac12)$ и тоже независимы. Тогда х.ф.  свёрточной  симметризации $U(0,1)$ имеет вид
\begin{multline*}
f_{\xi-\xi'}(t)=|f_\xi(t)|^2=\Big|\frac{e^{it}-1}{it}\Big|^2=\frac{2(1-\cos t)}{t^2}=\frac{\sin^2\frac{t}2}{(t/2)^2}=
\\
=(f_\zeta(t))^2=f_{\zeta+\zeta'}(t).
\end{multline*}
Заметим, что полученная х.ф. с точностью до нормировочной константы совпадает с плотностью из предыдущего примера, следовательно, по формуле обращения, с.в. $\xi-\xi'\eqd\zeta+\zeta'$ имеют распределение с треугольной плотностью $p(x)=(1-|x|)^+$, которая выполняла роль характеристической функции в предыдущем примере.

Рассмотрим второй тип симметризации равномерного распределения $U(0,1)$. Пусть с.в. $\eta\sim Ber(\frac12)$ не зависит от $\xi$, тогда 
$$
f_{\xi\eta}(t)=\Re f_\xi(t)=\frac{\sin t}{t}=f_{2\zeta}(t),
$$
т.е. $\xi\eta\eqd2\zeta\sim U(-1,1).$

\subsection{Распределение Стьюдента}

Плотность распределения Стьюдента $St(\nu)$ с $\nu>0$ степенями свободы имеет вид
$$
s_\nu(x)= \frac{\Gamma(\frac{\nu+1}2)}{\sqrt{\pi\nu}\Gamma(\nu/2)} \Big(1+\frac{x^2}\nu\Big)^{-\frac{\nu+1}{2}},\quad x\in\R.
$$
В частности, случай $\nu=1$ соответствует распределению Коши с плотностью 
$$
s_1(x)=\frac1{\pi(1+x^2)} \text{\quad и х.ф. \quad } f(t)=e^{-|t|}.
$$ 
Несложно убедиться, что распределение Стьюдента является гамма-масштабной смесью нормального закона. А именно, пусть с.в. $X_\nu\sim St(\nu)$, $Z\sim N(0,1),$ $\Lambda_\nu\sim\Gamma(\frac\nu2,\frac\nu2)$ и с.в. $Z,\Lambda_\nu$ независимы. Тогда
$$
X_\nu\eqd Z/\sqrt{\Lambda_\nu}.
$$
%в частности, при $\nu=1$ плотность $\Lambda_1\sim\Gamma(\frac12,\frac12)$ имеет вид $p_{\Lambda_1}(\la)=(2\pi \la)^{-1/2}e^{-\la/2}\I(\la>0).$ 
Действительно, плотность с.в. $Z/\sqrt{\Lambda_\nu}$  в силу независимости можно искать по формуле
\begin{multline*}
p_{Z/\sqrt{\Lambda_\nu}}(x) =\int_0^\infty\sqrt{\la}\phi(x\sqrt{\la}) \frac{(\nu/2)^{\nu/2}}{\Gamma(\nu/2)}\la^{\nu/2-1}e^{-\nu\la/2}\dd\la=
\\
=\frac{(\nu/2)^{\nu/2}}{\sqrt{2\pi}\Gamma(\nu/2)}\int_0^\infty \la^{(\nu-1)/2}e^{-\la(\nu+x^2)/2}\dd\la=s_\nu(x).
 \end{multline*}

С другой стороны, если с.в. $Y,Y_1,\ldots,Y_n\sim N(0,1)$  независимы, то, по определению, с.в. $Y_1^2+\ldots+Y_n^2\sim\chi^2(n)=\Gamma(\frac12,\frac n2)$ имеет $\chi^2$-распределение с $n$ степенями свободы. После нормировки получаем $\frac1n(Y_1^2+\ldots+Y_n^2)\sim\Gamma(\frac{n}2,\frac{n}2),$ а следовательно, по доказанному выше,
$$
\frac{Y}{\sqrt{\frac1n(Y_1^2+\ldots+Y_n^2)}}= \frac{Y\sqrt{n}}{\sqrt{Y_1^2+\ldots+Y_n^2}}\sim St(n).
$$

\subsection{(Сложное) распределение Пуассона}

Распределение Пуассона $\xi\sim Pois(\la)$ с параметром $\la>0$ определяется рядом распределения
$$
\Prob(\xi=k)= \frac{\la^k}{k!}e^{-\la},\quad k=0,1,2,\ldots,\quad\text{с х.ф. }\ h(t)=e^{\la(e^{it}-1)}.
$$
Определим теперь обобщённое (сложное, составное)\footnote{В англоязычной литературе~--- \textit{compound}.} распределение Пуассона. Пусть с.в. $X_1,X_2,\ldots$ независимы и имеют общую ф.р. $F$ с х.ф. $f$. Тогда $F_n(x)\coloneqq F^{*n}(x)$ и $f_n(t)=(f(t))^n$ являются соответственно ф.р. и х.ф. суммы $S_n=X_1+\ldots+X_n.$ Пусть с.в. $N_\la\sim Pois(\la)$ независима от $X_1,\ldots,X_n$. Тогда, по свойствам х.ф.,
$$
g(t)\coloneqq \sum_{n=0}^\infty \Prob(N_\la=n)f_n(t)= \sum_{n=0}^\infty\frac{\la^n}{n!}e^{-\la}(f(t))^n=\exp\{\la(f(t)-1)\}
$$
является х.ф. Несложно убедиться, что $g(t)$ является х.ф. пуассоновской случайной суммы $S_{N_\la}=X_1+\ldots+X_{N_\la}.$

\section{Неравенства сглаживания}
\label{SectionSmoothIneq}

Из формулы обращения~\eqref{FBVInversionFormulaExpcilit} для ФОВ~$R$ с преобразованием Фурье--Стилтьеса $r(t)=\int e^{itx}\dd R(x),$ $t\in\R,$
$$
\frac{R(x+0)+R(x-0)}2 = \frac{R(+\infty)+R(-\infty)}2 + \frac1{2\pi}\vp\int_{-\infty}^{\infty}\frac{e^{-itx}}{-it}r(t)\dd t, 
$$
$x\in\R,$ вытекает, что разность $R\coloneqq F-G$ двух непрерывных ф.р.~$F$ и~$G$ с х.ф.~$f$ и~$g$ можно равномерно оценить интегралом
\begin{equation}\label{rho(F,G)<=int|f(t)-g(t)|/|t|dt}
\rho(F,G)\coloneqq \sup_{x\in\R}|F(x)-G(x)| \le \frac1{2\pi}\int_{-\infty}^\infty \abs{\frac{f(t)-g(t)}{t}}\dd t,
\end{equation}
принимающим конечные значения, если, например, выполнены условия:
$$
\int_\R|x|^\d \dd F(x)<\infty,\quad \int_\R |x|^\d\dd G(x)<\infty,
$$
$$
\int_{|t|>A}\abs{\frac{f(t)}{t}}\dd t<\infty,\quad \int_{|t|>A}\abs{\frac{g(t)}{t}}\dd t<\infty
$$
для некоторого $\d>0$ и $A>0$. В частности, для выполнения последних условий достаточно абсолютной интегрируемости $f$ и~$g$. Цель данного раздела~--- получить аналог~\eqref{rho(F,G)<=int|f(t)-g(t)|/|t|dt} для необязательно гладких распределений.

Для достижения этой цели мы будем использовать \textit{сглаживание.} Пусть $\xi$~--- произвольная с.в. с ф.р. $F$ и х.ф. $f$, а $\eps$~--- независимая от $\xi$ абсолютно непрерывная с.в. с плотностью $p$. Тогда, как известно, свёртка $\xi+\eps$ тоже абсолютно непрерывна с плотностью
$$
p_{\xi+\eps}(x) =\int_\R p(x-y)\dd F(y),\quad x\in\R,
$$
и если $\eps$ имеет абсолютно интегрируемую х.ф. $\widehat p$, то х.ф. свёртки $f_{\xi+\eps}=f\cdot\widehat p$ тоже абсолютно интегрируема. Таким образом, свёртка $\xi+\eps$ является более \textit{гладким} распределением по отношению к распределению $\xi$, и поэтому её называют ещё $\eps$-\textit{сглаженным} распределением~$\xi$. При этом можно надеяться, что сглаженное распределение близк\'о к исходному, если с.в.~$\eps$ в определенном смысле мала (например, с большой вероятностью принимает значения из малой окрестности нуля). 

Заметим, что если $\eta$~--- с.в. с ф.р. $G$, независимая от $\eps$, то разность сглаженных ф.р. $\xi$ и $\eta$ выражается в виде
$$
F_{\xi+\eps}(x)-F_{\eta+\eps}(x)=
$$
$$
=\int_\R (F(x-y)-G(x-y))p(y)\dd y= \int_\R (F(y)-G(y))p(x-y)\dd y.
$$
%и  равномерное расстояние между сглаженными распределениями принимает  вид
%$$
%\sup_{x\in\R}\abs{F_{\xi+\eps}(x)-F_{\eta+\eps}(x)} =\sup_{x\in\R}\int_\R (F(x-y)-G(x-y))p(y)\dd y=
%$$
%$$
%= \int_\R (F(y)-G(y))p(x-y)\dd y\eqqcolon\Delta_p
%$$

Сглаживающая функция $p$ здесь называется \textit{ядром} или \textit{фильтром}.

\begin{theorem}\label{ThSmoothingIneq}
Пусть $F(x)$~--- неубывающая  ограниченная функция, $G(x)$~---
ФОВ с 
$$
A\coloneqq\sup_{x\in\R}|G'(x)|<\infty,
$$
${F(\pm\infty)=G(\pm\infty)}$, $p$~--- плотность симметричного распределения вероятностей на $\R$. Положим
$$
V_p(\alpha)\coloneqq\int_{-\alpha}^\alpha p(x)\dd x= 2\int_0^\alpha p(x)\dd x,\quad \alpha>0,
$$
$$
\Delta_p\coloneqq\sup_{x\in\R}\abs{\int_{-\infty}^\infty (F(x-y)-G(x-y))p(y)\dd y},
$$
$$
\Delta\coloneqq\sup_{x\in\R}|F(x)-G(x)|.
$$
Тогда для любого $\alpha>0$
\begin{equation}\label{SmoothingIneq}
\Delta_p\ge\Delta(2V_p(\alpha)-1)- A\alpha V_p(\alpha).
\end{equation}
\end{theorem}

\begin{proof}%[Доказательство теоремы~\ref{ThSmoothingIneq}.]
Заметим, что
$$
\Delta=\max\Big\{\sup_x(F(x)-G(x)),\, \sup_x(G(x)-F(x)) \Big\}.
$$

1)~Предположим  сначала, что $\Delta=\sup_x(F(x)\hm-G(x))$. Тогда для
любого $\eta>0$ найдётся точка $x_\eta\in\R$ такая, что
$$
\Delta-\eta\le F(x_\eta)-G(x_\eta)\le\Delta.
$$
Зафиксируем произвольное $\alpha>0$ и положим $x\coloneqq x_\eta+\alpha$, тогда, в силу монотонности $F$ и с применением формулы конечных приращений для $G$, при  всех $y\le\alpha$ получим
$$
F(x-y)-G(x-y)=F(x_\eta+\alpha-y)-G(x_\eta+\alpha-y) \ge
$$
$$
\ge  F(x_\eta)\mp G(x_\eta)-G(x_\eta+\alpha-y) \ge\Delta-\eta- A(\alpha-y).
$$
Используя полученную оценку для $y\in[-\alpha,\alpha]$ и  тривиальную оценку $F(x-y)\hm-G(x-y)\ge-\Delta$ для $|y|>\alpha$, а также чётность $p$, заключаем, что
$$
\Delta_p\ge \int_{-\infty}^{-\alpha}+ \int_{-\alpha}^\alpha+\int_{\alpha}^\infty (F(x-y)-G(x-y))p(y)\dd y\ge
$$
$$
\ge\int_{-\alpha}^{\alpha}(\Delta-\eta-A(\alpha-y))p(y)\dd y -\Delta\bigg[ \int_{-\infty}^{-\alpha}+\int_{\alpha}^\infty p(y)\dd y\bigg] = 
$$
$$
=(\Delta-\eta-A\alpha)V_p(\alpha) - \Delta(1-V_p(\alpha) )=\Delta(2V_p(\alpha) -1) -A\alpha V_p(\alpha) -\eta V_p(\alpha).
$$
Устремляя теперь $\eta$ к нулю и замечая, что функция $V_p$  ограничена на прямой (снизу~--- нулём, сверху~--- единицей), получаем искомое неравенство.

2)  Пусть теперь $\Delta=\sup_x(G(x)\hm-F(x))$. Тогда для
любого $\eta>0$ найдётся точка $x_\eta\in\R$ такая, что
$$
\Delta-\eta\le G(x_\eta)-F(x_\eta)\le\Delta.
$$
Для $\alpha>0$ положим $x\coloneqq x_\eta-\alpha$, тогда аналогично п.\,1 для  всех $y\ge-\alpha$ получим
$$
G(x-y)-F(x-y)=G(x_\eta-\alpha-y)-F(x_\eta-\alpha-y) \ge
$$
$$
\ge  G(x_\eta)-F(x_\eta)-A|-\alpha-y| \ge\Delta-\eta- A(\alpha+y),
$$
$$
\Delta_p
\ge\int_{-\alpha}^{\alpha}(\Delta-\eta-A(\alpha+y))p(y)\dd y -\Delta\bigg[ \int_{-\infty}^{-\alpha}+\int_{\alpha}^\infty p(y)\dd y\bigg] = 
$$
$$
%=(\Delta-\eta-A\alpha)V_p(\alpha) - \Delta(1-V_p(\alpha) )
=\Delta(2V_p(\alpha) -1) -A\alpha V_p(\alpha) -\eta V_p(\alpha)\to \Delta(2V_p(\alpha) -1) -A\alpha V_p(\alpha) 
$$
при $\eta\to0+$.
\end{proof}

Если преобразование Фурье\footnote{В теории интегралов Фурье \textit{преобразованием} (\textit{образом}) \textit{Фурье}  интегрируемой функции $p$ называется функция, комплексно сопряжённая к  $\widehat p/\sqrt{2\pi}$. Но поскольку комплексное сопряжение и нормировка постоянным множителем не оказывают принципиального влияния на изучаемые свойства данного объекта, мы используем те же термины ``преобразование/образ Фурье'' и для $\widehat p.$} (характеристическая функция) 
$$
\widehat p(t)\coloneqq \int_\R e^{itx}p(x)\dd x
$$
плотности $p$ абсолютно интегрируемо,  то в силу~\eqref{rho(F,G)<=int|f(t)-g(t)|/|t|dt}
$$
\Delta_p\le \frac1{2\pi}\int_{-\infty}^\infty \abs{\widehat p(t)}\cdot \abs{\frac{f(t)-g(t)}{t}}\dd t,
$$
где $f,g,$~-- преобразования Фурье--Стилтьеса ФОВ $F$ и $G$. Введём ещё один параметр $T>0$ и положим
$$
p^{}_T(x)\coloneqq Tp(Tx),\quad x\in\R.
$$
Ясно, что если $p$~--- плотность с.в. $\eps$, то $p^{}_T$~--- плотность с.в. $\eps/T$, и последнюю можно сделать сколь угодно ``малой'' (по вероятности и с вероятностью $1$) за счёт выбора достаточно большого $T$. Заметим, что 
$$
\widehat p^{}_T(t)=\widehat p\left(\frac{t}{T}\right),\quad V_{p^{}_T}(\alpha)=T\int_{-\alpha}^\alpha p(Tx)\dd x=V_p(\alpha T).
$$
Записывая теперь для оценки $\Delta_{p^{}_T}$ снизу теорему~\ref{ThSmoothingIneq} с $p^{}_T$ вместо $p$ и используя обозначение $\alpha$ вместо $\alpha T$, получаем

\begin{corollary}\label{CorSmoothingIneqCh.F.}
Пусть, в дополнение к условиям теоремы~\ref{ThSmoothingIneq}, преобразование Фурье $\widehat p$ плотности $p$ абсолютно интегрируемо, и  $f,g$~--- преобразования Фурье--Стилтьеса ФОВ $F$ и $G$ соответственно.  Тогда для любых $\alpha,T>0$
\begin{multline*}
Q_p(T)\coloneqq \frac{1}{2\pi} \int_{-\infty}^\infty\Big|\widehat p\Big(\frac{t}T\Big)\Big|\cdot \abs{\frac{f(t)-g(t)}{t}}\dd t\ge \Delta_p\ge
\\
\ge\Delta(2V_p(\alpha)-1) -\frac{A\alpha}{T}V_p(\alpha).
\end{multline*}
\end{corollary}

Рассмотрим функцию $V_p(\alpha)=2\int_0^\alpha p(x)\dd x$. Она непрерывно и монотонно возрастает для $\alpha\ge0$ с  
$$
V_p(0)=0,\quad \lim_{\alpha\to\infty}V_p(\alpha)=1,
$$ 
следовательно, уравнение $2V_p(\alpha)-1=0$ имеет хотя бы один корень. Обозначим $\alpha_p$ точную верхнюю грань всех корней этого уравнения; тогда для всех $\alpha>\alpha_p$ будем иметь $2V_p(\alpha)-1>0$ и, следовательно,
\begin{equation}\label{SmoothingIneqCh.F.}
\Delta\le \frac1{2V_p(\alpha)-1}\bigg[ \frac{1}{2\pi} \int_{-\infty}^\infty\Big|\widehat p\Big(\frac{t}T\Big)\Big|\cdot \abs{\frac{f(t)-g(t)}{t}}\dd t +\frac{A\alpha}{T}V_p(\alpha)\bigg].
\end{equation}

Из~\eqref{SmoothingIneqCh.F.} вытекает известное неравенство Феллера~\cite[Гл.\,16,  \S\,3, лемма\,1]{Feller1967}, если в качестве сглаживающего ядра взять функцию из раздела~\ref{Ex:SmoothingDistrTriangChF}, использовавшуюся Берри и Эссееном:
\begin{equation}\label{BE-smoothing-filter}
p(x)=\frac{1-\cos x}{\pi x^2}=\frac1{2\pi}\cdot\frac{\sin^2(x/2)}{(x/2)^2}, \quad x\in\R,
\end{equation}
с %$\alpha_p=1.6995\ldots$ и 
треугольной х.ф. $\widehat{p}(t)=(1-|t|)^+\le\I(|t|\hm\le1).$ Заметим, что
$$
V_p(\alpha)=1-2\int_{\alpha}^\infty p(x)\dd x =1-2\int_{\alpha}^\infty \frac{1-\cos x}{\pi x^2}\dd x\ge
$$
$$
\ge 1-\frac4\pi\int_\alpha^\infty\frac{\dd x}{x^2} =1-\frac{4}{\pi\alpha},\quad \alpha >0,
$$
$$
2V_p(\alpha)-1\ge 1-\frac{8}{\pi\alpha}>0\quad \text{при}\quad\alpha>\frac8\pi=2.5464\ldots,
$$
так что $\alpha_p\le8/\pi=2.5464\ldots$ (можно убедиться, что точное значение $\alpha_p\hm=1.6995\ldots$).  Замечая теперь, что функция $\frac{v}{2v-1}$ монотонно убывает при $v>1/2$, и значит,
$$
\frac{V_p(\alpha)}{2V_p(\alpha)-1}\le \frac{\pi\alpha-4}{\pi\alpha-8},\quad \alpha>\frac8\pi,
$$
и выполняя %в~\eqref{SmoothingIneqCh.F.} 
замену переменных
$$
%b\coloneqq \frac{\pi\alpha}{\pi\alpha-8}>1,
\alpha = \frac{8b}{\pi(b-1)},\quad b>1,
$$
%так что $1-\frac8{\pi\alpha}=b^{-1},$ 
получаем

\begin{corollary}\label{CorFellerSmoothIneq}
В условиях следствия~\ref{CorSmoothingIneqCh.F.} для всех $b>1$ и $T>0$ справедлива оценка
\begin{equation}\label{SmoothingIneqEsseenKernel}
\Delta\le \frac{b}{2\pi} \int_{-T}^T\abs{\frac{f(t)-g(t)}{t}}\dd t + \frac{4b(b+1)}{\pi(b-1)}\cdot \frac{A}{T}.
\end{equation}
В частности, при $b=2$
\begin{equation}\label{FellerSmoothIneq}
\Delta\le \frac{1}{\pi} \int_{-T}^T\abs{\frac{f(t)-g(t)}{t}}\dd t + \frac{24}{\pi}\cdot\frac{A}{T},\quad T>0.
\end{equation}
\end{corollary}

Неравенство~\eqref{FellerSmoothIneq} и носит имя Феллера.

Формулы обращения, позволяющие вместо функций множеств (мер, распределений) работать с более удобными комплекснозначными функциями вещественного аргумента (преобразований Фурье--Стилтьеса), лежат в основе \textit{метода характеристических функций}.

Отметим, что рассмотренные нами условия на ядро $p$ в части знакопостоянства и чётности не являются принципиальными. Такие условия рассматривались в исторически первых работах, использовавших сглаживание, а именно, в  работах 1941--1942\,гг.  А.\,Берри~\cite{Berry1941} и  К.-Г.\,Эссеена~\cite{Esseen1942} (при этом в~\cite{Esseen1942} дополнительно предполагалась конечность интеграла $\int|x|p(x)\dd x$), где для построения ядра использовалась функция~\eqref{BE-smoothing-filter}, Г.\,Бергстрёма~\cite{Bergstrom1944}, где в качестве $p$ выбиралась плотность стандартного нормального закона, и  несколько более поздних статьях 1965--1967\,гг. В.\,М.\,Золотарёва~\cite{Zolotarev1965,Zolotarev1967a,Zolotarev1967b}, в которых было показано, что ядро $p$ может являться плотностью любого симметричного распределения вероятностей, и доказано неравенство~\eqref{SmoothingIneqCh.F.}.  В опубликованной в 1966\,г. диссертации С.\,Цаля~\cite{Zahl1966} отмечено, что условие знакопостоянства не является необходимым, и предложено использовать знакопеременные несимметричные ядра, однако само неравенство сглаживания в явном виде не выписано. В диссертации П.\,ван~Бика~\cite{VanBeek1971}, не содержащей ссылок на работу Цаля, неравенство сглаживания Золотарёва обобщено на знакопеременные симметричные ядра. Одновременно с работой ван Бика была опубликована статья В.\,И.\,Паулаускаса~\cite{Paulauskas1971} (также не имеющая ссылок на работу Цаля), в которой доказан аналог неравенства Золотарёва для знакопостоянных несимметричных ядер. Наконец, в работе~\cite{Shevtsova2010SmoothIneq} (также см.~\cite[\S\,1.10.1]{Shevtsova2016}) был доказан аналог~\eqref{SmoothingIneq} для произвольной абсолютно интегрируемой функции $p$ с $\int_\R p(x)\dd x>0.$

%Неравенство Паулаускаса~--- единственное из вышеперечисленных, не применявшееся для уточнения верхней оценки абсолютной постоянной в неравенстве Берри--Эссеена, хотя в конце статьи Паулаускаса и отмечено, что оно предназначено в первую очередь именно для этой цели.

%\subsection{Неравенство сглаживания Правитца}

Отметим также, что в некоторых задачах, связанных с оцениванием точности нормальной аппроксимации, более точным, чем неравенства Золотарёва--Ван\,Бика--Паулаускаса, оказывается следующее неравенство типа сглаживания, доказанное в 1972\,г. Х.\,Правитцем~\cite{Prawitz1972} и уточняющее результаты Бомана~\cite{Bohman1963}. В нём роль образа Фурье $\widehat p$ сглаживающего ядра, не являющегося ни знакопостоянным, ни симметричным, играет функция $K(t)-\frac{i}{2\pi}$, определённая ниже.

\begin{lemma}[см.~\cite{Prawitz1972}] 
Для любой ф.р. $F$ с х.ф. $f$ при всех $x\in\R$ и $T>0$ справедливы неравенства
$$
F(x+0)\le \frac12+\vp\int_{-T}^Te^{-itx}\frac1TK\left(\frac{t}T\right)f(t)\dd t,
$$
$$
F(x-0)\ge \frac12-\vp\int_{-T}^Te^{-itx}\frac1TK\left(-\frac{t}T\right)f(t)\dd t,
$$
где  
$$
K(t)\coloneqq\frac12(1-|t|)+\frac i2\left[(1-|t|)\cot\pi t+\frac{\sign
t}\pi\right],\quad  t\in(-1,1)\setminus\{0\}.
$$
При этом для $K(t)$ при всех $t\in(-1,1)\setminus\{0\}$ справедливы оценки
$$
|K(t)|\le\frac{1.0253}{2\pi|t|},\quad \left|K(t)-\frac{i}{2\pi
t}\right|\le \frac12\bigg(1-|t|+\frac{\pi^2t^2}{18}\bigg).
$$
\end{lemma}

В качестве следствия данной леммы, применяя формулу обращения~\eqref{FBVInversionFormulaExpcilit} для ф.р. $\Phi$ стандартного нормального закона с х.ф. $e^{-t^2/2}$ 
$$
\Phi(x)=\frac12+\frac{i}{2\pi}\vp\int_{-\infty}^\infty e^{-itx}e^{-t^2/2}\frac{\dd t}{t},\quad x\in\R,
$$
можно получить следующую оценку равномерного расстояния между произвольной ф.р. $F$ с х.ф. $f$ и стандартной нормальной ф.р.,  принадлежащую Правитцу~\cite{Prawitz1972}:
\begin{multline*}%\label{LemPrawitzSmoothIneq}
\sup_{x\in\R}|F(x)-\Phi(x)|\le
\\
\le
2\int_0^{t_0}|K(t)|\cdot\big|f(Tt)-e^{-T^2t^2/2}\big|\dd t+
2\int_{t_0}^{1}|K(t)|\cdot|f(Tt)|\dd t+
\\
+2\int_0^{t_0}\left|K(t)-\frac i{2\pi t}\right|e^{-T^2t^2/2}\dd t +
\frac1{\pi}\int_{t_0}^\infty e^{-T^2t^2/2}\frac{\dd t}t%\quad t_0\in(0,1],\ T>0.
\end{multline*}
для  всех $t_0\in(0,1]$ и $T>0$.

Как заметил J.\,D.\,Vaaler~\cite{Vaaler1985}, неравенство Правитца
оптимально в некоторых задачах приближения разрывных функций
непрерывными.

\section{Неравенства для интегралов}

Напомним, что для данного пространства с $\sigma$-конечной мерой $(\Omega,\mathcal A,\mu)$ пространство $\mathcal L_p\hm=\mathcal L_p(\Omega)$, где  $p\in[1,\infty]$, состоит из всех измеримых функций $f\colon\Omega\to\R$, для которых величина
$$ 
\norm{f}_p\coloneqq 
\begin{cases}
\displaystyle \Big(\int|f|^p\dd\mu\Big)^{1/p},&p\in[1,\infty),
\\
%\ess\sup |f|\coloneqq
\inf\big\{M>0\colon \mu\big(\{\omega\in\Omega\colon |f(\omega)|>M\}\big)=0 \big\},&p=\infty,
\end{cases}
$$
конечна. Как известно  (см., например,~\cite[Гл.\,8, \S\,42, теорема\,1]{Halmos1950} или~\cite[Гл.\,4, \S\,22]{DyachenkoUlyanov2002}), так определённая функция $\norm{\cdot}_p$ является нормой в пространстве $L_p$ классов эквивалентности (т.е. совпадающих $\mu$-почти всюду) функций из $\mathcal L_p$ и, стало быть, $(L_p,\norm{\cdot}_p)$~--- линейным нормированным пространством. %функция $\rho(f,g)\coloneqq\norm{f-g}_p$, $f,g\in\mathcal L_p,$~--- метрикой.

%Приведём несколько классических результатов из теории меры, которые нам понадобятся в дальнейшем.

\medskip

\noindent\textbf{Неравенство Гёльдера}  (см., например,~\cite[Гл.\,8, \S\,42, теорема\,1]{Halmos1950} или~\cite[Гл.\,4, \S\,22]{DyachenkoUlyanov2002}): пусть $p,q>1$ таковы, что $p^{-1}+q^{-1}=1,$ и $f\in\mathcal L_p$, $g\in\mathcal L_q$, тогда
\begin{equation}\label{HolderIneqIntegral}
\norm{f\cdot g}_1\le \norm{f}_p\cdot\norm{g}_q.
\end{equation}
%или, в развёрнутой форме,
%$$
%\int|fg|\dd\mu\le \Big(\int|f|^p\dd\mu\Big)^{1/p}\Big(\int|g|^q\dd\mu\Big)^{1/q}.
%$$
%Иногда удобнее использовать неравенство Гёльдера с $p\coloneqq1/\d,$ $\d\in(0,1),$ когда оно принимает вид
В частности, для вероятностной меры $\mu$ и интегрируемых неотрицательных  с.в. $\xi$, $\eta$ 
\begin{equation}\label{HolderIneqForExpectations}
\E\xi^\d\eta^{1-\d}\le (\E\xi)^{\d}(\E\eta)^{1-\d},\quad \d\in(0,1),
\end{equation}
для $\Omega\coloneqq\N$, считающей меры $\mu$ и неотрицательных последовательностей $\{a_k\}_{k\in\N},$ $\{b_k\}_{k\in\N}$
\begin{equation}\label{HolderIneqForSeries}
\smallsum_{k=1}^\infty a_k^\d \cdot b_k^{1-\d}\le \Big(\smallsum_{k=1}^\infty a_k\Big)^{\d} \cdot \Big(\smallsum_{k=1}^\infty b_k\Big)^{1-\d},\quad \d\in(0,1).
\end{equation}

\noindent\textbf{Неравенство Ляпунова}: для любых $0\le r\le s\le t<\infty$
и $h\in\mathcal L_t$
$$
\norm{h}_s\le \norm{h}_r^\theta\cdot\norm{h}_t^{1-\theta},\quad \text{где }\ \theta\in[0,1]\colon \ \ \frac\theta r+\frac{1-\theta}t=\frac1s,
$$
вытекает из неравенства Гёльдера~\eqref{HolderIneqIntegral} с $p^{-1}\coloneqq\frac{\theta s}r,$ $q^{-1}\coloneqq\frac{(1-\theta)s}t,$ $f\coloneqq|h|^{\theta s}$ и $g\coloneqq|h|^{(1-\theta)s}$.
В частности, для вероятностной меры $\mu$ и с.в. $\xi$ c $\beta_s\coloneqq\E|\xi|^s<\infty,$ $0\le s\le t$
%$$
%(\E|\xi|^s)^{t-r}\le (\E|\xi|^r)^{t-s} (\E|\xi|^t)^{s-r},\quad 0\le r\le s\le t,
%$$
%$$
%\big(\E|\xi|^s\big)^{t-r}\le \big(\E|\xi|^r\big)^{t-s} \big(\E|\xi|^t\big)^{s-r},\quad 0\le r\le s\le t,
%$$
%$\beta_u\coloneqq\E|\xi|^u<\infty$, $u\in(0,t],$
$$
\beta_s^{t-r}\le \beta_r^{t-s}\beta_t^{s-r},\quad 0\le r\le s\le t,
$$
и при $r=0$
\begin{equation}\label{LyapunovMomentIneq}
\beta_s^{1/s}\le \beta_t^{1/t},\quad 0\le s\le t.
\end{equation}

\noindent\textbf{Неравенство Минковского} (см., например,~\cite[Гл.\,8, \S\,42, теорема\,2]{Halmos1950} или~\cite[Гл.\,4, \S\,22]{DyachenkoUlyanov2002}): для любых $p\ge1,$ $f,g\in\mathcal L_p$ 
% $f+g\in\mathcal L_p$
$$
\norm{f+g}_p\le \norm{f}_p+\norm{g}_p.
$$

\noindent\textbf{Неравенство Йенсена}: если $\xi$~--- с.в., а $g$~--- действительная функция, выпуклая на множестве возможных значений $\xi$ и  $\E|g(\xi)|\hm<\infty$, то 
$$
g(\E\xi)\le \E g(\xi).
$$

\noindent\textbf{Корреляционное неравенство Чебышева}: для любой с.в. $\xi$ и любой пары возрастающих или убывающих функций $f,g$ на $\R$
%\begin{equation}\label{ChebyshevCorrelationIneq}
$$
\E f(\xi)g(\xi)\ge \E f(\xi)\cdot\E g(\xi).
$$
%\end{equation}
%Неравенство~\eqref{ChebyshevCorrelationIneq} 
Это неравенство вытекает из неравенства $\E(f(\xi)-f(\xi'))(g(\xi)-g(\xi'))\ge0$, справедливого для любых с.в. $\xi,\xi'$ (в силу одинаковой монотонности $f$ и $g$), если в качестве $\xi'$ в нем взять независимую копию $\xi$.

\chapter{Оценки скорости сходимости в классической ЦПТ}
\thispagestyle{empty}
\section{Теоремы Линдеберга--Феллера и Ляпунова}

Рассмотрим схему серий $X_{n,1},X_{n,2}\ldots,X_{n,n}$, $n\in\N,$ независимых в каждой серии случайных величин, где первый индекс обозначает номер серии, а второй~--- номер случайной величины внутри серии. Пусть $F_{n,k}$~--- ф.р. с.в. $X_{n,k},$
$$
a_{n,k}\coloneqq\E X_{n,k},\ \sigma_{n,k}^2\coloneqq\D X_{n,k}<\infty,\ k=1,\ldots,n,\quad  B_n^2\coloneqq\smallsum_{k=1}^n\sigma_{n,k}^2>0.
$$ 
%Вообще говоря, каждая серия случайных величин может быть определена на своём в.п., поэтому математические ожидания $\E$ тоже можно снабжать индексами, но поскольку мы будем работать только с моментами определённого порядка, которые уже снабжены индексом $n$, то позволим себе опустить индексы у вероятностных пространств и соответствующих математических ожиданий. В любом случае можно выбрать новое достаточно богатое вероятностное пространство, на котором определить сразу все серии случайных величин, обеспечивая независимость внутри каждой серии.
Положим
$$
S_n\coloneqq\sum_{k=1}^nX_{n,k},\quad \widetilde S_n\coloneqq\frac{S_n-\E S_n}{\sqrt{\D S_n}}=\frac{S_n-\sum_{k=1}^n a_{n,k}}{B_n},
$$
$$
\overline F_n(x)\coloneqq \Prob\big(\widetilde S_n<x\big),\quad \Phi(x)=\frac1{\sqrt{2\pi}}\int_{-\infty}^x e^{-t^2/2}\dd t,\quad x\in\R.
$$

\begin{definition}
Будем говорить, что схема серий $X_{n,1},X_{n,2},\ldots,X_{n,n}$ удовлетворяет \textit{центральной предельной теореме} $($ЦПТ$)$, если
$$
\Delta_n\coloneqq%\rho\big(\overline F_n,\Phi)=
\sup_{x\in\R}|\overline F_n(x)-\Phi(x)| \to0,\quad n\to\infty.\eqno(\text{ЦПТ})
$$
\end{definition}

Поскольку центрирование суммы $S_n$ математическим ожиданием равносильно центрированию каждого слагаемого своим математическим ожиданием, без ограничения общности и для краткости дальнейших обозначений будем считать, что случайные слагаемые центрированы, т.е. $a_{n,k}\coloneqq\E X_{n,k}=0,$ $k=1,\ldots,n$.

Выполнение ЦПТ позволяет использовать при больших $n$ аппроксимацию распределения суммы $S_n$ нормальным законом с соответствующими  параметрами сдвига и масштаба. Сформулируем условия, которые обеспечивают справедливость ЦПТ.

Будем говорить, что выполняется \textit{условие Линдеберга}, если для любого $ \eps>0$
$$
L_n(\eps)\coloneqq \frac1{B_n^2}\sum_{k=1}^n\E X_{n,k}^2\I(|X_{n,k}|>\eps B_n)\to0, \quad n\to\infty.\eqno(L)
$$
Величина $L_n(\eps)$ называется \textit{дробью Линдеберга}.

В 1922\,г. Линдеберг~\cite{Lindeberg1922} доказал, что условие $(L)$ достаточно для выполнения $($\textit{ЦПТ}$)$. В 1935\,г. Феллер~\cite{Feller1935} дополнил теорему Линдеберга, показав, что условие $(L)$ не только достаточно, но и необходимо для выполнения $($\textit{ЦПТ}$)$, если случайные слагаемые являются \textit{равномерно пренебрежимо малыми} в следующем смысле:
$$
\lim_{n\to\infty}B_n^{-2}\max_{1\le k\le n}\sigma_{n,k}^2=0.\eqno (F)
$$
Условие $(F)$ называется \textit{условием Феллера}. Теорему Линдеберга--Феллера можно записать в следующем виде.

\begin{theorem}[теорема Линдеберга--Феллера]
В сделанных выше предположениях
$$
(L) \quad \Longleftrightarrow\quad (\text{ЦПТ})\ +\ (F)
$$
\end{theorem}

\begin{myremark}\label{Rem:(Lindeberg)=>(Feller)}
Импликация 
$$
(L)\quad \Longrightarrow\quad (F)
$$
тривиальна, т.к. для любого $k=1,\ldots,n$
$$
B_n^{-2}\sigma_{n,k}^2= B_n^{-2}\E X_{n,k}^2\I(|X_{n,k}|\le \eps B_n) + B_n^{-2}\E X_{n,k}^2\I(|X_{n,k}|> \eps B_n) \le
$$
$$
\le\eps^2+L_n(\eps)\to0,
$$
если устремить сначала $n$ к бесконечности, %воспользовавшись условием $(L),$ 
а затем $\eps$ к нулю.
\end{myremark}

%\begin{myremark}
Для о.р.с.в. имеем $B_n^2=n\sigma_{n,1}^2$, поэтому условие Феллера, очевидно, выполняется, а следовательно, условие Линдеберга \textit{равносильно} выполнению ЦПТ. Дробь Линдеберга в этом случае принимает вид 
$$
L_n(\eps)=\sigma_{n,1}^{-2}\E X_{n,1}^2\I(|X_{n,1}|>\eps \sigma_{n,1}\sqrt{n}) =: \E Y_n^2\I(|Y_n|>\eps\sqrt{n}),\ \eps>0,
$$
где $Y_n:\eqd X_{n,1}/\sigma_{n,1}$~--- центрированная с.в. с единичной дисперсией.
Ясно, что если распределение %нормированных слагаемых 
$Y_n$ не зависит от $n$, т.е. распределение случайных слагаемых не меняется от серии к серии, то условие Линдеберга $\lim\limits_{n\to\infty}L_n(\eps)=0$ выполняется автоматически, и в качестве следствия мы получаем простейший вариант ЦПТ для н.о.р.с.в. в схеме нарастающих сумм. Если же  некоторая масса распределения $|Y_n|$ смещается к бесконечности с ростом $n$, то условие Линдеберга может не выполняться, как показывает следующий хорошо известный пример. 

Рассмотрим схему Пуассона: пусть н.о.р.с.в. $\xi_{n,k}\sim Ber(p)$, $X_{n,k}:\eqd\xi_{n,k}-\E\xi_{n,k}=\xi_{n,k}-p,$  $k=1,\ldots,n,$  где значение параметра $p\in\big(0,\frac12\big)$ может зависеть от $n$. Тогда %$\frac pq<1<\frac qp$, 
с.в.~$Y_n$ имеет стандартизованное распределение Бернулли:
$$
\Prob\Big(Y_n=\sqrt{\tfrac qp}\,\Big)=p,%\in\big(0,\tfrac12),
\quad \Prob\Big(Y_n=-\sqrt{\tfrac pq}\, \Big)=q\coloneqq 1-p, \quad \tfrac pq<1<\tfrac qp,
$$
%где выбор параметра $p$ может зависеть от $n$. Тогда 
%с $\frac pq<1<\frac qp$  
$$
L_n(\eps)= \I\left(n\eps^2<\frac{p}q \right )+
q\cdot \I\left (\frac{p}q\le n\eps^2<\frac{q}p\right ),\quad n\in\N,\ \eps>0.
$$
Поскольку $q>1/2$,  условие Линдеберга $\sup\limits_{\eps>0}\lim\limits_{n\to\infty}L_n(\eps)=0$ равносильно выполнению условия $L_n(\eps)=0$, т.е. $n\eps^2\ge\frac{q}p$, для каждого $\eps>0$ и  всех достаточно больших $n\ge N(p,\eps)$. Последнее, в свою очередь, выполняется тогда и только тогда, когда
$$
np\To\infty,\quad n\to\infty,
$$ 
что равносильно  условию теоремы Муавра--Лапласа
$$
npq=\D(\xi_{n,1}+\ldots+\xi_{n,n})=\D S_n\ \To\ \infty,\quad n\to\infty.
$$
Если же $np\to\lambda\in(0,\infty),$ то ЦПТ не выполняется и, согласно теореме Пуассона, предельным распределением для $S_n+np\eqd \xi_{n,1}+\ldots\xi_{n,n}$ будет распределение Пуассона с параметром $\la$.
%\end{myremark}

\medskip

Предположим теперь, что при каждом $n\in\N$
$$
\max_{1\le k\le n}\E|X_{n,k}|^{2+\d}<\infty\quad  \text{для некоторого}\quad \d>0,
$$
не зависящего от $n$, и будем говорить, что выполнено \textit{условие Ляпунова}, если %для некоторого $\d>0$
$$
\exists\d>0\colon\ \lyapd\coloneqq \frac1{B_n^{2+\d}}\sum_{k=1}^n\E|X_{n,k}|^{2+\d}\To0,\quad n\to\infty.\eqno(\Lambda)
$$
Величина $\lyapd$ называется \textit{дробью Ляпунова} порядка $2+\d$. В случае н.о.р.с.в. 
$$
\lyapd=\frac{\E|X_{n,1}|^{2+\d}}{\sigma_1^{2+\d}n^{\d/2}}.
$$

\begin{excersize}
Доказать, что выполнение условия $(\Lambda)$ при некотором $\d_0>0$ влечёт его выполнение при всех $\d\in(0,\d_0].$
\end{excersize}

Заметим, что 
%условие Ляпунова влечёт условие Линдеберга, 
$$
(\Lambda)\quad \Longrightarrow\quad (L),
$$
так как
$$
L_n(\eps)=\frac1{B_n^2}\sum_{k=1}^n\E X_{n,k}^2\I(|X_{n,k}|>\eps B_n)\le 
$$
$$
\le\frac1{\eps^\d B_n^{2+\d}}\sum_{k=1}^n\E|X_{n,k}|^{2+\d}\I(|X_{n,k}|>\eps B_n) \le\eps^{-\d}\lyapd,\quad \eps>0,
$$
а следовательно, из теоремы Линдеберга вытекает %теорема Ляпунова~\cite{Lyapounov1901}:

\begin{theorem}[теорема Ляпунова~\cite{Lyapounov1901}]\label{ThLyapunov}
В сделанных выше предположениях
$$
(\Lambda)\quad  \Longrightarrow\quad  (\text{ЦПТ}).
$$
\end{theorem}
 
Докажем теорему Линдеберга--Феллера. Для этого и неоднократно в дальнейшем нам понадобится одно вспомогательное утверждение.

\begin{lemma} Справедливы оценки:
\begin{equation}\label{|e^z-1|<=|z|e^|z|}
\abs{e^z-1}\le |z|e^{|z|},\quad z\in\C,
\end{equation}
\begin{equation}\label{|ln(1+z)-z|<=|z|^2}
\abs{\ln(1+z)-z}\le |z|^2,\quad z\in\C,\ |z|\le\tfrac12.
\end{equation}
\end{lemma}

\begin{proof}
Используя разложения экспоненты и лографма в ряды Тейлора, для любого $z\in\C$ получаем
$$
\abs{e^z-1}= \bigg|\sum_{k=0}^\infty\frac{z^{k+1}}{(k+1)!}\bigg|\le %|z|\sum_{k=0}^\infty\frac{|z|^k}{(k+1)!} \le
|z|\sum_{k=0}^\infty\frac{|z|^k}{k!}=|z|e^{|z|},
$$
и, если $|z|\le1/2,$
$$
\abs{\ln(1+z)-z}=\abs{-\frac{z^2}2+\frac{z^3}3-\frac{z^4}4+\ldots} \le|z|^2\sum_{k=0}^\infty\frac{|z|^k}{k+2}\le 
$$
$$
\le\tfrac{1}{2}|z|^2\smallsum_{k=0}^\infty 2^{-k}=|z|^2.\qedhere
$$

\end{proof}

Всюду далее, если не оговорено иное, без  ограничения общности будем считать, что случайные слагаемые нормированы так, что 
$$
B_n^2=1\quad \text{и тогда}\quad \widetilde S_n=S_n,
$$
и для краткости будем опускать индексы, отвечающие за номер серии.

\begin{proof}[Доказательство теоремы Линдеберга--Феллера] Импликация $(L)\ \Rightarrow\ (F)$ была доказана выше, поэтому для доказательства \textit{достаточности} осталось показать, что $(L)\ \Rightarrow$ (\textit{ЦПТ}). Последнее условие, в свою очередь, в силу теоремы непрерывности равносильно поточечной сходимости характеристических функций
\begin{equation}\label{CLTviaCh.F.}
\overline f_n(t)\coloneqq\E e^{itS_n}=\prod_{k=1}^nf_k(t) \to e^{-t^2/2},\quad n\to\infty,
\end{equation}
где $f_k(t)=\E e^{itX_k},$ $t\in\R.$

Зафиксируем произвольное  $t\in\R$. C учётом того, что $\E X_k=0$, по теореме~\ref{ThCh.F.TaylorEstim(n,d)Classic} имеем 
\begin{equation}\label{|f_k(t)-1|<=sigma_k^2t^2/2}
\abs{f_k(t)-1}=\abs{\E e^{itX_k}-1-it\E X_k}\le \tfrac{1}2t^2\E X_k^2=\tfrac{1}2t^2\sigma_k^2\to0
\end{equation}
при $n\to\infty$ в силу условия $(F)$, вытекающего из $(L)$, поэтому для всех достаточно больших $n$ и  $k=1,\ldots,n$  
%для достаточно большого $n_0$ и  всех $n>n_0$
$$
\abs{f_k(t)-1} \le \tfrac12\quad\text{и}\quad \abs{f_k(t)}\ge 1-\abs{f_k(t)-1}\ge\tfrac12,
$$
а следовательно, определён логарифм $\ln \overline f_n(t)=\ln\prod_{k=1}^n f_k(t)$. Выбирая главную ветвь логарифма (т.е. такую, что $\ln1=0$), замечаем, что искомое условие~\eqref{CLTviaCh.F.} можно переписать в виде
$$
\delta_n(t)\coloneqq\frac{t^2}2+\ln \overline f_n(t)=\frac{t^2}2+\sum_{k=1}^n \ln f_k(t)\to0, \quad n\to\infty.
$$

Далее, в силу~\eqref{|ln(1+z)-z|<=|z|^2} с $z\coloneqq f_k(t)-1$ имеем
$$
\abs{\delta_n(t)}\le \sum_{k=1}^n\abs{\ln f_k(t)+\frac{\sigma_k^2t^2}2}\le  \sum_{k=1}^n\abs{f_k(t)-1+\frac{\sigma_k^2t^2}2}+
$$
$$
+\sum_{k=1}^n\abs{f_k(t)-1}^2\eqqcolon I_1+I_2,
$$
где для $I_2$ с учётом~\eqref{|f_k(t)-1|<=sigma_k^2t^2/2} и нормировки $B_n^2=1$ получаем
$$
I_2\le \max_{1\le i\le n}\abs{f_i(t)-1}\cdot\sum_{k=1}^n\abs{f_k(t)-1} \le\max_{1\le i\le n}\sigma_i^2 \Big(\frac{t^2}2\Big)^2\to0,\quad n\to\infty,
$$
в силу условия $(F)$. Величину  $I_1$ перепишем в виде
$$
I_1=\sum_{k=1}^n\abs{\E\Big[e^{itX_k}-1-it X_k\ -\frac{(itX_k)^2}2\Big]\Big[\I(|X_k|\le\eps) +\I(|X_k|>\eps)\Big]}
$$
для любого $\eps>0$. Оценивая подынтегральную функцию в (первом) интеграле по ограниченной области величиной $|t X_k|^3/6\le\eps|t|^3X_k^2/6$, а по неограниченной~--- величиной $|tX_k|^2$, получаем
$$
I_1\le \smallsum_{k=1}^n \tfrac{1}6\eps|t|^3\E X_k^2\I(|X_k|\le\eps)+ \smallsum_{k=1}^nt^2\E X_k^2\I(|X_k|>\eps)\le \tfrac16\eps|t|^3 + t^2L_n(\eps)
$$
для любых $t\in\R$ и $\eps>0$. Учитывая полученную выше оценку величины $I_2$ и условие $(L)$, окончательно получаем
$$
\lim_{n\to\infty}|\delta_n(t)|\le \tfrac16\eps|t|^3\to 0,\quad \eps\to0,
$$
для каждого $t\in\R$, что и требовалось доказать.

\textit{Необходимость.} Как и ранее, замечаем, что в силу условия $(F)$ справедливо соотношение~\eqref{|f_k(t)-1|<=sigma_k^2t^2/2} и определён логарифм $\ln\overline f_n(t),$ а значит $(\textit{ЦПТ\,})$ можно переписать в виде $\sup\limits_{t\in\R}\lim\limits_{n\to\infty}|\delta_n(t)|\to0$, или~--- с учётом~\eqref{|ln(1+z)-z|<=|z|^2} и того, что $I_2\to0$ при $n\to\infty$ в силу $(F)$,~--- в виде
$$
\abs{\frac{t^2}2+\smallsum_{k=1}^n(f_k(t)-1)}\le\abs{\frac{t^2}2 +\smallsum_{k=1}^n\ln f_k(t)} + \smallsum_{k=1}^n\abs{\ln f_k(t)-(f_k(t)-1)}\le
$$
$$
\le |\delta_n(t)|+I_2\to0,\quad n\to\infty.
$$
Левая часть последнего неравенства не может быть меньше, чем её вещественная часть, поэтому бесконечно малой при $n\to\infty$ для каждого $t\in\R$ также должна являться величина
$$
\frac{t^2}2 +\smallsum_{k=1}^n(\Re f_k(t)-1)= 
$$
$$
=\frac{t^2}2 -\smallsum_{k=1}^n\E\big[1-\cos(tX_k)\big]\big[\I(|X_k|\le\eps)+\I(|X_k|>\eps)\big]\ge
$$
$$
\ge\frac{t^2}2 -\sum_{k=1}^n\E\left[\frac{(tX_k)^2}2\I(|X_k|\le\eps) +2\I(|X_k|>\eps)\right]=
$$
$$
=\frac{t^2}2\smallsum_{k=1}^n\E X_k^2\I(|X_k|>\eps)  -2\smallsum_{k=1}^n\E\I(|X_k|>\eps)\ge
$$
$$
\ge \left(\frac{t^2}2-\frac{2}{\eps^2}\right)\sum_{k=1}^n\E X_k^2\I(|X_k|>\eps) = \left(\frac{t^2}2-\frac{2}{\eps^2}\right)L_n(\eps)
$$
при каждом $\eps>0,$ откуда вытекает, что $\lim\limits_{n\to\infty}L_n(\eps)=0$, если  $\eps|t|>2$.
%, в частности, при $t\coloneqq4/\eps.$ 
Теперь справедливость условия $(L)$ для каждого фиксированного ${\eps>0}$ вытекает из произвольности выбора $t$ (например, можно положить $t\coloneqq4/\eps$).
\end{proof}

В следующих разделах будут построены оценки скорости сходимости в теоремах Линдеберга--Феллера и Ляпунова.

\section[Скорость сходимости в теоремах Линдеберга и Ляпунова]{Оценки скорости сходимости в теоремах Линдеберга и Ляпунова}

Сначала мы перечислим наиболее известные и наиболее пригодные к применению на практике оценки скорости сходимости  и укажем связи между ними. Доказательства же ключевых оценок будут даны позже, в разделах~\ref{SectBEineqProof} и~\ref{SectOsipovIneqProof}.

Пусть, как и ранее, не исключая схему серий,  $X_1,\ldots,X_n$~--- независимые с.в. с 
$$
\E X_k=0,\quad \sigma_k^2\coloneqq\E X_k^2<\infty,\quad k=1,\ldots,n,
$$
$$
S_n\coloneqq\sum_{k=1}^nX_k,\quad 
B_n^2\coloneqq \D S_n=\sum_{k=1}^n\sigma_k^2>0,
$$
$$ 
\Delta_n\coloneqq \sup_{x\in\R}\abs{\Prob(S_n<xB_n)-\Phi(x)}.
$$

\subsection{Неравенство Осипова}

В дополнение к введённым выше обозначениям положим также
$$
M_n(\eps)\coloneqq \frac1{B_n^3}\sum_{k=1}^n\E|X_k|^3\I(|X_{k}|\le\eps B_n),\quad \eps>0,
$$
и напомним определение дроби Линдеберга:
$$
L_n(\eps)\coloneqq \frac1{B_n^2}\sum_{k=1}^n\E X_k^2\I(|X_{k}|>\eps B_n),\quad \eps>0.
$$

\begin{theorem}[неравенство Осипова~\cite{Osipov1966}]
\label{ThOsipovIneq(eps)}
В сделанных выше предположениях %для любого $n\in\N$ и $\eps>0$ справедлива оценка
\begin{equation}\label{OsipovIneq(eps)}
\Delta_n\le C(L_n(\eps)+M_n(\eps)),\quad n\in\N,\quad \eps>0,
\end{equation}
где $C$~--- некоторая абсолютная постоянная, $C\le1.87$~$\cite{KorolevDorofeyeva2017}$.
\end{theorem}

Заметим, что
$$
M_n(\eps)\le \frac{\eps}{B_n^2}\sum_{k=1}^n\E X_k^2\I(|X_{k}|\le\eps B_n)\le\eps,\quad \eps>0,
$$
поэтому из неравенства Осипова вытекает оценка
$$
\Delta_n\le C(\eps+L_n(\eps)),
$$
правая часть которой  при выполнении условия  Линдеберга $(L)$ может быть сделана сколь угодно малой за счёт выбора сначала достаточно малого $\eps$, а затем~--- достаточно большого $n$. Таким образом, неравенство Осипова~\eqref{OsipovIneq(eps)} влечёт теорему Линдеберга. С другой стороны, если случайные слагаемые асимптотически пренебрежимо малы в смысле условия $(F)$, то, в силу теоремы Феллера, стремление к нулю левой части~\eqref{OsipovIneq(eps)}  необходимо влечёт стремление к нулю и правой, т.е. левая и правая части неравенства Осипова~\eqref{OsipovIneq(eps)} либо стремятся к нулю, либо не являются бесконечно малыми величинами \textit{одновременно}. Другими словами, неравенство~\eqref{OsipovIneq(eps)} связывает \textit{критерий} сходимости со \textit{скоростью} сходимости в классической ЦПТ, и поэтому, согласно  терминологии В.\,М.\,Золотарёва~\cite{Zolotarev1986} является \textit{естественной} оценкой скорости сходимости. 

Заметим также, что значение $\eps=1$ минимизирует правую часть~\eqref{OsipovIneq(eps)}. Действительно, для любого $\eps>0$
$$
M_n(\eps )+L_n(\eps )= \sum_{k=1}^n\bigg[\E\Big|\frac{X_k}{ B_n}\Big|^3 \I\Big(\Big|\frac{X_k}{B_n}\Big|\le\eps\Big) +\E\Big|\frac{X_k}{B_n}\Big|^2 \I\Big(\Big|\frac{X_k}{B_n}\Big|>\eps\Big)\bigg], %= \sum_{k=1}^n\E \Big(\frac{X_k}{B_n}\Big)^2\min\left\{ 1, \frac{|X_k|}{B_n}\right\}
$$
а для каждой с.в.  $X\coloneqq X_k/B_n,$ $k=1\ldots,n,$ и любого борелевского множества $B$ на прямой
$$
\E|X|^3\I(|X|\le1)+ \E X^2\I(|X|>1)= 
$$
$$
=\E X^2\min\{ 1,|X|\}(\I(X\in B) +\I(X\notin B))\le 
$$
$$
\le \E|X|^3\I(X \in B) + \E X^2\I(X\notin B)
$$
(на оптимальность множества $B=[-1,1]$ здесь впервые было указано, по-видимому, в работе~\cite{Loh1975}, которую цитируем по~\cite{BarbourHall1984}). Выбор $B\hm\coloneqq[-\eps,\eps]$ приводит к искомому утверждению.

\begin{excersize}
Доказать, что для любой с.в.~$X$ с $\E X^2<\infty$ и любого $a>0$
$$
\int_0^a\E X^2\I(|X|>z)\,\dd z=\E X^2\min\{|X|,\,a\}.
$$
\end{excersize}

Таким образом, неравенство Осипова~\eqref{OsipovIneq(eps)} с $\eps=1$
\begin{multline}\label{OsipovIneqEps=1}
\Delta_n\le C\bigg[\frac{1}{B_n^3}\sum_{k=1}^n\E|X_k|^3\I(|X_k|\le B_n)+\frac{1}{B_n^2}\sum_{k=1}^n\E X_k^2\I(|X_k|> B_n)\bigg]= 
\\
=C\sum_{k=1}^n\E \Big(\frac{X_k}{B_n}\Big)^2\min\left\{ 1, \frac{|X_k|}{B_n}\right\} =C\int_0^{1}L_n(\eps)d\eps
\end{multline}
влечёт выполнение~\eqref{OsipovIneq(eps)} с произвольным $\eps>0.$
%имеет интересную историю. Для начала упомянем, что оно является тривиальным следствием более ранних и формально более общих результатов Каца~\cite{Katz1963} и Петрова~\cite{Petrov1965}, но эта связь долгое время оставалась незамеченной даже и самим Петровым (см., например,~\cite{Petrov1972}). Далее, неравенство~\eqref{OsipovIneqEps=1} бело передоказано в формально общей, но фактически в эквивалентной форме, Феллером~\cite{Feller1968} с $C=6$ и Падитцем~\cite{Paditz1980,Paditz1984} для о.р.с.в. с $C=4.77$. Двумя годами позже Падитц~\cite{Paditz1986} объявил более точную оценку $C\le3.51$. Работы Феллера и Падитца не содержали ссылок на работу. Позже неравенство Осипова~\eqref{OsipovIneq(eps)} также было передоказано Барбуром и Холлом~\cite{BarbourHall1984}, цитировавшими работу Феллера, но применившими новый метод Тихомирова--Стейна, который позволил им получить оценку константы лишь $C\le18$. Чен и Шао~\cite{ChenShao2001}, цитировавшие только работу Фллера, передоказали неравенство~\eqref{OsipovIneq} с $C=4.1$. В последние годы, Королёв со своими учениками последовательно снизили верхнюю оценку для $C$ с $2.011$ в~\cite{KorolevPopov2011,KorolevPopov2012} до $1.87$ в~\cite{KorolevDorofeyeva2017}.

\subsection{Неравенство Каца--Петрова}

Пусть $\mathcal G$~--- класс всех неотрицательных чётных функций $g$, определённых на $\R$, таких что $g(x)>0$ при $x>0$, и функции $g(x)$, $x/g(x)$ не убывают в области $x>0.$

Примеры функций $g\in\mathcal G$: $g(x)=const>0,$ $g(x)=\ln(1+|x|),$ $g(x)=\min\{|x|,a\}$, $g(x)=\max\{|x|,a\}$ с $a>0,$ $g(x)=|x|^\d$ с $\d\in(0,1].$

\begin{lemma}[см.~\cite{GabdullinMakarenkoShevtsova2018}]
\label{LemKatzPetrovGproperties}
Каждая функция $g\in\mathcal G$ удовлетворяет следующим неравенствам при любых $x\in\R\setminus\{0\}$ и $a>0:$
$$
g_*(x;a)\coloneqq \min\Big\{1,\frac{|x|}a\Big\}\ \le\ \frac{g(x)}{g(a)}\ \le\ \max\Big\{1,\frac{|x|}a\Big\}\eqqcolon g^*(x;a),
$$
при этом $g_*(\,\cdot\,;a),g^*(\,\cdot\,;a)\in\mathcal G$ для каждого $a>0$.
\end{lemma}

\begin{proof}
В силу чётности можно ограничиться случаем ${x>0}.$ Если $x\le a$, т.е. $\frac{x}a\le1$, то в силу монотонности $g(x)$ и $x/g(x)$ имеем
$$
\left\{
\begin{array}{rcl}
g(x)&\le&g(a),
\\[2mm]
 \frac{x}{g(x)}&\le& \frac{a}{g(a)},
\end{array}
\right.
\quad\Longrightarrow\quad
\left\{
\begin{array}{rcl}
\frac{g(x)}{g(a)}&\le&1=\max\big\{1,\frac{x}a\big\},
\\[2mm]
 \frac{g(x)}{g(a)}&\ge&\frac{x}a=\min\big\{1,\frac{x}a\big\}.
\end{array}
\right.
$$
Если же $x/a>1,$ то аналогично
$$
\left\{
\begin{array}{rcl}
g(x)&\ge&g(a),
\\[2mm]
 \frac{x}{g(x)}&\ge& \frac{a}{g(a)},
\end{array}
\right.
\quad\Longrightarrow\quad
\left\{
\begin{array}{rcl}
\frac{g(x)}{g(a)}&\ge&1=\min\big\{1,\frac{x}a\big\},
\\[2mm]
 \frac{g(x)}{g(a)}&\le&\frac{x}a=\max\big\{1,\frac{x}a\big\}.
\end{array}
\right.
$$
Принадлежность функций $g_*$ и $g^*$ классу $\mathcal G$ очевидна.
\end{proof}

\begin{excersize}
Доказать, что каждая функция $g\in\mathcal G$ непрерывна при $x>0$.
\end{excersize}

\begin{theorem}[неравенство Каца--Петрова~\cite{Katz1963,Petrov1965}]
\label{ThKatzPetrovIneq}
В сделанных выше предположениях для любой функции $g\in\mathcal G$ 
\begin{equation}\label{KatzPetrovIneq}
\Delta_n\le \frac{A}{B_n^2g(B_n)}\sum_{k=1}^n\E X_k^2g(X_k),\quad n\in\N,
\end{equation}
где $A$~--- некоторая универсальная постоянная $($не зависящая также от выбора $g)$.
\end{theorem}

%\subsection{Связь неравенств Осипова и Каца--Петрова}

Заметим, что неравенство Каца--Петрова~\eqref{KatzPetrovIneq} с $g_*(\,\cdot\,;B_n)\hm=\min\big\{1,\frac{|\,\cdot\,|}{B_n}\big\}\in\mathcal G,$
превращается в неравенство Осипова~\eqref{OsipovIneq(eps)} c  $\eps=1$, а значит, влечёт~\eqref{OsipovIneq(eps)} и с произвольным $\eps>0$. С другой стороны, как вытекает из леммы~\ref{LemKatzPetrovGproperties},  функция $g_*$ минимизирует правую часть неравенства Каца--Петрова~\eqref{KatzPetrovIneq}, т.е. неравенство Осипова~\eqref{OsipovIneq(eps)} c  $\eps=1$ влечёт неравенство Каца--Петрова с произвольной функцией $g\in\mathcal G$ и той же константой $A=C$. Таким образом,
$$
\eqref{OsipovIneq(eps)}\ \forall\eps>0\quad\stackrel{C=A\vphantom{\frac12}}{\Longleftrightarrow}  \quad \eqref{KatzPetrovIneq} \ \forall g\in\mathcal G.
$$

%\section{Оценки скорости сходимости в теореме Ляпунова}

\subsection{Неравенство Берри--Эcсеена} 

Перейдём теперь к оценкам скорости сходимости в теореме Ляпунова (теорема~\ref{ThLyapunov}). Для этого дополнительно предположим, что при  некотором $\d\in(0,1]$
$$
\beta_{2+\d,k}\coloneqq\E|X_k|^{2+\d}<\infty,\quad k=1,\ldots,n,
$$
и напомним определение дроби Ляпунова:
$$
\lyapd\coloneqq \frac1{B_n^{2+\d}}\sum_{k=1}^n\beta_{2+\d,k}.
$$

\begin{excersize}\label{Ex:L>=T>=n^(d/2)}
Доказать, что для любых $\sigma_1,\ldots,\sigma_n\ge0$ таких, что $B_n^2:=\sum_{k=1}^n\sigma_k^2>0$
$$
\frac1{B_n^{2+\d}}\sum_{k=1}^n\sigma_k^{2+\d}\ge \frac1{n^{\d/2}},\quad \forall n\in\N.
$$
\end{excersize}

Выбирая в неравенстве Каца--Петрова~\eqref{KatzPetrovIneq}
$$
g(x)\coloneqq |x|^\d\in\mathcal G,
$$
получаем следующую теорему, которая носит имена  Э.\,Берри~\cite{Berry1941} и К.-Г.\,Эссеена~\cite{Esseen1942}, доказавшим её независимо и одновременно в 1941--1942\,гг. в наиболее трудном случае $\d=1$ (Берри~--- для о.р.с.в.). Для $\d\in(0,1)$ данная теорема фактически была доказана ещё в оригинальной работе А.\,М.\,Ляпунова 1900\,г.~\cite{Lyapounov1900}, в которой одноимённая центральная предельная теорема \textit{о сходимости} была получена как следствие \textit{оценки скорости сходимости}.

%По сложившейся традиции фамилию Ляпунова опускаем (так как наибольшую трудность доказательство этого неравенства представляло именно при $\d=1$).

\begin{theorem}
В сделанных выше предположениях для каждого $\d\in(0,1]$
\begin{equation}\label{BEineq-intro}
\Delta_n\le C_0(\d)\cdot\lyapd,\quad n\in\N,
\end{equation}
где $C_0(\d)$~--- некоторые константы, которые могут быть выбраны и не зависящими от $\d,$ например,
$C_0(\d)= C,$ где $C$~--- абсолютная константа из неравенства Осипова~\eqref{OsipovIneq(eps)}. В частности, в случае о.р.с.в.
\begin{equation}\label{BEineq-intro-iid}
\Delta_n\le \frac{C_0(\d)}{n^{\d/2}}\cdot\frac{\beta_{2+\d,1}}{\sigma_1^{2+\d}},\quad n\in\N.
\end{equation}
\end{theorem}

Заметим, что с ростом $\d$ порядок скорости сходимости  $\mathcal O(n^{-\d/2})$, устанавливаемый неравенством Берри--Эсееена в случае о.р.с.в.~\eqref{BEineq-intro-iid}, повышается, но при $\d\ge1$ скорость сходимости остаётся в общем случае такой же, как при $\d=1$.  Это объясняется тем, что неравенство Берри--Эссеена является универсальной оценкой и потому должно учитывать как <<тяжелохвостость>>, так и структуру носителя (непрерывный/дискретный/решётчатый) распределения случайных слагаемых. При этом тяжелохвостость распределения выражается порядком используемого момента $2+\d$ и вносит в правую часть~\eqref{BEineq-intro-iid} вклад порядка $\mathcal O(n^{-\d/2})$. Вклад же информации о носителе распределения имеет порядок $\mathcal O(n^{-1/2})$ (в <<наихудшем>> решётчатом случае) вне зависимости от количества конечных моментов. Чтобы продемонстрировать сказанное на  примере, с одной стороны, а с другой~--- формализовать высказывание о неулучшаемости порядка $\mathcal O(n^{-1/2})$ правой части~\eqref{BEineq-intro-iid}, рассмотрим следующий пример.

Пусть н.о.р.с.в $X_1,X_2,\ldots,X_n$  имеют симметричное двухточечное распределение: $\Prob(X_1=\pm1)=\frac12.$ Это распределение, с одной стороны, является решётчатым (с шагом $2$), а с другой~--- имеет моменты любого порядка. В частности,
$$
\E X_1\ =\ 0,\quad \sigma_1^2=\E X_1^2\ =\ 1\ =\ \E|X_1|^3=\beta_{3,1},
$$
и для чётных $n$
$$
\Prob(S_n=0)=\Prob\Big(\smallsum_{k=1}^n X_k=0\Big)= C_n^{n/2}\big(\tfrac12\big)^n =\frac{n!\cdot2^{-n}}{(n/2)!\cdot (n/2)!}.
$$
По формуле Стирлинга 
$$
n!=\Big(\frac{n}e\Big)^n\sqrt{2\pi n} \cdot e^{\theta/(12n)},\quad\text{где}\quad 0<\theta<1,
$$
получаем
$$
\Prob(S_n=0)= \frac2{\sqrt{2\pi n}} +\mathcal O\left (\frac1n\right )
$$
для чётных $n\to\infty$. Следовательно, ф.р. нормированной суммы
$$
\overline F_n(x)\coloneqq \Prob(S_n<x\sqrt{n})
$$
имеет в точке $x=0$ скачок, асимптотически равный $\frac2{\sqrt{2\pi n}}(1+o(1))$, а значит, $\overline F_n(x)$ не может быть равномерно по $x\in\R$ аппроксимирована непрерывной функцией ($\Phi(x)$) с точностью, превышающей половину её скачка  $\frac1{\sqrt{2\pi n}}(1+o(1))$, т.е. 
$$
\Delta_n\ge \frac1{\sqrt{2\pi n}}(1+o(1)),\quad n\to\infty.
$$
При этом в рассматриваемом случае
$$
\lyap=\frac{\beta_{3,1}}{\sigma_1^3\sqrt{n}} =\frac1{\sqrt{n}}.
$$
Таким образом, в неравенстве Берри--Эссеена~\eqref{BEineq-intro-iid}
$$
C_0(1)\ge \sup\limsup_{n\to\infty}\frac{\Delta_n}{\lyap}%\frac{\sqrt{n}\sigma_1^3}{\beta_{3,1}}
\ge \frac1{\sqrt{2\pi}}=0.3989\ldots,
$$
где точная верхняя грань берётся по всем (одинаковым) распределениям независимых случайных слагаемых $X_1,\ldots,X_n$ с конечными третьими моментами и оценивается снизу значением на симметричном двухточечном распределении.  Сформулируем вышеизложенное в виде теоремы.

\begin{theorem}[нижняя оценка скорости сходимости в ЦПТ]
Для абсолютной константы $C_0(1)$ в неравенстве Берри--Эссеена~\eqref{BEineq-intro} c $\d=1$ справедлива нижняя оценка
$$
C_0(1)\ge \sup_{X_1\eqd\ldots\eqd  X_n\colon \lyap<\infty}\limsup_{n\to\infty}\frac{\Delta_n}{\lyap} \ge \frac1{\sqrt{2\pi}}=0.3989\ldots
$$
\end{theorem}

Таким образом, неравенство  Берри--Эссеена~\eqref{BEineq-intro} c $\d=1$ представляет наибольший интерес.

В 2007 г. К.\,Хипп и Л.\,Маттнер~\cite{HippMattner2007} доказали, что найденная  нижняя оценка $1/\sqrt{2\pi}$ является точным значением $C_0(1)$ для рассмотренных выше симметричных бернуллиевских  н.о.р.с.в., а именно,  в этом случае
$$
\Delta_n(X_1) = \left\{
  \begin{array}{ll} \Phi(\frac1{\sqrt{n}})-\frac12,
& n \hbox{ нечётно,} 
\\[2mm]
C_n^{n/2}\big(\frac12\big)^{n+1}, & n \hbox{ чётно,}
  \end{array}
\right. <\ \ \frac1{\sqrt{2\pi n}},\quad n\in\N.
$$

Рассматривая двухточечные распределения общего вида, Эссеен~\cite{Esseen1956}, действуя по той же схеме, получил более точную нижнюю оценку 
$$
C_0(1)\ge C_E\coloneqq\frac{\sqrt{10}+3}{6\sqrt{2\pi}}=0.4097\ldots,
$$
которая является наилучшей известной и по сегодняшний день. В 2016 г. в диссертации Шульца~\cite{Schulz2016} было показано, что $C_0(1)=C_E$, если н.о.р. случайные слагаемые принимают только по два возможных значения.

\subsection{Неравенство типа Берри--Эcсеена с уточнённой структурой}

Заметим, что в правую часть неравенства~\eqref{BEineq-intro} можно добавить неотрицательную величину
$$
T_{2+\d,n}\coloneqq \frac1{B_n^{2+\d}}\sum_{k=1}^n\sigma_k^{2+\d},
$$
принимающую значения от $n^{-\d/2}$ до $\lyapd$ (см. Упражнение~\ref{Ex:L>=T>=n^(d/2)}) и в случае о.р.с.в. совпадающую с $n^{-\d/2}$, и записать неравенство вида 
\begin{equation}\label{BEineq-struct-intro}
\Delta_n\le C_s(\d)\big(\lyapd+sT_{2+\d,n}\big),\quad s\ge0,
\end{equation}
c некоторыми константами $C_s(\d)\le C_0(\d).$ При этом ясно, что в качестве верхней оценки константы $C_0(\delta)$ можно взять $\inf_{s\ge0}(1+s)C_s(\delta)$, т.к. $T_{2+\d,n}\le\lyapd$.

\begin{excersize} Привести пример н.о.р.с.в. $X_1,\ldots,X_n,$ таких что отношение $\lyapd/T_{2+\d,n}$ может быть сколь угодно большим.
\end{excersize}

Оказывается, что наличие ``добавки'' $T_{2+\d,n}$, которая может быть существенно меньше ляпуновской дроби $\lyapd$ даже в случае о.р.с.в., позволяет значительно уменьшить входящие в новое неравенство~\eqref{BEineq-struct-intro} константы $C_s(\d)$ по сравнению с константами $C_0(\d)$ в классическом неравенстве Берри--Эссеена~\eqref{BEineq-intro}. А именно, справедлива следующая теорема, цитирующая верхние оценки констант из работ~\cite{Shevtsova2013Inf,Shevtsova2014DAN} (см.  таблицы~\ref{TabBEC0UpperBounds} и~\ref{TabBECstruct2SidedBounds}).

\begin{table}[h]
\centering
\begin{tabular}{||c|c|c||c|c|c||}
\hline&о.р.с.в.&р.р.с.в.&&о.р.с.в.&р.р.с.в.\\
$\delta$ & $C_0(\d)\le$ & $C_0(\d)\le$ &
$\delta$ & $C_0(\d)\le$ & $C_0(\d)\le$\\
\hline
$1$&$0.4690$&$0.5583$ &$0.5$&$0.5385$&$0.6497$\\
$1-$&$0.4748$&$0.5591$ &$0.4$&$0.5495$&$0.6965$\\
$0.9$&$0.4955$&$0.5631$&$0.3$&$0.5641$&$0.7519$\\
$0.8$&$0.5179$&$0.5743$&$0.2$&$0.5798$&$0.8126$\\
$0.7$&$0.5234$&$0.5907$&$0.1$&$0.5842$&$0.8790$\\
$0.6$&$0.5284$&$0.6152$&$0$&$0.5410$&$0.5410$\\
\hline
\end{tabular}
\caption{Округлённые вверх значения выражения
$(1\hm+s_0(\d))C_{s_0}(\d)$ из
неравенства~\eqref{BEineq} (см.
таблицу~\ref{TabBECstruct2SidedBounds}), мажорирующие значения
констант $C_0(\delta)$.} \label{TabBEC0UpperBounds}
\end{table}

\begin{table}[p]
\centering\small
\rotatebox{90}
{
\begin{tabular}{||c||c||c|c|c|c||c|c||c|c||}
\hline & &\multicolumn{4}{c||}{Верхние оценки для о.р.с.в.}
&\multicolumn{4}{c||}{Верхние оценки для р.р.с.в.}\\
\cline{3-10} $\delta$ & $C_s(\d)\ge$ & $s_0$ & $C_{s_0}(\d)$&
$s_1$ & $C_{s_1}(\d)$ & $s_0$ & $C_{s_0}(\d)$ & $s_1$ &
$C_{s_1}(\d)$
\\\hline
%$1$  &0.2659&0.646&0.30310&0.5&0.37226&1&0.3057\\
%$1-$ &0.2659&0.635&0.30410&0.61&0.34726&1&0.3057\\
$0.9$&0.2698&0.410&0.3514&0.626&0.3073&0.52&0.37046&1&0.3108\\
$0.8$&0.2819&0.6356&0.3166&0.6356&0.3166&0.15&0.49939&1&0.3215\\
$0.7$&0.2961&0.5830&0.3306&0.5830&0.3306&0.06&0.5572&1&0.3367\\
$0.6$&0.3128&0.5131&0.3492&0.5131&0.3492&0.02&0.60313&0.859&0.3557\\
$0.5$&0.3328&0.4444&0.3728&0.4444&0.3728&0.01&0.6432&0.834&0.3795\\
$0.4$&0.3568&0.3652&0.4025&0.47&0.4022&0.02&0.6828&0.806&0.4091\\
$0.3$&0.3862&0.2823&0.4399&0.52&0.4384&0.04&0.7229&0.778&0.4457\\
$0.2$&0.4232&0.1920&0.4868&0.58&0.4828&0.06&0.7666&0.748&0.4905\\
$0.1$&0.4714&0.0744&0.5439&0.63&0.5372&0.08&0.81388&0.710&0.5454\\
\hline
\end{tabular}
}
\caption{Нижние оценки констант $C_s(\d)$, справедливые при всех
$s\hm\ge0$, и верхние оценки констант $C_s(\delta)$ из
неравенства~\eqref{BEineq} для некоторых $s\in[0,1]$ и
${\delta\in(0,1)}$.} \label{TabBECstruct2SidedBounds}
\end{table}

\begin{theorem}\label{ThBEineq}
В сделанных выше предположениях при каждом $\d\in(0,1]$
\begin{equation}\label{BEineq}
\Delta_n\le \min_{s\ge0}C_s(\d)(\lyapd+sT_{2+\d,n}),\quad n\in\N,
\end{equation}
где $C_s(\d)$ зависят только от $s$ и $\d.$ 
В частности, для о.р.с.в.
\begin{equation}\label{BEineq(delta)iid}
\Delta_n\le \min_{s\ge0}\frac{C_s(\d)}{n^{\d/2}}\left (\frac{\beta_{2+\d,1}}{\sigma_1^{2+\d}}+s\right),\quad n\in\N,
\end{equation}
для $\d=1$ в общем случае
%\begin{equation}\label{BEineq(delta=1)}
$$
\Delta_n\le \lyap\min\big\{0.5583, 0.3723(\lyap+0.5T_{3,n}),0.3057(\lyap+T_{3,n})\},
$$
%\end{equation}
для $\d=1$ в случае о.р.с.в. с $\rho\coloneqq \E|X_1|^3/(\E X_1^2)^{3/2}$
\begin{equation}\label{BEineq(delta=1)iid} 
\Delta_n\le \frac1{\sqrt{n}}\min\big\{0.469\rho,0.3322(\rho+0.429),0.3031(\rho+0.646)\big\}.
\end{equation}
\end{theorem}

Заметим, что верхние оценки констант $C_s(1)$ при $s>0$ ($C_1\le0.3057$ в общем случае и $C_{0.646}\le0.3031$ для о.р.с.в.) строго меньше нижней оценки абсолютной константы $C_0(1)\ge1/\sqrt{2\pi}>0.3989.$ А нижние оценки для констант $C_s(1)$ при $s>0$, конечно, являются более оптимистичными, чем для $C_0(1)$. Например, в~\cite{Shevtsova2016} показано, что
$$
\inf_{s>0}C_s(1)\ge\sup_{\gamma>0,m\in\N_0}\sqrt\gamma\bigg(e^{-\gamma}\sum_{k=0}^m\frac{\gamma^k}{k!}-\Phi\Big(\frac{m-\gamma}{\sqrt\gamma}\Big)\bigg). %>\frac2{3\sqrt{2\pi}}=0.2659\ldots
$$
Выбирая $m=6$ и $\gamma=6.42$, получаем 
$$
\inf_{s>0}C_s(1)\ge\sqrt\gamma\bigg(e^{-\gamma}\sum_{k=0}^6\frac{\gamma^k}{k!}-\Phi\Big(\frac{6-\gamma}{\sqrt\gamma}\Big)\bigg)\bigg|_{\gamma=6.42}>0.266012,
$$
тогда как выбор $\gamma=m\to\infty$ приводит к более изящной, но менее точной нижней оценке $2/(3\sqrt{2\pi})=0.2659\ldots$.

Значения констант $C_s(\d)$ при других $\d\in(0,1)$  и некоторых $s\ge0$ указаны в таблице~\ref{TabBECstruct2SidedBounds} наряду с нижними оценками, здесь: $s_0$~---значение, минимизирующее (в рамках использованного метода доказательства) сумму $(1+s)C_s(\d),$ служащую верхней оценкой для $C_0(\d),$ а $s_1$~--- значение, минимизирующее (в рамках использованного метода доказательства) саму константу $C_s(\d)$~--- коэффициент при ляпуновской дроби $\lyapd$. Таблица~\ref{TabBEC0UpperBounds} приведена для удобства подсчёта верхних оценок констант $C_0(\d)$ по формуле 
$$
C_0(\d)\le \min_{s\ge0}(1+s)C_s(\d),
$$
справедливой в силу неравенства Ляпунова~\eqref{LyapunovMomentIneq}.

Из проведённого анализа вышеперечисленных оценок скорости сходимости в теоремах Линдеберга и Ляпунова вытекает, что достаточно доказать лишь неравенство Осипова с $\eps=1$~\eqref{OsipovIneq(eps=1)}. Однако, сразу нам это сделать будет затруднительно, поэтому сначала мы  докажем неравенство Берри--Эссеена с $\d=1$ (методом характеристических функций), затем методом усечений выведем из него неравенство Осипова с $\eps=1$, из которого уже и вытекают все остальные рассматриваемые здесь неравенства, в том числе неравенство Берри--Эссеена с произвольным $\d\in(0,1].$

\section[Доказательство неравенства Берри--Эссеена]{Доказательство неравенства Берри--Эссеена. Метод характеристических функций}\label{SectBEineqProof}

\let\oldlyap\lyap
\def\lyap{L^3}

Пусть $X_1,\ldots,X_n$~--- независимые с.в. с $\E X_k=0,$ $\sigma_k^2\coloneqq\E X_k^2,$ $\beta_k\hm\coloneqq\E|X_k|^3<\infty,$ $f_k(t)\coloneqq\E e^{itX_k}$, $t\in\R$,  $k=1,\ldots,n,$
$$
S_n\coloneqq\sum_{k=1}^nX_k,\quad 
B_n^2\coloneqq \D S_n=\sum_{k=1}^n\sigma_k^2>0,\quad \lyap\coloneqq \frac1{B_n^3}\sum_{k=1}^n\beta_k,
$$ 
$$
%\widetilde S_n\coloneqq \frac{S_n-\E S_n}{\sqrt{\D S_n}} =\frac{S_n}{B_n},\quad 
\overline f_n(t)=\E e^{it\widetilde S_n}=\prod_{k=1}^n f_k\left(\frac{t}{B_n}\right),\quad t\in\R.
$$
$\Phi$~--- ф.р. стандартного нормального закона.

Целью данного раздела является доказательство следующей теоремы.

\begin{theorem}[неравенство Берри--Эссеена]
\label{BEineqForE|Xi|^3<infty}
В сделанных выше предположениях
$$
\Delta_n\coloneqq \sup_{x\in\R}\abs{\Prob(S_n<xB_n)-\Phi(x)} \le C_0L^3,
$$
где $C_0$~--- некоторая абсолютная константа. При этом можно положить
$$
C_0=\frac2{\sqrt{2\pi}}\inf\left \{\frac{b}{(1-2d/3)^{3/2}} +\frac{2b(b+1)}{\pi d(b-1)}\colon b>1,\ d\in\left(0,\tfrac32\right) \right\}.
$$
\end{theorem}

В частности, полагая $b= 2$ и $d=3/4$, легко получить оценку
$$
C_0\le \frac2{\sqrt{2\pi}}\left(4\sqrt2+\frac{48}{3\pi}\right)=\frac8{\sqrt{2\pi}}\Big(\sqrt2+\frac4{\pi}\Big)=8.57\ldots<9.
$$
Более аккуратный подсчёт показывает, что оптимальные значения $b=1.72\ldots,$ $d=0.703\ldots$ и соответственно $C_0=8.23\ldots\ $. 

Напомним, что наилучшие известные на данный момент значения константы $C_0$ получены в~\cite{Shevtsova2013Inf}, приведены в таблице\,\ref{TabBEC0UpperBounds} и имеют вид:
$$
C_0\le 
\begin{cases}
0.5583&\text{в общем случае,}
\\
0.4690&\text{в случае о.р.с.в.}
\end{cases}
$$

Доказательство неравенства Берри--Эссеена будем проводить, как и в оригинальных работах~\cite{Berry1941,Esseen1942}, \textit{методом характеристических функций} с помощью неравенства сглаживания~\eqref{SmoothingIneqEsseenKernel}. Для этого нам необходимо построить оценку модуля разности соответствующих х.ф. $\big|\overline f_n(t)-e^{-t^2/2}\big|$ под знаком интеграла в правой части~\eqref{SmoothingIneqEsseenKernel}. 

Прежде чем перейти к оцениванию характеристических функций, докажем несколько вспомогательных утверждений, два из которых являются элементарными неравенствами. 

\begin{lemma}[см.~\cite{Prawitz1975}]\label{Lem:sum(x_i)^r<=(sum(x_i))^r}
Для любых $r\ge1,$ $n\in\N$ и $x_1,\ldots,x_n\ge0$
$$
\smallsum_{i=1}^nx_i^r\le  \Big(\smallsum_{i=1}^nx_i\Big)^r.
$$
\end{lemma}
\begin{myremark} Утверждение леммы~\ref{Lem:sum(x_i)^r<=(sum(x_i))^r} хорошо известно для $r=2$:
$$
\smallsum_{i=1}^nx_i^2\le \Big(\smallsum_{i=1}^nx_i\Big)^2,\quad x_i\ge0,\ n\in\N.
$$
\end{myremark}

\begin{proof}%[Доказательство леммы~\ref{Lem:sum(x_i)^r<=(sum(x_i))^r}]
Утверждение леммы равносильно тому, что
$$
Q_n(x_1,\ldots,x_n)\coloneqq \sum_{i=1}^n x_i^r-\bigg(\sum_{i=1}^n x_i\bigg)^r\le0.
$$
Имеем 
$$
\frac{\partial}{\partial x_n}Q_n(x_1,\ldots,x_n)=rx_n^{r-1}-r\bigg(\sum_{i=1}^n x_i\bigg)^{r-1}\le0,
$$
так как $x_n\le\sum_{i=1}^nx_i$ для любых неотрицательных $x_1,\ldots,x_n.$ Следовательно, $Q_n(x_1,\ldots,x_n)$ убывает по $x_n$ и 
$$
\max_{x_n\ge0}Q_n(x_1,\ldots,x_{n-1},x_n)=Q_n(x_1,\ldots,x_{n-1},0)
=Q_{n-1}(x_1,\ldots,x_{n-1}).
$$
Повторяя проведённые рассуждения ещё $n-2$ раза, заключаем, что для всех $x_1,\ldots,x_n\ge0$
$$
\qquad\qquad\quad Q_n(x_1,\ldots,x_{n-1},x_n)\le Q_1(x_1)=0.\qquad  \qedhere
$$
\end{proof}

\begin{myremark} Лемму~\ref{Lem:sum(x_i)^r<=(sum(x_i))^r} можно переписать в  эквивалентном виде
\begin{equation}\label{sum(x_i^b)^(1/b)<=(sum(x_i^a))^(1/a)}
\Big(\smallsum_{i=1}^nx_i^b\Big)^{1/b}\le \Big(\smallsum_{i=1}^nx_i^{a}\Big)^{1/a},\quad 0<a\le b,\quad x_i\ge0,\ i=1\ldots,n.
\end{equation}
Заметим, что неравенство Ляпунова для дискретного равномерного распределения на множестве чисел $\{x_1,\ldots,x_n\}$  позволяет дополнить~\eqref{sum(x_i^b)^(1/b)<=(sum(x_i^a))^(1/a)} неравенством
$$
\Big(\smallsum_{i=1}^nx_i^{a}\Big)^{1/a} \le n^{\textstyle\frac1{a}-\frac1{b}} \Big(\smallsum_{i=1}^nx_i^b\Big)^{1/b}
$$
с растущим коэффициентом $n^{\frac1{a}-\frac1{b}}$.
\end{myremark}

\begin{lemma}\label{Lem:cos(x)<=1-x^2/2+|x|^3/6}
Для любого $x\in\R$
$$
\cos x\le 1-\tfrac12x^2+\tfrac16|x|^3.
$$
\end{lemma}

\begin{proof}
Для любого $x\in\R$ имеем 
$$
0\le \cos x-1+\tfrac12x^2=\Re\left(e^{ix}-1-ix-\tfrac{(ix)^2}{2}\right)\le
$$
$$
\le \abs{e^{ix}-1-ix-\tfrac{(ix)^2}{2}}\le \tfrac{|x|^3}{6},
$$
где последнее неравенство представляет собой известную оценку остаточного члена в  формуле Тейлора для комплексной экспоненты.
\end{proof}

\begin{myremark}
Более аккуратные вычисления~\cite{Prawitz1973} показывают, что 
$$
\sup_{x\in\R}\frac{\cos x-1+x^2/2}{|x|^3} = 0.0991\ldots\ .
$$
\end{myremark}

\begin{lemma}\label{Lem|f_n(t)|<=e^(-t^2/2+L^3|t|^3/3)}
Для любого $t\in\R$
$$
\abs{\overline f_n(t)}\le \exp\big\{- \tfrac{t^2}2+\tfrac13\lyap|t|^3\big\}.
% e^{-{t^2}/2+L^3|t|^3/3}.
$$
\end{lemma}

\begin{proof} 
Пусть $X_k'$~--- независимая копия $X_k$. Тогда в  силу леммы~\ref{Lem:cos(x)<=1-x^2/2+|x|^3/6} 
$$
0\le |f_k(t)|^2=\E\cos t(X_k-X_k')\le 1-\tfrac{t^2}2\E(X_k-X_k')^2+\tfrac{|t|^3}6\E|X_k-X_k'|^3
$$
для любого $t\in\R.$ Оценим моменты симметризованного распределения. Имеем
$$
\E(X_k-X_k')^2=\D(X_k-X_k')=\D X_k+\D X_k'=2\sigma_k^2,
$$
$$
\E|X_k-X_k'|^3\le\E(X_k-X_k')^2(|X_k|+|X_k'|)=2\left(\beta_k+\sigma_k^2\E|X_k|\right)\le 4\beta_k
$$
в силу неравенства Ляпунова. И следовательно, используя также оценку $1+x\le e^x,$ $x\in\R,$ получаем 
$$
|f_k(t)|^2\le 1-t^2\sigma_k^2+\tfrac23\beta_k|t|^3\le \exp\left\{-\sigma_k^2t^2+\tfrac23\beta_k|t|^3\right\},\quad t\in\R,
$$
$$
\abs{\overline f_n(t)}^2 =\prod_{k=1}^n\Big|f_k\Big(\frac{t}{B_n}\Big)\Big|^2\le \exp\left\{- \frac{t^2}{B_n^2}\smallsum_{k=1}^n\sigma_k^2+ \frac23\cdot\frac{|t|^3}{B_n^3}\smallsum_{k=1}^n\beta_k\right\} =
$$
$$
=\exp\left\{- t^2+\tfrac23\lyap|t|^3\right\},
$$
что равносильно утверждению леммы.
\end{proof}

\begin{lemma}\label{LemAbsCh.F.DiffEstim}
Для любых $d>0$ и $|t|\le d/\lyap$
$$
r_n(t)\coloneqq\abs{\overline f_n(t)-e^{-t^2/2}}\le 2\lyap|t|^3\exp\left\{-\tfrac{t^2}2\left(1-\tfrac{2d}3\right)\right\}.
$$
\end{lemma}

\begin{proof} 
Сначала отсечём область ``больших'' значений $t\colon 1/L\le|t|\le d/L^3$ (которая непуста, только если $d\ge L^2$), оценив модуль разности тривиально суммой модулей. В силу леммы~\ref{Lem|f_n(t)|<=e^(-t^2/2+L^3|t|^3/3)} для таких $t$ имеем
$$
r_n(t)\le\abs{\overline f_n(t)}+e^{-t^2/2}\le \exp\big\{-\tfrac{t^2}2+\tfrac13\lyap|t|^3\big\}+e^{-t^2/2} \le
$$
$$
\le2\exp\big\{-\tfrac{t^2}2+\tfrac{d}3 t^2\big\}\le 2L^3|t|^3\exp\big\{-\tfrac{t^2}2+\tfrac{d}3 t^2\big\}
$$
для любого $d\ge0,$ что доказывает утверждение леммы при $1/L\hm\le|t|\le d/L^3$. Рассмотренный случай является экстремальным.

Пусть теперь $|t|<L^{-1}$. Заметим, что для всех $t\in\R$
$$
\abs{f_k(t)-1}=\abs{\E\left(e^{itX_k}-1-itX_k\right)}\le \E\left |\frac{(tX_k)^2}{2}\right | =\frac{\sigma_k^2t^2}{2},
$$
$$
\abs{f_k(t)-1+\frac{\sigma_k^2t^2}2}\le\E\abs{e^{itX_k}-1-itX_k-\frac{(itX_k)^2}{2}}\le \frac{\beta_k|t|^3}6,
$$
и, так как
$$
\sigma_k\le\beta_k^{1/3}\le \Big(\smallsum_{k=1}^n\beta_k\Big)^{1/3}=LB_n,\quad k=1,\ldots,n,
$$
то 
$$
\abs{f_k\left (\frac t{B_n}\right )-1}\le \frac{t^2}2\cdot\frac{\sigma_k^2}{B_n^2}\le \frac{L^2t^2}{2}\le \frac12\quad \text{при }\quad L|t|\le1.
$$
Следовательно, $|f_k(t/B_n)|\ge1/2$ и определён логарифм 
$$
\ln\overline f_n(t) = \sum_{k=1}^n\ln f_k\left (\frac t{B_n}\right )\quad\text{при}\quad |t|\le\frac1L.
$$
Выбирая главную ветвь логарифма (проходящую через нуль в единице) и используя далее неравенство $\abs{e^z-1}\le|z|e^{|z|},$ $z\in\C,$ получаем
$$
r_n(t)=e^{-t^2/2}\abs{\exp\left\{\tfrac{t^2}2+\ln\overline f_n(t)  \right\}-1}\le \delta_n(t)e^{\delta_n(t)-t^2/2},
$$
где
$$
\delta_n(t)\coloneqq\abs{\frac{t^2}2+\ln\overline f_n(t)} =\abs{ \frac{t^2}{2B_n^2}\sum_{k=1}^n\sigma_k^2 + \sum_{k=1}^n\ln\left(1+ \left[f_k\left(\frac{t}{B_n} \right)-1 \right]\right) }.
$$
В силу неравенства $\abs{\ln(1+z)-z}\le|z|^2,$ $|z|\le\frac12,$ (см.~\eqref{|ln(1+z)-z|<=|z|^2}) имеем
$$
\delta_n(t)\le \abs{\smallsum_{k=1}^n\left[f_k\left(\frac{t}{B_n}\right)-1+\frac{t^2\sigma_k^2}{2B_n^2}\right]} +\smallsum_{k=1}^n\abs{f_k\left(\frac{t}{B_n} \right)-1}^2\le
$$
$$
\le \smallsum_{k=1}^n\frac{\beta_k|t|^3}{6B_n^3} +\smallsum_{k=1}^n\frac{\sigma_k^4t^4}{4B_n^4} \le \frac{L^3|t|^3}{6} +\frac{t^4}{4B_n^4}\smallsum_{k=1}^n\beta_k^{4/3}.
$$
В силу леммы~\ref{Lem:sum(x_i)^r<=(sum(x_i))^r} с $r=4/3$ %(см.~\eqref{sum(x_i^b)^(1/b)<=(sum(x_i^a))^(1/a)} с $a\coloneqq1,$ $b\coloneqq4/3$),
$$
\smallsum_{k=1}^n\beta_k^{4/3}\le   \Big(\smallsum_{i=1}^n\beta_k\Big)^{4/3},
$$
и следовательно, %при  рассматриваемых $L|t|\le1$
$$
\delta_n(t)\le  \tfrac16{L^3|t|^3} +\tfrac14{L^4t^4}\le \tfrac5{12} L^3|t|^3\le \tfrac5{12},\quad L|t|\le1.
$$
Используя обе последние оценки для $\delta_n(t)$ и замечая далее, что $\tfrac5{12}\,e^{5/12}=0.63\ldots<2$, получаем
$$
r_n(t)\le\tfrac5{12}L^3|t|^3\cdot e^{5/12-t^2/2}\le 2L^3|t|^3 e^{-t^2/2},
$$
откуда тривиально вытекает утверждение леммы.
\end{proof}

\begin{proof}[Доказательство теоремы~\ref{BEineqForE|Xi|^3<infty}]
По неравенству сглаживания~\eqref{SmoothingIneqEsseenKernel} с $F(x)\coloneqq\Prob(S_n<xB_n),$ $G=\Phi,$ 
$$
A\coloneqq\sup_{x\in\R}\Phi'(x)=\sup_{x\in\R}\frac{e^{-x^2/2}}{\sqrt{2\pi}}=\frac1{\sqrt{2\pi}}
$$
при всех $b>1$ и $T>0$ имеем
$$
\Delta_n\le \frac{b}{2\pi} \int_{-T}^T\bigg|\frac{\overline f_n(t)-e^{-t^2/2}}{t}\bigg|\dd t + \frac{4b(b+1)}{\pi\sqrt{2\pi}(b-1)}\cdot \frac1{T}\eqqcolon I_1+I_2.
$$
Положим  $T\coloneqq d/L^3,$ $d\in\big(0,\frac32\big),$ тогда
$$
I_2= \frac{4}{\pi\sqrt{2\pi}}\cdot \frac{b(b+1)}{d(b-1)}\cdot L^3
$$
имеет искомый порядок $\mathcal O(L^3)$, а интеграл $I_1$ оценим  с помощью леммы~\ref{LemAbsCh.F.DiffEstim}:
$$
I_1\le\frac{2bL^3}{\pi} \int_0^{d/L^{3}} t^2\exp\left\{-\tfrac{t^2}2\left(1-\tfrac{2d}3\right)\right\}\dd t \le
$$
$$
\le\frac{2bL^3}{\pi} \int_0^\infty t^2\exp\left\{-\tfrac{t^2}2\left(1-\tfrac{2d}3\right)\right\}\dd t =
$$
$$
= \frac{2bL^3}{\pi}\cdot \frac{\sqrt2\Gamma(3/2)}{(1-2d/3)^{3/2}}= \frac{2}{\sqrt{2\pi}}\cdot \frac{bL^3}{(1-2d/3)^{3/2}}.
$$
Суммируя построенные оценки для $I_1$ и $I_2$,  получаем неравенство
$$
\Delta_n\le \frac{2L^3}{\sqrt{2\pi}}\left[ \frac{b}{(1-2d/3)^{3/2}} +\frac{2b(b+1)}{\pi d(b-1)}\right],
$$
справедливое при любых $b>1$ и $d\in\big(0,\frac32\big).$ Дальнейшая оптимизация по $b$ и $d$ приводит к утверждению теоремы.
\end{proof}

Более аккуратное оценивание величин $\abs{\overline f_n(t)}$, $\big|\overline f_n(t)-e^{-t^2/2}\big|$ и использование ЭВМ для дальнейших расчётов приводит к более точным оценкам абсолютной константы~$C_0$.

\let\lyap\oldlyap

\section[Доказательство неравенства Осипова]{Доказательство неравенства Осипова. Метод усечений}\label{SectOsipovIneqProof}

Пусть $X_1,X_2,\ldots,X_n$ -- независимые случайные величины с ф.р. $F_1,F_2,\ldots,F_n,$ $\E X_k=0,$ $\sigma_k^2\coloneqq\E X_k<\infty$ и 
$$
S_n\coloneqq\sum_{k=1}^nX_k,\quad B_n^2\coloneqq\sum_{k=1}^n\sigma_k^2=\D S_n>0,
$$
$$
\overline F_n(x)\coloneqq \Prob(S_n<xB_n),\quad \Delta_n=\sup_{x\in\R}|\overline F_n(x)-\Phi(x)|,
$$
$$
L_n(\eps)\coloneqq \frac1{B_n^2}\sum_{k=1}^n\E X_k^2\I(|X_k|>\eps B_n),
$$
$$
M_n(\eps)\coloneqq \frac1{B_n^3}\sum_{k=1}^n\E|X_k|^3\I(|X_k|\le\eps B_n),\quad \eps>0.
$$

\begin{theorem}\label{ThOsipovIneq(eps=1)}
В сделанных выше предположениях
\begin{equation}\label{OsipovIneq(eps=1)}
\Delta_n\le C(L_n(1)+M_n(1)),
\end{equation}
где 
$$
C=\inf_{0<b<1}\max\bigg\{\frac2{1-b^2}, 1+\frac{4.25C_0}{b^3}+\frac1{\sqrt{2\pi}b}\Big(1+\frac{2e^{-1/2}}{1+b}\Big)\bigg\},
$$
а $C_0$~--- абсолютная константа из классического неравенства Берри--Эссеена с $\d=1$ $($см. теоремы~\ref{ThBEineq} и~\ref{BEineqForE|Xi|^3<infty}$)$.
\end{theorem}

\begin{myremark}
Поскольку первое выражение под знаком максимума в определении константы $C$ монотонно возрастает по $b\in(0,1)$, а второе~--- монотонно убывает, то оптимальное значение $b$  удовлетворяет условию равенства аргументов максимума друг другу. В частности:\\
$\bullet$ для $C_0=9$ имеем $C=42.75\ldots$ (при $b=0.97\ldots$);
\\
$\bullet$ для $C_0=0.5583$ имеем $C=6.11\ldots$ (при $b=0.82\ldots$).
\\
$\bullet$ для $C_0=0.4690$ (случай о.р.с.в.) $C=5.66\ldots$ (при $b=0.80\ldots$).
\end{myremark}

Напомним, что наилучшая известная верхняя оценка константы $C\le1.87$ получена в~\cite{KorolevDorofeyeva2017}.

\begin{lemma}\label{LemE|X-a|^3<E|X|^3+3.25|a|}
Пусть $X$~--- с.в. с $\beta_{2+\d}\coloneqq\E|X|^{2+\d}<\infty$ при некотором $\d\in(0,1]$ и $a\coloneqq \E X,$ $\beta_r\coloneqq\E|X|^r,$ $0<r\le2+\d.$ Тогда 
$$
\E|X-a|^{2+\d}\le \beta_{2+\d}+3.25\beta_2|a|^\d\le 4.25\beta_{2+\d}.
$$
\end{lemma}
\begin{proof}
Если $\E X^2=0,$ то $X\eqp0$ и утверждение леммы тривиально. Поэтому и в силу инвариантности доказываемого неравенства относительно преобразования масштаба, всюду далее без ограничения общности будем считать, что $\beta_2\coloneqq\E X^2=1.$ Тогда в силу неравенства Ляпунова $|a|\le1$, $\beta_\d\le1,$ $\beta_{1+\d}\le1\le\beta_{2+\d}.$ По неравенству треугольника $|x+y|^\d\le|x|^\d+|y|^\d,$ $x,y\in\R,$ имеем
$$
\E|X-a|^{2+\d}=\E|X-a|^\d(X-a)^2\le\E(|X|^\d+|a|^\d)(X^2-2aX+a^2)=
$$
$$
=\beta_{2+\d}-2a\E X|X|^\d+a^2\beta_\d+|a|^\d\beta_2-|a|^{2+\d}\le
$$
$$
\le\beta_{2+\d}+2|a|\beta_{1+\d}+a^2\beta_\d+|a|^\d\beta_2-|a|^{2+\d}\le
$$
$$
\le \beta_{2+\d}+|a|^\d\left( 2|a|^{1-\d}+|a|^{2-\d}+1-|a|^2\right)\le
$$
$$
\le \beta_{2+\d} +|a|^\d\max_{0\le x\le1}\left(1+2x^{1-\d}+x^{2-\d}-x^2\right),
$$
и, так как $\d\in(0,1]$,
\begin{multline*}
\E|X-a|^{2+\d}\le \beta_{2+\d} +|a|^\d\max_{0\le x\le1}\left(3+x-x^2\right) =
\\
= \beta_{2+\d}+3.25|a|^\d\le 4.25 \beta_{2+\d}.\qedhere
\end{multline*}
\end{proof}

\begin{lemma}\label{LemPhiIncrementsEstims}
Для любых $q\in\R$ и $p>0$
\begin{equation}\label{sup_x|Phi(x+q)-Phi(x)|=2Phi(q)-1}
\max_{x\in\R}\abs{\Phi(x+q)-\Phi(x)}=2\Phi\big(\tfrac{|q|}2\big)-1\le \frac{|q|}{\sqrt{2\pi}},
\end{equation}
\begin{multline}\label{sup_x|Phi(px)-Phi(x)|=...<=...}
\max_{x\in\R}\abs{\Phi(px)-\Phi(x)}=\abs{\Phi\bigg(p\sqrt{\frac{2\ln p}{p-1}}\bigg) -\Phi\bigg(\sqrt{\frac{2\ln p}{p-1}}\bigg)}\le 
\\
\le\min\bigg\{\sqrt{\frac{(p-1)\ln p}\pi},\ \frac{p\vee p^{-1}-1}{\sqrt{2\pi e}}\bigg\}.
\end{multline}
\end{lemma}
Соотношение~\eqref{sup_x|Phi(x+q)-Phi(x)|=2Phi(q)-1} доказано в~\cite[Lemma\,2]{DorofeevaKorolevZeifman2020}.

\begin{proof}
Обозначим
$$
g_1(x)=\Phi(x+q)-\Phi(x),\quad g_2(x)\coloneqq\Phi(px)-\Phi(x),\quad x\in\R,
$$
и заметим, что $g_i(-x)=-g_i(x)$, $x\in\R$, а значит, функции  $|g_i|$ четные и, следовательно, 
$$
\sup_{x\in\R}|g_i(x)|=\sup_{x\ge0}|g_i(x)|= \max\Big\{\sup_{x\ge0}g_i(x),-\inf_{x\ge0}g_i(x)\Big\},\quad i=1,2.
$$
В силу того, что $g_i(0)=g_i(\infty)=0$ и $g_i$ непрерывна, функция $|g_i|$ достигает своего максимального и минимального значений в стационарных точках $g_i$, $i=1,2$. Найдем эти точки. Приравнивая производные 
$$
g_1'(x)=\phi(x+q)-\phi(x)=\tfrac1{\sqrt{2\pi}}\big (e^{-(x+q)^2/2}-e^{-x^2/2}\big),
$$
$$
g_2'(x)=p\phi(px)-\phi(x)=\tfrac1{\sqrt{2\pi}}\big (pe^{-px^2/2}-e^{-x^2/2}\big)
$$
к нулю и логарифмируя, находим единственную на $[0,\infty)$ стационарную точку
$$
x_1=-\frac q2\quad (\text{для } g_1),\quad x_2=\sqrt{\frac{2\ln p}{p-1}}\quad (\text{для } g_2).
$$
Чтобы установить тип экстремума в этой точке, найдем вторую производную:
$$
\sqrt{2\pi}g_1''(x)=xe^{-x^2/2}-(x+q)e^{-(x+q)^2/2},
$$
$$
\sqrt{2\pi}g_2''(x)=x\big(e^{-x^2/2}-p^2e^{-px^2/2}).
$$

Ясно, что 
$$
\sqrt{2\pi}g_1''(x_1)=-qe^{-q^2/8}<0\quad \Leftrightarrow\quad q>0,
$$
т.е. при $q>0$ функция $g_1$ достигает глобального максимума в точке~$x_1$, и ее глобальный минимум равняется нулю, а при $q<0$ точка~$x_1$ является точкой глобального минимума,  а глобальный максимум $g_1$ равняется нулю. Таким образом,
$$
\max_{x\in\R}|g_1(x)|=
\begin{cases}
g_1(x_1),&q>0,
\\
-g_1(x_1),&q<0,
\end{cases} 
=|g_1(x_1)|=\abs{\Phi(q/2)-\Phi(-q/2)}=
$$
$$
=\abs{2\Phi(q/2)-1}=2\Phi(|q|/2)-1,
$$
что доказывает равенство в~\eqref{sup_x|Phi(x+q)-Phi(x)|=2Phi(q)-1}. Следующее за ним неравенство вытекает из формулы конечных приращений с учетом того, что $2\Phi(0)=1$ и $\sup\limits_{x\in\R}\Phi'(x)=\phi(0)=1/\sqrt{2\pi}$.

Логарифмируя уравнение 
$$
\sqrt{2\pi}g_2''(x_2)/x_2=e^{-\ln p/(p-1)} -p^2e^{-p\ln p/(p-1)}=0,
$$
убеждаемся, что $g_2''(x_2)<0$ тогда и только тогда, когда $\ln p>0$,
%$$
% -\frac{\ln p}{(p-1)}>2\ln p -\frac{p\ln p}{p-1}
%$$
%$$
%-\frac{(p-1)\ln p}{(p-1)}>2\ln p
%$$
%$$
%-\ln p>2\ln p
%$$
%$$
%0>3\ln p
%$$
%$$
%p<1
%$$
%$$
%g''(x_*)<0\quad \Leftrightarrow\quad \ln p>0\quad \Leftrightarrow\quad p>1,
%$$
т.е. при $p>1$ в точке $x_2$ достигается глобальный максимум функции~$g_2$ и при этом ее глобальный минимум равняется нулю, а при $p<1$ в точке $x_2$ достигается глобальный минимум $g_2$ и глобальный максимум~$g_2$ равняется нулю. Таким образом, 
$$
\max_{x\in\R}|g_2(x)|=
\begin{cases}
g_2(x_2),&p>1,
\\
-g_2(x_2),&p<1,
\end{cases} 
=|g_2(x_2)|=\abs{\Phi(px_2)-\Phi(x_2)},
$$
что совпадает с искомым выражением в~\eqref{sup_x|Phi(px)-Phi(x)|=...<=...}. Далее, т.к.
$$
\abs{px_2-x_2}=|p-1|\sqrt{\frac{2\ln p}{p-1}}=\sqrt{2(p-1)\ln p},
$$
то из формулы конечных приращений получаем, что
$$
\abs{\Phi(px_2)-\Phi(x_2)}\le |px_2-x_2|\max_{x\in\R}\Phi'(x) =\sqrt{ \frac{(p-1)\ln p}\pi}, 
$$
что совпадает с первым аргументом минимума в правой части~\eqref{sup_x|Phi(px)-Phi(x)|=...<=...}. Второй аргумент минимума получим, если формулу Лагранжа используем с самого начала:
$$
\Phi(px)-\Phi(x)=(px-x)\phi(\theta px+(1-\theta)x),
$$
с некоторым $\theta=\theta(x)\in[0,1]$,
$$
\frac{\abs{\Phi(px)-\Phi(x)}}{|p-1|}\le \sup_{x\in\R,\theta\in[0,1]}|x|\phi(\theta px+(1-\theta)x) =
$$
$$
= \sup_{y\in\R,\theta\in[0,1]}\frac{|y|\phi(y)}{|p\theta+1-\theta|} =  \max\{p^{-1},1\}\phi(1) = \frac{p^{-1}\vee1}{\sqrt{2\pi e}},
$$
что равносильно искомому неравенству, т.к.
$$
|p-1|\max\{p^{-1},1\}=
\left\{\begin{array}{ll}
p-1&\text{при } p\ge1
\\
1-p^{-1}&\text{при } p<1
\end{array}\right\}
=p\vee p^{-1}-1.
\qedhere
$$
\end{proof}

\begin{proof}[Доказательство теоремы~\ref{ThOsipovIneq(eps=1)}.]
Будем использовать \textit{метод усечений}. Положим для $k=1,\ldots,n$
$$
\overline X_k\coloneqq X_k\I(|X_k|\le B_n),\quad \overline a_k=\E\overline X_k,\quad \overline \sigma_k^2=\D\overline X_k,\quad\overline  B_n^2=\sum_{k=1}^n\overline \sigma_k^2.
$$
В силу того, что $\E X_k=0$, имеем
$$
\overline a_k = \E X_k\I(|X_k|\le B_n) = -\E X_k\I(|X_k|>B_n),
$$
и следовательно,
$$
|\overline a_k|\le \E|X_k|\I(|X_k|>B_n)\le B_n^{-1}\E X_k^2\I(|X_k|>B_n),\quad k=1,\ldots,n.
$$
Далее, 
$$
0\le \sigma_k^2-\overline \sigma_k^2=\E X_k^2-\E\overline X_k^2+\overline a_k^2 = 
$$
$$
=\E X_k^2\I(|X_k|>B_n)+ (\E X_k\I(|X_k|>B_n))^2\le 2\E X_k^2\I(|X_k|>B_n),
$$
и следовательно,
$$
0\le B_n^2-\overline B_n^2\le 2\sum_{k=1}^n\E X_k^2\I(|X_k|>B_n)=2B_n^2L_n(1).
$$

Пусть $b$ --- произвольное число из интервала $(0,1).$ Отбросим сначала тривиальный случай, когда $\overline B_n\hm\le b B_n,$ т.е. $\overline B_n$ ``далеко'' отстоит от~$B_n$. Тогда $B_n^2-\overline B_n^2\ge (1-b^2)B_n^2$, и следовательно, $(1-b^2)B_n^2\le 2B_n^2L_n(1)$, откуда тривиально получаем
$$
\Delta_n\le 1\le \frac{2}{1-b^2}L_n(1),
$$
т.е. в данном случае неравенство~\eqref{OsipovIneq(eps=1)} справедливо с $C\le \frac{2}{1-b^2}.$

Пусть теперь $\overline B_n>bB_n.$ Положим $\overline S_n\coloneqq\overline  X_1+\ldots+\overline X_n$. Прибавляя и вычитая под знаком модуля в нижеследующей формуле величины $\Prob(\overline S_n<x)$, $\Phi\left((x-\sum\overline  a_j)/\overline B_n\right)$ и $\Phi(x/\overline B_n)$, получаем
$$
\Delta_n=\sup_{x\in\R}\abs{\Prob(S_n<x)-\Phi(x/B_n)}\le I_1+I_2+I_3+I_4,
$$
где 
$$
I_1\coloneqq\sup_{x\in\R} \abs{\Prob(S_n<x)-\Prob(\overline S_n<x)},
$$
$$
I_2\coloneqq\sup_{x\in\R} \abs{\Prob(\overline S_n<x)-\Phi\bigg(\frac{x-\sum \overline a_k}{\overline B_n}\bigg)},
$$
$$
I_3\coloneqq\sup_{x\in\R} \abs{\Phi\bigg(\frac{x-\sum \overline a_k}{\overline B_n}\bigg) - \Phi\bigg(\frac{x}{\overline B_n}\bigg)},
$$
$$
I_4\coloneqq\sup_{x\in\R} \abs{\Phi\bigg(\frac{x}{\overline B_n}\bigg) - \Phi\bigg(\frac{x}{B_n}\bigg)}.
$$

Для оценки $I_1$ введем событие $A_n\coloneqq \bigcap_{k=1}^n\{|X_k|\le B_n\}.$ Тогда $S_n=\overline S_n$ на $A_n,$ и следовательно,
$$
\Prob(S_n<x)-\Prob(\overline S_n<x)=\Prob( S_n<x,\overline A_n) + \Prob(S_n<x,A_n)-\Prob(\overline S_n<x)=
$$
$$
=\Prob( S_n<x,\overline A_n) + \Prob(\overline S_n<x,A_n)-\Prob(\overline S_n<x)\le \Prob( S_n<x,\overline A_n) \le \Prob(\overline A_n).
$$
Аналогично
$$
\Prob(\overline S_n<x)-\Prob(S_n<x)\le\Prob(\overline S_n<x,\overline A_n) \le \Prob(\overline A_n).
$$
Таким образом, учитывая обе полученные оценки, имеем
$$
I_1\le \Prob(\overline A_n) =\Prob\bigg(\bigcup_{k=1}^n\{|X_k|>B_n\}\bigg)\le \sum_{k=1}^n\Prob(|X_k|>B_n)=
$$
$$
=\sum_{k=1}^n\E\I(|X_k|>B_n)\le\frac1{B_n^2}\sum_{k=1}^n\E X_k^2\I(|X_k|>B_n)= \boxed{L_n(1)}.
$$

Обозначив 
$$
Z_n=\frac{\overline S_n-\E\overline S_n}{\D\overline S_n}=\frac{\overline S_n-\sum \overline a_k}{\overline B_n},
$$
величину $I_2$ можем оценить с помощью неравенства Берри--Эссеена (теорема~\ref{BEineqForE|Xi|^3<infty}) следующим образом
$$
I_2=\sup_{x\in\R}\abs{\Prob(Z_n<x)-\Phi(x)}\le \frac{C_0}{\overline B_n^3}\sum_{k=1}^n\E\abs{\overline X_k-\overline a_k}^3.
$$
Далее, по лемме~\ref{LemE|X-a|^3<E|X|^3+3.25|a|} с $\d=1$ и с учетом того, что $B_n/\overline B_n\le b^{-1},$ имеем
$$
I_2\le \frac{4.25C_0}{\overline B_n^3}\sum_{k=1}^n\E\abs{\overline X_k}^3 = 4.25C_0M_n(1){B_n^3}/{\overline B_n^3}\le \boxed{\frac{4.25C_0}{b^3}M_n(1)}.
$$

Величины $I_3,$ $I_4$ оценим с помощью леммы~\ref{LemPhiIncrementsEstims}, дополнительно учитывая, что $|\overline a_k|\le B_n^{-1}\E X_k^2\I(|X_k|>B_n)$, $\overline B_n>bB_n$ и $p\hm\coloneqq B_n/\overline B_n\ge1$, $B_n^2-\overline B_n^2\le2B_n^2L_n(1)$. Имеем
$$
I_3=\sup_{x\in\R} \abs{\Phi\bigg(x-\frac1{\overline B_n}\sum_{k=1}^n \overline a_k\bigg) - \Phi(x)}\le \frac1{\sqrt{2\pi}\cdot\overline B_n}\sum_{k=1}^n|\overline a_k|\le
$$
$$
\le \frac1{\sqrt{2\pi}bB_n^2}\sum_{k=1}^n\E X_k^2\I(|X_k|>B_n) =\boxed{\frac{L_n(1)}{\sqrt{2\pi}b}},
$$
$$
I_4=\sup_{x\in\R} \abs{\Phi\bigg(\frac{B_n}{\overline B_n}x\bigg) - \Phi(x)}\le \frac1{\sqrt{2\pi e}}\left(\frac{B_n}{\overline B_n}-1\right)= 
$$
$$
=\frac{B_n^2-\overline B_n^2}{\sqrt{2\pi e}(B_n+\overline B_n)\overline B_n}\le
 \boxed{\frac{2L_n(1)}{\sqrt{2\pi e}(1+b)b}}.
$$

Наконец, применяя тривиальное неравенство $\max\{L_n(1),M_n(1)\}\hm\le L_n(1)+M_n(1)$ к построенным оценкам для $I_1$, $I_2$, $I_3$, $I_4$ и суммируя затем полученные мажоранты, а также учитывая свободу в выборе параметра $b\in(0,1),$ приходим к утверждению теоремы.
\end{proof}

\begin{myremark}
На самом деле мы доказали неравенство
\begin{equation}\label{OsipovIneq:max(Ln(1),Mn(1))}
\Delta_n\le C\max\{L_n(1),M_n(1)\},
\end{equation}
с той же самой константой $C$, что определена в формулировке теоремы~\ref{ThOsipovIneq(eps=1)}. Неравенство~\eqref{OsipovIneq:max(Ln(1),Mn(1))} на первый взгляд кажется сильнее неравенства Осипова~\eqref{OsipovIneq(eps=1)}, объявленного в формулировке теоремы~\ref{ThOsipovIneq(eps=1)}. Однако, учитывая, что
$$
\max\{L_n(1),M_n(1)\}\le L_n(1)+M_n(1)\le 2\max\{L_n(1),M_n(1)\}
$$
видим, что оценки~\eqref{OsipovIneq(eps=1)} и~\eqref{OsipovIneq:max(Ln(1),Mn(1))}  отличаются только значениями абсолютной константы $C$, но, поскольку мы не ставим здесь цели \textit{оптимизации} значений абсолютных констант, можно считать, что неравенства ~\eqref{OsipovIneq(eps=1)} и~\eqref{OsipovIneq:max(Ln(1),Mn(1))} эквивалентны. 
\end{myremark}

\section[Неравномерные оценки]{Неравномерные оценки скорости сходимости в ЦПТ}

В данном и следующем разделе будут представлены без доказательств ещё некоторые полезные оценки точности нормальной аппроксимации для распределений сумм независимых случайных величин. Напомним основные предположения и обозначения.

Рассматривается последовательность $X_1,X_2,\ldots,X_n,\ldots$ независимых с.в. с ф.р. $F_1,F_2,\ldots,F_n,\ldots,$ $\E X_k=0,$ $\sigma_k^2\coloneqq\E X_k<\infty$ и 
$$
S_n\coloneqq\sum_{k=1}^nX_k,\quad B_n^2\coloneqq\sum_{k=1}^n\sigma_k^2=\D S_n>0.
$$
Положим
$$
\overline F_n(x)\coloneqq \Prob(S_n<xB_n),\quad \Delta_n(x)=|\overline F_n(x)-\Phi(x)|,\quad  x\in\R.
$$
Заметим, что поскольку и $\overline F_n(x)$, и $\Phi(x)$~--- ф.р., то величина $\Delta_n(x)$ должна стремиться к нулю при $|x|\to\infty$ вне зависимости от значения~$n$. Более того, если у с.в. $X_1,\ldots,X_n$ конечны моменты некоторого порядка $r>0$, то в силу неравенства Маркова
$$
\overline F_n(-x+0)+ 1-\overline F_n(x) = \Prob(|S_n|\ge x)\le x^{-r}\E|S_n|^r,\quad x>0,
$$
т.е. порядок малости величины $\Delta_n(x)$ должен быть $\mathcal O\left (|x|^{-r}\right)$ при $|x|\to\infty$. Оценки равномерного расстояния 
$$
\Delta_n=\sup_{x\in\R}\Delta_n(x),
$$
которые мы рассматривали выше, не учитывают эту особенность. Вместе с тем точность нормальной аппроксимации для ф.р. сумм с.в. именно при больших значениях аргумента представляет особый интерес, например, при вычислении рисков критически больших потерь. Неравномерные оценки, рассматриваемые в данном разделе, не только учитывают большие значения аргумента, являясь бесконечно малыми при $|x|\to\infty$, но и устанавливают ``правильный'' порядок малости $\Delta_n(x)=\mathcal O\left (|x|^{-r}\right)$, $n\in\N,$ в зависимости от моментных предположений относительно случайных слагаемых.

Напомним также следующие обозначения:
$$
L_n(\eps)\coloneqq \frac1{B_n^2}\sum_{k=1}^n\E X_k^2\I(|X_k|>\eps B_n),
$$
$$
M_n(\eps)\coloneqq \frac1{B_n^3}\sum_{k=1}^n\E|X_k|^3\I(|X_k|\le\eps B_n),\quad \eps>0,
$$
и, если $\max_{1\le k\le n}\E|X_{k}|^{2+\d}<\infty$ для некоторого $\d>0$,
$$
\lyapd\coloneqq \frac1{B_n^{2+\d}}\sum_{k=1}^n\E|X_{k}|^{2+\d}
$$
для дроби Линдеберга, суммы усеченных абсолютных третьих моментов и дроби Ляпунова соответственно.

В 1965\,г. С.\,В. Нагаевым~\cite{Nagaev1965} для о.р.с.в. $\d=1$ и в 1966\,г. А.\,Бикялисом~\cite{Bikelis1966} для общего случая был получен следующий результат.

\begin{theorem}[неравенство Нагаева--Бикялиса]\label{ThNagaevBikelisIneq}
Для каждого $\d\in(0,1]$
\begin{equation}\label{NagaevBikelisIneq}
\Delta_n(x)\le K_0\frac{\lyapd}{1+|x|^{2+\d}},\quad x\in\R,\ n\in\N,
\end{equation}
где 
$K_0$~--- некоторая абсолютная константа.
\end{theorem}

Впервые верхние оценки константы $K_0=K_0(\d)$ (в зависимости от значений $\d$) были получены Л.\,Падитцем~\cite{Paditz1976,Paditz1977,Paditz1978,Paditz1979} в 1976--1979\,гг. и имели порядок нескольких сотен вплоть до двух тысяч. Далее оценки констант $K_0(\d)$ рассматривались и уточнялись в статье Р.\,Михеля~\cite{Michel1981} (случай о.р.с.в. и $\d=1$), диссертации В.\,Тысиака~\cite{Tysiak1983}, работах Ш.\,А.\,Мирахмедова, Л.\,Падитца~\cite{Mirachmedov1984,PaditzMirachmedov1986,Paditz1989} 1984--1989\,гг., Ю.\,С.\,Нефедовой, И.\,Г.\,Шевцовой, М.\,Е.\,Григорьевой, С.\,В.\,Попова~\cite{NefedovaShevtsova2011,NefedovaShevtsova2012, GrigorievaPopov2012SMI,Shevtsova2013Inf,Shevtsova2016} 2011--2016\,гг. Наилучшие известные на сегодняшний день верхние оценки получены в~\cite{Shevtsova2016} (также см.~\cite{Shevtsova2013Inf,Shevtsova2017book}) и представлены в табл.\,\ref{TabBEnonuni1+x_struct} во втором и пятом столбцах.

\begin{table}[h]
\centering
\begin{tabular}{||c||c|c|c||c|c|c||}
\hline &\multicolumn{3}{c||}{Случай р.р.с.в.}
&\multicolumn{3}{c||}{Случай о.р.с.в.}\\
\cline{2-7} $\delta$ & \vphantom{$\frac{1^1}2$}
$K_0(\d)$ & $K_{s_1}(\d)$ & $s_1$ &
$K_0(\d)$ & $K_{s_1}(\d)$ & $s_1$
\\\hline
$1$  &21.82&18.19&1    &17.36&15.70&0.646\\
$0.9$&20.07&16.65&1    &16.24&14.61&0.619\\
$0.8$&18.53&15.34&1    &15.20&13.61&0.625\\
$0.7$&17.14&14.20&1    &14.13&12.71&0.570\\
$0.6$&15.91&13.19&0.859&13.15&11.90&0.498\\
$0.5$&14.84&12.30&0.834&12.26&11.17&0.428\\
$0.4$&13.92&11.53&0.806&11.43&10.51&0.350\\
$0.3$&13.10&10.86&0.778&10.66&9.93 &0.273\\
$0.2$&12.35&10.28&0.748&9.92 &9.42 &0.183\\
$0.1$&11.67& 9.77&0.710&9.18 &8.97 &0.074\\
\hline
\end{tabular}
\caption{Верхние оценки констант $K_s(\delta)$ из
неравенств~\eqref{NagaevBikelisIneq} и~\eqref{BEnonuni1+x_struct} для некоторых $s\in[0,1]$ и
${\delta\in(0,1)}$.}
\label{TabBEnonuni1+x_struct}
\end{table}

Стоит отметить, что на самом деле Бикялис в~\cite[Теорема 4]{Bikelis1966} получил более сильный результат, который можно назвать неравномерным аналогом неравенства Осипова~\eqref{OsipovIneqEps=1}. Приводим его в следующей теореме.

\begin{theorem}[неравенство Бикялиса]\label{ThBikelisIneq}
В сделанных выше предположениях для всех $x\in\R$ и  $n\in\N$
\begin{eqnarray}\label{BikelisIneqInt}
\Delta_n(x)&\le& \frac{A}{(1+|x|)^3B_n}\int_0^{(1+|x|)B_n}L_n(z)\dd z =
\\
\label{BikelisIneqMin}
&=&\frac{A} {(1+|x|)^3B_n^3}\sum_{k=1}^n\E X_k^2\min\big\{|X_k|,\,(1+|x|)B_n\big\}=\qquad
\end{eqnarray}
\begin{equation}\label{BikelisIneq}
=A\sum_{k=1}^n\bigg[\frac{\E X_k^2\I(|X_k|>(1+|x|)B_n)}{(1+|x|)^2B_n^2}
+ \frac{\E|X_k|^3\I(|X_k|\le (1+|x|)B_n)}{(1+|x|)^3B_n^3}\bigg]
\end{equation}
с некоторой абсолютной константой $A$.
\end{theorem}

И из~\eqref{BikelisIneqInt} Бикялис уже вывел~\eqref{NagaevBikelisIneq}. Действительно, домножая квадратичную  и кубическую функции под знаками интегралов в~\eqref{BikelisIneq}  на 
$$
\Big(\frac{|X_k|}{(1+|x|)B_n}\Big)^\d\ge1 \quad \text{и}\quad \Big(\frac{(1+|x|)B_n}{|X_k|}\Big)^{1-\d}\ge1
$$
соответственно, из~\eqref{BikelisIneq}  получаем~\eqref{NagaevBikelisIneq} с
%\begin{equation}\label{K0<=A*2^(1+d)}
%K_0\le A\sup_{x>0}\frac{(1+x)^{2+\d}}{1+x^{2+\d}}= A\cdot\frac{(1+x)^{2+\d}}{1+x^{2+\d}}\bigg|_{x=1}= A\cdot 2^{1+\d}\le 4A.
%\end{equation}
\begin{equation}\label{K0<=A*2^(1+d)}
K_0\le A\sup_{x>0}\frac{1+x^{2+\d}}{(1+x)^{2+\d}}= A\cdot\frac{1+x^{2+\d}}{(1+x)^{2+\d}}\bigg|_{x\to\infty}= A.
\end{equation}
Однако выражение~\eqref{BikelisIneq}, хотя, очевидно, и подразумевалось Бикялисом при выводе~\eqref{NagaevBikelisIneq}, в явном виде в~\cite{Bikelis1966} выписано не было и впервые встречается лишь в работе В.\,В.\,Петрова~\cite{Petrov1979} 1979\,г., где он выводит его из~\eqref{BikelisIneqInt} в ходе доказательства некоторой леммы, сформулированной ниже в виде теоремы~\ref{ThPetrovIneq}.

Напомним, что $\mathcal G$~--- класс всех неотрицательных чётных функций~$g$, определённых на $\R$, таких что $g(x)>0$ при $x>0$, и функции $g(x)$, $x/g(x)$ не убывают в области $x>0.$

\begin{theorem}[неравенство Петрова]\label{ThPetrovIneq}
В сделанных выше предположениях для любой функции $g\in\mathcal G$ 
\begin{equation}\label{PetrovIneq}
\Delta_n(x)\le \frac{A\sum_{k=1}^n\E X_k^2g(X_k)}{(1+|x|)^2B_n^2g((1+|x|)B_n)},\quad x\in\R,\ n\in\N,
\end{equation}
с той же константой  $A$, что и в~\eqref{BikelisIneq} $($универсальной для всех функций $g\in\mathcal G)$.
\end{theorem}

Неравенство~\eqref{PetrovIneq}  тривиально вытекает из~\eqref{BikelisIneq} с учетом леммы~\ref{LemKatzPetrovGproperties}; оно доказывалось в~\cite{Petrov1979} с использованием ровно тех же идей применительно к конкретному выражению от функции $g\in\mathcal G$, а не сразу к классу всех функций $g\in\mathcal G$, как в лемме~\ref{LemKatzPetrovGproperties}.  С другой стороны, неравенство Бикялиса~\eqref{BikelisIneqMin} вытекает из~\eqref{PetrovIneq}  с $g(u)=\min\big\{|u|,\,(1+|x|)B_n\big\}\in\mathcal G$. Кроме того, \eqref{PetrovIneq} с $g(u)=|u|^\d\in\mathcal G$ также влечет неравенство Нагаева--Бикялиса~\eqref{NagaevBikelisIneq} с $K_0\le A$.

Отметим также, что в 2001 Л.\,Чен и К.-М.\,Шао~\cite{ChenShao2001} передоказали неравенство~\eqref{BikelisIneq} методом Стейна, при этом авторы~\cite{ChenShao2001} ссылаются на работу Бикялиса~\cite{Bikelis1966}, цитируя только более слабое неравенство~\eqref{NagaevBikelisIneq} и ошибочно утверждая, что  результаты~\cite{Bikelis1966} имеют менее общий характер и получены при условии конечности третьих моментов.

Значение константы $A$ также долгое время оставалось неизвестным. Первые ее верхние оценки появились лишь в 2005--2007\,гг. в работах К.\,Неаммане и П.\,Тонгта~\cite{Neammanee2005,ThongthaNeammanee2007,NeammaneeThongtha2007}. Затем они были уточнены В.\,Ю.\,Королевым и С.\,В.\,Поповым в~\cite{KorolevPopov2012} до наилучших известных на данный момент значений: $A\le 47.65$ в общем случае и  $A\le 39.32$ в случае о.р.с.в.  Более того, в~\cite{KorolevPopov2012} также показано, что при $|x|\ge10$ справедливы уточненные оценки: $A\le 29.62$ в общем случае и $A\le 24.13$ в случае о.р.с.в.

Кроме того, в работах С.\,В.\,Гавриленко, Ю.\,С.\,Нефедовой и И.\,Г.\,Шевцовой~\cite{Gavrilenko2011,NefedovaShevtsova2011DAN,Shevtsova2013Inf,Shevtsova2017book} были получены структурные уточнения неравенства Нагаева--Бикялиса~\eqref{NagaevBikelisIneq} в духе~\eqref{BEineq-struct-intro}
\begin{equation}\label{BEnonuni1+x_struct}
\sup_{x\in\R}(1+|x|^{2+\d})\Delta_n(x)\le \inf_{s\ge0}K_s(\d)\big(\lyapd+sT_{2+\d,n}\big),
\end{equation}
где, как и ранее, 
$$
T_{2+\d,n}:= \frac1{B_n^{2+\d}}\sum_{j=1}^n\sigma_j^{2+\d}\ \le\ \frac1{B_n^{2+\d}}\sum_{j=1}^n\bet =\lyapd.
$$
При этом значения констант $K_s(\d)$ при $s>0$ оказываются строго меньше, чем при  $s=0$, что дает преимущество оценкам~\eqref{BEnonuni1+x_struct} по сравнению с классическим неравенством Нагаева--Бикялиса~\eqref{NagaevBikelisIneq} при больших значениях отношения $\lyapd/T_{2+\d,n}$ (которое, напомним, не может быть меньше единицы). Наилучшие известные на данный момент верхние оценки констант $K_s(\d)$ получены в~\cite{Shevtsova2013Inf,Shevtsova2017book} и приведены в таблице~\ref{TabBEnonuni1+x_struct} при некоторых $s\in[0,1]$ и $\d\in(0,1]$, где в качестве $s_1(\d)$ указаны значения $s\ge0$, минимизирующие $K_s(\d)$ (в рамках использованного метода), т.е. $K_s(\d)=K_{s_1}(\d)$ при $s>s_1(\d)$.

Что касается нижних оценок $K_0$, то, как вытекает из результатов работы Г.\,П.\,Чистякова~\cite{Chistyakov1990} 1990\,г.,
$$
K_0(1)\ge \lim_{|x|\to\infty}\limsup_{\ell\to0}\sup_{n,X_1,\ldots,X_n\colon L_{3,n}=\ell}|x|^3\Delta_n(x)/\ell =1.
$$
В 2013\,г. И.\,Пинелис~\cite{Pinelis2013} улучшил эту нижнюю оценку для $K_0(1)$ до 
$$
K_0(1)\ge \sup_{x\in\R,X_1}
\left (1+|x|^3\right )\Delta_1(x)\frac{\sigma_1^3}{\beta_{3,\,1}}>1.0135\ldots,
$$
рассматривая $n=1$, $X_1\eqd Y-p$, $Y\sim Ber(p)$, ${x=1-p}$ и $p=0.08$.

С помощью нижней оценки для $K_0(1)$ из~\eqref{K0<=A*2^(1+d)} легко получить нижнюю оценку константы $A$ в неравенствах Бикялиса~\eqref{BikelisIneqInt}, \eqref{BikelisIneqMin}, \eqref{BikelisIneq} и Петрова~\eqref{PetrovIneq}:
%$$
%A\ge\sup_{\d\in(0,1]}K_0(\d)/2^{1+\d}\ge K_0(1)/4>0.2533.
%$$
$$
A\ge\sup_{\d\in(0,1]}K_0(\d)\ge K_0(1)>1.0135.
$$
Однако можно действовать и более деликатно, аналогично получению нижней оценки для $K_0(1)$. А именно, полагая в~\eqref{PetrovIneq} 
$$
g(u)\equiv1\in\mathcal G,\ n=1,\ X_1\eqd Y-p,\ Y\sim Ber(p),\ {x=1-p},\ p=0.15,
$$
получаем следующую нижнюю оценку для $A$:
$$
A\ge\sup_{x\in\R,X_1}(1+|x|)^2\Delta_1(x)\ge 
$$
$$
\ge \sup_{p\in(0,1),\,q=1-p}\big(1+\sqrt{q/p}\,\big)^2 \abs{q-\Phi(\sqrt{q/p})} >1.6153 \quad (p=0.15).
$$
Таким образом, для константы $A$ в неравенствах~\eqref{BikelisIneqInt}, \eqref{BikelisIneqMin}, \eqref{BikelisIneq}, \eqref{PetrovIneq} справедливы следующие двусторонние оценки:
$$
1.6153<A\le 
\begin{cases}
47.65,&\text{в общем случае,}\\
39.32,&\text{в случае о.р.с.в.}
\end{cases}
$$

Аналогичные рассуждения приводят и к следующим нижним оценкам для $K_0(\d)$ в~\eqref{NagaevBikelisIneq} при произвольном $\d\in[0,1]$ (неравенство~\eqref{NagaevBikelisIneq} остается справедливым и при $\d=0$ с  $\lyapd=1$, как вытекает, например, из~\eqref{PetrovIneq} с $g(u)\equiv1$):
\begin{multline}\label{NBconst(delta)>=}
K_0(\d)\ge \sup_{x\in\R,X_1}
\left (1+|x|^{2+\d}\right)\Delta_1(x)\frac{\sigma_1^{2+\d}}{\beta_{2+\d,\,1}} \ge 
\\
\ge\sup_{0<p<1,\,q=1-p} \left(1+\Big(\frac{q}p\Big)^{1+\d/2}\right)\cdot \abs{q-\Phi(\sqrt{q/p}) }\frac{(pq)^{\d/2}}{p^{1+\d}+q^{1+\d}}= 
\\
=\sup_{0<p<1,\,q=1-p}  q^{\d/2} \cdot \frac{p^{1+\d/2}+q^{1+\d/2}}{p^{1+\d}+q^{1+\d}} \abs{1-\frac1p\Phi\left(-\sqrt{\frac{q}p}\,\right )}.
\end{multline}
Устремляя $p\to0+$ и учитывая, что $\Phi(-x)\sim x^{-1}e^{-x^2/2}/\sqrt{2\pi}$ при $x\to\infty$ (см., например,~\cite[Задача 3.228]{ZubkSevastChist1989}), получаем нижнюю оценку 
$$
K_0(\d)\ge \lim_{p\to0}\abs{1-\frac1p\Phi\left(-\sqrt{\frac{1-p}p}\,\right )} =  1-\lim_{x\to\infty}(1+x^2)\Phi(-x)=1,
$$
универсальную для всех $0\le\d\le1$. Аккуратная же оптимизация по $p\in(0,1)$ при каждом $\d\in[0,1]$ приводит к более точным нижним оценкам, указанным в таблице~\ref{TabNBconst(delta)>=} во второй строке. Наряду со значениями миноранты~\eqref{NBconst(delta)>=} таблица~\ref{TabNBconst(delta)>=}  также  содержит в третьей строке  близкие к экстремальным значения параметра~$p$ в~\eqref{NBconst(delta)>=}, обеспечивающие справедливость объявленных нижних оценок.

\begin{table}[h]
\centering
\begin{tabular}{||c||c|c|c|c|c|c|c|c||}
\hline
$\d$&0&0.1&0.3&0.5&0.7&0.9&1
\\\hline
$K_0\ge$ &1&1.0061&1.0139&1.0167&1.0164&1.0147&1.0135
\\\hline
$p$ &0+&0.06&0.07&0.076&0.08&0.08&0.08
\\\hline
\end{tabular}
\caption{Нижние оценки констант $K_0(\delta)$ из
неравенства~\eqref{NagaevBikelisIneq}, построенные по формуле~\eqref{NBconst(delta)>=}, для некоторых 
${\delta\in[0,1]}$.}
\label{TabNBconst(delta)>=}
\end{table}

\section[Оценки $\zeta$-метрик]{Оценки скорости сходимости в ЦПТ в $\zeta$-метриках}

\subsection{Оценки $\zeta_1$-метрики}

Перейдем теперь к оценкам расстояний между допредельным распределением нормированной суммы $\widetilde S_n=(X_1+\ldots+X_n)/B_n$ с ф.р. $\overline F_n$ и предельным нормальным распределением в $\zeta_s$-метриках. Для удобства обозначим через $Z$ стандартную нормальную с.в. с ф.р. $\Phi$ и плотностью $\varphi(x)=\Phi'(x)$. Сначала рассмотрим случай $s=1$.

В 1967\,г. Студневым и Игнатом~\cite{StudnevIgnat1967} и немного позже, в 1973\,г., Эриксоном~\cite{Erickson1973}, не цитирующим работу~\cite{StudnevIgnat1967}, был доказан аналог неравенства Осипова для средней метрики:
\begin{equation}\label{IgnatStudnevErickson}
\zeta_1\big(\widetilde S_n,Z\big)\le C\int_0^1L_n(z)\dd z=\frac{C}{B_n^3}\sum_{k=1}^n\E X_k^2\min\{|X_k|,B_n\},
\end{equation}
где $C$~--- некоторая абсолютная константа, значение которой не указано в~\cite{StudnevIgnat1967}, и $C\le72$ в~\cite{Erickson1973} (причем в этой же работе намечены идеи, как снизить оценку $C$ до $36$).

Используя лемму~\ref{LemKatzPetrovGproperties}, из~\eqref{IgnatStudnevErickson} легко получить следующие аналоги неравенств Каца--Петрова и Берри--Эссеена для метрики Канторовича:
\begin{equation}\label{KatzPetrovForZeta1}
\zeta_1\big(\widetilde S_n,Z\big)\le \frac{C}{B_n^2g(B_n)}\sum_{k=1}^n\E X_k^2g(X_k),\quad g\in\mathcal G,
\end{equation}
\begin{equation}\label{BE(delta)ForZeta1}
\zeta_1\big(\widetilde S_n,Z\big)\le C\cdot \lyapd= \frac{C}{B_n^{2+\d}}\sum_{k=1}^n\E|X_k|^{2+\d},\quad \d\in[0,1],
\end{equation}
где $\mathcal G$~--- класс всех неотрицательных чётных функций~$g$, определённых на $\R$, таких что $g(x)>0$ при $x>0$, и функции $g(x)$, $x/g(x)$ не убывают в области $x>0$, а $\lyapd$~--- дробь Ляпунова порядка $2+\d$ при $0<\d\le1$ и $L_{2,n}:=1$ (в предположении, что соответствующие моменты конечны).

Значения универсальной константы $C$ в~\eqref{BE(delta)ForZeta1} можно улучшить, если учесть, что в~\eqref{BE(delta)ForZeta1} стоит не произвольная функция  $g\in\mathcal G$, а степенная: $g(u)=|u|^\d$, и если позволить ей также зависеть от конкретного значения $\delta$: $C=C(\d)$. В таком случае для $C(\d)$ известны следующие верхние оценки:
\begin{equation}\label{BEForZeta1const(delta)<=}
C(\d)\le \frac2{1+\d}\Big(\frac\pi2\Big)^{(1-\d)/2},\quad \d\in[0,\,1],
\end{equation}
в частности, $C(1)\le1,$ $C(0)\le\sqrt{2\pi}<2.51$, значения для других $\d$ см. в таблице~\ref{TabBEForZeta1const(delta)>=} во второй строке. Данные оценки были получены при $\d=1$ одновременно Голдстейном и Тюриным в 2009\,г. в~\cite{Goldstein2009,Tyurin2009DAN,Tyurin2010TVP}, а при $\d<1$ приводятся впервые.

Что касается нижних оценок  $C(\d)$, первая такая оценка была найдена для $\d=1$ в 1964\,г. Золотаревым в~\cite{Zolotarev1964}, который рассмотрел симметричную схему Бернулли $X_k\eqd Y-1/2,$ $Y\sim Ber(1/2),$ $n\to\infty$, и, используя асимптотическое разложение Эссеена~\cite{Esseen1945}, доказал, что
$$
C(1)\ge \sup_{X_1\eqd\ldots\eqd X_n}\limsup_{n\to\infty} \sqrt{n}\norm{\overline F_n-\Phi}_1 \frac{\sigma_1^3}{\beta_{3,1}}\ge \frac12
$$
(напомним, что $\zeta_1(F,G)=\norm{F-G}_1$). Здесь точная верхняя грань берется по всем одинаковым распределениям случайных слагаемых $X_1,\ldots,X_n$ с конечным третьим моментом (схема серий не допускается) и оценивается снизу на конкретном (симметричном двухточечном) распределении из данного класса.

В 2009\,г. Голдстейн~\cite{Goldstein2009} дополнил результат Золотарева, рассмотрев то же распределение при $n=1$, в результате чего оказалось, что
$$
C(1)\ge \norm{F_1-\Phi}_1\frac{\sigma_1^3}{\beta_{3,1}}\ge C_{\textsc{g}}:=4(\Phi(1)+\varphi(1))-\sqrt{\frac2\pi}-3=0.5353\ldots
$$

Идея Голдстейна легко обобщается на случай произвольного ${\d\in[0,1]}$ (при $\d=0$ полагаем $\lyapd:=1$). Действительно, для двухточечных н.о.р.с.в.
\begin{equation}\label{standard2point-distr}
X_k=
\begin{cases}
-\sqrt{p/q},&q:=1-p\in(0,1),\\
\sqrt{q/p},&p,
\end{cases}
\quad k=1,\ldots,n,
\end{equation}
и стандартной нормальной с.в. $Z$ имеем 
$$
\E X_1=0,\quad \sigma_1^2=\E X_k^2=1,\quad \beta_{2+\d,\,1}=\E|X_1|^{2+\d}=\frac{p^{1+\d}+q^{1+\d}}{(pq)^{\d/2}},
$$
$$
\zeta_1(X_1,Z)=\norm{F_1-\Phi}_1= 2\sqrt{\frac qp} \Phi\left(\sqrt{\frac qp}\right) +2\sqrt{\frac pq}\Phi\left(\sqrt{\frac pq}\right)+
$$
$$
+ 2\varphi\left(\sqrt{\frac qp}\right) + 2\varphi\left(\sqrt{\frac pq}\right) - 2\varphi\left(\Phi^{-1}(q)\right)-2\,\frac{1-pq}{\sqrt{pq}}=:\psi(p),
$$
и следовательно, для любого $\d\in[0,1]$
\begin{equation}\label{BEForZeta1const(delta)>=}
C(\d)\ge \zeta_1(X_1,Z)\frac{\sigma_1^{2+\d}}{\beta_{2+\d,1}} \ge \sup_{p\in(0,1),\,q=1-p}\frac{(pq)^{\d/2}}{p^{1+\d}+q^{1+\d}}\,\psi(p).
\end{equation}
В частности, при $p=q=1/2$ получаем $\beta_{2+\d,\,1}=1$ для всех $\d\in[0,1]$ и 
$$
C(\d)\ge\psi(1/2)=C_{\textsc{g}}>0.5353,
$$
т.е. константа Голдстейна является универсальной нижней оценкой $C(\d)$ при всех $\d\in[0,1]$, при этом более аккуратная оптимизация по $p\in(0,1)$ (на самом деле достаточно рассмотреть лишь один из интервалов $p\in(0,1/2]$ или $p\in[1/2,1)$ в силу симметрии) при $\d\in[0.5,1]$ не находит ничего лучшего, чем значение $p=1/2$, а при ${\d\in(0,0.5)}$ приводит к оптимальным значениям $p$ отличающимся от~$1/2$, что позволяет повысить полученную универсальную нижнюю оценку. Результаты численной оптимизации по $p\in(0,1/2]$ при некоторых $\d\in[0,0.5]$ приведены в таблице~\ref{TabBEForZeta1const(delta)>=}.

\begin{table}[h]
\centering
\begin{tabular}{||c||c|c|c|c|c|c|c||}
\hline
$\d\vphantom{\frac{1^1}{1^1}}$&0&0.1&0.2&0.3&0.4&$\in[0.5,\,1]$
\\\hline
$C(\d)\le\vphantom{\frac{1^1}{1^1}}$ &2.5067&2.2279&1.9967&1.8019&1.6359&1.4927
\\\hline
$C(\d)>\vphantom{\frac{1^1}{1^1}}$ &0.8479&0.7162&0.6328&0.5760&0.5412&0.5353
\\\hline
$p\vphantom{\frac{1^1}{1^1}}$ &0.013&0.046&0.082&0.13&0.23&0.5
\\\hline
\end{tabular}
\caption{Верхние и нижние оценки констант $C=C(\delta)$ в
неравенстве~\eqref{BE(delta)ForZeta1}, построенные соответственно по формулам~\eqref{BEForZeta1const(delta)>=} и~\eqref{BEForZeta1const(delta)<=}, для некоторых ${\delta\in[0,1]}$.}
\label{TabBEForZeta1const(delta)>=}
\end{table}

Поскольку константы $C(\d)$ в~\eqref{BE(delta)ForZeta1} и $C$ в неравенстве Студнева--Игната--Эриксона~\eqref{IgnatStudnevErickson} связаны неравенством
$$
\max_{0\le\d\le1}C(\d)\le C
$$
(в силу того, что $g(u)=u^{\d}\in\mathcal G$ при любом $\d\in[0,1]$), найденные нижние оценки для $C(\d)$ остаются справедливыми и для~$C$, в частности, $C(0)>0.8479$ влечет $C>0.8479.$ С другой стороны, заметим, что так как $C(0)\ge\psi(p)=\zeta_1(X_1,Z)$ при $p=0.013$ и $X_1+p\sim Ber(p)$, то фактически мы получили миноранту для $C$ в~\eqref{KatzPetrovForZeta1} c функцией $g(u)\equiv1$. Естественно ожидать, что эту миноранту можно уточнить, если рассмотреть функцию $g(u)=\min\{|u|,B_n\}$, минимизирующую правую часть~\eqref{KatzPetrovForZeta1}, которая в этом случае совпадает как раз с правой частью исследуемого неравенства~\eqref{IgnatStudnevErickson}. Действительно, для схемы суммирования н.о.р.с.в. с распределением~\eqref{standard2point-distr} имеем $B_n=\sqrt n$,
$$
\E X_1^2\min\{|X_1|,\sqrt n\}=
\begin{cases}
\E|X|^3=(p^2+q^2)/\sqrt{pq},& np>q,
\\
q\sqrt{n}+p^{3/2}q^{-1/2},&np\le q,
\end{cases}
$$
$$
C\ge \sup_{n\in\N,\,p\in(0,1/2]} \frac{\sqrt{n}\zeta_1\big(\widetilde S_n,Z\big)}{\E X_1^2\min\{|X_1|,\sqrt n\}} 
\ge \sup_{p\in(0,1/2],\,q=1-p} \frac{\zeta_1(X_1,Z)}{q+p^{3/2}q^{-1/2}}=
$$
$$
=\sup_{p\in(0,1/2],\,q=1-p} \frac{\sqrt q\,\psi(p)}{q^{3/2}+p^{3/2}}>0.8614\quad (p=0.024).
$$
Таким образом, мы действительно уточнили нижнюю оценку для $C$, однако она все равно далеко отстоит от верхней:
$$
0.8614\le C\le36.
$$
По-видимому, используя современные методы оптимизации абсолютных констант в неравенствах типа Берри--Эссеена, эту верхнюю оценку можно существенно уточнить.

\subsection{Оценки $\zeta$-метрик более высокого порядка}

В рассматриваемой нами ЦПТ предполагается конечность моментов порядка $2+\d$ с $\d\in[0,1]$, поэтому имеет смысл рассматривать расстояния между $\widetilde S_n$ и $Z$ в $\zeta_s$-метриках и более высокого порядка вплоть до $s\le2+\d$.

Аналоги неравенств Берри--Эссеена для $\zeta_s$-метрик при ${s\in[1,2+\d]}$ известны давно, приводим сразу неравенства с наилучшими известными на данный момент значениями констант:
$$
\zeta_2\big(\widetilde S_n,Z\big)\le \frac{\sqrt{2\pi}}{8}L_{3,n}\quad(\d=1),
$$
\begin{equation}\label{zeta2+dTyurin}
\zeta_{2+\d}\big(\widetilde S_n,Z\big)\le \frac{\lyapd}{(1+\d)(2+\d)},\quad \d\in(0,1].
\end{equation}
Первая оценка получена Тюриным в~\cite{Tyurin2009DAN,Tyurin2010TVP}, вторая~--- там же при $\d=1$ и в~\cite{Tyurin2012} в общем случае. Интересно отметить, что мультипликативные константы $((1+\d)(2+\d))^{-1}$ в оценке для $\zeta_{2+\d}$ являются оптимальными. Чтобы в этом убедиться, достаточно, как и ранее, рассмотреть стандартное двухточечное распределение~\eqref{standard2point-distr} и $n=1$, а также вспомнить, что
$$
\zeta_{2+\d}\big(\widetilde S_n,Z\big)\ge \frac{\Gamma(1+\delta)}{\Gamma(3+\d)}\abs{\E|X_1|^{2+\d}-\E|Z|^{2+\d}}.
$$
Это позволяет заключить, что коэффициент при ляпуновской дроби $\lyapd$ в~\eqref{zeta2+dTyurin} не может быть меньше, чем:
$$
\sup_{n\ge1,\,X_1\ldots,X_n}\frac{\zeta_{2+\d}\big(\widetilde S_n,Z\big)}{\lyapd} \ge \frac{\zeta_{2+\d}(X_1,Z)}{\E|X_1|^{2+\d}} \ge
$$
$$
\ge \frac{\Gamma(1+\delta)}{\Gamma(3+\d)}\sup_{X_1} \frac{\abs{\E|X_1|^{2+\d}-\E|Z|^{2+\d}}}{\E|X_1|^{2+\d}} \ge
$$
$$
\ge \frac{\Gamma(1+\delta)}{\Gamma(3+\d)}\lim_{p\to0} \abs{1-\frac{(pq)^{\d/2}\E|Z|^{2+\d}}{p^{1+\d}+q^{1+\d}}} =\frac{1}{(1+\d)(2+\d)}
$$
для всех $\d\in(0,1]$. В этом смысле неравенство~\eqref{zeta2+dTyurin} оптимально. Однако его можно улучшить!

Идея заключается в том, чтобы заменить \textit{постоянный} коэффициент при ляпуновской дроби ограниченной функцией от моментов распределений случайных слагаемых, строго меньшей постоянного коэффициента. Эта идея была реализована в недавней работе~\cite{MattnerShevtsova2019} для случая $\d=1$, в том числе за счет добавления в правую часть дополнительного слагаемого, убывающего быстрее. Чтобы сформулировать этот результат, введем следующие обозначения. Для $\rho\ge1$ обозначим $\xi_\rho$ стандартизованную двухточечную с.в. вида~\eqref{standard2point-distr} с 
$$
\E|\xi_\rho|^3=\rho,\quad \E\xi_\rho^3\ge0.
$$
При этих дополнительных условиях распределение $\xi_\rho$ уже определяется однозначно, а именно, в качестве $p$ в~\eqref{standard2point-distr} нужно взять 
$$
p=\frac12-\frac12\sqrt{\frac\rho2\sqrt{\rho^2+8}-\frac{\rho^2}2-1}.
$$
Обозначим 
$$
A(\rho)=\frac{\E\xi_\rho^3}{\E|\xi_\rho|^3} = \frac{\E\xi_\rho^3}{\rho} = \sqrt{\frac12\sqrt{1+8\rho^{-2}} +\frac12-2\rho^{-2} }.
$$
Несложно убедиться, что для всех $\rho\ge1$ функция $A(\rho)$ строго монотонно возрастает,
$$
A(\rho)<1,\quad \lim_{\rho\to\infty}A(\rho)=1,\quad
A(\rho)\sim\sqrt{\frac83(\rho-1)}\to0,\quad \rho\to1+.
$$
Автор в~\cite{Shevtsova2014JMAA} доказала, что для любой с.в. $X$ с $\E X=0,$ $\E X^2=1$ и $\E|X|^3=\rho\ge1$ справедливо неравенство
$$
|\E X^3|\le A(\rho)\E|X|^3,
$$
причем равенство достигается тогда и только тогда, когда $X\eqd\xi_\rho$ или $X\eqd-\xi_\rho$, т.е. $A(\rho)$~--- это максимальное значение отношения третьего алгебраического момента к абсолютному для центрированных распределений. Далее, перенумеруем слагаемые так, чтобы $\sigma_1\ge\sigma_2\ge\ldots\ge\sigma_n$. В этих обозначениях в~\cite{MattnerShevtsova2019} доказано следующее неравенство:
\begin{multline}\label{zeta3MaSheNormalAppr}
\zeta_3\big(\widetilde S_n,Z\big)\ \le\ \frac16\sum_{k=1}^n\frac{\beta_{3,k}}{B_n^3} A\left(\frac{\beta_{3,\,k}}{\sigma_k^3}\right)+
\\
+ 0.0993\frac{\sigma_1^3}{B_n^3}+ \frac{0.0665}{B_n^3}\sum_{k=1}^{n-1}\frac{\sigma_{k+1}^3}{\sqrt{k}}.
\end{multline}
В случае о.р.с.в. с $\rho:=\beta_{3,1}/\sigma_1^3$ неравенство~\eqref{zeta3MaSheNormalAppr} упрощается до
\begin{equation}\label{zeta3MaSheNormalApprIID}
\zeta_3\big(\widetilde S_n,Z\big)\ \le\ \frac{\rho A(\rho)}{6\sqrt{n}} +\frac{0.1352}{n}, %= \frac{\E\xi_\rho^3}{6\sqrt{n}} +\frac{0.1352}{n},
\end{equation}
т.е. по сравнению с~\eqref{zeta2+dTyurin} при $\d=1$ и для о.р.с.в.
\begin{equation}\label{zeta3TyurinIID}
\zeta_3\big(\widetilde S_n,Z\big)\ \le\ \frac{\rho}{6\sqrt{n}}
\end{equation}
появился множитель $A(\rho)<1$, который является бесконечно малым при $\rho\to1$, или, в терминах моментов, абсолютный третий момент $\rho$ заменен алгебраическим $\rho A(\rho)=\E\xi_\rho^3$ экстремального распределения, но появилось дополнительное слагаемое порядка $\mathcal O(n^{-1})$, имеющее более высокий порядок малости при $n\to\infty$. Несложно убедиться, что неравенство~\eqref{zeta3MaSheNormalApprIID} точнее, чем~\eqref{zeta3TyurinIID}, при 
$$
n\ge\frac{0.66}{\rho^2(1-A(\rho))^2},
$$
в частности, для $1\le\rho\le1.015$~--- при всех $n\in\N$.

%\chapter{Теорема Пуассона}
%\thispagestyle{empty}
%\input{poisson_thm}

\chapter{Аппроксимация распределений случайных сумм}
\thispagestyle{empty}

\section[Теорема Пуассона]{Теорема Пуассона. Метод одного вероятностного пространства}

\begin{definition}
Пусть $\xi_1,\ldots,\xi_n$~--- независимые с.в. на некотором в.п. $(\Omega,\mathcal A,\Prob)$, с распределениями $\xi_i\sim Ber(p_i)$:
$$
\Prob(\xi_i=1)=p_i=1-\Prob(\xi_i=0)\in(0,1],\quad i=\overline{1,n}.
$$
Тогда распределение с.в.
$$
\Bin:\eqd \xi_1+\ldots\xi_n
$$
называется \textit{пуассон-биномиальным} с параметрами $n\in\N$ и $\pr\hm\coloneqq(p_1,\ldots,p_n)$. Обозначение: $\Bin\sim PB(n,\pr).$
\end{definition}

В частности, если $p_1=\ldots=p_n\eqqcolon p,$ то $\Bin\sim Bi(n,p).$ Если при этом $p=1$, то $\Bin\eqp n.$ Из базового курса теории вероятностей хорошо известно, что если $n$ велико, а $p$ мало, то  распределение $\Bin$ близко к распределению Пуассона с параметром $np$. Нижеследующая теорема формализует и обобщает этот факт на случай различающихся $p_1,\ldots,p_n$, а также устанавливает оценку скорости сближения пуассон-биномиального распределения с пуассоновским.

Обозначим $N_\la$ случайную величину, имеющую пуассоновское распределение с параметром $\lambda>0$ и положим
$$
\lambda\coloneqq \sum_{i=1}^n p_i.
$$
%Напомним, что функция
%$$
%\tvd(\xi,\eta)\coloneqq\sup_{B\in\mathcal B}\abs{\Prob(\xi\in B)-\Prob(\eta\in B)}
%$$
%от распределений с.в. $\xi$ и $\eta$ называется метрикой полной вариации

\begin{theorem}\label{ThPoissona}
В сделанных выше предположениях для любого $n\in\N$ %для любого $B\in\mathcal B$
\begin{equation}\label{|poisbin-pois|<=np^2}
\tvd(\Bin, N_\la)\coloneqq \sup_{B\in\mathcal B}\abs{\Prob(\Bin\in B)-\Prob(N_\la\in B)}\le \sum_{i=1}^np_i^2.
\end{equation}
\end{theorem}

\begin{proof} Будем использовать \textit{метод одного (общего) вероятностного пространства}. Для этого мы специальным образом определим случайные величины $\{\xi_i\}_{i=1}^n$ и $N_\la$ на \textit{общем} в.п., которое, вообще говоря, отличается от уже имеющегося $(\Omega,\mathcal A,\Prob)$, но для которого мы будем использовать то же обозначение. Итак, пусть
$$
\Omega=[0,1]^n=\{(\omega_1,\ldots,\omega_n)\colon \omega_i\in[0,1]\},\quad \mathcal A=\mathcal B(\Omega),\quad \Prob=\lambda,
$$
$$
\xi_i\coloneqq
\begin{cases}
0,&0\le\omega_i\le1-p_i,
\\
1,&1-p_i<\omega_i\le1,
\end{cases}
\quad
\xi_i^*\coloneqq
\begin{cases}
0,&0\le\omega_i\le e^{-p_i}\eqqcolon\pi_{i,0},
\\
k\in\N,&\pi_{i,k-1}<\omega_i\le\pi_{i,k},
\end{cases}
$$
где 
$$
\pi_{i,k}\coloneqq e^{-p_i}\sum_{m=0}^k\frac{p_i^m}{m!}=\Prob(\xi_i^*\le k),\quad i=\overline{1,n},\quad k=0,1,2,\ldots,
$$
так что $\xi_i\sim Ber(p_i),$ $i=\overline{1,n}$, имеют искомые распределения и являются независимыми, а $\xi_i^*\sim Pois(p_i)$, $i=\overline{1,n}$, и также являются независимыми. Следовательно, можно положить
$$
N_\la\coloneqq\sum_{i=1}^n\xi_i^*\sim Pois(\la),\ \text{ где }\ \la=\sum_{i=1}^np_i.
$$
Поскольку $1-p\le e^{-p}$, $p\in[0,1]$, и 
$$
\{\xi_i\neq\xi_i^*,\,\xi_i=0\}= \{\xi_i^*\ge1,\, \xi_i=0\}=\emptyset,
$$
имеем
$$
\{\xi_i\neq\xi_i^*\} =\{\xi_i\neq\xi_i^*,\xi_i=0\}\cup \{\xi_i\neq\xi_i^*,\xi_i=1\}=
$$
$$
=\emptyset \cup \{\xi_i^*=0,\xi_i=1\}\cup \{\xi_i^*\ge2,\xi_i=1\} = \{\xi_i^*=0,\xi_i=1\} \cup \{\xi_i^*\ge2\}.
$$
Следовательно,
$$
\Prob(\xi_i\neq\xi_i^*)=\Prob(\{\xi_i^*=0,\xi_i=1\}\cup\{\xi_i^*\ge2\})=
$$
$$
=\lambda\left( (1-p_i,e^{-p_i}] \cup (\pi_{i,1},1]\right)=p_i(1-e^{-p_i})\le p_i^2,\quad i=\overline{1,n},
$$
а значит,
$$
\Prob(\Bin\neq N_\la)\le \Prob\bigg(\bigcup_{i=1}^n\{\xi_i\neq\xi_i^*\}\bigg)\le \sum_{i=1}^n\Prob\left(\xi\neq\xi_i^*\right)\le  \sum_{i=1}^n p_i^2.
$$
Теперь для доказательства теоремы осталось заметить, что для любого $B\in\mathcal B(\R)$
$$
\Prob(\Bin\in B)-\Prob(N_\la\in B)= 
$$
$$
=\Prob(\Bin\in B,\Bin=N_\la) +\Prob(\Bin\in B,\Bin\neq N_\la) -
$$
$$
-\Prob(N_\la\in B,N_\la=\Bin)-\Prob(N_\la\in B,N_\la\neq \Bin)=
$$
$$
=\Prob(\Bin\in B,\Bin\neq N_\la)-\Prob(N_\la\in B,N_\la\neq \Bin),
$$
и следовательно,
$$
\abs{\Prob(\Bin\in B)-\Prob(N_\la\in B)}\le
$$
$$
\le\max\{\Prob(\Bin\in B,\Bin\neq N_\la), \Prob(N_\la\in B,N_\la\neq \Bin)\}\le
$$
$$
\le\Prob(\Bin\neq N_\la)\le  \sum_{i=1}^n p_i^2.\qquad\qedhere
$$
\end{proof}
% !TeX root = book.tex

\section{Случайные суммы}

\begin{definition}
Пусть $N,X_1,X_2,\ldots$ --- независимые с.в. на некотором в.п. $(\Omega,\mathcal F,\Prob),$ такие что $X_1,X_2,\ldots$ распределены одинаково, а $N$ принимает целые неотрицательные значения, т.е. ${\Prob(N\in\N_0)=1}$. Тогда с.в.
$$
S_N(\omega)\coloneqq \sum_{k=1}^{N(\omega)}X_k(\omega),\quad \omega\in\Omega,
$$
где по определению полагаем $\sum_{k=1}^0(\,\cdot\,)\coloneqq0,$ называется \textit{случайной суммой}.
\end{definition}

Для удобства обозначений бывает полезно ввести ещё одну независимую от $N$ с.в. $X$, имеющую одинаковое с $X_1,\ldots,X_n$ распределение (так называемое ``случайное слагаемое''). Обозначим ф.р. и х.ф. $X$ через
$$
F(x)=\Prob(X<x),\quad x\in\R,\qquad f(t)=\E e^{it X},\quad t\in\R,
$$
соответственно, и пусть
$$
p_n=\Prob(N=n),\ n\in\N_0,\quad \psi(z)\coloneqq\E z^N=\sum_{n=0}^\infty p_nz^n,\ z\in\C,\ |z|\le1.
$$

\begin{theorem}[Элементарные свойства случайных сумм]\label{ThRandSumElemProp}
\mbox{}\\
{\rm (i)} $\displaystyle F_{S_N}(x)\coloneqq\Prob(S_N<x)=\sum_{n=0}^\infty p_nF^{*n}(x),$ $x\in\R.$
\\
{\rm (ii)} $\E S_N=\E N\cdot\E X,$ $\D S_N=\E N\cdot\D X+\D N\cdot(\E X)^2,$
если соответствующие моменты с.в. $N$ и $X$ конечны.
\\[2mm]
{\rm (iii)} $f_{S_N}(t)\coloneqq \E e^{itS_N}=\psi(f(t)),$ $t\in\R.$
\\[2mm]
{\rm (iv)} Если с.в. $X$ целочисленная неотрицательная с пр.ф. $\phi(z)\hm\coloneqq\E z^{X},$ то пр.ф. случайной суммы $S_N$ имеет вид
$$
\E z^{S_N}=\psi(\phi(z)),\quad  z\in\C,\ |z|\le1.
$$
\end{theorem}

Несложное доказательство данной теоремы, например, основанное на представлении $g(S_N)\eqp\sum_{n=0}^\infty g(S_n)\I(N=n)$ для любой борелевской функции $g$, предлагается в качестве упражнения для читателя.

\section{Биномиальные и пуассон-би\-но\-ми\-аль\-ные случайные суммы}
%\setrightmark{Биномиальные и пуассон-биномиальные случайные суммы}

\begin{definition}
Пусть $X_1,X_2,\ldots$~--- н.о.р.с.в, $\Bin\sim PB(n,\pr)$, $n\in\N,$ $\pr=(p_1,\ldots,p_n),$ $p_k\in(0,1],$ $k=\overline{1,n}$, и $\Bin,X_1,\ldots,X_n$ независимы при каждом $n$ и  $\pr.$ Тогда с.в.
$$
\Sbin(\omega)\coloneqq \sum\limits_{k=1}^{\Bin(\omega)}X_k(\omega),\quad \omega\in\Omega,
$$
называется \textit{пуассон-биномиальной случайной суммой}, а её распределение --- \textit{обобщённым\footnote{В англоязычной литературе принят термин ``compound''}  (сложным, составным) пуассон-биномиальным}.
\end{definition}

\begin{theorem}
В сделанных выше предположениях
\begin{equation}\label{Spbin=sum_k=1^nXkIk}
\Sbin\eqd \smallsum_{k=1}^n\xi_k X_k,
\end{equation}
где $\xi_1,\ldots,\xi_n,X_1,\ldots,X_n$ независимы и $\xi_k\sim Ber(p_k),$ $k=\overline{1,n}.$
\end{theorem}

\begin{proof}
По теореме~\ref{ThRandSumElemProp} (iii) для х.ф. с.в. $\Sbin$ имеем 
$$
f_{\Sbin}(t) \coloneqq \E e^{it\Sbin}=\E(f(t))^{\Bin},\quad t\in\R.
$$
C другой стороны, х.ф. суммы $\sum_{k=1}^n\xi_k X_k$ в правой части~\eqref{Spbin=sum_k=1^nXkIk} в силу независимости всех фигурирующих в ней с.в. имеет вид
$$
\E e^{it\sum_{k=1}^n\xi_kX_k}= \prod_{k=1}^n\E e^{it\xi_kX_k} =\prod_{k=1}^n(p_kf(t)+1-p_k) = \prod_{k=1}^n\E(f(t))^{\xi_k}=
$$
$$
=\E\prod_{k=1}^n(f(t))^{\xi_k}=\E(f(t))^{\sum_{k=1}^n\xi_k}=\E(f(t))^{\Bin},\quad t\in\R,
$$
и, как теперь видно, совпадает с х.ф. $f_{\Sbin}(t)$ с.в. $\Sbin$. Доказательство теоремы завершает ссылка на теорему единственности.
\end{proof}

Приведём один пример из области страхования, где возникают пуассон-биномиальные суммы. Рассмотрим портфель, состоящий из $n$ страховых полисов, и пусть $X_k$~--- (случайный) объем страховых выплат по $k$-му полису, $X_k\stackrel{\text{п.н.}}{>}0$, а $\xi_k$~--- индикатор наступления какого-либо страхового случая, $k=\overline{1,n}$. Тогда $\Sbin$ есть суммарный объем страховых выплат страховой компанией по данному портфелю. При этом предположение об одинаковой распределённости с.в. $\{X_k\}_{k=1}^n$, если речь идет, например, об автостраховании, означает, что в одном портфеле страхуются автомобили-``одноклассники'', т.е. автомобили примерно одинаковой стоимости, но в силу различающихся условий использования автомобилей, застрахованных в данном портфеле, как-то: хранение в гараже или на улице, различный опыт и стаж водителей, вероятности наступления страховых событий $p_1,\ldots,p_n$, вообще говоря, разные для разных полисов, т.е. бернуллиевские с.в. $\{\xi_k\}_{k=1}^n$ могут иметь разные распределения. 

Как правило, вероятности наступления страховых событий $p_1,\ldots,p_n$ малы, а объем страхового портфеля $n$ велик. Нас может интересовать вероятность $\Prob(\Sbin>u)$ того, что страховые выплаты превысят некоторой порог $u\ge0$, в частности, вероятность ``разорения'' страховой компании $\Prob\left(\Sbin>\sum_{k=1}^nb_k\right)$, где $b_k$~--- стоимость $k$-го страхового полиса, $k=\overline{1,n}$. Как оценить эти вероятности при больших объёмах портфеля $n$? В контексте нашего курса ``оценить'' значит записать предельную теорему с оценкой скорости сходимости.

Предположим, что $0<\beta_{2+\d}\coloneqq\E|X|^{2+\d}<\infty$ при некотором $0<\d\le1$ (теперь $X,X_1,\ldots,X_n$ не обязательно положительные, но невырожденные в нуле) и обозначим
$$
a\coloneqq \E X,\quad \beta_2\coloneqq\E|X|^2,\quad \sigma^2\coloneqq\D X=\beta_2-a^2,
$$
$$
\la\coloneqq \smallsum_{k=1}^np_k,\quad \la_2\coloneqq\smallsum_{k=1}^np_k^2,\quad \theta\coloneqq\frac{\la_2}{\la}.
$$
Тогда  $\theta\le\max\limits_{1\le k\le n}p_k\le1$ и $\theta=p,$ если $p_1=\dots=p_n=p$. В связи с пуассон-биномиальным распределением иногда нам будет удобно использовать ``вспомогательную'' дискретную с.в. $\eta$, распределение которой задаётся формулами 
$$
\Prob(\eta=p_k)=\frac{p_k}{\la},\quad k=1,\ldots,n, \quad \text{c}\quad \E\eta=\theta.
$$

Для моментов имеем
$$
\E\Bin=\la,\quad \D\Bin=\smallsum_{k=1}^np_k(1-p_k)=\la-\la_2,
$$
а также, в силу теоремы~\ref{ThRandSumElemProp}~(ii) или представления~\eqref{Spbin=sum_k=1^nXkIk},
$$
\E\Sbin=\la a,\quad \D\Sbin=\la(\beta_2-a^2)+(\la-\la_2)a^2=\la(\beta_2-a^2\theta).
$$
Положим
$$
\widetilde\Sbin\coloneqq \frac{\Sbin-\E\Sbin}{\sqrt{\D\Sbin}}=\frac{\Sbin-\la a} {\sqrt{\la(\beta_2-a^2\theta)}}.
$$

Напомним неравенство Берри--Эссеена~\ref{ThBEineq}, а также его структурные уточнения. Для случайных сумм эти уточнения оказываются как нельзя кстати. Значения констант приводим наилучшие известные в мире, а не те, которые смогли получить в рамках данного курса.

\begin{lemma}\label{LemBEineq->RS}
Если $Y_1,\ldots,Y_n$ --- независимые с.в. с $\E Y_k=0,$ $B_n^2\hm\coloneqq\sum_{k=1}^n\E Y_k^2>0$ и $\E|Y_k|^{2+\d}<\infty$ при некотором $\d\in(0,1]$ для всех $k=\overline{1,n},$ то для всех $n\in\N$
$$
\Delta_n\coloneqq \sup_{x\in\R}\abs{\Prob\left(\smallsum_{k=1}^nY_k<xB_n\right)-\Phi(x)}\le \min_{s\ge0}C_s(\d)(\lyapd+sT_{2+\d,n}),
$$
где 
$$
\lyapd = \frac1{B_n^{2+\d}}\sum_{k=1}^n\E|Y_k|^{2+\d},\quad T_{2+\d,n} = \frac1{B_n^{2+\d}}\sum_{k=1}^n(\E Y_k^2)^{1+\d/2},
$$
константы $C_s(\d)$ зависят только от $\d\in(0,1]$ и $s\ge0$, в частности, при $\d=1$
$$
0.4097\ldots\le \frac{\sqrt{10}+3}{6\sqrt{2\pi}}\le C_0\le\begin{cases}
0.5583&\text{в общем случае},\\
0.4690&\text{в случае о.р.с.в.},
\end{cases}
$$
$C_1\le0.3057$ в общем случае и  $C_{0.646}\le0.3031$ в случае о.р.с.в.
\end{lemma}

Заметим, что в силу представления~\eqref{Spbin=sum_k=1^nXkIk}
$$
\widetilde\Sbin\eqd \sum_{k=1}^n\frac{\xi_kX_k- ap_k} {\sqrt{\la(\beta_2-a^2\theta)}}\eqqcolon\frac1{B_n}\sum_{k=1}^nY_k,
$$
где  $Y_k\coloneqq\xi_kX_k-ap_k$ и
$$
B_n^2\coloneqq \la(\beta_2-a^2\theta) =\sum_{k=1}^np_k(\beta_2-a^2p_k) =\sum_{k=1}^n\E Y_k^2>0,
$$
при этом $\E Y_k=0$ и, в силу леммы~\ref{LemE|X-a|^3<E|X|^3+3.25|a|},
$$
\E|Y_k|^{2+\d}=\E|\xi_kX_k-\E\xi_kX_k|^{2+\d}\le \E|\xi_kX_k|^{2+\d} +3.25\E(\xi_kX_k)^2|\E\xi_kX_k|^\d=
$$
$$
=p_k\beta_{2+\d}+3.25|a|^\d\beta_2p_k^{1+\d}\le \beta_{2+\d}p_k(1+3.25p_k^\d)
<\infty, \quad k=\overline{1,n}.
$$
Таким образом, набор с.в. $\{Y_k\}_{k=1}^n$ удовлетворяет условиям леммы~\ref{LemBEineq->RS} c
%$$
%\lyapd \le \frac{\beta_{2+\d}(\la+3.25\sum_{k=1}^np_k^{1+\d})} {(\la(\beta_2-a^2\theta))^{1+\d/2}}=  \frac{\beta_{2+\d}(1+3.25\E\eta^\d)} {\la^{\d/2}(\beta_2-a^2\theta)^{1+\d/2}}\le
%$$
$$
\lyapd \le \sum_{k=1}^n \frac{\beta_{2+\d}p_k(1+3.25p_k^\d)} {(\la(\beta_2-a^2\theta))^{1+\d/2}}  =  \frac{\beta_{2+\d}\E(1+3.25\eta^\d)} {\la^{\d/2}(\beta_2-a^2\theta)^{1+\d/2}}\le
$$
$$
\le \frac{\beta_{2+\d}(1+3.25(\E\eta)^\d)} {\la^{\d/2}\beta_2^{1+\d/2}(1-\theta)^{1+\d/2}} = \frac{\beta_{2+\d}(1+3.25\theta^\d)} {\la^{\d/2}\beta_2^{1+\d/2}(1-\theta)^{1+\d/2}},
$$
$$
T_{2+\d,n}= \sum_{k=1}^n \Big(\frac{p_k(\beta_2-a^2p_k)} {\la(\beta_2-a^2\theta)}\Big)^{1+\d/2}  =\frac{\E\eta^{\d/2}(1-\alpha\eta)^{1+\d/2}}{\la^{\d/2}(1-\alpha\theta)^{1+\d/2}},
$$
где $\alpha\coloneqq {a^2}/{\beta_2}\in[0,1].$ В силу неравенства Коши--Буняковского и затем Гёльдера~\eqref{HolderIneqForExpectations} имеем
$$
\la^{\d}T_{2+\d,n}^2= \frac{\left[ \E\eta^{\d/2}(1-\alpha\eta)^{(1+\d)/2}\cdot\sqrt{1-\alpha\eta}\, \right]^2}{(1-\alpha\theta)^{2+\d}} 
\le 
$$
$$
\le \frac{\E\eta^{\d}(1-\alpha\eta)^{1+\d}\E(1-\alpha\eta)}{(1-\alpha\theta)^{2+\d}} =\frac{\E[\eta(1-\alpha\eta)^2]^{\d}\cdot(1-\alpha\eta)^{1-\d}} {(1-\alpha\theta)^{1+\d}}\le
$$
$$
\le \frac{[\E\eta(1-\alpha\eta)^2]^{\d}[\E(1-\alpha\eta)]^{1-\d}} {(1-\alpha\theta)^{1+\d}} =\frac{[\E\eta(1-\alpha\eta)^2]^{\d}} {(1-\alpha\theta)^{2\d}}.
$$
Так как $\eta\in[0,1]$ почти наверное и  $\alpha\in[0,1]$, то $\alpha^2\E\eta^3\le\alpha\E\eta^2$ и, следовательно,
$$
\E\eta(1-\alpha\eta)^2=\E\eta-2\alpha\E\eta^2+\alpha^2\E\eta^3 \le\E\eta-\alpha\E\eta^2\le
$$
$$
\le \E\eta-\alpha(\E\eta)^2 =\theta(1-\alpha\theta),
$$
откуда окончательно получаем 
$$
T_{2+\d,n}\le\frac{\theta^{\d/2}}{\la^{\d/2}(1-\alpha\theta)^{\d/2}} \le \frac{\theta^{\d/2}}{\la^{\d/2}(1-\theta)^{\d/2}}.
% =\Big(\frac{\theta}{\la}\Big)^{\d/2}(1-\theta)^{-1-\d/2},
$$
Отметим, что при $a=0$
$$
\lyapd=\frac{\beta_{2+\d}}{\la^{\d/2}\beta_2^{1+\d/2}},\quad T_{2+\d,n}=\la^{-\d/2}\E\eta^{\d/2}
%=\sum_{k=1}^n \Big(\frac{p_k}{\la}\Big)^{1+\d/2} 
\le\la^{-\d/2}(\E\eta)^{\d/2}=\frac{\theta^{\d/2}}{\la^{\d/2}},
$$
причём последнее равенство  справедливо и при $a\neq0$, если $p_1=\ldots\hm=p_n$.  Таким образом, из леммы~\ref{LemBEineq->RS} мы получаем следующее утверждение.

\begin{theorem}\label{ThPoisBinRSNormApprEstim}
В сделанных выше предположениях и обозначениях, если $\theta<1,$ то
\begin{multline*}
\Delta_{n,\pr}\coloneqq\sup_{x\in\R}\abs{\Prob\left(\widetilde\Sbin<x\right)-\Phi(x)}\le
\\
\le\min_{s\ge0}\frac{C_s(\d)}{\la^{\d/2}(1-\theta)^{\d/2}} \bigg[ \frac{\beta_{2+\d}}{\beta_2^{(2+\d)/2}} \cdot \frac{1+3.25\theta^\d}{1-\theta} +s\theta^{\d/2}\bigg],
\end{multline*}
%в частности, при $\d=1$
%$$
%\Delta_{n,\pr}\le
%\frac1{\sqrt{\la(1-\theta)}} \min\left\{0.469\cdot\frac{\beta_3}{\beta_2^{3/2}} \cdot\frac{1+3.25\theta}{1-\theta} ,\ 0.3031\cdot\frac{\beta_3}{\beta_2^{3/2}}\cdot\frac{1+3.25\theta}{1-\theta}  +0.646\sqrt\theta\right\}.
%$$
а если $a\coloneqq \E X=0,$ то  при любом значении $\theta\in(0,1]$ справедлива уточнённая оценка
$$
\Delta_{n,\pr}\le
\min_{s\ge0}\frac{C_s(\d)}{\la^{\d/2}} \bigg[\frac{\beta_{2+\d}}{\beta_2^{(2+\d)/2}}+s\theta^{\d/2} \bigg],
$$
в частности, при $\d=1$ и $p_1,\ldots,p_n\in(0,1]$
$$
\Delta_{n,\pr}\le
\frac1{\sqrt\la}\min\left\{0.5583\cdot{\beta_3}{\beta_2^{-3/2}},\ 0.3057\cdot{\beta_3}{\beta_2^{-3/2}} +\sqrt\theta\right\},
$$
а при $\d=1$ и $p_1=\ldots=p_n$
$$
\Delta_{n,\pr}\le
\frac1{\sqrt\la}\min\left\{0.469\cdot{\beta_3}{\beta_2^{-3/2}},\ 0.3031\cdot{\beta_3}{\beta_2^{-3/2}} +0.646\sqrt\theta\right\}.
$$
\end{theorem}

\begin{myremark}
Заметим, что %величина третьего абсолютного нормированного момента
$\beta_{2+\d}/\beta_2^{(2+\d)/2}\ge1$ в силу неравенства Ляпунова, и при больших значениях  этой дроби предпочтительнее оценка с меньшей константой $C_s(\d)$ ($0.3057$ или $0.3031$ для $\d=1$).
\end{myremark}

Таким образом, из теоремы~\ref{ThPoisBinRSNormApprEstim} видно, что $\Sbin$ асимптотически нормальна, если $\la\coloneqq\sum_{k=1}^np_k\to\infty,$ т.е. $p_k$ могут быть малы, но ряд $\sum p_k$ должен расходиться. Если же $\la \nrightarrow\infty,$ но $\theta\coloneqq\frac1\la\sum_{k=1}^np_k^2\to0,$ то необходимо $\sum_{k=1}^np_k^2\to0$ и, по теореме Пуассона~\ref{ThPoissona}, $\Bin\hm{\Rightarrow} N_\la\sim Pois(\la)$, т.е. индекс суммирования $\Bin$ имеет приближённо пуассоновское распределение. Как мы скоро увидим (теорема~\ref{ThTVD(Yn,Ym)<=TVD(n,m)} и вытекающая из неё формула~\eqref{TVD(Sbin,Spois)<=theta*la}), распределения и соответствующих случайных сумм в этом случае тоже будут близки.

\section{Пуассоновские случайные суммы}
%\setrightmark{Пуассоновские случайные суммы}

\begin{definition}
Пусть $X_1,X_2,\ldots$~--- н.о.р.с.в, $N_\la\sim Pois(\la)$, $\la>0,$ и $N_\la,X_1,\ldots,X_n$ независимы при каждом $\la>0.$ Тогда с.в.
$$
\Spois(\omega)\coloneqq\sum_{k=1}^{N_\la(\omega)}X_k(\omega),\quad \omega\in\Omega,
$$
называется \textit{пуассоновской случайной суммой}, а её распределение~--- \textit{обобщённым\footnote{В англоязычной литературе принят термин ``compound''}  (сложным, составным) пуассоновским}.
\end{definition}

Нижеследующая теорема позволяет судить о близости случайно проиндексированных последовательностей, в частности, случайных сумм, по близости соответствующих случайных индексов по метрике полной вариации между распределениями с.в. $\xi$ и $\eta$
$$
\tvd(\xi,\eta)\coloneqq\sup_{B\in\mathcal B}\abs{\Prob(\xi\in B)-\Prob(\eta\in B)} = \frac12\int_{-\infty}^\infty\abs{\dd(F_\xi-F_\eta)(x)}.
$$

\begin{theorem}\label{ThTVD(Yn,Ym)<=TVD(n,m)}
Пусть $\{Y_n\}_{n\in\N_0}$~--- последовательность произвольных с.в., $M,N$~--- целочисленные неотрицательные с.в., такие что $N,Y_0,Y_1,\ldots$ независимы и $M,Y_0,Y_1,\ldots$  независимы. Тогда
$$
\tvd(Y_N,Y_M)\le\tvd(N,M).
$$
\end{theorem}

\begin{proof}
Используя выражение для метрики полной вариации через интеграл Лебега--Стилтьеса по вариации разности соответствующих ф.р., можем записать
$$
2\tvd(Y_N,Y_M)=\int\abs{\dd\big[\Prob(Y_N<x)-\Prob(Y_M<x)\big]}.
$$
Заметим, что в силу независимости
$$
\Prob(Y_N<x)-\Prob(Y_M<x)= \smallsum_{n=0}^\infty\big[\Prob(Y_n<x,N=n)-\Prob(Y_n<x,M=n)\big]=
$$
$$
= \smallsum_{n=0}^\infty\big[\Prob(N=n)-\Prob(M=n)\big]\Prob(Y_n<x),\quad x\in\R,
$$
поэтому 
$$
2\tvd(Y_N,Y_M)= \int\abs{\smallsum_{n=0}^\infty\big[\Prob(N=n)-\Prob(M=n)\big] \dd\Prob(Y_n<x)}\le
$$
$$
\le \int\smallsum_{n=0}^\infty\abs{\Prob(N=n)-\Prob(M=n)} \dd\Prob(Y_n<x)=
$$
$$
 =\smallsum_{n=0}^\infty\abs{\Prob(N=n)-\Prob(M=n)}=2\tvd(N,M).\qedhere 
$$
\end{proof}

Записывая теорему~\ref{ThTVD(Yn,Ym)<=TVD(n,m)} для сумм $Y_n\coloneqq X_1+\ldots+X_n$ с ${N=\Bin}$ и $M=N_\la,$ и учитывая оценку скорости сходимости в теореме Пуассона~\ref{ThPoissona}, получаем
\begin{equation}\label{TVD(Sbin,Spois)<=theta*la}
\rho(\Sbin,\Spois)\le\tvd(\Sbin,\Spois)\le \tvd(\Bin,N_\la)\le \sum_{k=1}^np_k^2=\la_2=\theta\la,
\end{equation}
откуда вытекает, что аппроксимация обобщённого пуассон-биномиального распределения обобщённым пуассоновским имеет место не только при ограниченном~$\la$ и малом~$\theta$, что уже обсуждалось выше после теоремы~\ref{ThPoisBinRSNormApprEstim}, но и при растущем~$\la$, если $\theta\la\to0.$ С другой стороны, при $\la\to\infty$ пуассоновские случайные суммы, как мы увидим, сами по себе асимптотически нормальны, поэтому при больших~$\lambda$ сближение пуассон-биномиальных сумм с нормальным законом неизбежно (в рассматриваемых моментных условиях).

\subsection{Аналог неравенства Берри--Эссеена}

\begin{theorem}[Аналог неравенства Берри--Эссеена для пуассоновских случайных сумм]
\label{ThBEineqPoisSum}
Пусть $\beta_{2+\d}\coloneqq\E|X|^{2+\d}<\infty$ с некоторым $\d\in(0,1],$  $a\coloneqq \E X,$ $\beta_2\coloneqq\E X^2>0,$ 
$$
\widetilde \Spois\coloneqq \frac{\Spois-\E\Spois}{\sqrt{\D\Spois}}=\frac{\Spois-\la a}{\sqrt{\la\beta_2}},\quad\Lpoisd\coloneqq\frac{\beta_{2+\d}}{\la^{\d/2}\beta_2^{1+\d/2}},\quad \la>0,
$$
и $Z\sim N(0,1).$ Тогда для всех $\la>0$
\begin{equation}\label{B-Epois}
\rho(\widetilde\Spois,Z)\coloneqq \sup_{x\in\R}\abs{\Prob(\widetilde\Spois<x)-\Phi(x)}\le \poisbec(\d)\cdot\Lpoisd,
\end{equation}
где 
$$
\poisbec( \d)=\min\limits_{0\le s\le1}C_s(\d),\quad 0<\d\le1,
$$
$C_s(\d)$ --- константы из классического неравенства Берри--Эссееена~\eqref{BEineq(delta)iid} для о.р.с.в. В частности, $\poisbec(1)\hm=C_{0.646}(1)\le0.3031$ и при $\d=1$
$$
\rho(\widetilde\Spois,Z)\le \frac{0.3031}{\sqrt{\la}}\cdot\frac{\beta_3}{\beta_2^{3/2}},\quad \la>0.
$$
\end{theorem}

\begin{proof}
Выберем произвольное натуральное $n>\la$ и рассмотрим ``вспомогательную'' биномиальную случайную сумму $\Sbin$ с $p_k=\la/n=\theta\in(0,1),$ $k=\overline{1,n}.$ 
Заметим, что
$$
\E\Sbin=\la a=\E\Spois,
\quad 
q^2\coloneqq \frac{\D\Sbin}{\D\Spois}=1-\frac{a^2\theta}{\beta_2}\le1,
$$
где по определению полагаем $q>0.$ По неравенству треугольника имеем
$$
\rho(\widetilde\Spois,Z)\le\rho(\widetilde\Spois,q\widetilde\Sbin) + \rho(q\widetilde\Sbin,qZ) +\rho(qZ,Z).
$$
В силу инвариантности метрики Колмогорова относительно линейных преобразований случайных величин и оценки~\eqref{TVD(Sbin,Spois)<=theta*la} получаем
$$
\rho(\widetilde\Spois,q\widetilde\Sbin) 
=\rho\left(\frac{\Spois-\la a}{\sqrt{\D\Spois}},\frac{\Sbin-\la a}{\sqrt{\D\Spois}}\right)
=\rho(\Spois,\Sbin)\le\theta\la.
$$
Далее, по теореме~\ref{ThPoisBinRSNormApprEstim} имеем
\begin{multline*}
\rho(q\widetilde\Sbin,qZ)= \rho(\widetilde\Sbin,Z)\le 
\\
\le \min_{s\ge0}\frac{C_s(\d)}{(\la(1-\theta))^{\d/2}}  \bigg[ \frac{\beta_{2+\d}}{\beta_2^{(2+\d)/2}}\cdot\frac{1+3.25\theta^\d}{1-\theta} +s\theta^{\d/2}\bigg].
\end{multline*}
И наконец, в силу леммы~\ref{LemPhiIncrementsEstims} и с учётом того, что $q\in(0,1],$ получаем
\begin{multline*}
\rho(qZ,Z) =\sup_{x\in\R}\abs{\Phi(x)-\Phi(qx)}\le \frac{q^{-1}-1}{\sqrt{2\pi e}} \le \frac{q^{-2}-1}{\sqrt{2\pi e}}=
\\
= \frac1{\sqrt{2\pi e}}\cdot\frac{a^2\theta}{\beta_2-a^2\theta}\le \frac{1}{\sqrt{2\pi e}}\cdot \frac\theta{1-\theta},
\end{multline*}
где последнее неравенство справедливо в силу монотонного возрастания выражения слева по $a^2\le\beta_2.$

Таким образом, мы получили оценку
\begin{multline*}
\rho(\widetilde\Spois,Z)\le  \theta\la + \frac{1}{\sqrt{2\pi e}}\cdot \frac\theta{1-\theta}+
\\
+ \min_{s\ge0}\frac{C_s(\d)}{(\la(1-\theta))^{\d/2}}  \bigg[ \frac{\beta_{2+\d}}{\beta_2^{(2+\d)/2}}\cdot\frac{1+3.25\theta}{1-\theta} +s\theta^{\d/2}\bigg]
\end{multline*}
для любого натурального $n>\la$, где $\theta=\la/n$. Теперь переходя к пределу при $n\to\infty$ и замечая, что $\theta\to0$ при каждом фиксированном $\la>0$, приходим к утверждению теоремы.
\end{proof}

\begin{myremark}
Абсолютная константа $0.3031$ в аналоге неравенства Берри--Эссеена для пуассоновских случайных сумм, устанавливаемая теоремой~\ref{ThBEineqPoisSum} (с $\d=1$), оказывается строго меньше нижней оценки $0.3989<\frac1{\sqrt{2\pi}}\le C_0(1)$ аналогичной абсолютной константы в классическом неравенстве Берри--Эссеена. Это обстоятельство вызывает особый интерес к нижним оценкам  констант $\poisbec(\delta)$ в аналоге неравенства Берри--Эссеена для пуассоновских случайных сумм~\eqref{B-Epois}.
\end{myremark}

\subsection{Нижние оценки}

%Для $0\le\d\le1$ обозначим $\F_{2+\d}$ множество всех ф.р. $F$ с.в. $X$ таких, что $\E|X|^{2+\d}\in(0,\infty)$, и для $F\in\F_{2+\d}$ положим

\begin{theorem}\label{ThB-EpoisLowerBound}
Для каждого $0<\d\le1$
\begin{equation}\label{B-EpoisLowerBound}
\poisbec(\delta)\ge \limsup_{\la\to\infty}\sup_X\rho(\widetilde\Spois,Z)/\Lpoisd\ge \frac12\sup_{\gamma>0}\gamma^{\d/2}e^{-\gamma}I_0(\gamma),
\end{equation}
где точная верхняя грань берётся по множеству всех распределений ``случайного слагаемого'' $X,$ таких что $0\hm<\E|X|^{2+\d}\hm<\infty,$ а $\displaystyle I_0(\gamma)\coloneqq1+\sum\limits_{k=1}^\infty\frac{(\gamma/2)^{2k}}{(k!)^2},$ $\gamma>0,$ --- модифицированная функция Бесселя первого рода нулевого порядка. В частности,  справедлива универсальная оценка
$$
\inf_{0<\d\le1}\poisbec(\delta)\ge \frac{I_0(1)}{2e} >\frac{81}{128e}=0.2327\ldots
$$
\end{theorem}

\begin{myremark}
Как будет видно из доказательства, полученные в теореме~\ref{ThB-EpoisLowerBound} нижние оценки констант $\poisbec(\delta)$ из неравенства~\eqref{B-Epois} остаются справедливыми и для симметричного распределения $X$.
\end{myremark}

\begin{proof}
Пусть с.в. $X$ имеет симметричное распределение с $\Prob(|X|=1)=p=1-\Prob(X=0),$ $p\in(0,1]$. Тогда 
$$
\E X=0,\quad \beta_2\coloneqq\E X^2=p=\E|X|^{2+\d}\eqqcolon\beta_{2+\d},\quad \Lpoisd=(\la p)^{-\d/2}.
$$
В силу симметрии $\Prob(\Spois\hm<0)=\Prob(\Spois>0)$ и следовательно, $2\Prob(\Spois\hm<0)+\Prob(\Spois=0)=1$. Ограничивая снизу равномерное расстояние значением в нуле, получаем
$$
\rho(\widetilde\Spois,Z)\ge \Phi(0)-\Prob(\Spois<0) =\frac12-\frac{1-\Prob(\Spois=0)}{2} =\frac12\Prob(\Spois=0)=
$$
$$
=\frac12\sum_{n=0}^\infty\Prob(S_n^*=0,N_\la=n) =\frac{e^{-\la}}2\sum_{n=0}^\infty\frac{\la^n}{n!}\Prob(S_n^*=0),
$$
где $S_n^*=X_1+\ldots+X_n$ для  $n\in\N$ и $S_0^*\coloneqq0.$ Учитывая, что 
$$
\Prob(S_n^*=0) = \sum_{k=0}^{\lfloor n/2\rfloor} \frac{n!}{(k!)^2(n-2k)!}(p/2)^{2k}(1-p)^{n-2k},\quad n\in\N_0,
$$
получаем
$$
\rho(\widetilde\Spois,Z)\ge \frac{e^{-\la}}2\sum_{n=0}^\infty \sum_{k=0}^{\lfloor n/2\rfloor} \frac{\la^n}{n!} \cdot \frac{(1-p)^nn!}{(k!)^2(n-2k)!}\cdot\Big(\frac{p/2}{1-p}\Big)^{2k} =
$$
$$
=\frac{e^{-\la}}2\sum_{k=0}^\infty \sum_{n=2k}^{\infty} \frac{(\la(1-p))^n}{(n-2k)!}\cdot\frac1{(k!)^2}\Big(\frac{p/2}{1-p}\Big)^{2k}=
$$
$$
=\frac{e^{-\la}}2\sum_{k=0}^\infty(\la(1-p))^{2k}e^{\la(1-p)} \cdot\frac1{(k!)^2}\Big(\frac{p/2}{1-p}\Big)^{2k}=
$$
$$
=\frac{e^{-\la p}}2 \sum_{k=0}^\infty\frac{(\la p/2)^{2k}}{(k!)^2}= \frac12e^{-\la p}I_0(\la p).
$$
%где $\gamma\coloneqq \la p>0.$ 
Таким образом, выбирая указанное выше распределение $X$ с $p\coloneqq \frac\gamma\la$, $0<\gamma\le\la$, получаем нижнюю оценку
$$
\poisbec(\delta)\ge \frac12\limsup_{\la\to\infty}\sup_{p\in(0,1]}(\la p)^{\d/2}e^{-\la p}I_0(\la p)\ge
$$
$$
\ge \frac12\limsup_{\la\to\infty}\sup_{0<\gamma\le\la}\gamma^{\d/2}e^{-\gamma}I_0(\gamma)= \frac12\sup_{\gamma>0}\gamma^{\d/2}e^{-\gamma}I_0(\gamma)
$$
для всех $\d\in(0,1].$ 

Универсальную для всех $\d\in(0,1]$ миноранту можно получить, если положить $\gamma=1$ и заметить, что, как вытекает из определения функции~$I_0$, 
$$
I_0(1)> 1+\frac14+\frac1{16\cdot4}=\frac{81}{64}.\qedhere
$$
\end{proof}

Можно убедиться, что при $\d=1$ максимум в~\eqref{B-EpoisLowerBound}  достигается при $\gamma=0.7899\ldots,$ что влечёт оценку $\poisbec(1)\ge0.2344\ldots.$ Заметим, что, как вытекает из доказательства, ляпуновская дробь экстремального распределения связана с $\gamma$ следующим образом: $\Lpoisd=\gamma^{-\d/2}$, т.е. при $\d=1$ экстремальное $\Lpoisd=1.12\ldots.$ Точные значения миноранты в~\eqref{B-EpoisLowerBound}  при других $0<\d\le1$ приведены в нижеследующей таблице во второй строке, а соответствующие оптимальные значения параметра~$\gamma$~--- в третьей.

{\centering
\begin{tabular}{||c||c|c|c|c|c||}
\hline
$\d$ & 1&0.9&0.5&0.1&0+
\\\hline
$\poisbec(\d)\ge$ &0.2383&0.2534&0.2803&0.4097&$1/2$
\\\hline
$\gamma=$&0.7899&0.6520&0.2922&0.0513&0+
\\\hline
\end{tabular}

}

\section{Смешанные пуассоновские случайные суммы}
%\setrightmark{Смешанные пуассоновские случайные суммы}

\begin{definition}
Пусть $\Lambda$ --- почти наверное положительная с.в. с ф.р. $G.$ Будем говорить, что целочисленная неотрицательная с.в. $N$ имеет \textit{смешанное распределение Пуассона} со структурным {\rm(}смешивающим{\rm)} распределением $\Lambda$ $($или $G$$)$ и писать $N\sim\MP(\Lambda),$ если
$$
\Prob(N=k)=\frac1{k!}\int_0^\infty \la^ke^{-\la}\dd G(\la),\quad k=0,1,2,\ldots.
$$
\end{definition}

%С.в. $N$ можно также записать в виде пуассоновской с.в. $N_\la\sim Pois(\la)$ со случайным параметром $\Lambda\qp0$, где с.в. $N_\la$ и $\Lambda$ независимы при каждом $\la>0$
Если ввести с.в. $N_\la\sim Pois(\la)$, независимую от $\Lambda$ при каждом $\la>0$, то можно записать $N\eqd N_\Lambda$.

Известно несколько конкретных примеров смешанных пуассоновских распределений, наиболее широко используемым среди которых, пожалуй, является отрицательное биномиальное распределение (это распределение было использовано в виде смешанного пуассоновского еще в работе Гринвуда и Юла~\cite{GreenwoodYule1920} для моделирования частот несчастных случаев на производстве). Отрицательное биномиальное распределение порождается структурным гамма-распределением. Другими примерами смешанных пуассоновских распределений являются распределение Делапорте~\cite{Delaporte1960}, порождаемое сдвинутым гамма-структурным распределением: $\Lambda+\alpha\sim\Gamma(r,\beta)$, $r,\beta>0$, $\alpha\in\R,$ распределение Зихеля~\cite{Holla1967,Sichel1971,Willmot1987}, порождённое обратным нормальным структурным распределением с плотностью
$$
G'(x)=\frac{(a/b)^{r/2}}{2K_r(\sqrt{ab})} x^{r-1}e^{-(b/x+ax)/2}\I(x>0),\quad 
$$
при $r=-1/2,$ где $K_r$ --- модифицированная функция Бесселя третьего рода, бета-пуассоновское распределение~\cite{Quinkert1957,BeallRescia1953,Gurland1958}, порождаемое бета-структурным распределением с плотностью
$$
G'(x)= \frac{x^{a-1}(\alpha-x)^{b-1}}{B(a,b)\alpha^{a+b-1}}\I(0<x<\alpha),\quad \alpha,a,b>0,
$$
обобщённое распределение Варинга~\cite{Irwin1968,Seal1978}, порождаемое структурным распределением $\Lambda=\Lambda_1\cdot\Lambda_2$, где $\Lambda_1$, $\Lambda_2$ независимы, $\Lambda_1\sim\Gamma(\gamma,\gamma),$ а $\Lambda_2$ имеет плотность
$$
\frac{\dd\Prob(\Lambda_2<x)}{\dd x} =\frac{\gamma^{\rho}x^{\varkappa-1}}{(\gamma+x)^{\rho+\varkappa}} \cdot\frac{\I(x>0)}{B(\rho,\varkappa)},\quad \gamma,\rho,\varkappa>0,
$$
Свойства смешанных пуассоновских распределений подробно описаны в книгах~\cite{Grandell1997,BeningKorolev2002}.

\begin{lemma}\label{LemNegBinomial=MixedPoisson}
Если $\Lambda\sim\Gamma\big(r,\frac{p}{1-p}\big),$ $r>0,$ $p\in(0,1),$ то $\MP(\Lambda)\hm=\NB(r,p).$ В частности, если $\Lambda\sim exp\big(\frac{p}{1-p}\big),$ то $\MP(\Lambda)\hm=\mathcal G(p).$
\end{lemma}

\begin{proof} Пусть $\beta\coloneqq\frac{p}{1-p}$ и $N\sim\MP(\Lambda),$ тогда, по определению, для всех  $k\in\N_0$ имеем
$$
\Prob(N=k)=\frac1{k!}\int_0^\infty e^{-\la}\la^k\cdot \frac{\beta^r}{\Gamma(r)}\la^{r-1}e^{-\beta\la}\dd\la=
$$
$$
= \frac{\beta^r}{k!\Gamma(r)} \int_0^\infty\la^{r+k-1}e^{-(\beta+1)\la}\dd\la= \frac{\beta^r}{k!\Gamma(r)}\cdot \frac{\Gamma(r+k)}{(\beta+1)^{r+k}}=
$$
$$
=  \frac{\Gamma(r+k)}{k!\Gamma(r)} \Big(\frac{\beta}{\beta+1}\Big)^r\Big(\frac1{\beta+1}\Big)^k = C_{r+k-1}^kp^r(1-p)^k.\qedhere 
$$
\end{proof}

\begin{definition} Если $X_1,X_2,\ldots$ одинаково распределены, $N\sim\MP(\Lambda)$ и $N,X_1,X_2,\ldots$ независимы на некотором в.п. $(\Omega,\mathcal F,\Prob)$, то с.в.
$$
S(\omega)\coloneqq X_1(\omega)+\ldots+X_{N(\omega)}(\omega),\quad \omega\in\Omega,
$$
{\rm(}как и ранее, для определённости мы считаем, что $S=0$, если $N=0${\rm)} называется \textit{смешанной пуассоновской случайной суммой}, а ее распределение~---
\textit{обобщённым {\rm(}составным, сложным{\rm)} смешанным пуассоновским} {\rm(}compound mixed Poisson{\rm)}.
\end{definition}

Возвращаясь к записи смешанной пуассоновской с.в. $N$ в виде суперпозиции $N\eqd N_\Lambda$ пуассоновской с.в. $N_\la\sim Pois(\la)$ и случайного параметра $\Lambda\gp0$, независимого от $N_\la$, мы можем записать и с.в. $S$ в виде суперпозиции
$$
S\eqd S_{N_{\Lambda}}\eqd X_1+\ldots+X_{N_{\Lambda}},
$$
пуассоновской случайной суммы $S_{N_{\la}}\coloneqq X_1+\ldots+X_{N_\la}$ со случайным параметром $\Lambda$, где $\Lambda,N_\la,X_1,X_2,\ldots$ независимы при каждом $\la>0$.

Пусть теперь $\Lambda=\Lambda_t\gp0$ --- с.в., распределение которой зависит от параметра $t>0$, и с.в. $N(t)\sim\MP(\Lambda_t)$ такова, что $N(t),X_1,X_2,\ldots$ независимы при каждом $t>0$. Положим
$$
S(t)\coloneqq X_1+\ldots+X_{N(t)}.
$$

\begin{myremark}
Если $N(t)=N_1(\Lambda_t)$, где $N_1$ --- стандартный пуассоновский процесс, $\{\Lambda_t,t\ge0\}$ --- случайный процесс с п.н. неубывающими непрерывными справа траекториями, т.ч. $\Prob(\Lambda_0\hm=0)=1=\Prob(\Lambda_t<\infty))$, и процессы $N_1(t),\Lambda_t$ независимы, то, очевидно, $N(t)\sim\MP(\Lambda_t)$, а $S(t)$ в этом случае называется \textit{обобщённым {\rm(}сложным, составным{\rm)} дважды стохастическим пуассоновским процессом}, или \textit{трижды стохастическим пуассоновским процессом}, или \textit{обобщённым {\rm(}сложным, составным{\rm)} процессом Кокса},  управляемым процессом $\Lambda_t.$
\end{myremark}
 
\begin{myremark}
Мы используем запись в терминах случайных величин лишь для удобства и наглядности. На самом деле речь идёт о соответствующих соотношениях для распределений, но такая форма записи оказывается более громоздкой. Поэтому предположение о существовании вероятностного пространства, на котором определены упомянутые выше случайные величины с указанными свойствами, ни в коей мере не ограничивает общность.
\end{myremark}

\subsection{Предельные распределения}

Пусть, как и ранее, $X$ обозначает ``случайное слагаемое'', т.е. с.в., имеющую общее с $X_1,X_2,\ldots$ распределение. Предположим, что $\beta_2=\E X^2\in(0,\infty).$ Асимптотическое поведение смешанных пуассоновских случайных сумм $S(t)$, когда $N(t)$ в определённом смысле неограниченно возрастает, принципиально различно в зависимости от того, центрированы ли случайные слагаемые или нет. Пусть, для простоты, $\E X=0$. В этом случае предельными распределениями для стандартизованных смешанных пуассоновских случайных сумм являются масштабные смеси нормальных законов. 

Всюду далее, не ограничивая общности, будем считать, что $\beta_2\hm=1$. Сходимость по распределению и по вероятности будет обозначаться
символами $\dto$ и $\pto$ соответственно.

\begin{theorem}[см.~\cite{Korolev1996,BeningKorolev2002}]
\label{ThMixedPoisRSLimitLawsKorolev}
Предположим, что $\E X=0,$ $\E X^2=1$ и $\Lambda_t\pto\infty$ при $t\to\infty$. Тогда для положительной неограниченно возрастающей функции $d(t)$ имеет место слабая сходимость
$$
\frac{S(t)}{\sqrt{d(t)}}\dto Y
$$
к некоторой с.в. $Y$, тогда и только тогда, когда найдётся такая с.в. $\Lambda,$ что при той же функции $d(t)$
\begin{equation}\label{CompoundMixedPoissonLimitLaw}
\frac{\Lambda_t}{d(t)}\dto \Lambda\quad\text{и}\quad  Y\eqd Z\sqrt{\Lambda}.
\end{equation}
\end{theorem}

Заметим, что, в отличие от рассмотренных ранее предельных теорем, класс предельных законов для смешанных пуассоновских случайных сумм довольно широк, в  частности, он содержит распределения с произвольно тяжёлыми хвостами, например, такие как распределение Стьюдента и --- как частный случай последнего --- распределение Коши, у которого нет даже математического ожидания.

Для удобства дальнейших обозначений введем с.в. $\Lambda(r,\alpha)\sim\Gamma(r,\alpha)$, $r,\alpha>0.$

\begin{example}\label{Ex:St=Z/sqrt(Gamma)}
Покажем, что если с.в. $\Lambda^{-1}\eqd\Lambda(\frac{r}2,\frac{r}2)\sim\Gamma(\frac{r}2,\frac{r}2)$, $r>0$, и $Z\sim N(0,1)$ независимы, то
$$
Z\sqrt{\Lambda}\eqd Z/\sqrt{\Lambda(\tfrac{r}2,\tfrac{r}2)}\sim St(r),\quad r>0,
$$
т.е. если в~\eqref{CompoundMixedPoissonLimitLaw} с.в. $\Lambda$ имеет обратное гамма-распределение, то $Y$ имеет распределение Стьюдента.  Действительно, плотность с.в.  $Z/\sqrt{\Lambda}$ имеет вид
$$
\frac{\dd}{\dd x}\E\Phi\big(x\sqrt{\Lambda}\,\big) = \frac{\dd}{\dd x}\int_0^\infty\Phi\big(x\sqrt{\la}\,\big)\frac{(r/2)^{r/2}}{\Gamma(r/2)} e^{-\la r/2}\la^{r/2-1}\dd\la=
$$
$$
=\frac{(r/2)^{r/2}}{\Gamma(r/2)} \int_0^\infty\sqrt{\la}\phi\big(x\sqrt{\la}\,\big)e^{-\la r/2}\la^{r/2-1}\dd\la=
$$
$$
=\frac{(r/2)^{r/2}}{\sqrt{2\pi}\Gamma(r/2)} \int_0^\infty \la^{(r-1)/2}e^{-(x^2+r)\la/2}\dd\la=
$$
$$
=\frac{(r/2)^{r/2}\Gamma\big(\frac{r+1}2\big)}{\sqrt{2\pi}\Gamma(r/2)} \Big(\frac2{x^2+r}\Big)^{(r+1)/2} =\frac{\Gamma\big(\frac{r+1}2\big)}{\sqrt{\pi r}\,\Gamma\big(\frac{r}2\big)} \Big(1+\frac{x^2}r\Big)^{-(r+1)/2}
$$
для любого $x\in\R$, в частности, при $r=1$ получаем $(\pi(1+x^2))^{-1}$ --- плотность распределения Коши.
\end{example}

\begin{example}\label{Ex:SymmGamma=Z*sqrt(Gamma)}
Если с.в. $\Lambda\eqd\Lambda(r,r)\sim\Gamma(r,r)$, $r>0$, и $Z\sim N(0,1)$ независимы, то с.в.  $Z\sqrt{\Lambda}$ имеет симметризованное гамма-распределение с параметрами~$r$ и $\sqrt{2r}$:
$$
Z\sqrt{\Lambda(r,r)}\eqd\Lambda\big(r,\sqrt{2r}\,\big) -\Lambda'\big(r,{\sqrt{2r}}\,\big),
$$
где с.в. $\Lambda\big(r,{\sqrt{2r}}\big)$, $\Lambda'\big(r,{\sqrt{2r}}\big)\sim\Gamma\big(r,{\sqrt{2r}}\big)$ и независимы. Чтобы в этом убедиться, достаточно заметить, что х.ф. с.в. $Z\sqrt\Lambda$
$$
\E e^{itZ\sqrt{\Lambda}}=\E e^{-\Lambda t^2/2}=\frac{r^r}{\Gamma(r)}\int_0^\infty \lambda^{r-1}e^{-\lambda(t^2/2+r)}\dd \lambda=
$$
$$
=\left (\frac{1}{1+t^2/(2r)}\right )^r= \left (\frac1{1+it/\sqrt{2r}}\right )^r\left (\frac1{1-it/\sqrt{2r}}\right )^r
$$
есть $r$-я степень лапласовской х.ф., или произведение характеристической функции гамма-распределения с параметрами $r$ и $\sqrt{2r}$ и комплексно сопряженной к ней.
\end{example}

\begin{excersize}
Доказать, что если $X,X_1,X_2$~--- н.о.р.с.в. с  $\E X=0,\E X^2=1,$ с.в. $N_{r,p}\sim\NB(r,p)$ независима от $\{X_k\}_{k\ge1}$,  то для любого $r>0$
$$
\sqrt{p}\sum_{k=1}^{N_{r,p}}X_k\dto \xi_r-\xi_r',\quad p\to0,
$$
где $\xi_r,\xi_r'$~--- н.о.р.с.в, имеющие гамма-распределение $\Gamma(r,\sqrt2)$. В частности, при целых $r$ имеем $\sqrt{p}\sum_{k=1}^{N_{r,p}}X_k\dto \eta_1+\ldots+\eta_r,$ $p\to0,$ где $\{\eta_k\}_{k=1}^r$~--- н.о.р.с.в, имеющие распределение Лапласа c параметром~$\sqrt2$.
\end{excersize}

\begin{excersize}\label{Ex:NegBin->Normal}
Доказать, что если $X,X_1,X_2$~--- н.о.р.с.в. с  $0<\E X^2<\infty,$ с.в. $N_{r,p}\sim\NB(r,p)$ независима от $\{X_k\}_{k\ge1}$ и $S_{r,p}=X_1+\ldots+X_{N_{r,p}}$,  то 
$$
%\sqrt{\frac{rp}{1-p}}\sum_{k=1}^{N_{r,p}}X_k\dto Z\sim N(0,1),\quad  r\to\infty.
\frac{S_{r,p}-\E S_{r,p}}{\sqrt{\D S_{r,p}}}\dto Z\sim N(0,1),\quad  r\to\infty,
$$
для всех $p$ таких, что $rp\to\infty$.
\end{excersize}

\subsection{Оценки скорости сходимости}

В данном разделе мы приведём оценки скорости сходимости в теореме~\ref{ThMixedPoisRSLimitLawsKorolev} и её частных случаях. Расстояние по метрике Колмогорова между распределениями с.в. $\xi,\eta$ будет обозначаться через 
$$
\rho(\xi,\eta)\coloneqq \sup_{x\in\R}\abs{\Prob(\xi<x)-\Prob(\eta<x)}.
$$

%Некоторые оценки скорости сходимости обобщенных смешанных пуассоновских распределений с нулевыми средними приведены в книге (Bening and Korolev, 2002). Однако эти оценки имеют сложный вид и справедливы при дополнительных моментных ограничениях. В следующей теореме представлена более простая оценка, справедливая при достаточно общих условиях.

\begin{theorem}\label{ThMixedPoisRSLimitLawApproxEstim}
Пусть $\E X=0,$ $\E X^2=1,$ $\beta_{2+\d}\coloneqq\E|X|^{2+\d}<\infty$ для некоторого $\d\in(0,1],$ с.в. $Z\sim N(0,1),$ $\Lambda\gp0,$ $\Lambda_t$ --- независимы при каждом $t>0,$ и $d(t)$~--- положительная функция, определённая для всех $t>0$. Положим
$$
\Delta_t\coloneqq\rho\bigg( \frac{S(t)}{\sqrt{d(t)}},Z\sqrt{\Lambda}\bigg),
\quad
\delta_t\coloneqq\rho\bigg(Z\sqrt{\frac{\Lambda_t}{d(t)}},Z\sqrt{\Lambda}\bigg),\quad t>0.
$$
Тогда для всех $t>0$
$$
\Delta_t \le \poisbec(\d)\cdot\beta_{2+\d}\E\Lambda_t^{-\d/2} + \delta_t,
$$
где константы  $\poisbec(\d)$ определены в теореме~\ref{ThBEineqPoisSum}, в частности, $\poisbec(1)\hm\le0.3031.$ При этом
$$
\delta_t\le\frac12\rho\bigg(\frac{\Lambda_t}{d(t)},\Lambda\bigg),\quad t>0.
$$
\end{theorem}

\begin{myremark} Оценка равномерного расстояния $\delta_t=\rho\big(Z\sqrt{{\Lambda_t}/{d(t)}},Z\sqrt{\Lambda}\big)$ между смешанными нормальными  распределениями, имеющими \textit{непрерывные} ф.р., через половину равномерного расстояния $\rho\big({\Lambda_t}/{d(t)},\Lambda\big)$ между смешивающими распределениями, является довольно грубой  с точки зрения предельной теоремы~\ref{ThMixedPoisRSLimitLawsKorolev} в случае предельного распределения $\Lambda$ с  \textit{разрывной}  (например, дискретной) ф.р., так как в теореме~\ref{ThMixedPoisRSLimitLawsKorolev} речь идёт о \textit{слабой} сходимости $\Lambda_t/d(t)$ к $\Lambda$, которая не метризуется равномерной метрикой для предельных законов c разрывной ф.р., т.е. может оказаться, что $\Lambda_t/{d(t)}\pto\Lambda$ и, следовательно, $\delta_t\to0$ в силу непрерывности и ограниченности $\Phi$, но $\rho\big({\Lambda_t}/{d(t)},\Lambda\big)$ не стремится к нулю с ростом~$t$.
\end{myremark}

Теорема~\ref{ThMixedPoisRSLimitLawApproxEstim} с немного худшими значениями констант $M(\d)$ была доказана в~\cite{GavrilenkoKorolev2006}.

\begin{proof}
Как упоминалось выше, распределение смешанной пуассоновской случайной
суммы $S(t)$ можно записать в виде суперпозиции $S(t)\hm\eqd S_{N_{\Lambda_t}}\eqd X_1+\ldots+X_{N_{\Lambda_t}}\vphantom{\frac{1}{\displaystyle1^1}}$, где $S_0\coloneqq0,$ $S_n\coloneqq X_1+\ldots+X_n,$ $n\in\N,$ $N_\lambda$ --- пуассоновская с.в. с параметром
$\lambda>0$, такая что при каждом $\lambda>0$ и $t>0$ с.в.
$\Lambda_t,N_\lambda,X_1,X_2,\ldots$ независимы, т.е.
$$
\Prob(S(t)<x)=\int_{0}^{\infty} \Prob\left (S_{N_\la}<x\right) \dd G_t(\la),\quad
x\in\R,\ t>0,
$$
где $G_t(\la)\coloneqq \Prob(\Lambda_t<\la)$, $\la\in\R$, --- ф.р. с.в. $\Lambda_t.$ Обозначим также $G(\la)\coloneqq\Prob(\Lambda<\la)$, $\la\in\R,$ --- ф.р. с.в. $\Lambda.$ С помощью вышеприведённого  представления ф.р. $S(t)$ и неравенства треугольника получаем
$$
\Delta_t=\sup_{x\in\R}\bigg| \Prob\bigg(\frac{S(t)}{\sqrt{d(t)}}<x\bigg) - \Prob\left(Z\sqrt{\Lambda}<x\right)\bigg|=
$$
$$
=\sup_{x\in\R}\bigg| \int_0^\infty\bigg[ \Prob\bigg(\frac{S_{N_\la}}{\sqrt{\la}}<x\sqrt{\frac{d(t)}{\la}}\bigg) \pm\Phi\bigg(x\sqrt{\frac{d(t)}{\la}}\,\bigg) \bigg]\dd G_t(\la)-
$$
$$
-\Prob\left(Z\sqrt{\Lambda}<x\right)\bigg|\le I_1+ \sup_{x\in\R}|I_2(x)|,
$$
где
$$
I_1\coloneqq\sup_{x\in\R}\bigg| \int_0^\infty\bigg[ \Prob\bigg(\frac{S_{N_\la}}{\sqrt{\la}}<x\sqrt{\frac{d(t)}{\la}}\bigg) -\Phi\bigg(x\sqrt{\frac{d(t)}{\la}}\,\bigg) \bigg]\dd G_t(\la) \le
$$
$$
\le\int_0^\infty \sup_{x\in\R}\bigg|\Prob\bigg(\frac{S_{N_\la}}{\sqrt{\la}}<x\bigg) -\Phi(x)\bigg|\dd G_t(\la)\le
$$
$$
\le\int_0^\infty \poisbec(\d)\cdot\frac{\beta_{2+\d}}{\la^{\d/2}} \dd G_t(\la) =\poisbec(\d)\cdot\beta_{2+\d}\E\Lambda_t^{-\d/2}
$$
 в силу аналога неравенства Берри--Эссеена для пуассоновских случайных сумм из теоремы~\ref{ThBEineqPoisSum} и 
$$
I_2(x)\coloneqq \int_0^\infty\Phi\bigg(x\sqrt{\frac{d(t)}{\la}}\,\bigg)\dd G_t(\la)
-\Prob\left(Z\sqrt{\Lambda}<x\right)=
$$
$$
= \Prob\Big(Z\sqrt{{\Lambda_t}/{d(t)}}<x\Big) -\Prob\left(Z\sqrt{\Lambda}<x\right),
$$
%$$
%= \int_0^\infty\Phi\left(\frac{x}{\sqrt{\la}}\right) \dd_\la\left[G_t\big(\la d(t)\big)-G(\la)\right],
%$$
откуда сразу видно, что 
$$
\sup_x|I_2(x)|=\rho\big(Z\sqrt{{\Lambda_t}/{d(t)}},Z\sqrt{\Lambda}\big) =\delta_t.
$$
Далее интегрированием по частям получаем
$$
I_2(x)= \int_0^\infty\Phi\bigg(x\sqrt{\frac{d(t)}{\la}}\,\bigg)\dd G_t(\la)
-\int_0^\infty\Phi\bigg(\frac{x}{\sqrt\la}\bigg)\dd G(\la)=
$$
$$
=\int_0^\infty\Phi\bigg(\frac{x}{\sqrt\la}\bigg)\dd_\la[G_t(\la d(t))-G(\la)]=
$$
$$
=\Phi\left(\frac{x}{\sqrt{\la}}\right) \left[G_t\big(\la d(t)\big)-G(\la)\right]_{\la=0}^{\la=\infty} - 
$$
$$
-\int_0^\infty\left[G_t\big(\la d(t)\big)-G(\la)\right]\dd_\la\Phi\left(\frac{x}{\sqrt{\la}}\right)=
$$
$$
=-\int_0^\infty\left[\Prob\left(\frac{\Lambda_t}{d(t)}<\la\right) -\Prob(\Lambda<\la)\right]\dd_\la\Phi\left(\frac{x}{\sqrt{\la}}\right),\quad x\in\R.
$$
Оценивая подыинтегральную функцию равномерно по $\la>0$, приходим к неравенству
$$
\delta_t=\sup_{x\in\R}|I_2(x)|\le \rho\left(\frac{\Lambda_t}{d(t)},\Lambda\right) \sup_{x\in\R}\abs{\int_0^\infty\dd_\la\Phi\left(\frac{x}{\sqrt{\la}}\right)},
$$
интеграл в правой части которого равен $\Phi(0)-\Phi(+\infty)=-\frac12$ при $x>0,$ нулю при $x=0$ и $\Phi(0)-\Phi(-\infty)=\frac12$ при $x<0,$ т.е. по абсолютной величине не превосходит $\frac12.$
\end{proof}

В качестве примера применения теоремы~\ref{ThMixedPoisRSLimitLawApproxEstim} рассмотрим случай, когда при каждом $t>0$ случайная величина $\Lambda_t$ или $1/\Lambda_t$ имеет гамма-распределение. Первый случай представляет особый интерес с точки зрения его применения в финансовой математике для асимптотического обоснования адекватности таких популярных моделей эволюции финансовых индексов как дисперсионные гамма-процессы Леви (variance-gamma L{\'e}vy processes, VG-processes)~\cite{MadanSeneta1990} или двусторонние гамма-процессы~\cite{CarrMadanChang1998}. Второй случай, в котором предельными являются распределения Стьюдента, интересен как альтернатива  устойчивым законам, возникающим в качестве предельных для сумм н.о.р.с.в. с бесконечными дисперсиями, так как на практике обычно нет оснований предполагать отсутствие вторых моментов; более того, случайные слагаемые зачастую вообще можно считать ограниченными по смыслу задачи, и следовательно, имеющими все моменты, зато предположение о случайности числа этих слагаемых, как правило, бывает весьма реалистично.

Для удобства обозначений введём с.в. $\Lambda(r,\alpha)\sim\Gamma(r,\alpha)$ с плотностью
$$
g_{r,\alpha}(x)=\frac{\alpha^r}{\Gamma(r)}x^{r-1}e^{-\alpha x}\I(x>0),\quad \alpha,\ r>0,
$$
и заметим, что $\E\Lambda(r,\alpha)=r/\alpha$, $\Lambda(r,c\alpha)\eqd \frac1c\Lambda(r,\alpha),$ $c>0.$

\begin{example} Пусть с.в. $\Lambda_t:\eqd\left(\Lambda\left(\frac{r}2,\frac{t}2\right)\right)^{-1}$ имеет обратное гамма-распределение. Положим $d(t)\coloneqq t/r$. Тогда 
$$
\frac{\Lambda_t}{d(t)}\eqd\frac{r}{t\Lambda(\frac r2,\frac t2)}\eqd \frac1{\Lambda(\frac r2,\frac r2)} \eqqcolon\Lambda
$$
не зависит от $t$ и, следовательно, $\delta_t\le\frac12\rho\big(\frac{\Lambda_t}{d(t)},\Lambda\big)=0$ для всех  $t>0.$
\end{example}

Вычислим 
$$
\E\Lambda_t^{-\d/2}=\E\Big(\Lambda\Big(\frac{r}2,\frac{t}2\Big)\Big)^{\d/2} =\E\Big(\frac{\Lambda(r/2,1)}{t/2}\Big)^{\d/2} = 
$$
$$
=\Big(\frac t2\Big)^{-\d/2}\int_0^\infty \frac{x^{\d/2+r/2-1}}{\Gamma(r/2)}\,e^{-x} \dd x=\frac{2^{\d/2}\Gamma\big(\frac{r+\d}2\big)}{t^{\d/2}\Gamma\big(\frac{r}2\big)}.
$$
Для оценки отношения гамма-функций может быть полезна следующая лемма, устанавливающая, по крайней мере, правильную асимптотику ${\Gamma\big(\frac{r+\d}2\big)}/{\Gamma\big(\frac{r}2\big)} \sim \left(\frac{r}2\right)^{\d/2}$ при $r\to\infty$ и принадлежащая Вендeлю~\cite{Wendel1948}.

\begin{lemma}[\cite{Wendel1948}]\label{WendelGammaRatioBounds}
Для любого $x>0$ и $s\in[0,1]$
$$
\Big(\frac{x}{x+s}\Big)^{1-s} \le \frac{\Gamma(x+s)}{x^s\Gamma(x)}\le1.
$$
\end{lemma}

\begin{proof}
Пусть $\xi\sim exp(1)$ --- с.в. с плотностью $p_\xi(t)\hm=e^{-t}\I(t>0).$ Тогда 
$$
\Gamma(x+s)=\int_0^\infty t^{x+s-1}e^{-t}\dd t=\E\xi^{x+s-1}.
$$
В силу неравенства Гёльдера~\eqref{HolderIneqForExpectations} и соотношения $\Gamma(x+1)=x\Gamma(x)$, $x>0,$ имеем 
$$
\Gamma(x+s)=\E\xi^{xs}\cdot\xi^{(x-1)(1-s)} \le(\E\xi^x)^s(\E\xi^{x-1})^{1-s}=
$$
$$
=(\Gamma(x+1))^s(\Gamma(x))^{1-s}=x^s\Gamma(x).
$$
Заменяя теперь в полученном неравенстве $s$ на $1-s$ и затем $x$ на $x+s$, получаем
$$
x\Gamma(x)=\Gamma(x+1)%=x\Gamma(x)
\le (x+s)^{1-s}\Gamma(x+s).
$$ 
Из построенной таким образом двусторонней оценки для $\Gamma(x+s)$ непосредственно вытекает  утверждение леммы.
\end{proof}

Из  леммы~\ref{WendelGammaRatioBounds} с $x=r/2$ и $s=\d/2$ вытекает оценка $\E\Lambda_t^{-\d/2}\hm\le \left(\frac{r}t\right)^{\d/2}.$

Как вытекает из теоремы~\ref{ThMixedPoisRSLimitLawsKorolev} и примера~\ref{Ex:St=Z/sqrt(Gamma)}, предельным для $S(t)\sqrt{\frac{r}t}$ при $t\to\infty$ законом в рассматриваемом  случае будет распределение Стьюдента с $r$ степенями свободы c  ф.р.
$$
\E\Phi\big(x\sqrt{\Lambda}\,\big)= \E\Phi\left({x}/{\sqrt{\Lambda(\tfrac r2,\tfrac r2)}}\, \right)=
$$
$$
=\frac{\Gamma\big(\frac{r+1}2\big)}{\sqrt{\pi r}\,\Gamma\big(\frac{r}2\big)} \int_{-\infty}^x\Big(1+\frac{y^2}r\Big)^{-(r+1)/2}\dd y\eqqcolon  T_r(x),\quad x\in\R.
$$
Пусть $\tau_r$ --- с.в. с ф.р. $T_r(x)\coloneqq\Prob(\tau_r<x),$ $x\in\R,$ и
$$
\Delta_{r,t}\coloneqq\rho\left(S(t)\sqrt{\tfrac{r}{t}},\tau_r\right) =\sup_{x\in\R}\Big|\Prob\Big(S(t)<x\sqrt{\tfrac tr}\,\Big)-T_r(x)\Big|,\quad t>0.
$$

\begin{corollary}\label{CorMixedPoisStudentApproxErrBound}
Пусть $\Lambda_t^{-1}\sim\Gamma\left(\frac{r}2,\frac{t}2\right)$ и $\E X=0,$ $\E X^2=1,$ $\beta_{2+\d}\coloneqq\E|X|^{2+\d}<\infty$ с некоторым $\d\in(0,1].$ Тогда для всех $t,r>0$
$$
\Delta_{r,t}\le 2^{\d/2}\poisbec(\d) \frac{\Gamma\big(\frac{r+\d}2\big)}{\Gamma\big(\frac{r}2\big)}\cdot \frac{\beta_{2+\d}}{t^{\d/2}} \le \poisbec(\d)\beta_{2+\d} \Big(\frac{r}{t}\Big)^{\d/2} .
$$
В частности, при $r=1$ $($т.е. для предельного распределения Коши $T_1)$ и $\d=1$ 
$$
\Delta_{1,t}\le 0.3031\sqrt{\frac2{\pi}}\cdot \frac{\beta_3}{\sqrt t}<0.2419\cdot\frac{\beta_3}{\sqrt t},\quad t>0.
$$
\end{corollary}

%Как известно, $\tau_r\dto Z\sim N(0,1)$ при $r\to\infty$. В то же время асимптотика ${\Gamma\big(\frac{r+\d}2\big)}/{\Gamma\big(\frac{r}2\big)} \sim \left(\frac{r}2\right)^{\d/2}$ при $r\to\infty$ устанавливается следующей леммой, 

Заметим, что $\tau_r\dto Z\sim N(0,1)$  при $r\to\infty$, поэтому, если вместе с $t\to\infty$ одновременно и $r\to\infty$, но c такой скоростью, что $\frac{r}t\to0,$ то, как вытекает из следствия~\ref{CorMixedPoisStudentApproxErrBound}, имеет место сближение распределений нормированных сумм $S(t)$ с нормальным законом, т.е. ЦПТ. Более конкретно, с использованием асимптотически правильной оценки точности нормальной аппроксимации для распределения Стьюдента
$$
\rho(\tau_r,Z)\le \frac{A}{r}\quad\text{ с }\quad A\coloneqq\frac14\sqrt{\frac{7+5\sqrt2}{\pi e^{1+\sqrt2}}}=0.1582\ldots,
$$
полученной И.\,Пинелисом~\cite{Pinelis2013}, где постоянный множитель $A$ не может быть уменьшен в том смысле, что~\cite{Pinelis2013}
$$
\lim_{r\to\infty}r\cdot \rho(\tau_r,Z)=A,
$$
из следствия~\ref{CorMixedPoisStudentApproxErrBound}  можно получить следующую оценку точности \textit{нормальной} аппроксимации для распределений смешанных пуассоновских случайных сумм с обратным $\Gamma\big(\frac{r}2,\frac{t}2\big)$-структурным распределением при согласованно растущих $r$ и $t$ (так чтобы $\frac{r}t\to0$).

%Оценка точности нормальной аппроксимации для распределения Стьюдента получена И.\,Пинелисом~\cite{Pinelis2013} в виде
%$$
%\rho(\tau_r,Z)%\sup_{x\in\R}\abs{T_r(x)-\Phi(x)}
%\le \frac{A}{r},\quad\text{где }\ A=\frac14\sqrt{\frac{7+5\sqrt2}{\pi e^{1+\sqrt2}}}=0.1582\ldots,
%$$
%причём постоянный множитель $A$ здесь является наилучшим возможным в том смысле, что~\cite{Pinelis2013} 
%$$
%\lim_{r\to\infty}r\cdot \rho(\tau_r,Z)=A.
%$$
%
%Таким образом, из следствия~\ref{CorMixedPoisStudentApproxErrBound}  вытекает следующая оценка точности \textit{нормальной} аппроксимации для распределений смешанных пуассоновских случайных сумм с обратным $\Gamma\big(\frac{r}2,\frac{t}2\big)$-структурным распределением при согласованно растущих $r$ и $t$ (так чтобы $\frac{r}t\to0$).

\begin{corollary} 
В условиях следствия~\ref{CorMixedPoisStudentApproxErrBound} для всех $r,t>0$ 
\begin{equation}\label{MixedPois-Student-NormalApproxErrBound}
\rho\left(\sqrt{\frac{r}{t}}S(t),Z\right)\le \poisbec(\d)\beta_{2+\d} \Big(\frac{r}{t}\Big)^{\d/2} + \frac{A}{r}.
\end{equation}
\end{corollary}

Задаваясь целью выбора оптимального $r$, минимизирующего правую часть~\eqref{MixedPois-Student-NormalApproxErrBound}, можно убедиться, что с точностью до множителя, зависящего только от $\d$, это оптимальное $r$ имеет вид
$$
r=\Big(\frac{\beta_{2+\d}}{t^{\d/2}}\Big)^{-\frac2{2+\d}}.
$$
При указанном  $r$ из~\eqref{MixedPois-Student-NormalApproxErrBound} получаем
$$
\rho\left((\beta_{2+\d}t)^{-\frac{1}{2+\d}}S(t),Z\right)\le \big(\poisbec(\d)+A\big)\Big(\frac{\beta_{2+\d}}{t^{\d/2}}\Big)^{\frac2{2+\d}},\quad t>0,
$$
а если положить $r\coloneqq t^{\frac{\d}{2+\d}}$, чтобы избежать присутствия абсолютного момента $\beta_{2+\d}$ в нормировочном множителе при $S(t)$, то из~\eqref{MixedPois-Student-NormalApproxErrBound} получим оценку
$$
\rho\left(t^{-\frac{1}{2+\d}}S(t),Z\right)\le \big(\poisbec(\d)\beta_{2+\d}+A\big)\cdot t^{-\frac{\d}{2+\d}},\quad t>0,
$$
в частности, при $\d=1$
$$
\rho\left((\beta_3t)^{-1/3}S(t),Z\right)\le 0.4614 \Big(\frac{\beta_3}{\sqrt t}\Big)^{2/3}, 
$$
$$
\rho\big(t^{-1/3}S(t),Z\big)\le \big(0.3031\beta_3+0.1583\big) t^{-1/3},\quad t>0.
$$

\begin{example}
Пусть теперь 
$$
\Lambda_t:\eqd\Lambda\left(r,t^{-1}\right)\eqd t\Lambda(r,1)\sim\Gamma\left(r,t^{-1}\right).
$$
% с $t\coloneqq(1-p)/p$, $p\in(0,1)$. 
Тогда $N(t)\sim\NB\big(r,\frac1{1+t}\big)$ в силу леммы~\ref{LemNegBinomial=MixedPoisson}. Положим $d(t)\coloneqq\E\Lambda_t=t\E\Lambda(r,1)=rt\to\infty$ при $t\to\infty.$ Тогда 
$$
\frac{\Lambda_t}{d(t)}\eqd \frac{t\Lambda(r,1)}{rt}\eqd \Lambda(r,r)\eqqcolon\Lambda\sim\Gamma(r,r)
$$
не зависит от $t$, и следовательно, $\delta_t=0$ для всех  $t>0.$
\end{example}

Пусть, как и ранее, $Z\sim N(0,1)$ и независима от $\Lambda(r,t^{-1})$ при всех $r,t>0$. Вспомним далее, что  с.в. $Z\sqrt{\Lambda(r,r)}$
имеет симметризованное гамма-распределение (см. пример~\ref{Ex:SymmGamma=Z*sqrt(Gamma)}):
$$
Z\sqrt{\Lambda(r,r)}\eqd\Lambda\big(r,\sqrt{2r}\,\big) -\Lambda'\big(r,{\sqrt{2r}}\,\big)\eqqcolon\Lambda^{(s)}\big(r,\sqrt{2r}\,\big),
$$
где с.в. $\Lambda\big(r,{\sqrt{2r}}\big)$, $\Lambda'\big(r,{\sqrt{2r}}\big)\sim\Gamma\big(r,{\sqrt{2r}}\big)$ и независимы.

Вычислим 
$$
\E\Lambda_t^{-\d/2}=\E\big(t\Lambda(r,1)\big)^{-\d/2} = \frac{t^{-\d/2}}{\Gamma(r)} \int_0^\infty x^{-\d/2+r-1}e^{-x}\dd x =  \frac{\Gamma(r-\frac\d2)}{\Gamma(r)t^{\d/2}},
$$
предполагая, что $r>\d/2.$ Используя  лемму~\ref{WendelGammaRatioBounds} с $x=r-\d/2$, $s=\d/2$ (так что $x+s=r$), также получаем оценку
$$
\E\Lambda_t^{-\d/2}\le (rt)^{-\d/2}\left(1+\frac{\d}{2r-\d}\right).
$$
Теперь делая замену переменных $t=(1-p)/p,$ $p\in(0,1),$ и замечаяя, что $t\to\infty$ при $p\to0$, получаем следующую оценку скорости сходимости для отрицательных биномиальных сумм с малой вероятностью ``успеха'' $p$.

\begin{corollary}\label{Cor:rho(NegBin,ConvolvedLaplace)}
Пусть с.в. $N_{r,p}\sim\NB(r,p)$ независима от $X_1,X_2,\ldots$ для всех $r>0,$ $p\in(0,1)$ и $\E X=0,$ $\E X^2=1,$ $\beta_{2+\d}\coloneqq\E|X|^{2+\d}<\infty$ с некоторым $\d\in(0,1].$ Тогда для $S_{r,p}\coloneqq X_1+\ldots+X_{N_{r,p}}$ при всех $r>\d/2$ и $p\in(0,1)$ справедливы оценки 
\begin{multline*}
\Delta_{r,p}\coloneqq\rho\Big(S_{r,p}\sqrt{\tfrac{p}{r(1-p)}}, \Lambda^{(s)}\big(r,{\sqrt{2r}}\big)\Big) \coloneqq
\\
\coloneqq\sup_{x\in\R}\biggl|\Prob\Big(S_{r,p}<x\sqrt{\tfrac{r(1-p)}p}\Big)- \int_0^\infty \Phi\Big(\tfrac{x}{\sqrt{\la}}\Big)g_{r,r}(\la)\dd\la\biggr| \le
\\
\le \poisbec(\d)\beta_{2+\d} \frac{\Gamma(r-\frac\d2)}{\Gamma(r)} \Big(\frac{p}{1-p}\Big)^{\d/2}\le \poisbec(\d)\beta_{2+\d}\frac{1+\frac{\d}{2r-\d}}{(1-p)^{\d/2}}\Big(\frac{p}{r}\Big)^{\d/2}.
\end{multline*}
В частности, при $\d=1$ для всех $r>1/2$ и $p\in(0,1)$
$$
\Delta_{r,p}\le 0.3031\cdot\beta_3\frac{\Gamma(r-1/2)}{\Gamma(r)} \sqrt{\frac{p}{1-p}} \le0.3031\beta_3\frac{1+\frac1{2r-1}}{\sqrt{1-p}} \sqrt{\frac{p}{r}}.
$$
\end{corollary}

\begin{myremark}
Заметим, что правая часть последнего неравенства имеет порядок  $\mathcal O\big(\sqrt{p}\big)$ при $p\to0$ для фиксированного $r>0$, но если ещё и $r\to\infty$ (в этом случае ``предельное распределение'' $Z\sqrt{\Lambda(r,r)}$ будет сближаться с нормальным), то скорость сходимости возрастает до $\mathcal O\big(\sqrt{{p}/r}\,\big).$
\end{myremark}

Следствие~\ref{Cor:rho(NegBin,ConvolvedLaplace)} без оценок, полученных с помощью леммы~\ref{WendelGammaRatioBounds}, и с немного худшими значениями констант $M(\d)$ было доказано в работе~\cite{GavrilenkoKorolev2006}.

Итак, как видно из следствия~\ref{Cor:rho(NegBin,ConvolvedLaplace)}, сближение распределений $S_{r,p}\sqrt{\tfrac{p}{r(1-p)}}$ и $\Lambda^{(s)}\big(r,{\sqrt{2r}}\big)$ имеет место и при фиксированном $p$, если $r\to\infty$: в этом случае оба распределения притягиваются к нормальному. Используя оценку точности нормальной аппроксимации для дисперсионного гамма-нормального (симметризованного гамма) распределения 
$$
\rho(Z\sqrt{\Lambda(r,r)},Z)\le \frac{A_r}{r},\quad A_r=
\begin{cases}
0.8593,& 0<r\le1.59,\\
\displaystyle \frac{\pi^{-1}}{r-1},&r>1.59,
\end{cases}
\le 0.8593,
$$
полученную в работе~\cite{Shevtsova2018}, и неравенство треугольника для равномерной метрики, из следствия~\ref{Cor:rho(NegBin,ConvolvedLaplace)} легко получить следующую оценку точности \textit{нормальной} аппроксимации нормированных отрицательных биномиальных случайных сумм.

\begin{corollary}
В условиях следствия~\ref{Cor:rho(NegBin,ConvolvedLaplace)} для всех $r>\delta/2$ и ${p\in(0,1)}$
\begin{equation}\label{NegBin-NormalApproxErrBound}
\rho\big(S_{r,p}\sqrt{\tfrac{p}{r(1-p)}}, Z\big)\le \poisbec(\d)\beta_{2+\d}\frac{1+\frac{\d}{2r-\d}}{(1-p)^{\d/2}}\Big(\frac{p}{r}\Big)^{\d/2} +\frac{A_r}{r}.
\end{equation}
\end{corollary}
Задаваясь целью выбора оптимального значения $p\in(0,1)$ в полученной оценке, приравняем порядки двух ее составляющих, игнорируя (асимптотические) константы:
$$
\bigg(\frac{p}{r(1-p)}\bigg)^{\d/2}=\frac1{r}\quad \Leftrightarrow\quad p=\left (1+r^{2/\delta-1}\right )^{-1}\in(0,1).
$$
При данном $p$ из~\eqref{NegBin-NormalApproxErrBound} получаем
$$
\rho\left (r^{-1/\d}S_{r,p}, Z\right )\le \frac1r\left(\poisbec(\d)\beta_{2+\d}\big(1+\tfrac{\d}{2r-\d}\big) +A_r\right),\quad r>\d/2,
$$
в частности, при $\d=1$
$$
\rho\bigg(\frac{S_{r,p}}{r}, Z\bigg)\le \frac{0.3031\big(1+(2r-1)^{-1}\big)\beta_3 +A_r}{r},\quad r>1/2.
$$

%%%%%%%%%%%%%%%%
%%%%%%%%%%%%%%%%

\smallskip

Если же $r=1$, то  $N_{r,p}\eqqcolon N_p\sim G(p)$ имеет геометрическое распределение, а предельным законом для нормированных геометрических случайных сумм будет симметризованное  (двустороннее) показательное распределение (распределение \textit{Лапласа}) с параметром~$\sqrt2$ и плотностью $\ell(x)=e^{-\sqrt2|x|}/\sqrt2.$ Обозначим $L(x)\coloneqq\int_{-\infty}^x\ell(y)\dd y$, $x\in\R,$ ф.р. Лапласа с плотностью $\ell.$ Учитывая, что для всех $\d\in(0,1]$
$$
\Gamma(1-\d/2)\le \Gamma(1/2)=\sqrt{\pi}<1.7725
$$
в силу монотонности $\Gamma(x)$ при $0<x\le1$, из следствия~\ref{Cor:rho(NegBin,ConvolvedLaplace)} получаем

\begin{corollary} Пусть с.в. $N_{p}\sim G(p)$ независима от $X_1,X_2,\ldots$ для всех $p\in(0,1)$ и $\E X=0,$ $\E X^2=1,$ $\beta_{2+\d}\coloneqq\E|X|^{2+\d}<\infty$ для некоторого $\d\in(0,1]$. Тогда при любом  $p\in(0,1)$ и $\d\in(0,1]$
\begin{multline*}
\Delta_p\coloneqq\sup_{x\in\R}\bigg| \Prob\bigg(\smallsum_{k=1}^{N_{p}}X_k<x\sqrt{\tfrac{1-p}p}\bigg)- L(x)\bigg| \le
\\
\le M(\d)\Gamma\big(1-\d/2\big)\beta_{2+\d}\Big(\frac{p}{1-p}\Big)^{\d/2} \le  \sqrt\pi M(\d)\beta_{2+\d}\Big(\frac{p}{1-p}\Big)^{\d/2},
\end{multline*}
в частности,
$$
\Delta_p \le  0.3031\sqrt{\pi}\beta_3\sqrt{\frac{p}{1-p}} < 0.5373\cdot\beta_3\sqrt{\frac{p}{1-p}}.
$$
\end{corollary}

В книге~\cite[Теорема~8.3.1]{KruglovKorolev1990} (также см.~\cite[Гл.\,2, теорема\,2.1.3, с.\,124]{KorolevBeningShorgin2011}) получена оценка
$$
(1-p)\sup_{x\in\R}\bigg| \Prob\bigg(\sqrt p\smallsum_{k=1}^{N_{p}}X_k<x\bigg)- L(x)\bigg|  \le \frac{C_0(\d)\Gamma(1-\frac\d2)}{(\ln2)^{1-\d/2}}\beta_{2+\d}p^{\d/2}+ \frac{p}2,
$$
где $C_0(\d)$~--- константы из классического неравенства Берри--Эссеена~\eqref{BEineq} для о.р.с.в., в частности, $C_0(1)\le0.469$, которая отличается от полученной здесь оценки множителем $C_0(\d)(\ln2)^{\d/2-1}$ (вместо $M(\d)$), наличием второго слагаемого в правой части $p/2$ и общей нормировкой $(1-p)^{-1}>(1-p)^{-\d/2}$, увеличивающей итоговую оценку. Кроме того, немного различается нормировка случайных сумм. Сравнивая эти две оценки, с уверенностью можно утверждать, что полученная нами~--- точнее, по крайней мере при $\d=1$, т.к. $C_0(1)/\sqrt{\ln2}>0.5633>0.3031=M(1)$ и при этом она не содержит второго слагаемого.

\section[Случайные суммы с безгранично делимыми индексами]{Случайные суммы с безгранично делимыми в классе неотрицательных целочисленных случайных величин индексами}
%\section{Сложные пуассоновские случайные суммы}

%\begin{definition}
Пусть $N$ --- неотрицательная целочисленная с.в. с пр.ф. 
$$
\psi_N(z)=\E z^N=\Prob(N=0)+\sum_{n=1}^\infty z^n\Prob(N=n),\quad z\in\C,\  |z|\le1,
$$
а $X$ --- произвольная с.в. с х.ф.
$$
f_X(t)=\E e^{itX},\quad t\in\R.
$$
Всюду в данном разделе под $\{N,X\}$ будем понимать с.в. с х.ф. 
$$
f_{\{N,X\}}(t)=\E e^{it\{N,X\}}\coloneqq\psi_N(f_X(t)),\quad \in\R,
$$
т.е.  
$$
\{N,X\}:\eqd X_1+\ldots+X_N\eqqcolon S_N,
$$
где с.в. $N,X_1,X_2,\ldots$ независимы, а $X,X_1,X_2,\ldots$  одинаково распределены. При этом, как и ранее, полагаем $S_N=0$ при $N=0,$ с.в. $X$  называем \textit{случайным слагаемым}, $N$ --- \textit{индексом}, а $S_N$ и $\{N,X\}$ --- случайными  суммами.
%\end{definition}

Ясно, что если $\Prob(X\in\N_0)=1$, то $\{N,X\}$ также неотрицательная целочисленная с.в. с пр.ф. 
$$
\psi_{\{N,X\}}(z)=\psi_N(\psi_X(z)),\quad z\in\C,\ |z|\le1.
$$

Напомним, что если $N\eqd N_\la\sim Pois(\la),$ то распределение с.в. $\{N_\la,X\}$ называется \textit{обобщённым (сложным, составным) пуассоновским}, при этом
$$
f_{\{N_\la,X\}}(t)=\exp\{\la(f_X(t)-1)\},\quad t\in\R.
$$
Если же $Y$ --- неотрицательная целочисленная с.в. с пр.ф. $\psi_Y$, то $\{N_\la,Y\}$ тоже неотрицательная целочисленная с.в. с пр.ф.
$$
\psi_{\{N_\la,Y\}}(z)=e^{\la(\psi_Y(z)-1)},\quad z\in\C,\ |z|\le1.
$$

\begin{excersize} Пусть $X$~--- произвольная с.в., $N_\la\sim Pois(\la)$. Доказать, что $\{N_\la,X\}$ является целочисленной неотрицательной случайной величиной тогда и только тогда, когда таковой является с.в.~$X$.
\end{excersize}

\begin{excersize} Пусть $X$~--- произвольная с.в., $N$~--- неотрицательная целочисленная с.в. с $\Prob(N=1)>0$. Доказать, что $\{N,X\}$ является целочисленной неотрицательной случайной величиной тогда и только тогда, когда таковой является с.в.~$X$.
\end{excersize}
%% Решение:
%Пусть $N=N_\la\sim Pois(\la)$ (на самом деле важно только то, что $\Prob(N\ge0)=1$ и $\Prob(N=1)>0$), $M=\{N,X\}\in\N_0$ п.н. Если $p:=\Prob(X<0)>0$, то $\Prob(M<0)=\sum_{n=0}^\infty\Prob(N=n)\Prob(X_1+\ldots+X_n<0)\ge \sum_{n=0}^\infty\Prob(N=n)\Prob(X_1<0,\ldots,X_n<0)= \sum_{n=0}^\infty\Prob(N=n)p^n>0$, что противоречит $\Prob(M\ge0)=1$.
%
%Если $\Prob(X=a)>0$ для некоторого $a>0,a\notin\N$, то $\Prob(M=a)=\sum_{n=0}^\infty\Prob(N=n)\Prob(X_1+\ldots+X_n=a)\ge \Prob(N=1)\Prob(X=a)>0$, что противоречит $\Prob(M\in\N_0)=1$.

Класс целочисленных сложных пуассоновских распределений чрезвычайно широк. В теореме~\ref{ThInfDivisInteger=CompoundPoisson} ниже будет показано, что любая безгранично делимая в классе неотрицательных целочисленных случайная величина (см. определение~\ref{Def:InfDivisibilityNonnegInteger} ниже) является сложной пуассоновской. В частности, неотрицательный целочисленный однородный случайный процесс с независимыми приращениями, стартующий из нуля, представляет собой совокупность случайных величин, имеющих такое безгранично делимое распределение, а значит, сложное пуассоновское. 

Приведём еще несколько конкретных примеров сложных пуассоновских распределений.

\begin{example}[Отрицательное биномиальное распределение]
Ранее мы показали, что  отрицательное биномиальное распределение является смешанным пуассоновским (см. лемму~\ref{LemNegBinomial=MixedPoisson}). Оно также является и сложным пуассоновским (более конкретно~--- см. лемму~\ref{LemNegBinom=CompoundPoisson} на стр.\,\pageref{LemNegBinom=CompoundPoisson} ниже). 
\end{example}

\begin{example}[Кратное распределение Пуассона, см., например,~\protect{\cite[Т.\,I, Гл.\,XII, \S\,2]{Feller1967}}]\label{Exa:MultiplePoissonDistr}
В 1924\,г. Эггенбергер~\cite{Eggenberger1924}
% Mitteilungen / Schweizerische Vereinigung der Versicherungsmathematiker = Bulletin / Association Suisse des Actuaires = Bulletin / Swiss Association of Actuaries
показал, что числа летальных исходов при скарлатине и оспе в Швейцарии, а также при взрывах паровых котлов имеют сложное пуассоновское распределение. В 1965\,г. Куппер~\cite{Kupper1965} использовал аналогичное распределение для описания числа жертв катастроф на транспорте. В действительности во всех этих примерах возникают кратные пуассоновские распределения, которые одновременно являются сложными пуассоновскими. Убедимся в этом. 

Вполне естественно предположить, что, во-первых, вспышки болезней/аварии с паровыми котлами или на транспорте (происшествия) происходят независимо друг от друга; во-вторых, число $\xi_j$ происшествий, вызвавших ровно $j$ летальных исходов, в силу малой вероятности каждого отдельного происшествия и огромного числа потенциально возможных, по теореме о редких событиях близко к распределению Пуассона с некоторым параметром $\la_j.$ Тогда общее число жертв всех происшествий рассматриваемого типа имеет вид
\begin{equation}\label{MultiPoissonDistr}
N=\xi_1+2\xi_2+\ldots+n\xi_n,
\end{equation}
где параметр $n$ называется \textit{максимальной множественностью} и, вообще говоря, может быть бесконечным. Модели вида~\eqref{MultiPoissonDistr} использовались также в эпидемиологии, энтомологии и бактериологии~\cite{GreenwoodYule1920}, при изучении пространственного распределения особей различных видов в биологии и экологии~\cite{Skellam1952}. 

В предположении, что  в~\eqref{MultiPoissonDistr} случайные величины $\xi_k\sim Pois(\la_k)$ и независимы, пр.ф. с.в. $N$ можно записать в виде
$$
\psi_N(z)\coloneqq \E z^N=\E \prod_{k=1}^n z^{k\xi_k}=\prod_{k=1}^n\E  z^{k\xi_k}=\prod_{k=1}^n\psi_{\xi_k}(z^k) =
$$
$$
=\prod_{k=1}^n\exp\left\{\la_k(z^k-1)\right\}=\exp\left\{\smallsum_{k=1}^n\la_k(z^k-1)\right\} \eqqcolon\exp\left\{\la(\psi(z)-1)\right\}
$$
с  $\la\coloneqq\sum_{k=1}^n\la_k,$ 
$$
\psi(z)\coloneqq  \sum_{k=1}^n z^k\cdot\frac{\la_k}\la= \E z^Y,\quad z\in\C,\ |z|\le1,
$$
где $Y$ --- положительная целочисленная с.в. с $\Prob(Y=k)={\lambda_k}/\la,$ $k\hm=1,\ldots,n$. Отметим также, что вышеприведенные выкладки остаются справедливыми и для $n=\infty$ в предположении, что ряд $\sum_{k=1}^\infty\lambda_k$ сходится. 
\end{example}

%\begin{example}[Кратное распределение Пуассона~\protect{\cite[Т.\,I, Гл.\,XII, \S\,2]{Feller1967}}]\label{Exa:MultiplePoissonDistr}
%Допустим, что мы классифицируем автомобильные катастрофы в зависимости от числа пострадавших машин на одиночные, двойные и т.д. Предположим, далее, что число $\xi_1$ одиночных, $\xi_2$ двойных и т.д. катастроф имеет распределение Пуассона со средними значениями $\la_1,$ $\la_2,\ldots$ и что случайные величины $\xi_1,\xi_2,\ldots$ независимы. Тогда общее число пострадавших машин
%$$
%N\coloneqq\sum_{k=1}^\infty k\xi_k
%$$
%имеет пр.ф.
%$$
%\psi_N(z)\coloneqq \E z^N=\E \prod_{k=1}^\infty z^{k\xi_k}=\prod_{k=1}^\infty \E \psi_{\xi_k}(z^k) =\prod_{k=1}^\infty\exp\left\{\la_k(z^k-1)\right\}=
%$$
%$$
%=\exp\left\{\smallsum_{k=1}^\infty\la_k(z^k-1)\right\} \eqqcolon\exp\left\{\la(\psi(z)-1)\right\}
%$$
%с  $\la\coloneqq\sum_{k=1}^\infty\la_k,$ 
%$$
%\psi(z)\coloneqq  \sum_{k=1}^\infty z^k\cdot\frac{\la_k}\la= \E z^Y,\quad z\in\C,\ |z|\le1,
%$$
%где $Y$ --- положительная целочисленная с.в. с $\Prob(Y=k)={\lambda_k}/\la,$ $k\in\N.$
%\end{example}

\subsection{Безграничная делимость в классе неотрицательных целочисленных распределений}

\begin{definition}\label{Def:InfDivisibilityNonnegInteger}
Распределение неотрицательной целочисленной с.в. $N$  называется \textit{безгранично делимым в классе распределений неотрицательных целочисленных} (БДНЦ) случайных величин, если для любого $m\in\N$ найдётся распределение неотрицательной целочисленной с.в. $N_m$, т.ч. $N\eqd\{m,N_m\},$ т.е.
$$
N\eqd N_{m,1}+\ldots+N_{m,m} \quad\text{для любого }\ m\in\N,
$$
где $N_{m,i}$ ---  неотрицательные целочисленные н.о.р.с.в., или, в терминах пр.ф.,
\begin{equation}\label{InfDivisIntegerDef}
\psi_N(z)=(\psi_{m}(z))^m,\quad z\in\C,\ |z|\le1, \quad\text{для любого }\ m\in\N,
\end{equation}
где $\psi_m$ является некоторой производящей функцией (случайной величины $N_m$). При этом с.в. $N$ также  называется \textit{безгранично делимой в классе неотрицательных целочисленных} (БДНЦ) с.в.
\end{definition}

\begin{myremark} 
Если $N$ --- БДНЦ с.в., то необходимо
$$
\Prob(N=0)>0.
$$
Действительно, если бы $\Prob(N=0)=0,$ то в представлении $N\hm\eqd N_{m,1}+\ldots+N_{m,m}$, где $N_{m,i}$ --- \textit{неотрицательные целочисленные} с.в.,  все с.в. $N_{m,i}$ оказались бы п.н. \textit{положительными целочисленными}, т.е. $N_{m,i}\gep1$, и тогда бы $N\gep m$ для любого $m\in\N,$ что невозможно.
\end{myremark}

\begin{myremark} 
Неотрицательная целочисленная безгранично делимая с.в. необязательно является 
БДНЦ. Например, пусть $N\hm\coloneqq N_\la+1$ с $N_\la\sim Pois(\la)$ имеет сдвинутое распределение Пуассона  c произвольным  параметром $\la>0$:
$$
\Prob(N=k)=\frac{\la^k}{k!}e^{-\la},\quad k\in\N.
$$
Тогда $N$ не может быть БДНЦ с.в., т.к. $\Prob(N=0)=0,$ но в силу безграничной делимости пуассоновского распределения ($N_\la\hm\eqd\left\{m,N_{\la/m}\right\}$, $m\in\N$) очевидно, что 
$$
N\eqd\left\{m,N_{\la/m}+\tfrac1m\right\}\quad\text{для любого }\ m\in\N,
$$
при этом с.в. $N_{\la/m}+1/m$, разумеется, не является целочисленной для $m\ge2.$
\end{myremark}

Обратим внимание, что в приведенном выше примере отсутствие свойства БДНЦ у сдвинутого пуассоновского распределения тривиально вытекало из отсутствия у последнего атома в нуле. Оказывается, что наличие этого свойства является своего рода критерием того, что безгранично делимое распределение является БДНЦ.

\begin{theorem}
Пусть $X$~--- неотрицательная целочисленная с.в. с $\Prob(X=0)>0$. Тогда $X$ является безгранично делимой тогда и только тогда, когда $X$ является БДНЦ.
\end{theorem}

\begin{proof}
Достаточность очевидна. Докажем необходимость.    Пусть $X$~--- неотрицательная целочисленная безгранично делимая с.в. с $\Prob(X=0)>0$. Тогда $X\eqd X_{n,1}+\ldots+X_{n,n}$ для любого $n\in\N$, где $\{X_{n,k}\}_{k=1}^n$~--- н.о.р.с.в. (не обязательно целочисленные!) при каждом $n\in\N$. Покажем, что необходимо $\Prob(X_{n,1}\in\N_0)=1$ для всех $n\in\N$. Это и будет означать, что $X$ является БДНЦ. 

%Во-первых, заметим, что $X_{n,1}\gep0$, т.к. если бы $p:=\Prob(X_{n,1}\hm<0)>0$, то 
%$$
%\Prob(X<0)\ge\Prob(X_{n,k}<0\ \forall k=\overline{1,n})=p^n>0,
%$$
%что противоречит предположению неотрицательности с.в. $X$. Таким образом, слагаемые $X_{n,k}$ неотрицательны.

Пусть $n\in\N$~--- фиксировано. Прежде всего заметим, что слагаемые $X_{n,k}$ обязаны иметь атом в нуле: $p:=\Prob(X_{n,1}=0)>0$, т.к. если бы $p=0$, т.е. $\Prob(X_{n,1}>0)=1$, то 
$$
\Prob(X>0)\ge \Prob(X_{n,k}>0\ \forall k=\overline{1,n})=1^n=1,
$$
что противоречит наличию атома в нуле у с.в. $X$. Предполагая теперь, что $q:=\Prob(X_{n,1}\in A)>0$ для некоторого $A\in\mathcal B(\R)$ такого, что  $A\cap\N_0=\emptyset$ (т.е. $A$ не содержит целых неотрицательных точек), получаем
$$
\Prob(X\in A)\ge\Prob(X_{n,1}\in A,X_{n,2}=\ldots=X_{n,n}=0)=qp^{n-1}>0,
$$
что противоречит тому, что $\Prob(X\in\N_0)=1$ (и, следовательно, $\Prob(X\hm\in A)=0$).
\end{proof}

Следующая теорема устанавливает довольно любопытный факт: БДНЦ распределениями являются  сложные пуассоновские и только они.

\begin{theorem}\label{ThInfDivisInteger=CompoundPoisson}
Неотрицательная целочисленная с.в. $N$ является БДНЦ тогда и только тогда, когда найдётся $\la>0$ такое, что $N\eqd\{N_\la,Y\},$ где $N_\la\sim Pois(\la),$ а $Y$~--- некоторая неотрицательная целочисленная с.в. 
%При этом можно положить $Y\eqp0,$ $\la$~--- произвольно, если $N\eqp0,$ и
%$$
%\la=-\ln\Prob(N=0)>0,\quad \Prob(Y=k)=\frac{\gamma_k}{\la}, \quad k\in\N,
%$$
%где $\gamma_k$ --- (единственным образом определяемые) коэффициенты разложения логарифма пр.ф. $\psi$ с.в. $N$ в ряд Маклорена
%$$
%\ln\psi_N(z)=\gamma_0+\sum_{k=1}^\infty\gamma_kz^k, 
%$$
%справедливого в некоторой вещественной окрестности нуля $|z|<a\le1,$ $z\in\R,$ если с.в. $N$ не вырождена в нуле.
\end{theorem}

\begin{proof}
\textit{Достаточность} является прямым следствием безграничной делимости пуассоновского и сложного пуассоновского распределений. Действительно, если  $N\eqd\{N_\la,Y\}$, где $N_\la \hm\sim Pois(\la)$, а  $Y$~---  неотрицательная целочисленная с.в. с пр.ф. $h$, то пр.ф. с.в. $N$
\begin{equation}\label{psi(z)=exp(lambda(h(z)-1))}
\psi(z)\coloneqq\E z^N=\psi_{N_\la}(h(z))=e^{\la(h(z)-1)}%\quad z\in\C,\ |z|\le1,
\end{equation}
является $m$-й степенью пр.ф. $e^{\frac{\la}m(h(z)-1)}$ неотрицательной целочисленной с.в.  (пуассоновской случайной суммы) $\{N_{\la/m},Y\}$, т.е. $N\eqd\{m,\{N_{\la/m},Y\}\}$,  для любого $m\in\N$, что и означает безграничную делимость с.в. $N$ в классе неотрицательных целочисленных с.в.

\textit{Необходимость.} Прежде всего заметим, что распределения неотрицательных целочисленных случайных величин совпадают тогда и только тогда, когда совпадают их пр.ф. в некоторой невырожденной \textit{вещественной} окрестности нуля. Таким образом, достаточно показать справедливость~\eqref{psi(z)=exp(lambda(h(z)-1))} в такой окрестности с пр.ф.~$h$ некоторой неотрицальной целочисленной с.в.

Если $N\eqp0$, то, очевидно, $N\eqd\{N_\la,0\}$ с произвольным $\la\hm>0$. Пусть теперь БДНЦ с.в.  $N$ не вырождена в нуле. Тогда $%p_0\coloneqq
\Prob(N\hm=0)\in(0,1)$, и следовательно, пр.ф. с.в. $N$
\begin{equation}\label{psi(z)=sum(p_kz^k)}
\psi(z)=\Prob(N\hm=0)+\sum_{k=1}^\infty z^n\Prob(N\hm=n)
\end{equation}
должна удовлетворять условию $0<\psi(z)<1$  в некоторой достаточно малой вещественной окрестности нуля $|z|<a\le1,$ $z\in\R$. Но в таком случае определён и логарифм 
$$
\ln\psi(z)=\ln(1-[1-\psi(z)])=-\sum_{k=1}^\infty\frac{(1-\psi(z))^k}{k},
$$
который в силу~\eqref{psi(z)=sum(p_kz^k)} в той же области можно разложить в степенной ряд
$$
\ln\psi(z)\ =\ \gamma_0+\sum_{k=1}^\infty\gamma_kz^k,\quad -a<z<a,
$$
где $\gamma_0=\ln\psi(0)=\ln\Prob(N=0)<0$ и $\sum_{k=0}^\infty\gamma_k=\ln\psi(1)=\ln1=0,$ т.е. 
$$
-\gamma_0=\gamma_1+\gamma_2+\ldots\ .
$$
Таким образом, формальное тождество~\eqref{psi(z)=exp(lambda(h(z)-1))} 
$$
\la(h(z)-1)\ =\ \gamma_0+\sum_{k=1}^\infty\gamma_kz^k
$$
имеет место в области $-a<z<a$ с 
$$
\la\coloneqq-\gamma_0=-\ln\Prob(N=0)>0, \quad
h(z)\coloneqq 1+ \frac1{\la}\sum_{k=0}^\infty\gamma_kz^k =\frac1{\la}\sum_{k=1}^\infty\gamma_kz^k,
$$
при этом $h(1)=-\sum_{k=1}^\infty\gamma_k/\gamma_0=1,$ и для того чтобы  $h(z)$ являлась производящей функцией некоторого вероятностного распределения при  $-a<z<a$, остаётся показать, что $\gamma_k\ge0$ для всех $k\in\N.$ 

Допустим обратное: пусть $\gamma_r<0$ при некотором $r\ge1.$ Для того чтобы избежать слишком длинных формул, положим
$$
A(z)\coloneqq\sum_{k=1}^{r-1}\gamma_kz^k,\quad B(z)\coloneqq\sum_{k=r+1}^\infty\gamma_kz^k,\quad \frac1m\eqqcolon\eps,\quad m\in\N,
$$
так что
$$
\psi^{1/m}(z)=e^{\eps\gamma_0}\cdot e^{\eps A(z)}\cdot e^{\eps\gamma_rz^r}\cdot e^{\eps B(z)},\quad -a<z<a.
$$
По предположению о БДНЦ с.в. $N$, $\psi^{1/m}(z)=\phi_0+\sum_{k=1}^\infty\phi_kz^k$ для любого $m\in\N$, где $\phi_k\ge0$ для всех $k\in\N_0$. Рассмотрим, в частности, коэффициент $\phi_r$ при  $z^r.$ Степенной ряд $B(z)$ содержит только члены со степенями, большими, чем $r$, и  следовательно,  не влияет на $\phi_r.$ Поэтому $\phi_r$ есть коэффициент при $z^r$ в выражении
$$
e^{\eps\gamma_0}\big(1+\eps A(z)+\frac12\eps^2A^2(z)+\ldots\big) \left(1+\eps\gamma_rz^r\right)=
$$
$$
=e^{\eps\gamma_0}\bigg(1+\eps A(z)+\eps^2\sum_{k=2}^\infty\eps^{k-2}\frac{A^k(z)}{k!}+\eps\gamma_rz^r+ \eps^2\gamma_rz^r\sum_{k=1}^\infty\eps^{k-1}\frac{A^k(z)}{k!}\bigg).
$$
Так как $A(z)$ есть полином степени не выше, чем $r-1$, то легко установить, что 
$$
\phi_r=e^{\eps\gamma_0}\left(\eps\gamma_r+\eps^2p(\eps)\right),
$$
где $p(\eps)$~--- коэффициент при $z^r$ в выражении
$$
\sum_{k=2}^\infty\eps^{k-2}\frac{A^k(z)}{k!} +\gamma_rz^r\sum_{k=1}^\infty\eps^{k-1}\frac{A^k(z)}{k!},
$$
откуда видно, что $p(\eps)$~--- полином (степенной ряд) от $\eps$ и поэтому $\phi_r=\eps\gamma_r+\mathcal O(\eps^2)$ при $\eps\to0$. Если $\gamma_r<0,$ то $\phi_r$ будет  отрицательно при $\eps$ достаточно малом, что невозможно (противоречит тому, что $N$ --- БДНЦ с.в.). Следовательно, $\gamma_r\ge0$ при каждом  $r\in\N,$ что и требовалось доказать.

Таким образом, с.в. $Y$ с пр.\,ф. $h(z)$ в представлении $N\eqd\{N_\la,Y\}$ может быть  выбрана положительной (без атома в нуле) целочисленной с
$$
%\la=-\ln\Prob(N=0)>0,\quad 
\Prob(Y=k)=\frac{\gamma_k}{\la}=-\frac{\gamma_k}{\ln\Prob(N=0)}, \quad k\in\N.
$$
\end{proof}

\begin{excersize}\label{Ex:N=(N_la,Y),Y>0)}
Доказать, что для любой неотрицательной целочисленной невырожденной в нуле с.в.~$Y$ и $\la>0$ найдётся положительная целочисленная с.в. $Y'$ и $\la'>0$ такие, что $\{N_\la,Y\}\eqd\{N_{\la'},Y'\}$, где $N_\la\sim Pois(\la)$.
\end{excersize}
%Решение:
%$\lambda'=\lambda(1-\Prob(Y=0))$, $\Prob(Y'=k)=\Prob(Y=k)/(1-\Prob(Y=0))$

Заметим, что представление $N\eqd\{N_\la,Y\}$ для БДНЦ с.в. $N$, доказанное в теореме~\ref{ThInfDivisInteger=CompoundPoisson}, не единственно, а требование положительности с.в. $Y$ в этом представлении для вырожденной в нуле с.в. $N$ делает его единственным.

%Пусть теперь $N\eqd\{N_\la,Y\}$, $S\eqd\{N,X\}\eqd\{\{N_\la,Y\},X\}$, где $N_\la\sim Pois(\la).$ 
Следующая лемма показывает, что пуассоновской случайной суммой является не только каждая  БДНЦ с.в., но и случайная сумма с произвольным БДНЦ индексом.

\begin{lemma}[см.~\cite{Shorgin1996TVP}]\label{Lem((N_la,Y),X)=(N_la,(Y,X))}
Если с.в. $N_\la\sim Pois(\la),$ $\la>0,$ и $Y$~--- неотрицательная целочисленная с.в., то для любого распределения с.в. $X$
$$
\{\{N_\la,Y\},X\}\ \eqd\ \{N_\la,\{Y,X\}\},\quad \la>0.
$$
\end{lemma}

\begin{proof}
Пусть с.в. $S:\eqd\{\{N_\la,Y\},X\}$. Тогда, учитывая свойства случайных сумм, х.ф. $f_S(t)\hm=\E e^{itS},$ $t\in\R,$ можно записать в виде
%$$
%f_S(t)=\psi_{\{N_\la,Y\}}(f_X(t))=\psi_{N_\la}\left(\psi_Y(f_X(t))\right)= 
%$$
%$$
%=\psi_{N_\la}\left(f_{\{Y,X\}}(t)\right)=f_{\{N_\la,\{Y,X\}\}}(t)
%$$
%для любого $t\in\R$, 
$$
f_S=\psi_{\{N_\la,Y\}}(f_X)=\psi_{N_\la}\left(\psi_Y(f_X)\right)
=\psi_{N_\la}\left(f_{\{Y,X\}}\right)=f_{\{N_\la,\{Y,X\}\}},
$$
откуда в силу теоремы единственности и вытекает требуемое.
\end{proof}

\subsection{Оценки точности нормальной аппроксимации}

Как и ранее, относительно случайного слагаемого $X$ в данном разделе мы будем предполагать, что 
$$
\beta_{2+\d}\coloneqq\E|X|^{2+\d}\in(0,\infty)\quad \text{для некоторого}\quad \d\in(0,1]
$$
(откуда, в частности, вытекает, что с.в. $X$ не вырождена в нуле) и положим  $\beta_r\coloneqq\E|X|^r$, $0<r\le2+\d,$ $a\coloneqq\E X,$ $\sigma^2\coloneqq\D X=\beta_2-a^2,$
$$
\widetilde X\coloneqq\frac{X-\E X}{\sqrt{\D X}}=\frac{X-a}{\sigma},\quad \text{если }\ \D X>0.
$$ 
Введём  \textit{нецентральную} и  \textit{центральную} ляпуновские  дроби порядка~$2+\d$
$$
L_1^{2+\d}(X)\coloneqq \frac{\E|X|^{2+\d}}{(\E X^2)^{(2+\d)/2}} =\frac{\beta_{2+\d}}{\beta_2^{1+\d/2}},\quad \beta_2>0,
$$
$$
L_0^{2+\d}(X)\coloneqq \frac{\E|X-\E X|^{2+\d}}{(\D X)^{(2+\d)/2}}=L_1^{2+\d}(X-\E X) =L_1^{2+\d}(\widetilde X) %=\E|\widetilde X|^{2+\d}
,\quad \sigma^2>0.
$$
%для невырожденного распределения~$X$, соответственно. 
Заметим, что для вырожденного распределения $X\eqp a\neq0$
$$
L_1^{2+\d}(X)=1\quad\text{для всех}\quad \d\in(0,1].
$$

Пусть $N$~--- неотрицательная целочисленная с.в. с $\E N^2<\infty,$
$$
S_N\coloneqq\{N,X\},\quad \widetilde S_N\coloneqq\frac{S_N-\E S_N}{\sqrt{\D S_N}} =\frac{S_N-a\E N}{\sqrt{\sigma^2\E N+a^2\D N}}.
$$
Далее, пусть, как и ранее, $Z\sim N(0,1)$~--- с.в. с ф.р. $\Phi$, $\rho$~--- равномерная метрика, в частности,
$$
\rho(\widetilde S_N,Z)= \sup_{x\in\R}\abs{\Prob(\widetilde S_N<x)-\Phi(x)}.
$$
Всюду в данном разделе $N_\la\sim Pois(\la)$~--- с.в., имеющая пуассоновское распределение с параметром $\la>0$.

Для случая, когда 
$$
N\eqd N_\la\eqd\{N_\la,1\},
$$
величину $\rho(\widetilde S_N,Z)$ в сделанных выше предположениях можно оценить с помощью аналога неравенства Берри--Эссеена для пуассоновских случайных сумм из  теоремы~\ref{ThBEineqPoisSum} следующим образом:
\begin{equation}\label{BEineqPoisSumForInfDivIndexes}
\rho(\widetilde S_N,Z)\le \frac{\poisbec(\d)}{\la^{\d/2}}L_1^{2+\d}(X),\quad\la>0,
\end{equation}
где константы $\poisbec(\d)$ определены в теореме~\ref{ThBEineqPoisSum}, в частности, $\poisbec(1)\hm\le0.3031.$ Наша цель~--- обобщить оценку~\eqref{BEineqPoisSumForInfDivIndexes} на случай, когда $N$ является произвольной невырожденной БДНЦ с.в., т.е., в силу теоремы~\ref{ThInfDivisInteger=CompoundPoisson} и упражнения~\ref{Ex:N=(N_la,Y),Y>0)}, когда $N\hm\eqd\{N_\la,Y\}$ с некоторой положительной целочисленной с.в. $Y$.

При доказательстве нам понадобятся оценки абсолютных моментов сумм через соответствующие  моменты слагаемых, для чего, наряду с известным неравенством 
\begin{equation}\label{(sum(x_i))^r<=n^(r-1)sum(x_i)^r}
\Big(\smallsum_{i=1}^nx_i\Big)^r\le n^{r-1}\smallsum_{i=1}^n x_i^r,\quad r\ge1,\ x_i\ge0,\ i=1,\ldots,n,
\end{equation}
вытекающим, например, из неравенства Йенсена для степенной функции, мы также будем пользоваться следующей леммой, которая позволит в некоторых случаях снизить порядок роста $n^{r-1}$ коэффициента при сумме в правой части до $n^{r/2-1}$.

\begin{lemma}[см.~\cite{DharmadhikariJogdeo1969}]
\label{LemDharmadhikariJogdeo}
Пусть $X_1,\ldots,X_n$~--- независимые с.в. с $\E X_i=0$ и $\E|X_i|^r<\infty,$ $i=\overline{1,n},$ для некоторого вещественного $r\ge2.$ Тогда
$$
\E\Big|\smallsum_{i=1}^n X_i\Big|^r\le K_r\cdot n^{r/2-1}\smallsum_{i=1}^n\E|X_i|^r,
$$
где 
$$
K_r\coloneqq \tfrac12{r(r-1)}\max\left\{1,2^{r-3}\right\}  \bigg[ 1+\frac2r\bigg(\smallsum_{k=1}^m\frac{k^{2m-1}}{(k-1)!}\bigg)^{\frac{r-2}{2m}} \bigg],
$$
а целое $m$ таково, что $2m\le r<2m+2.$ В частности, при $r\hm=2+\d\in(2,3]$ имеем $m=1,$ $K_{2+\d}\hm=(1+\d)(4+\d)/2$ и для одинаково распределённых $X_1,\ldots,X_n$ получаем
$$
\E\Big|\smallsum_{i=1}^n X_i\Big|^{2+\d}\le \tfrac12(1+\d)(4+\d) n^{1+\d/2}\E|X_1|^{2+\d}.
$$
\end{lemma}

\begin{theorem}\label{ThCompoundPoisRSNormApproxErrBound}
Пусть $N$~--- невырожденная в нуле БДНЦ с.в. и $N\eqd\{N_\la,Y\}$~--- ее представление в виде сложной пуассоновской суммы из теоремы~\ref{ThInfDivisInteger=CompoundPoisson}, где $Y$~--- положительная целочисленная с.в., $N_\la\sim Pois(\la)$, $\la>0.$ Предположим, что $\E Y^{2+\d}<\infty$ для некоторого $\d\in(0,1]$. Тогда для любого $\la>0$ справедливы оценки:\\
{\rm (i)}~при любом $a\in\R$$:$
\begin{equation}\label{CompoundPoisRSNormApproxErrBound}
\rho(\widetilde S_N,Z)\le \frac{\poisbec(\d)}{\la^{\d/2}}\cdot \frac{\beta_{2+\d}\E Y^{2+\d}} {(\sigma^2\E Y+a^2\E Y^2)^{1+\d/2}},
\end{equation}
{\rm (ii)}~при $a=0$$:$
\begin{equation}\label{CompoundPoisRSNormApproxErrBound(a=0)}
\rho\bigg(\frac{S_N}{\sqrt{\la\beta_2\E Y}},Z\bigg) \le\tfrac12(1+\d)(4+\d) \,\frac{\poisbec(\d)}{\la^{\d/2}}\cdot\frac{L_1^{2+\d}(X)\E Y^{1+\d/2}}{(\E Y)^{1+\d/2}},
\end{equation}
где константы $\poisbec(\d)$ определены в теореме~\ref{ThBEineqPoisSum}. В частности, при $\d=1$ имеем  $\poisbec(1)\hm\le0.3031,$ $\tfrac12(1+\d)(4+\d)\poisbec(\d)\le5\cdot0.3031<1.5155$ и, если $a=0,$ то 
$$
\rho\bigg(\frac{S_N}{\sqrt{\la\beta_2\E Y}},Z\bigg) \le 
 %5C(1)\,\frac{\E Y^{3/2}}{(\E Y)^{3/2}}\cdot \frac{\beta_3}{\sqrt\la} 
\frac{1.5155}{\sqrt\la}\cdot\frac{\beta_3}{\beta_2^{3/2}}\cdot\frac{\E Y^{3/2}}{(\E Y)^{3/2}},\quad \la>0.
$$
%\begin{eqnarray}
%\rho(\widetilde S_N,Z)&\le& \frac{\poisbec(\d)\big[ \beta_{2+\d}\E Y+\beta_2\beta_\d\E Y(Y-1)+ |a|\beta_{1+\d}\E Y^2(Y-1)\big]} {\la^{\d/2}[\sigma^2\E Y+a^2\E Y^2]^{1+\d/2}}
%\nonumber
%\\
%\rho(\widetilde S_N,Z)&\le& \frac{\poisbec(\d)}{\la^{\d/2}
%%[\beta_2\E Y+a^2\E Y(Y-1)]^{1+\d/2}}
%[\sigma^2\E Y+a^2\E Y^2]^{1+\d/2}}
%\big[ \beta_{2+\d}\E Y + \nonumber
%\\\label{CompoundPoisRSNormApproxErrBound}
%&&+\beta_2\beta_\d\E Y(Y-1)+ |a|\beta_{1+\d}\E Y^2(Y-1)\big],
%%(\beta_2\beta_\d+2|a|\beta_{1+\d})\E Y(Y-1)+ a^2\beta_\d\E Y(Y-1)(Y-2)\big],
%\\\label{CompoundPoisRSNormApproxErrBound_simplified}
%\rho(\widetilde S_N,Z)&\le& \frac{\poisbec(\d)}{\la^{\d/2}}\cdot \frac{\beta_{2+\d}\E Y^{2+\d}} {(\sigma^2\E Y+a^2\E Y^2)^{1+\d/2}},
%\end{eqnarray}
\end{theorem}

\comment{
С. Я. Шоргин, О точности нормальной аппроксимации распределений случайных сумм с безгранично делимыми индексами, Теория вероятн. и ее примен., 1996, том 41, выпуск 4, 920--926 (DOI: https://doi.org/10.4213/tvp3283), теорема 1, также процитировано в (Бенинг, Королев, Шоргин, 2011, теорема 2.4.6, стр.151): для $\d=1$
$$
\rho\big(\widetilde S_N, Z\big) \le \frac{\poisbec(1)}{\sqrt\la}\cdot \frac{\beta_3\E Y+3\beta_1\beta_2\E Y(Y-1)+\beta_1^3\E Y(Y-1)(Y-2)} {\sigma^2\E Y+a^2\E Y^2},
$$

Оценка для $\d=1$, $a=0$ и $M(1)=0.3041$ в виде
$$
\rho\bigg(\frac{S_N}{\sqrt{\la\beta_2\E Y}},Z\bigg) \le 
\frac{5M(1)}{\sqrt\la} \cdot\frac{\beta_3}{\beta_2^{3/2}}\cdot\frac{\E Y^{3/2}}{(\E Y)^{3/2}},\quad \la>0.
$$
была впервые получена в (Гавриленко, 2010)~\cite{Gavrilenko2010}

Таким образом, случай $\d=1$ рассматривается впервые.
}

Теорема~\ref{ThCompoundPoisRSNormApproxErrBound} для $\d=1$ и $a\neq0$ была получена в работе~\cite{Shorgin1996TVP}, для $\d=1$ и $a=0$~--- в~\cite{Gavrilenko2010} (также см.~\cite[теорема\,2.4.6, с.\,151]{KorolevBeningShorgin2011}). Случай $0<\d<1$ здесь рассматривается впервые.

%\begin{myremark}
%Для любой неотрицательной целочисленной с.в. $Y$
%$$
%\E Y(Y-1)\ge0,\quad \E Y(Y-1)(Y-2)\ge0.
%$$
%\end{myremark}

\begin{myremark}
Моменты с.в. $Y$ в правой части~\eqref{CompoundPoisRSNormApproxErrBound}  ``согласованы'' только при $a\neq0$: если с.в. $Y\eqd m Y_1,$ где $Y_1$~--- произвольная  неотрицательная целочисленная с.в. с фиксированным распределением, а натуральное  $m\to\infty,$  то правая часть~\eqref{CompoundPoisRSNormApproxErrBound} будет оставаться ограниченной при прочих фиксированных параметрах, если только $a\neq0$. При $a=0$ следует пользоваться оценкой~\eqref{CompoundPoisRSNormApproxErrBound(a=0)}, которая, к тому же, менее требовательна в отношении ``хвостов'' распределения с.в. $Y$: здесь достаточно, чтобы $\E Y^{1+\d/2}<\infty$ против требования $\E Y^{2+\d}<\infty$, необходимого для нетривиальности~\eqref{CompoundPoisRSNormApproxErrBound}.
\end{myremark}

\begin{proof}
По лемме~\ref{Lem((N_la,Y),X)=(N_la,(Y,X))} имеем $S_N\eqd\{N_\la,\{Y,X\}\}$, где, в силу невырожденности $N$ в нуле и упражнения~\ref{Ex:N=(N_la,Y),Y>0)}, случайную величину $Y$ можно выбрать положительной. Пусть $U\hm{:\eqd}\{Y,X\}\eqd X_1+\ldots+X_Y$. Записывая~\eqref{BEineqPoisSumForInfDivIndexes} с $U$ вместо $X$, получаем
$$
\rho(\widetilde S_N,Z)\le \frac{\poisbec(\d)}{\la^{\d/2}}L_1^{2+\d}(U) =\frac{\poisbec(\d)}{\la^{\d/2}}\cdot \frac{\E|U|^{2+\d}}{(\E U^2)^{1+\d/2}}.
$$
По элементарным свойствам случайных сумм (см. теорему~\ref{ThRandSumElemProp}(ii))
$$
\E U^2 =\D U+(\E U)^2=\sigma^2\E Y+a^2\D Y+a^2(\E Y)^2 =\sigma^2\E Y+a^2\E Y^2.
$$
%По  формуле полной вероятности имеем (как и ранее, считаем, что $\sum_{n=1}^0\coloneqq0$)
%\begin{multline*}
%\E U^2=\smallsum_{n=0}^\infty\E(X_1+\ldots+X_n)^2\Prob(Y=n)
%=\smallsum_{n=0}^\infty\left[\D S_n+(\E S_n)^2\right]\Prob(Y=n) =
%\\
%=\sigma^2\E Y+a^2\E Y^2.%=\beta_2\E Y+a^2\E Y(Y-1).
%\end{multline*}
Момент 
$$
\E|U|^{2+\d}=\smallsum_{n=0}^\infty\E|X_1+\ldots+X_n|^{2+\d}\Prob(Y=n)
$$
оценим с помощью неравенства~\eqref{(sum(x_i))^r<=n^(r-1)sum(x_i)^r} в общем случае и с помощью леммы~\ref{LemDharmadhikariJogdeo} в случае $a=0$. Имеем
$$
\E|X_1+\ldots+X_n|^{2+\d}\le n^{1+\d}\smallsum_{i=1}^n\E|X_i|^{2+\d}=n^{2+\d}\beta_{2+\d},\quad \forall a,
$$
$$
\E|X_1+\ldots+X_n|^{2+\d}\le K_{2+\d}\beta_{2+\d}\cdot n^{1+\d/2},\quad a=0,
$$
что приводит к оценкам $\E|U|^{2+\d}\le \beta_{2+\d}\E Y^{2+\d}$ в общем случае, и
$\E|U|^{2+\d}\le K_{2+\d}\beta_{2+\d}\E Y^{1+\d/2}$ в случае $a=0$. Оценивая теперь $L_1^{2+\d}(U)$ c помощью полученных оценок для  $\E|U|^{2+\d}$ и $\E U^2$, приходим к утверждению теоремы. 
\end{proof}

\begin{myremark}
Заметим, что можно скомбинировать оба метода оценивания абсолютных моментов суммы, процентрировав сначала сумму с помощью~\eqref{(sum(x_i))^r<=n^(r-1)sum(x_i)^r}  при  $n=2$ и применив затем лемму~\ref{LemDharmadhikariJogdeo} для оценивания момента уже центрированной суммы (можно потом еще раз применить~\eqref{(sum(x_i))^r<=n^(r-1)sum(x_i)^r}, чтобы перейти от центральных моментов слагаемых к нецентральным). Тогда получится некоторый промежуточный вариант между оценками~\eqref{CompoundPoisRSNormApproxErrBound} и~\eqref{CompoundPoisRSNormApproxErrBound(a=0)}, имеющий согласованное поведение моментов с.в. $Y$ как при $a=0$, так и при $a\neq0$, однако худшие константы и более громоздкий вид, а именно:
$$
\rho(\widetilde S_N,Z)\le \frac{2^\d\poisbec(\d)\big((1+\d)(4+\d)\E|X-a|^{2+\d}\E  Y^{1+\d/2} +2|a|^{2+\d}\E Y^{2+\d}\big)} {\la^{\d/2}(\sigma^2\E Y+a^2\E Y^2)^{1+\d/2}}.
$$
\end{myremark}

В качестве примера применения теоремы~\ref{ThCompoundPoisRSNormApproxErrBound} рассмотрим следующую задачу, связанную с примером~\ref{Exa:MultiplePoissonDistr}.

Страховая компания классифицирует автомобильные катастрофы в зависимости от числа пострадавших машин на одиночные, двойные, тройные и т.д., при этом предполагается, что число одиночных $\xi_1$, двойных $\xi_2$ и т.д. катастроф, произошедших за период $t$ дней, подчиняется пуассоновскому распределению со средним $\lambda_1t$, $\lambda_2t,\ldots$ и что эти случайные величины независимы. Страховые выплаты за каждый потерпевший автомобиль являются независимыми и одинаково распределенными случайными величинами со средним~$2$ у.е., дисперсией~$1$ и третьим абсолютным моментом~$12$. Оценить вероятность того, что суммарные выплаты страховой компании за год превзойдут $1600$ у.е., если $\lambda_k=2^{-k},$ $k=1,2,3,\ldots$

Пусть $X_1,X_2,\ldots$~--- (случайные) размеры страховых выплат по каждому страховому случаю из рассматриваемого портфеля. Тогда, по условию, $\{X_k\}_{k\ge1}$~--- н.о.р.с.в. с $a\coloneqq\E X_1=2,$ $\sigma^2\coloneqq\D X_1=1,$ $\beta_3\hm\coloneqq\E|X_1|^3=12$. Суммарные страховые выплаты за период $t$ дней имеют вид 
$$
S_N=\smallsum_{k=1}^{N} X_k,\quad\text{где } N=\sum_{k=1}^\infty k\xi_k,
$$
и, по условию, $\xi_k\sim Pois(\la_kt)$ независимы при $k=1,2,\ldots$ В примере~\ref{Exa:MultiplePoissonDistr} было показано, что $N\eqd\{N_\la,Y\}$, где 
$$
\la=\smallsum_{k=1}^\infty\la_kt=t\smallsum_{k=1}^\infty2^{-k}=t,\quad \Prob(Y=k)=\la_kt/\la=2^{-k},\ k\in\N,
$$
т.е. $Y-1\sim\mathcal G(0.5).$ Имеем 
$$
\E Y=2,\ \E Y^2=6,\ \E Y^3=26;\quad
\E N=\la\E Y=2t,\ \D N=\la\E Y^2=6t;
$$
$$
\E S_N=a\E N=2at,\quad \D S_N=\sigma^2\E N+a^2\D N=2t(\sigma^2+3a^2).
$$
Отметим, что средние выплаты за год $(t=365)$ равны $\E S_N=2at\hm=1460.$ Используя далее оценку~\eqref{CompoundPoisRSNormApproxErrBound}, получаем
\begin{multline*}
\Prob(S_N>1600)=1-\Phi\bigg(\frac{1600-\E S_N}{\sqrt{\D S_N}}\bigg) \pm\frac{0.3031\beta_3\E Y^3}{\sqrt\la(\sigma^2\E Y+a^2\E Y^2)^{3/2}}=
\\
=1-\Phi\bigg(\frac{1600-2at}{\sqrt{2t(\sigma^2+3a^2)}}\bigg) \pm \frac{0.3031\cdot26\beta_3}{\sqrt t(2\sigma^2+6a^2)^{3/2}}.
\end{multline*}
Подставляя в последнюю формулу значения $t=365,$ $a=2,$ $\sigma^2=1,$ $\beta_3=12,$ заданные в условии задачи, окончательно получаем
$$
\Prob(S_N>1600)=0.0753\ldots\pm0.0373\ldots\le0.1128,
$$
т.е. погрешность использованной нормальной аппроксимации не превышает $3.74\%$, и искомое событие гарантированно имеет вероятность не больше $11.28\%$. 

\smallskip

Рассмотрим еще один пример применения  теоремы~\ref{ThCompoundPoisRSNormApproxErrBound} для центрированных слагаемых и отрицательного биномиального индекса  $N$. Прежде всего найдем представление $N$ в виде пуассоновской случайной суммы.

\begin{lemma}[см.~\protect{\cite[Т.\,I, Гл.\,XII, \S\,2]{Feller1967}}] \label{LemNegBinom=CompoundPoisson}
Пусть с.в. $N\sim\NB(r,p),$ $r>0,$ $p\in(0,1).$ Тогда $N=\{N_\la,Y\},$ где $N_\la\sim Pois(\la)$ с $\la=r\ln \frac1p,$ а распределение положительной целочисленной с.в. $Y$ имеет вид:
\begin{equation}\label{LogDistrDef}
\Prob(Y=k)=\left (\ln\tfrac1p\right)^{-1}\cdot \frac{q^k}{k},\quad q=1-p, \quad k\in\N.
\end{equation}
\end{lemma}

\begin{proof}
Пр.ф. с.в. $N$
$$
\psi(z)=\sum_{k=0}^\infty z^kp^rq^k\frac{\Gamma(r+k)}{k!\Gamma(r)}= \left(\frac{p}{1-qz}\right)^r
$$
представима в виде
$$
\psi(z)=\exp\left\{ r\big[\ln p -\ln(1-qz)\big]\right\} \eqqcolon\exp\{\la(h(z)-1)\},
$$
где $\la\coloneqq -r\ln p=r\ln\frac1p$ и 
$$
h(z)\coloneqq -\frac{r}\la\ln(1-qz)= \frac{r}\la\sum_{k=1}^\infty \frac{(qz)^k}{k} = \sum_{k=1}^\infty z^k\Prob(Y=k)=
\E z^Y,
$$
$z\in\C,\ |z|\le1,$ является пр.ф. распределения с.в. $Y$, задаваемого формулой~\eqref{LogDistrDef}.
\end{proof}

\begin{definition}
Распределение, задаваемое формулой~\eqref{LogDistrDef}, называется  \textit{логарифмическим} или распределением  \textit{логарифмического ряда}~{\rm\cite{FisherCorbetWilliams1943}, \cite[\S\,5.16]{KendallStuart1945}}.
\end{definition}

Анализируя, например, характеристические функции, легко видеть, что стандартизованная отрицательная биномиальная случайная сумма $\widetilde S_N\dto Z\hm\sim N(0,1)$ при $r\to\infty,$ как только $\E X^2\in(0,\infty)$ (также см. упражнение~\ref{Ex:NegBin->Normal}). Оценим скорость этой сходимости в равномерной метрике с помощью оценки~\eqref{CompoundPoisRSNormApproxErrBound(a=0)} теоремы~\ref{ThCompoundPoisRSNormApproxErrBound}, предполагая, что $\E X=0$. Моменты с.в. $Y$ из леммы~\ref{LemNegBinom=CompoundPoisson}  имеют вид
$$
\E Y=-\frac1{\ln p}\sum_{k=1}^\infty q^k = -\frac{q}{p\ln p},\quad q=1-p,
$$
$$
\E Y^{1+\d/2}=-\frac1{\ln p}\sum_{k=1}^\infty k^{\d/2}q^k.
$$
Оценивая последний ряд с помощью неравенства Йенсена для вогнутой функции $g(x)=x^{\d/2}$ и геометрического распределения, получаем
$$
 \E Y^{1+\d/2}\ln\frac1p= \frac{q}p\smallsum_{k=1}^\infty k^{\d/2}\cdot pq^{k-1}\le \frac{q}p\Big(\smallsum_{k=1}^\infty k\cdot pq^{k-1}\Big)^{\d/2}= \frac{q}{p^{1+\d/2}},
$$
что приводит нас к итоговой оценке % $\E Y^{1+\d/2}\le(1-p)p^{-1-\d/2}(\ln\frac1p)^{-1},$
$$
\frac{\E Y^{1+\d/2}}{(\E Y)^{1+\d/2}}\le \bigg(\frac{-\ln p}{q}\bigg)^{\d/2} = \left(\frac{\la}{r(1-p)}\right)^{\d/2}.
$$
%Заметим также, что $\E N={r(1-p)}/p.$
При этом квадрат нормирующего множителя для случайной суммы~$S_N$ при $a\coloneqq\E X=0$ и, для упрощения обозначений, $\E X^2=1$ имеет вид
$$
\D S_N= \E N=\la\E Y= \tfrac{r(1-p)}p.
$$
\begin{corollary}
Пусть $N:\eqd N_{r,p}\sim\NB(r,p),$ $\E X=0,$ ${\E X^2=1},$ $\beta_{2+\d}\hm\coloneqq\E|X|^{2+\d}<\infty$ с некоторым $\d\in(0,1].$ Тогда  для $S_{r,p}\coloneqq\{N_{r,p},X\}$ при всех  $r>0$ и $p\in(0,1)$
\begin{equation}\label{NegBin(CompPois)-NormalApproxErrBound}
\rho\big(S_{r,p}\sqrt{\tfrac{p}{r(1-p)}},Z\big) \le
\tfrac12(1+\d)(4+\d)\poisbec(\d)\frac{\beta_{2+\d}}{[r(1-p)]^{\d/2}},
\end{equation}
в частности, при $\d=1$
$$
\rho\big(S_{r,p}\sqrt{\tfrac{p}{r(1-p)}},Z\big)\le \frac{1.5155\cdot\beta_3}{\sqrt{r(1-p)}},\quad r>0,\ p\in(0,1).
$$
\end{corollary}

\comment{Случай $\d=1$ доказан в (Гавриленко, 2010).
}

Сравнивая оценки точности нормальной аппроксимации для отрицательных биномиальных сумм, устанавливаемые неравенствами~\eqref{NegBin(CompPois)-NormalApproxErrBound} и~\eqref{NegBin-NormalApproxErrBound}, замечаем, что порядок по~$r$ у них одинаков:  $\mathcal O(r^{-\d/2})$, $r\to\infty$, однако~\eqref{NegBin-NormalApproxErrBound} еще и учитывает информацию о  малых значениях параметра $p$: например, если $p=r^{1-2/\d}\to0$, $r\to\infty$, то~\eqref{NegBin-NormalApproxErrBound} дает оценку порядка $O(r^{-1})$, тогда как  в~\eqref{NegBin(CompPois)-NormalApproxErrBound} скорость сходимости $\mathcal O(r^{-\d/2})$ остается такой же, как и при любом другом поведении $p\nrightarrow1$. Однако не стоит забывать, что оценка~\eqref{NegBin(CompPois)-NormalApproxErrBound} получена лишь как частный случай оценки более общего вида.

\clearpage
\addtocontents{toc}{\protect\nopagebreak\medskip}
\addcontentsline{toc}{section}{Литература}
\def\BibEmph#1{\emph{#1}}

\nocite{Senatov2009,ChristophUlyanov2020}
{
\footnotesize  

\bibliography{../../bib/my_pub_rus,../../bib/my_pub,../../bib/biblio_rus,../../bib/biblio,../../bib/BErefs}

\begin{thebibliography}{100}
\def\selectlanguageifdefined#1{
\expandafter\ifx\csname date#1\endcsname\relax
\else\selectlanguage{#1}\fi}
\providecommand*{\href}[2]{{\small #2}}
\providecommand*{\url}[1]{{\small #1}}
\providecommand*{\BibUrl}[1]{\url{#1}}
\providecommand{\BibAnnote}[1]{}
\providecommand*{\BibEmph}[1]{#1}
\ProvideTextCommandDefault{\cyrdash}{\iflanguage{russian}{\hbox
  to.8em{--\hss--}}{\textemdash}}
\providecommand*{\BibDash}{\ifdim\lastskip>0pt\unskip\nobreak\hskip.2em plus
  0.1em\fi
\cyrdash\hskip.2em plus 0.1em\ignorespaces}
\renewcommand{\newblock}{\ignorespaces}

\bibitem{Bikelis1966}
\selectlanguageifdefined{russian}
\BibEmph{Бикялис~A.} {Оценки остаточного члена в центральной предельной
  теореме}~// \BibEmph{Литов. матем. сб.} \BibDash
\newblock 1966. \BibDash
\newblock \CYRT.~6, {\cyr\textnumero}~3. \BibDash
\newblock {\cyr\CYRS.}~323--346.

\bibitem{Bogachev2003}
\selectlanguageifdefined{russian}
\BibEmph{Богачев~В.~И.} Основы теории меры. Т.\,1,\,2. \BibDash
\newblock Москва--Ижевск~: НИЦ ``Регулярная и хаотическая динамика'', 2003.

\bibitem{Bogachev2016}
\selectlanguageifdefined{russian}
\BibEmph{Богачев~В.~И.} Слабая сходимость мер. \BibDash
\newblock Москва--Ижевск~: Институт компьютерных исследований, 2016.

\bibitem{Gavrilenko2010}
\selectlanguageifdefined{russian}
\BibEmph{Гавриленко~С.~В.} {Оценки скорости сходимости распределений случайных
  сумм с безгранично делимыми индексами к нормальному закону}~//
  \BibEmph{Информ. и её примен.} \BibDash
\newblock 2010. \BibDash
\newblock \CYRT.~4, {\cyr\textnumero}~4. \BibDash
\newblock {\cyr\CYRS.}~80--87.

\bibitem{Gavrilenko2011}
\selectlanguageifdefined{russian}
\BibEmph{Гавриленко~С.~В.} {Уточнение неравномерных оценок скорости сходимости
  распределений пуассоновских случайных сумм к нормальному закону}~//
  \BibEmph{Информ. и её примен.} \BibDash
\newblock 2011. \BibDash
\newblock \CYRT.~5, {\cyr\textnumero}~1. \BibDash
\newblock {\cyr\CYRS.}~12--24.

\bibitem{GavrilenkoKorolev2006}
\selectlanguageifdefined{russian}
\BibEmph{Гавриленко~С.~В., Королев~В.~Ю.} {Оценки скорости сходимости смешанных
  пуассоновских случайных сумм}~// \BibEmph{Системы и средства информатики.
  Специальный выпуск.} \BibDash
\newblock 2006. \BibDash
\newblock {\cyr\CYRS.}~248--257.

\bibitem{Glivenko2007}
\selectlanguageifdefined{russian}
\BibEmph{Гливенко~В.~И.} Интеграл Стилтьеса. \BibDash
\newblock {II} {\cyr\cyri\cyrz\cyrd.} \BibDash
\newblock Москва~: УРСС, 2007.

\bibitem{GrigorievaPopov2012SMI}
\selectlanguageifdefined{russian}
\BibEmph{Григорьева~М.~Е., Попов~С.~В.} {О неравномерных оценках скорости
  сходимости в центральной предельной теореме}~// \BibEmph{Системы и средства
  информатики}. \BibDash
\newblock 2012. \BibDash
\newblock \CYRT.~22, {\cyr\textnumero}~1. \BibDash
\newblock {\cyr\CYRS.}~180--204.

\bibitem{DyachenkoUlyanov2002}
\selectlanguageifdefined{russian}
\BibEmph{Дьяченко~М.~И., Ульянов~П.~Л.} Мера и интеграл. \BibDash
\newblock Москва~: Факториал Пресс, 2002.

\bibitem{Zolotarev1964}
\selectlanguageifdefined{russian}
\BibEmph{Золотарёв~В.~М.} {Об асимптотически правильных константах в~уточнениях
  глобальной предельной теоремы}~// \BibEmph{Теория вероятн. примен.} \BibDash
\newblock 1964. \BibDash
\newblock \CYRT.~9, {\cyr\textnumero}~2. \BibDash
\newblock {\cyr\CYRS.}~293--302.

\bibitem{Zolotarev1965}
\selectlanguageifdefined{russian}
\BibEmph{Золотарёв~В.~М.} {О близости распределений двух сумм независимых
  случайных величин}~// \BibEmph{Теория вероятн. примен.} \BibDash
\newblock 1965. \BibDash
\newblock \CYRT.~10, {\cyr\textnumero}~3. \BibDash
\newblock {\cyr\CYRS.}~519--526.

\bibitem{Zolotarev1967a}
\selectlanguageifdefined{russian}
\BibEmph{Золотарёв~В.~М.} {Некоторые неравенства теории вероятностей и их
  применение к уточнению теоремы А.\,М.\,Ляпунова}~// \BibEmph{Докл. АН СССР}.
  \BibDash
\newblock 1967. \BibDash
\newblock \CYRT. 177, {\cyr\textnumero}~3. \BibDash
\newblock {\cyr\CYRS.}~501--504.

\bibitem{Zolotarev1983}
\selectlanguageifdefined{russian}
\BibEmph{Золотарёв~В.~М.} {Вероятностные метрики}~// \BibEmph{Теория вероятн.
  примен.} \BibDash
\newblock 1983. \BibDash
\newblock \CYRT.~28, {\cyr\textnumero}~2. \BibDash
\newblock {\cyr\CYRS.}~264--287.

\bibitem{Zolotarev1986}
\selectlanguageifdefined{russian}
\BibEmph{Золотарёв~В.~М.} Современная теория суммирования независимых случайных
  величин. \BibDash
\newblock Москва~: Наука, 1986.

\bibitem{ZubkSevastChist1989}
\selectlanguageifdefined{russian}
\BibEmph{Зубков~А.~М., Севастьянов~Б.~А., Чистяков~В.~П.} Сборник задач по
  теории вероятностей. \BibDash
\newblock Москва~: Наука, 1989.

\bibitem{KendallStuart1945}
\selectlanguageifdefined{russian}
\BibEmph{Кендалл~М.~Д., Стьюарт~А.} Теория распределений. \BibDash
\newblock Москва~: Физматлит, 1966.

\bibitem{KorolevBeningShorgin2011}
\selectlanguageifdefined{russian}
\BibEmph{Королев~В.~Ю., Бенинг~В.~Е., Шоргин~С.~Я.} Математические основы
  теории риска. \BibDash
\newblock 2 {\cyr\cyri\cyrz\cyrd.} \BibDash
\newblock Москва~: Физматлит, 2011.

\bibitem{KorolevPopov2012}
\selectlanguageifdefined{russian}
\BibEmph{Королев~В.~Ю., Попов~С.~В.} {Уточнение оценок скорости сходимости в
  центральной предельной теореме при ослабленных моментных условиях}~//
  \BibEmph{ДАН}. \BibDash
\newblock 2012. \BibDash
\newblock \CYRT. 445, {\cyr\textnumero}~3. \BibDash
\newblock {\cyr\CYRS.}~265--270.

\bibitem{KruglovKorolev1990}
\selectlanguageifdefined{russian}
\BibEmph{Круглов~В.~М., Королев~В.~Ю.} Предельные теоремы для случайных сумм.
  \BibDash
\newblock Москва~: МГУ, 1990.

\bibitem{Lukacs1970}
\selectlanguageifdefined{russian}
\BibEmph{Лукач~Е.} Характеристические функции. \BibDash
\newblock Москва~: Наука, 1979.

\bibitem{Lyapounov1900}
\selectlanguageifdefined{russian}
\BibEmph{Ляпунов~А.~М.} {Об одной теореме теории вероятностей}~//
  \BibEmph{Известия Академии Наук, V серия}. \BibDash
\newblock 1900. \BibDash
\newblock \CYRT.~13, {\cyr\textnumero}~4. \BibDash
\newblock {\cyr\CYRS.}~359--386.

\bibitem{Lyapounov1901}
\selectlanguageifdefined{russian}
\BibEmph{Ляпунов~А.~М.} {Новая форма теоремы о пределе вероятности}~//
  \BibEmph{Записки Академии Наук по физико-математическому отделению, VIII
  серия}. \BibDash
\newblock 1901. \BibDash
\newblock \CYRT.~12, {\cyr\textnumero}~5. \BibDash
\newblock {\cyr\CYRS.}~1--24.

\bibitem{Mirachmedov1984}
\selectlanguageifdefined{russian}
\BibEmph{Мирахмедов~Ш.~А.} {Об абсолютной постоянной в неравномерной оценке
  скорости сходимости в центральной предельной теореме}~// \BibEmph{Изв. АН
  УзССР, сер. физ.-мат. наук}. \BibDash
\newblock 1984. \BibDash
\newblock {\cyr\textnumero}~4. \BibDash
\newblock {\cyr\CYRS.}~26--30.

\bibitem{Nagaev1965}
\selectlanguageifdefined{russian}
\BibEmph{Нагаев~С.~В.} {Некоторые предельные теоремы для больших уклонений}~//
  \BibEmph{Теория вероятн. примен.} \BibDash
\newblock 1965. \BibDash
\newblock \CYRT.~10, {\cyr\textnumero}~2. \BibDash
\newblock {\cyr\CYRS.}~231--254.

\bibitem{NefedovaShevtsova2011}
\selectlanguageifdefined{russian}
\BibEmph{Нефедова~Ю.~С., Шевцова~И.~Г.} О точности нормальной аппроксимации для
  распределений пуассоновских случайных сумм~// \BibEmph{Информ. и её примен.}
  \BibDash
\newblock 2011. \BibDash
\newblock \CYRT.~5, {\cyr\textnumero}~1. \BibDash
\newblock {\cyr\CYRS.}~39--45.

\bibitem{NefedovaShevtsova2011DAN}
\selectlanguageifdefined{russian}
\BibEmph{Нефедова~Ю.~С., Шевцова~И.~Г.} Уточнение структуры неравномерных
  оценок скорости сходимости в центральной предельной теореме с приложением к
  пуассоновским случайным суммам~// \BibEmph{ДАН}. \BibDash
\newblock 2011. \BibDash
\newblock \CYRT. 440, {\cyr\textnumero}~5. \BibDash
\newblock {\cyr\CYRS.}~583--588.

\bibitem{NefedovaShevtsova2012}
\selectlanguageifdefined{russian}
\BibEmph{Нефедова~Ю.~С., Шевцова~И.~Г.} О неравномерных оценках скорости
  сходимости в центральной предельной теореме~// \BibEmph{Теория вероятн.
  примен.} \BibDash
\newblock 2012. \BibDash
\newblock \CYRT.~57, {\cyr\textnumero}~1. \BibDash
\newblock {\cyr\CYRS.}~62--97.

\bibitem{Osipov1966}
\selectlanguageifdefined{russian}
\BibEmph{Осипов~Л.~В.} {Уточнение теоремы Линдеберга}~// \BibEmph{Теория
  вероятн. примен.} \BibDash
\newblock 1965. \BibDash
\newblock \CYRT.~11, {\cyr\textnumero}~2. \BibDash
\newblock {\cyr\CYRS.}~339--342.

\bibitem{PaditzMirachmedov1986}
\selectlanguageifdefined{russian}
\BibEmph{Падитц~Л., Мирахмедов~Ш.~А.} {Письмо в редакцию (Замечание к оценке
  абсолютной постоянной в неравномерной оценке скорости сходимости в
  ц.п.т.)}~// \BibEmph{Изв. АН УзССР, сер. физ.-мат. наук}. \BibDash
\newblock 1986. \BibDash
\newblock {\cyr\textnumero}~3. \BibDash
\newblock {\cyr\CYRS.}~80.

\bibitem{Paulauskas1971}
\selectlanguageifdefined{russian}
\BibEmph{Паулаускас~В.~И.} {О неравенстве сглаживания}~// \BibEmph{Лит. матем.
  сб.} \BibDash
\newblock 1971. \BibDash
\newblock \CYRT.~11, {\cyr\textnumero}~4. \BibDash
\newblock {\cyr\CYRS.}~861--866.

\bibitem{Petrov1965}
\selectlanguageifdefined{russian}
\BibEmph{Петров~В.~В.} {Одна оценка отклонения распределения суммы независимых
  случайных величин от нормального закона}~// \BibEmph{Докл. АН СССР}. \BibDash
\newblock 1965. \BibDash
\newblock \CYRT. 160, {\cyr\textnumero}~5. \BibDash
\newblock {\cyr\CYRS.}~1013--1015.

\bibitem{Petrov1979}
\selectlanguageifdefined{russian}
\BibEmph{Петров~В.~В.} {Одна предельная теорема для сумм независимых
  неодинаково распределённых случайных величин}~// \BibEmph{Зап. научн. сем.
  ЛОМИ}. \BibDash
\newblock 1979. \BibDash
\newblock \CYRT.~85. \BibDash
\newblock {\cyr\CYRS.}~188--192.

\bibitem{Senatov2009}
\selectlanguageifdefined{russian}
\BibEmph{Сенатов~В.~В.} Центральная предельная теорема: точность аппроксимации
  и асимптотические разложения. \BibDash
\newblock Москва~: УРСС, Книжный дом Либроком, 2009.

\bibitem{StudnevIgnat1967}
\selectlanguageifdefined{russian}
\BibEmph{Студнев~Ю.~П., Игнат~Ю.~И.} {Об уточнении центральной предельной
  теоремы и ее глобального варианта}~// \BibEmph{Теория вероятн. примен.}
  \BibDash
\newblock 1967. \BibDash
\newblock \CYRT.~12, {\cyr\textnumero}~3. \BibDash
\newblock {\cyr\CYRS.}~562--567.

\bibitem{Tyurin2009DAN}
\selectlanguageifdefined{russian}
\BibEmph{Тюрин~И.~С.} {О точности гауссовской аппроксимации}~// \BibEmph{ДАН}.
  \BibDash
\newblock 2009. \BibDash
\newblock \CYRT. 429, {\cyr\textnumero}~3. \BibDash
\newblock {\cyr\CYRS.}~312--316.

\bibitem{Tyurin2010TVP}
\selectlanguageifdefined{russian}
\BibEmph{Тюрин~И.~С.} {О скорости сходимости в теореме Ляпунова}~//
  \BibEmph{Теория вероятн. примен.} \BibDash
\newblock 2010. \BibDash
\newblock \CYRT.~55, {\cyr\textnumero}~2. \BibDash
\newblock {\cyr\CYRS.}~250--270.

\bibitem{Feller1967}
\selectlanguageifdefined{russian}
\BibEmph{Феллер~В.} Введение в теорию вероятностей и её приложения. \BibDash
\newblock Москва~: Мир, 1984.

\bibitem{Halmos1950}
\selectlanguageifdefined{russian}
\BibEmph{Халмош~П.} Теория меры. \BibDash
\newblock Москва~: Издательство иностранной литературы, 1953.

\bibitem{HippMattner2007}
\selectlanguageifdefined{russian}
\BibEmph{Хипп~К., Маттнер~Л.} {On the normal approximation to symmetric
  binomial distributions}~// \BibEmph{Теория вероятн. примен.} \BibDash
\newblock 2007. \BibDash
\newblock \CYRT.~52, {\cyr\textnumero}~3. \BibDash
\newblock {\cyr\CYRS.}~610--617.

\bibitem{Chistyakov1990}
\selectlanguageifdefined{russian}
\BibEmph{Чистяков~Г.~П.} {Об одной задаче А.\,Н.\,Колмогорова}~// \BibEmph{Зап.
  научн. сем. ЛОМИ}. \BibDash
\newblock 1990. \BibDash
\newblock \CYRT. 184. \BibDash
\newblock {\cyr\CYRS.}~289--319.

\bibitem{Shevtsova2010SmoothIneq}
\selectlanguageifdefined{russian}
\BibEmph{Шевцова~И.~Г.} О неравенстве сглаживания~// \BibEmph{ДАН}. \BibDash
\newblock 2010. \BibDash
\newblock \CYRT. 430, {\cyr\textnumero}~5. \BibDash
\newblock {\cyr\CYRS.}~600--602.

\bibitem{Shevtsova2013Inf}
\selectlanguageifdefined{russian}
\BibEmph{Шевцова~И.~Г.} {Об абсолютных константах в неравенстве Берри--Эссеена
  и его структурных и неравномерных уточнениях}~// \BibEmph{Информ. и её
  примен.} \BibDash
\newblock 2013. \BibDash
\newblock \CYRT.~7, {\cyr\textnumero}~1. \BibDash
\newblock {\cyr\CYRS.}~124--125.

\bibitem{Shevtsova2014DAN}
\selectlanguageifdefined{russian}
\BibEmph{Шевцова~И.~Г.} {Об абсолютных константах в неравенствах типа
  Берри--Эссеена}~// \BibEmph{ДАН}. \BibDash
\newblock 2014. \BibDash
\newblock \CYRT. 456, {\cyr\textnumero}~6. \BibDash
\newblock {\cyr\CYRS.}~650--654.

\bibitem{Shevtsova2016}
\selectlanguageifdefined{russian}
\BibEmph{Шевцова~И.~Г.} Точность нормальной аппроксимации: методы оценивания и
  новые результаты. \BibDash
\newblock Москва~: Аргамак--Медиа, 2016.

\bibitem{Shevtsova2018}
\selectlanguageifdefined{russian}
\BibEmph{Шевцова~И.~Г.} {Оценки скорости сходимости в глобальной ЦПТ для
  обобщенных смешанных пуассоновских распределений}~//
  \href{http://dx.doi.org/10.4213/tvp5143}{\BibEmph{Теория вероятн. примен.}}
  \BibDash
\newblock 2018. \BibDash
\newblock \CYRT.~63, {\cyr\textnumero}~1. \BibDash
\newblock {\cyr\CYRS.}~89--116.

\bibitem{Shorgin1996TVP}
\selectlanguageifdefined{russian}
\BibEmph{Шоргин~С.~Я.} {О точности нормальной аппроксимации распределений
  случайных сумм с безгранично делимыми индексами}~//
  \href{http://dx.doi.org/10.4213/tvp3283}{\BibEmph{Теория вероятн. примен.}}
  \BibDash
\newblock 1996. \BibDash
\newblock \CYRT.~41, {\cyr\textnumero}~4. \BibDash
\newblock {\cyr\CYRS.}~920--926.

\bibitem{BarbourHall1984}
\selectlanguageifdefined{english}
\BibEmph{Barbour~A.~D., Hall~P.} {Stein's method and the Berry--Esseen
  theorem}~// \BibEmph{Australian Journal of Statistics}. \BibDash
\newblock 1984. \BibDash
\newblock Vol.~26. \BibDash
\newblock P.~8--15.

\bibitem{BeallRescia1953}
\selectlanguageifdefined{english}
\BibEmph{Beall~G., Rescia~R.~R.} {A generalization of Neyman’s contagious
  distribution}~// \BibEmph{Biometrics}. \BibDash
\newblock 1953. \BibDash
\newblock Vol.~9. \BibDash
\newblock P.~354--386.

\bibitem{BeningKorolev2002}
\selectlanguageifdefined{english}
\BibEmph{Bening~V.~E., Korolev~V.~Y.} Generalized Poisson Models and their
  Applications in Insurance and Finance. \BibDash
\newblock Utrecht, The Netherlands~: VSP, 2002.

\bibitem{Bergstrom1944}
\selectlanguageifdefined{english}
\BibEmph{Bergstr{\"o}m~H.} {On the central limit theorem}~// \BibEmph{Skand.
  Aktuarietidskr.} \BibDash
\newblock 1944. \BibDash
\newblock Vol.~27. \BibDash
\newblock P.~139--153.

\bibitem{Berry1941}
\selectlanguageifdefined{english}
\BibEmph{Berry~A.~C.} {The accuracy of the Gaussian approximation to the sum of
  independent variates}~// \BibEmph{Trans. Amer. Math. Soc.} \BibDash
\newblock 1941. \BibDash
\newblock Vol.~49. \BibDash
\newblock P.~122--136.

\bibitem{Bohman1963}
\selectlanguageifdefined{english}
\BibEmph{Bohman~H.} {To compute the distribution function when the
  characteristic function is known}~//
  \href{http://dx.doi.org/10.1080/03461238.1963.10404789}{\BibEmph{Skandinavsk
  Aktuarietidskrift}}. \BibDash
\newblock 1963. \BibDash
\newblock Vol. 1963, no. 1--2. \BibDash
\newblock P.~41--46.

\bibitem{CarrMadanChang1998}
\selectlanguageifdefined{english}
\BibEmph{Carr~P.~P., Madan~D.~B., Chang~E.~C.} {The variance gamma process and
  option pricing}~// \BibEmph{European Finance Review}. \BibDash
\newblock 1998. \BibDash
\newblock Vol.~2. \BibDash
\newblock P.~79--105.

\bibitem{ChenShao2001}
\selectlanguageifdefined{english}
\BibEmph{Chen~L. H.~Y., Shao~Q.~M.} {A non-uniform Berry--Esseen bound via
  Stein's method}~// \BibEmph{Probab. Theory Relat. Fields}. \BibDash
\newblock 2001. \BibDash
\newblock Vol. 120. \BibDash
\newblock P.~236--254.

\bibitem{ChristophUlyanov2020}
\selectlanguageifdefined{english}
\BibEmph{Christoph~G., Ulyanov~V.~V.} {Second Order Expansions for
  High-Dimension Low-Sample-Size Data Statistics in Random Setting}~//
  \href{http://dx.doi.org/10.3390/math8071151}{\BibEmph{Mathematics}}. \BibDash
\newblock 2020. \BibDash
\newblock Vol.~8. \BibDash
\newblock P.~1151.

\bibitem{Delaporte1960}
\selectlanguageifdefined{english}
\BibEmph{Delaporte~P.} {Un probl\`eme de tarification de l’assurance accidents
  d’automobile examin\'e par la statistique math\'ematique}~// Trans. 16th
  Intern. Congress of Actuaries. \BibDash
\newblock Vol.~2. \BibDash
\newblock 1960. \BibDash
\newblock P.~121--135.

\bibitem{DharmadhikariJogdeo1969}
\selectlanguageifdefined{english}
\BibEmph{Dharmadhikari~S.~W., Jogdeo~K.} {Bounds on moments of certain random
  variables}~// \BibEmph{Ann. Math. Statist.} \BibDash
\newblock 1969. \BibDash
\newblock Vol.~40, no.~4. \BibDash
\newblock P.~1506--1509.

\bibitem{DorofeevaKorolevZeifman2020}
\selectlanguageifdefined{english}
\BibEmph{Dorofeeva~A., Korolev~V., Zeifman~A.} {Bounds for the accuracy of
  invalid normal approximation}~// \BibEmph{arXiv:2010.13138}. \BibDash
\newblock 2020.

\bibitem{Eggenberger1924}
\selectlanguageifdefined{english}
\BibEmph{Eggenberger~F.} \href{http://dx.doi.org/10.3929/ethz-a-000114454}{Die
  Wahrscheinlichkeitsansteckung}~: {Ph.D. Thesis}~/ F.~Eggenberger~; ETH
  Z\"urich. \BibDash
\newblock Bern, 1924.

\bibitem{Erickson1973}
\selectlanguageifdefined{english}
\BibEmph{Erickson~R.~V.} {On an $L_p$ version of the Berry--Esseen theorem for
  independent and m-dependent variables}~// \BibEmph{Ann. Probab.} \BibDash
\newblock 1973. \BibDash
\newblock Vol.~1, no.~3. \BibDash
\newblock P.~497--503.

\bibitem{Esseen1942}
\selectlanguageifdefined{english}
\BibEmph{Esseen~C.-G.} {On the Liapounoff limit of error in the theory of
  probability}~// \BibEmph{Ark. Mat. Astron. Fys.} \BibDash
\newblock 1942. \BibDash
\newblock Vol. A28, no.~9. \BibDash
\newblock P.~1--19.

\bibitem{Esseen1945}
\selectlanguageifdefined{english}
\BibEmph{Esseen~C.-G.} {Fourier analysis of distribution functions. A
  mathematical study of the Laplace--Gaussian law}~// \BibEmph{Acta Math.}
  \BibDash
\newblock 1945. \BibDash
\newblock Vol.~77, no.~1. \BibDash
\newblock P.~1--125.

\bibitem{Esseen1956}
\selectlanguageifdefined{english}
\BibEmph{Esseen~C.-G.} {A moment inequality with an application to the central
  limit theorem}~// \BibEmph{Skand. Aktuarietidskr.} \BibDash
\newblock 1956. \BibDash
\newblock Vol.~39. \BibDash
\newblock P.~160--170.

\bibitem{Feller1935}
\selectlanguageifdefined{german}
\BibEmph{Feller~W.} {\"{U}ber den zentralen Genzwertsatz der
  Wahrscheinlichkeitsrechnung}~// \BibEmph{Math. Z.} \BibDash
\newblock 1935. \BibDash
\newblock {Bd.}~40. \BibDash
\newblock S.~521--559.

\bibitem{FisherCorbetWilliams1943}
\selectlanguageifdefined{english}
\BibEmph{Fisher~R.~A., Corbet~A.~S., Williams~C.~B.} {The relation between the
  number of species and the number of individuals in a random sample of an
  animal population}~// \href{http://dx.doi.org/10.2307/1411}{\BibEmph{J.
  Animal Ecology}}. \BibDash
\newblock 1943. \BibDash
\newblock Vol.~12, no.~1. \BibDash
\newblock P.~42--58.

\bibitem{GabdullinMakarenkoShevtsova2018}
\selectlanguageifdefined{english}
\BibEmph{Gabdullin~R., Makarenko~V., Shevtsova~I.} {Esseen--Rozovskii type
  estimates for the rate of convergence in the Lindeberg theorem}~//
  \BibEmph{J. Math. Sci.} \BibDash
\newblock 2018. \BibDash
\newblock Vol. 234, no.~6. \BibDash
\newblock P.~847--885.

\bibitem{Goldstein2009}
\selectlanguageifdefined{english}
\BibEmph{Goldstein~L.} {Bounds on the constant in the mean central limit
  theorem}~// \BibEmph{Ann. Probab.} \BibDash
\newblock 2010. \BibDash
\newblock Vol.~38, no.~4. \BibDash
\newblock P.~1672--1689. \BibDash
\newblock arXiv:0912.0726, 2009.

\bibitem{Grandell1997}
\selectlanguageifdefined{english}
\BibEmph{Grandell~J.} {Mixed Poisson processes}. \BibDash
\newblock London~: Chapman and Hall, 1997.

\bibitem{GreenwoodYule1920}
\selectlanguageifdefined{english}
\BibEmph{Greenwood~M., Yule~G.} {An inquiry into the nature of frequency
  distributions of multiple happenings with particular reference to the
  occurrence of multiple attacks of disease or of repeated accidents}~//
  \BibEmph{J. R. Stat. Soc.} \BibDash
\newblock 1920. \BibDash
\newblock Vol.~83, no.~2. \BibDash
\newblock P.~255--279.

\bibitem{Gurland1948}
\selectlanguageifdefined{english}
\BibEmph{Gurland~J.} Inversion formulae for the distribution of ratios~//
  \BibEmph{Ann. Math. Statistics}. \BibDash
\newblock 1948. \BibDash
\newblock Vol.~19, no.~2. \BibDash
\newblock P.~228– 237.

\bibitem{Gurland1958}
\selectlanguageifdefined{english}
\BibEmph{Gurland~J.} A generalized class of contagious distributions~//
  \BibEmph{Biometrics}. \BibDash
\newblock 1958. \BibDash
\newblock Vol.~14, no.~2. \BibDash
\newblock P.~229--249.

\bibitem{Holla1967}
\selectlanguageifdefined{english}
\BibEmph{Holla~M.~S.} {On a Poisson-inverse Gaussian distribution}~//
  \BibEmph{Metrika}. \BibDash
\newblock 1967. \BibDash
\newblock Vol.~11, no.~2. \BibDash
\newblock P.~115--121.

\bibitem{Irwin1968}
\selectlanguageifdefined{english}
\BibEmph{Irwin~J.~O.} {The generalized Waring distribution applied to accident
  theory}~// \BibEmph{J. R. Stat. Soc., Ser. A}. \BibDash
\newblock 1968. \BibDash
\newblock Vol. 130, no.~2. \BibDash
\newblock P.~205--225.

\bibitem{Katz1963}
\selectlanguageifdefined{english}
\BibEmph{Katz~M.~L.} {Note on the Berry--Esseen theorem}~// \BibEmph{Ann. Math.
  Statist.} \BibDash
\newblock 1963. \BibDash
\newblock Vol.~34. \BibDash
\newblock P.~1107--1108.

\bibitem{KorolevDorofeyeva2017}
\selectlanguageifdefined{english}
\BibEmph{Korolev~V., Dorofeyeva~A.} {Bounds of the accuracy of the normal
  approximation to the distributions of random sums under relaxed moment
  conditions}~// \BibEmph{Lith. Math. J.} \BibDash
\newblock 2017. \BibDash
\newblock Vol.~57, no.~1. \BibDash
\newblock P.~38--58.

\bibitem{Korolev1996}
\selectlanguageifdefined{english}
\BibEmph{Korolev~V.~Y.} {A general theorem on the limit behavior of
  superpositions of independent random processes with applications to Cox
  processes. Journal of Mathematical Sciences}~// \BibEmph{J. Math. Sci.}
  \BibDash
\newblock 1996. \BibDash
\newblock Vol.~81, no.~5. \BibDash
\newblock P.~2951--2956.

\bibitem{Kupper1965}
\selectlanguageifdefined{english}
\BibEmph{Kupper~J.} {Some aspects of cumulative risk}~// \BibEmph{Astin Bull.}
  \BibDash
\newblock 1965. \BibDash
\newblock Vol.~3. \BibDash
\newblock P.~85--103.

\bibitem{Lindeberg1922}
\selectlanguageifdefined{english}
\BibEmph{Lindeberg~J.~W.} {Eine neue Herleitung des Exponentialgesetzes in der
  Wahrscheinlichkeitsrechnung}~// \BibEmph{Mathematische Zeitschrift}. \BibDash
\newblock 1922. \BibDash
\newblock Vol.~15, no.~1. \BibDash
\newblock P.~211--225.

\bibitem{Loh1975}
\selectlanguageifdefined{english}
\BibEmph{Loh~W.~Y.} On the normal approximation for sums of mixing random
  variables~: Master thesis~/ W.~Y.~Loh~; Department of Mathematics, University
  of Singapore. \BibDash
\newblock 1975.

\bibitem{MadanSeneta1990}
\selectlanguageifdefined{english}
\BibEmph{Madan~D.~B., Seneta~E.} {The variance gamma (V.G.) model for share
  market returns}~// \BibEmph{Journal of Business}. \BibDash
\newblock 1990. \BibDash
\newblock Vol.~63, no.~4. \BibDash
\newblock P.~511--524.

\bibitem{MattnerShevtsova2019}
\selectlanguageifdefined{english}
\BibEmph{Mattner~L., Shevtsova~I.~G.} {An optimal Berry--Esseen type theorem
  for integrals of smooth functions}~// \BibEmph{Lat. Am. J. Probab. Math.
  Stat.} \BibDash
\newblock 2019. \BibDash
\newblock Vol.~16. \BibDash
\newblock P.~487--530.

\bibitem{Michel1981}
\selectlanguageifdefined{english}
\BibEmph{Michel~R.} {On the constant in the nonuniform version of the
  Berry--Esseen theorem}~// \BibEmph{Z. Wahrsch. Verw. Geb.} \BibDash
\newblock 1981. \BibDash
\newblock Vol.~55, no.~1. \BibDash
\newblock P.~109--117.

\bibitem{Neammanee2005}
\selectlanguageifdefined{english}
\BibEmph{Neammanee~K.} {On the constant in the nonuniform version of the
  Berry--Esseen theorem}~// \BibEmph{International Journal of Mathematics and
  Mathematical Sciences}. \BibDash
\newblock 2005. \BibDash
\newblock Vol.~12. \BibDash
\newblock P.~1951--–1967.

\bibitem{NeammaneeThongtha2007}
\selectlanguageifdefined{english}
\BibEmph{Neammanee~K., Thongtha~P.} {Improvement of the non-uniform version of
  Berry--Esseen inequality via Paditz--Siganov theorems}~// \BibEmph{J.
  Inequal. Pure and Appl. Math.} \BibDash
\newblock 2007. \BibDash
\newblock Vol.~8, no.~4. \BibDash
\newblock P.~1--20.

\bibitem{Paditz1976}
\selectlanguageifdefined{english}
\BibEmph{Paditz~L.} {Absch\"atzungen der Konvergenzgeschwindigkeit im zentralen
  Grenzwertsatz}~// \BibEmph{Wiss. Z. der TU Dresden}. \BibDash
\newblock 1976. \BibDash
\newblock Vol.~25. \BibDash
\newblock P.~1169--1177.

\bibitem{Paditz1977}
\selectlanguageifdefined{english}
\BibEmph{Paditz~L.} \"{U}ber die Ann\"{a}herung der Verteilungsfunktionen von
  Summen unabh\"{a}ngiger Zufallsgr\"{o}{\ss}en gegen unbegrenzt teilbare
  Verteilungsfunktionen unter besonderer Beachtung der Verteilungsfunktion der
  standardisierten Normalverteilung~: Dissertation {A}~/ L.~Paditz~; Technische
  Universit\"{a}t Dresden. \BibDash
\newblock Dresden, 1977.

\bibitem{Paditz1978}
\selectlanguageifdefined{english}
\BibEmph{Paditz~L.} {Absch\"atzungen der Konvergenzgeschwindigkeit zur
  Normalverteilung unter Voraussetzung einseitiger Momente}~// \BibEmph{Math.
  Nachr.} \BibDash
\newblock 1978. \BibDash
\newblock Vol.~82. \BibDash
\newblock P.~131--156.

\bibitem{Paditz1979}
\selectlanguageifdefined{english}
\BibEmph{Paditz~L.} {\"Uber eine Fehlerabsch\"atzung im zentralen
  Grenzwertsatz}~// \BibEmph{Wiss. Z. der TU Dresden}. \BibDash
\newblock 1979. \BibDash
\newblock Vol.~28, no.~5. \BibDash
\newblock P.~1197--1200.

\bibitem{Paditz1989}
\selectlanguageifdefined{english}
\BibEmph{Paditz~L.} {On the analytical structure of the constant in the
  nonuniform version of the Esseen inequality}~// \BibEmph{Statistics
  $($Berlin: Akademie-Verlag$)$}. \BibDash
\newblock 1989. \BibDash
\newblock Vol.~20, no.~3. \BibDash
\newblock P.~453--464.

\bibitem{Pinelis2013}
\selectlanguageifdefined{english}
\BibEmph{Pinelis~I.} {On the nonuniform Berry--Esseen bound}~//
  \BibEmph{arXiv:1301.2828}. \BibDash
\newblock 2013.

\bibitem{Prawitz1972}
\selectlanguageifdefined{english}
\BibEmph{Prawitz~H.} Limits for a distribution, if the characteristic function
  is given in a finite domain~// \BibEmph{{Skand. Aktuarietidskr.}} \BibDash
\newblock 1972. \BibDash
\newblock Vol.~55. \BibDash
\newblock P.~138--154.

\bibitem{Prawitz1973}
\selectlanguageifdefined{english}
\BibEmph{Prawitz~H.} Ungleichungen f\"ur den absoluten betrag einer
  charackteristischen funktion~// \BibEmph{{Skand. Aktuarietidskr.}} \BibDash
\newblock 1973. \BibDash
\newblock no.~1. \BibDash
\newblock P.~11--16.

\bibitem{Prawitz1975}
\selectlanguageifdefined{english}
\BibEmph{Prawitz~H.} {On the remainder in the central limit theorem.~I:
  One-dimensional independent variables with finite absolute moments of third
  order}~// \BibEmph{{Scand. Actuar. J.}} \BibDash
\newblock 1975. \BibDash
\newblock no.~3. \BibDash
\newblock P.~145--156.

\bibitem{Prawitz1991}
\selectlanguageifdefined{english}
\BibEmph{Prawitz~H.} {Noch einige Ungleichungen f\"ur charakteristische
  Funktionen}~// \BibEmph{{Scand. Actuar. J.}} \BibDash
\newblock 1991. \BibDash
\newblock no.~1. \BibDash
\newblock P.~49--73.

\bibitem{Quinkert1957}
\selectlanguageifdefined{english}
\BibEmph{Quinkert~W.} Die kollektive Risikotheorie unter Ber\"ucksichtigung
  schwankender Grundwahrscheinlichkeiten mit endlichen Schwankungsbereich~:
  {Ph.D. Thesis}~/ W.~Quinkert~; University of Cologne. \BibDash
\newblock Cologne, 1957.

\bibitem{Rachev1991}
\selectlanguageifdefined{english}
\BibEmph{Rachev~S.} Probability Metrics and the Stability of Stochastic Models.
  \BibDash
\newblock Chichester, England~: John Wiley ans Sons, 1991.

\bibitem{Schulz2016}
\selectlanguageifdefined{english}
\BibEmph{Schulz~J.} The Optimal Berry--Esseen Constant in the Binomial Cases~:
  {Ph.D. Thesis}~/ J.~Schulz~; Trier University. \BibDash
\newblock Trier, Germany, 2016. \BibDash
\newblock Access mode: \BibUrl{http://ubt.opus.hbz-nrw.de/volltexte/2016/1007}.

\bibitem{Seal1978}
\selectlanguageifdefined{english}
\BibEmph{Seal~H.} Survival Probabilities (The Goal of Risk Theory). \BibDash
\newblock Chichester~: Wiley, 1978.

\bibitem{Shevtsova2014JMAA}
\selectlanguageifdefined{english}
\BibEmph{Shevtsova~I.} {On the accuracy of the approximation of the complex
  exponent by the first terms of its Taylor expansion with applications}~//
  \BibEmph{Journal of Mathematical Analysis and Applications}. \BibDash
\newblock 2014. \BibDash
\newblock Vol. 418, no.~1. \BibDash
\newblock P.~185--210.

\bibitem{Shevtsova2017book}
\selectlanguageifdefined{english}
\BibEmph{Shevtsova~I.~G.} On the absolute constants in Nagaev--Bikelis--type
  inequalities~// Inequalities and Extremal Problems in Probability and
  Statistics~/ I.~Pinelis, V.~H.~de~la~Pe\~{n}a, R.~Ibragimov et~al.~; Ed.\ by\
  I.~Pinelis. \BibDash
\newblock London~: Elsevier. \BibDash
\newblock P.~47--102. \BibDash
\newblock
  ISBN:~\href{http://isbndb.com/search-all.html?kw=978-0-12-809818-9}{978-0-12-809818-9}.
  \BibDash
\newblock Access mode:
  \BibUrl{https://www.elsevier.com/books/inequalities-and-extremal-problems-in-probability-and-statistics/pinelis/978-0-12-809818-9}.

\bibitem{Sichel1971}
\selectlanguageifdefined{english}
\BibEmph{Sichel~H.~S.} {On a family of discrete distributions particular suited
  to represent long tailed frequency data}~// Proceedings of the Third
  Symposium on Mathematical Statistics~/ Ed.\ by\ N.~F.~Laubscher. \BibDash
\newblock 1971. \BibDash
\newblock P.~51--97.

\bibitem{Skellam1952}
\selectlanguageifdefined{english}
\BibEmph{Skellam~J.~G.} {Studies in statistical ecology: I. Spatial pattern}~//
  \href{http://dx.doi.org/10.2307/2334030}{\BibEmph{Biometrika}}. \BibDash
\newblock 1952. \BibDash
\newblock Vol.~39, no. 3/4. \BibDash
\newblock P.~346--362.

\bibitem{ThongthaNeammanee2007}
\selectlanguageifdefined{english}
\BibEmph{Thongtha~P., Neammanee~K.} {Refinement on the constants in the
  non-uniform version of the Berry--Esseen theorem}~// \BibEmph{Thai Journal of
  Mathematics}. \BibDash
\newblock 2007. \BibDash
\newblock Vol.~5. \BibDash
\newblock P.~1--13.

\bibitem{Tysiak1983}
\selectlanguageifdefined{english}
\BibEmph{Tysiak~W.} Gleichm{\"a}\ss ige und nicht-gleichm{\"a}\ss ige
  Berry--Esseen Absch{\"a}tzungen~: {Dissertation}~/ W.~Tysiak~;
  Gesamthochschule Wuppertal. \BibDash
\newblock Wuppertal, 1983.

\bibitem{Tyurin2012}
\selectlanguageifdefined{english}
\BibEmph{Tyurin~I.~S.} {Some optimal bounds in CLT using zero biasing}~//
  \BibEmph{Stat. Prob. Letters}. \BibDash
\newblock 2012. \BibDash
\newblock Vol.~82, no.~3. \BibDash
\newblock P.~514--518.

\bibitem{Vaaler1985}
\selectlanguageifdefined{english}
\BibEmph{Vaaler~J.~D.} {Some extremal functions in Fourier analysis}~//
  \BibEmph{Bull. Amer. Math. Soc. (N.S.)}. \BibDash
\newblock 1985. \BibDash
\newblock Vol.~12, no.~2. \BibDash
\newblock P.~183--216.

\bibitem{VanBeek1971}
\selectlanguageifdefined{english}
\BibEmph{Van~Beek~P.} Fourier-analytische Methoden zur Versch\"{a}rfung der
  Berry--Esseen Schranke~: {Doctoral dissertation}~/ P.~Van~Beek~; Friedrich
  Wilhelms Universit\"{a}t. \BibDash
\newblock Bonn, 1971.

\bibitem{Wendel1948}
\selectlanguageifdefined{english}
\BibEmph{Wendel~J.~G.} {Note on the Gamma Function}~// \BibEmph{The American
  Mathematical Monthly}. \BibDash
\newblock 1948. \BibDash
\newblock Vol.~55, no.~9. \BibDash
\newblock P.~563--564.

\bibitem{Willmot1987}
\selectlanguageifdefined{english}
\BibEmph{Willmot~G.~E.} {The Poisson–inverse Gaussian distribution as an
  alternative to the negative binomial}~//
  \href{http://dx.doi.org/10.1080/03461238.1987.10413823}{\BibEmph{{Scand.
  Actuar. J.}}} \BibDash
\newblock 1987. \BibDash
\newblock Vol. 1987, no. 3--4. \BibDash
\newblock P.~113--127.

\bibitem{Zahl1966}
\selectlanguageifdefined{english}
\BibEmph{Zahl~S.} {Bounds for the central limit theorem error}~// \BibEmph{SIAM
  J. Appl. Math.} \BibDash
\newblock 1966. \BibDash
\newblock Vol.~14, no.~6. \BibDash
\newblock P.~1225--1245.

\bibitem{Zolotarev1967b}
\selectlanguageifdefined{english}
\BibEmph{Zolotarev~V.~M.} {A sharpening of the inequality of Berry--Esseen}~//
  \BibEmph{Z. Wahrsch. Verw. Geb.} \BibDash
\newblock 1967. \BibDash
\newblock Vol.~8. \BibDash
\newblock P.~332--342.

\end{thebibliography}
\bibliographystyle{gost2008s}

}

\label{last_page}

% ПРОВЕРИТЬ:
% 1) отстутствие номера страницы на первой странице Списка литературы (файл .bbl)
% 2) Последняя строка Содержания (Литература) не переносится на отдельную страницу
% 3) Удалить опцию draft
% 

\end{document}